\documentclass[10pt,a4paper]{amsart}
\usepackage{amsmath, amscd, amsfonts, amssymb, amsthm, latexsym, url, color, todonotes, bm, framed, rotating} 
\usepackage{graphicx}
\usepackage[left=1.23in, right=1.23in, top=1in, bottom=1.3in, includefoot, headheight=13.6pt]{geometry}
\usepackage[colorlinks=true,hyperindex,citecolor=blue,linkcolor=magenta]{hyperref}
\input{xy}
\xyoption{all}

\newtheorem{theorem}{Theorem}[section]
\newtheorem{lemma}[theorem]{Lemma}
\newtheorem{corollary}[theorem]{Corollary}
\newtheorem{proposition}[theorem]{Proposition}
\newtheorem{definition}[theorem]{Definition}

\newtheorem{assumption}[theorem]{Assumption}

\theoremstyle{definition}
\newtheorem{remark}[theorem]{Remark}
\newtheorem{example}[theorem]{Example}
\newtheorem{construction}[theorem]{Construction}

\newcommand{\xysquare}[8]{
\[\xymatrix{
#1 \ar@{#5}[r] \ar@{#6}[d] & #2 \ar@{#7}[d]\\
#3 \ar@{#8}[r] & #4
}\]
}

\DeclareMathOperator*{\projlimf}{``\varprojlim''}

\newcommand{\bb}{\mathbb}

\newcommand{\blob}{\bullet}

\newcommand{\comment}[1]{}

\newcommand{\ep}{\varepsilon}

\newcommand{\isoto}{\stackrel{\simeq}{\to}}
\newcommand{\Isoto}{\stackrel{\simeq}{\longrightarrow}}
\newcommand{\isofrom}{\stackrel{\simeq}{\leftarrow}}

\newcommand{\onto}{\twoheadrightarrow}
\newcommand{\op}{\operatorname}
\newcommand{\pid}[1]{\langle #1 \rangle}
\renewcommand{\phi}{\varphi}
\newcommand{\quis}{\stackrel{\sim}{\to}}

\newcommand{\roi}{\mathcal{O}}

\newcommand{\sub}[1]{{\mbox{\rm \scriptsize #1}}}

\newcommand{\To}{\longrightarrow}
\newcommand{\ul}[1]{\underline{#1}}

\newcommand{\xto}{\xrightarrow}

\newcommand{\gS}{\mathfrak{S}}

\newcommand{\cal}{\mathcal}
\renewcommand{\hat}{\widehat}
\renewcommand{\frak}{\mathfrak}
\newcommand{\indlim}{\varinjlim}
\renewcommand{\tilde}{\widetilde}
\renewcommand{\Im}{\operatorname{Im}}
\renewcommand{\ker}{\operatorname{Ker}}
\newcommand{\coker}{\operatorname{Coker}}
\renewcommand{\projlim}{\varprojlim}
\newcommand{\dR}{{\mathrm{dR}}}
\renewcommand{\inf}{{\mathrm{inf}}}
\newcommand{\crys}{{\mathrm{crys}}}
\newcommand{\cont}{{\mathrm{cont}}}

\newcommand{\cycl}{{\mathrm{cycl}}}

\DeclareMathOperator{\Ann}{Ann}

\DeclareMathOperator{\dlog}{dlog}

\DeclareMathOperator{\Ext}{Ext}

\DeclareMathOperator{\Hom}{Hom}

\DeclareMathOperator{\Lie}{Lie}

\DeclareMathOperator{\Spec}{Spec}
\DeclareMathOperator{\Spa}{Spa}
\DeclareMathOperator{\Spf}{Spf}

\DeclareMathOperator{\Tor}{Tor}

\newcommand{\dotimes}{\otimes^{\bb L}}
\newcommand{\xTo}[1]{\stackrel{#1}{\To}}

\begin{document}

\title{Integral $p$-adic Hodge theory}

\author{Bhargav Bhatt}
\author{Matthew Morrow}
\author{Peter Scholze}

\begin{abstract} We construct a new cohomology theory for proper smooth (formal) schemes over the ring of integers of $\bb C_p$. It takes values in a mixed-characteristic analogue of Dieudonn\'e modules, which was previously defined by Fargues as a version of Breuil--Kisin modules. Notably, this cohomology theory specializes to all other known $p$-adic cohomology theories, such as crystalline, de~Rham and \'etale cohomology, which allows us to prove strong integral comparison theorems.

The construction of the cohomology theory relies on Faltings' almost purity theorem, along with a certain functor $L\eta$ on the derived category, defined previously by Berthelot--Ogus. On affine pieces, our cohomology theory admits a relation to the theory of de~Rham--Witt complexes of Langer--Zink, and can be computed as a $q$-deformation of de~Rham cohomology.
\end{abstract}

\date{\today}

\maketitle

\tableofcontents

\setcounter{section}{0}

\section{Introduction}

This paper deals with the following question: as an algebraic variety degenerates from characteristic $0$ to characteristic $p$, how does its cohomology degenerate?

\subsection{Background}

To explain the meaning and the history of the above question, let us fix some notation. Let $K$ be a finite extension of $\bb Q_p$, and let $\roi_K\subset K$ be its ring of integers. Let $\frak X$ be a proper smooth scheme over $\roi_K$;\footnote{We use the fractal letter for consistency with the main body of the paper, where $\frak X$ will be allowed to be a formal scheme.} in other words, we consider only the case of good reduction in this paper, although we expect our methods to generalize substantially. Let $k$ be the residue field of $\roi$, and let $\bar{k}$ and $\bar{K}$ be algebraic closures.

There are many different cohomology theories one can associate to this situation. The best understood theory is $\ell$-adic cohomology for $\ell\neq p$. In that case, we have \'etale cohomology groups $H^i_\sub{\'et}(\frak X_{\bar{K}},\bb Z_\ell)$ and $H^i_\sub{\'et}(\frak X_{\bar{k}},\bb Z_\ell)$, and proper smooth base change theorems in  \'etale cohomology imply that these cohomology groups are canonically isomorphic (once one fixes a specialization of geometric points),
\[
H^i_\sub{\'et}(\frak X_{\bar{K}},\bb Z_\ell)\cong H^i_\sub{\'et}(\frak X_{\bar{k}},\bb Z_\ell)\ .
\]
In particular, the action of the absolute Galois group $G_K$ of $K$ on the left side factors through the action of the absolute Galois group $G_k$ of the residue field $k$ on the right side; i.e., the action of $G_K$ is unramified.

Grothendieck raised the question of understanding what happens in the case $\ell=p$. In that case, one still has well-behaved \'etale cohomology groups $H^i_\sub{\'et}(\frak X_{\bar{K}},\bb Z_p)$ of the generic fibre, but the \'etale cohomology groups of the special fibre are usually too small; for example, if $i=1$, they capture at best half of the \'etale cohomology of the generic fibre. A related phenomenon is that the action of $G_K$ on $H^i_\sub{\'et}(\frak X_{\bar{K}},\bb Z_p)$ is much more interesting than in the $\ell$-adic case; in particular, it is usually not unramified. As a replacement for the \'etale cohomology groups of the special fibre, Grothendieck defined the crystalline cohomology groups $H^i_\crys(\frak X_k/W(k))$. These are Dieudonn\'e modules, i.e.~finitely generated $W(k)$-modules equipped with a Frobenius operator $\phi$ which is invertible up to a power of $p$. However, $H^i_\sub{\'et}(\frak X_{\bar{K}},\bb Z_p)$ and $H^i_\crys(\frak X_k/W(k))$ are cohomology theories of very different sorts: the first is a variant of singular cohomology, whereas the second is a variant of de~Rham cohomology. Over the complex numbers $\bb C$, integration of differential forms along cycles and the Poincar\'e lemma give a comparison between the two, but algebraically the two objects are quite unrelated. Grothendieck's question of the mysterious functor was to understand the relationship between $H^i_\sub{\'et}(\frak X_{\bar{K}},\bb Z_p)$ and $H^i_\crys(\frak X_k/W(k))$, and ideally describe each in terms of the other.

Fontaine obtained the conjectural answer to this question, using his period rings, after inverting $p$, in \cite{Fontainepadicrepr}. Notably, he defined a $W(k)[\tfrac 1p]$-algebra $B_\crys$ whose definition will be recalled below, which comes equipped with actions of a Frobenius $\phi$ and of $G_K$, and he conjectured the existence of a natural $\phi,G_K$-equivariant isomorphism
\[
H^i_\sub{\'et}(\frak X_{\bar{K}},\bb Q_p)\otimes_{\bb Q_p} B_\crys\cong H^i_\crys(\frak X_k/W(k))[\tfrac 1p]\otimes_{W(k)[\frac 1p]} B_\crys\ .
\]
The existence of such an isomorphism was proved by Tsuji, \cite{Tsuji}, after previous work by Fontaine--Messing, \cite{FontaineMessing}, Bloch--Kato, \cite{BlochKato}, and Faltings, \cite{FaltingsJAMS}. This allows one to recover $H^i_\crys(\frak X_k/W(k))[\tfrac 1p]$ from $H^i_\sub{\'et}(\frak X_{\bar{K}},\bb Q_p)$ by the formula
\[
H^i_\crys(\frak X_k/W(k))[\tfrac 1p] = (H^i_\crys(\frak X_k/W(k))[\tfrac 1p]\otimes_{W(k)[\frac 1p]} B_\crys)^{G_K}\cong (H^i_\sub{\'et}(\frak X_{\bar{K}},\bb Q_p)\otimes_{\bb Q_p} B_\crys)^{G_K}\ .
\]
Conversely, Fontaine showed that one can recover $H^i_\sub{\'et}(\frak X_{\bar{K}},\bb Q_p)$ from $H^i_\crys(\frak X_k/W(k))[\tfrac 1p]$ together with the Hodge filtration coming from the identificaton $H^i_\crys(\frak X_k/W(k))\otimes_{W(k)} K = H^i_{\dR}(\frak X_K)$.

Unfortunately, when $p$ is small or $K/\bb Q_p$ is ramified, the integral structure is not preserved by these isomorphisms; only when $ie<p-1$, where $e$ is the ramification index of $K/\bb Q_p$, most of the story works integrally, roughly using the integral version $A_\crys$ of $B_\crys$ instead, as for example in work of Caruso, \cite{Caruso}; cf.~also work of Faltings, \cite{FaltingsRamified}, in the case $i<p-1$ with $e$ arbitrary.

\subsection{Results}

In this paper, we make no restriction of the sort mentioned above: very ramified extensions and large cohomological degrees are allowed throughout. Our first main theorem is the following; it is formulated in terms of formal schemes for wider applicability, and it implies that the torsion in the crystalline cohomology is an upper bound for the torsion in the \'etale cohomology.

\begin{theorem}\label{ThmA} Let $\frak X$ be a proper smooth formal scheme over $\roi_K$, where $\roi_K$ is the ring of integers in a complete discretely valued nonarchimedean extension $K$ of $\bb Q_p$ with perfect residue field $k$. Let $C$ be a completed algebraic closure of $K$, and write $\mathfrak{X}_C$ for the (geometric) rigid-analytic generic fibre of $\mathfrak{X}$. Fix some $i \geq 0$.
\begin{enumerate}
\item There is a comparison isomorphism
\[
H^i_\sub{\'et}(\frak X_{C},\bb Z_p)\otimes_{\bb Z_p} B_\crys\cong H^i_\crys(\frak X_k/W(k))\otimes_{W(k)} B_\crys\ ,
\]
compatible with the Galois and Frobenius actions, and the filtration. In particular, $H^i_\sub{\'et}(\frak X_{C},\bb Q_p)$ is a crystalline Galois representation.
\item For all $n\geq 0$, we have the inequality
\[
\mathrm{length}_{W(k)} (H^i_\crys(\frak X_k/W(k))_\sub{tor}/p^n)\geq \mathrm{length}_{\bb Z_p} (H^i_\sub{\'et}(\frak X_{C},\bb Z_p)_\sub{tor}/p^n)\ .
\]
In particular, if $H^i_\crys(\frak X_k/W(k))$ is $p$-torsion-free\footnote{We show that this is equivalent to requiring $H^i_\dR(\frak X)$ being a torsion-free $\roi_K$-module (for any fixed $i$).}, then so is $H^i_\sub{\'et}(\frak X_{C},\bb Z_p)$.
\item Assume that $H^i_\crys(\frak X_k/W(k))$ and $H^{i+1}_\crys(\frak X_k/W(k))$ are $p$-torsion-free. Then one can recover $H^i_\crys(\frak X_k/W(k))$ with its $\varphi$-action from $H^i_\sub{\'et}(\frak X_{C},\bb Z_p)$ with its $G_K$-action.
\end{enumerate}
\end{theorem}

Part (i) is the analogue of Fontaine's conjecture for proper smooth formal schemes over $\roi_K$. In fact, our methods work more generally: we directly prove the comparison isomorphism in (i) and the inequalities in (ii) (as well as a variant of (iii), formulated below) for any proper smooth formal scheme that is merely defined over $\roi_C$. For formal schemes over discretely valued base fields, part (i) has also been proved recently by Colmez--Niziol, \cite{ColmezNiziolBst} (in the more general case of semistable reduction), and Tan--Tong, \cite{TanTongCcrys} (in the absolutely unramified case, building on previous work of Andreatta--Iovita, \cite{AndreattaIovita}).

Intuitively, part (ii) says the following. If one starts with a proper smooth variety over the complex numbers $\bb C$, then the comparison between de~Rham and singular (co)homology says that any class in singular homology gives an obstruction to integrating differential forms: the integral over the corresponding cycle has to be zero. However, for torsion classes, this is not an actual obstruction: a multiple of the integral, and thus the integral itself, is always zero. Nevertheless, part (ii) implies the following  inequality:
\begin{equation}
\label{IneqdRet}
 \dim_k H^i_\dR(\mathfrak{X}_k) \geq \dim_{\mathbb{F}_p} H^i_{\sub{\'et}}(\mathfrak{X}_{C},\mathbb{F}_p).
 \end{equation}
In other words, $p$-torsion classes in singular homology still produce non-zero obstructions to integrating differential forms on any (good) reduction modulo $p$ of the variety. The relation is however much more indirect, as there is no analogue of ``integrating a differential form against a cycle'' in the $p$-adic world.

\begin{remark}
Theorem~\ref{ThmA} (ii) ``explains'' certain pathologies in algebraic geometry in characteristic $p$. For example, it was observed (by classification and direct calculation, see \cite[Corollaire 7.3.4 (a)]{IllusieDRWitt}) that for any Enriques surface $S_k$ over a perfect field $k$ of characteristic $2$, the group $H^1_{\dR}(S_k)$ is never $0$, contrary to what happens in any other characteristic. Granting the fact that any such $S_k$ lifts to characteristic $0$ (which is known, see \cite{EkedahlMO,LiedtkeEnriques}), this phenomenon is explained by Theorem~\ref{ThmA} (ii): an Enriques surface $S_C$ over $C$ has $H^1_{\sub{\'et}}(S_C,\mathbb{F}_2) \cong \mathbb{F}_2 \neq 0$ as the fundamental group is $\mathbb{Z}/2$, so the inequality \eqref{IneqdRet} above forces $H^1_{\dR}(S_k) \neq 0$. 
\end{remark}

\begin{remark}
\label{rmk:IntroExamples}
We also give examples illustrating the sharpness of Theorem~\ref{ThmA} (ii) in two different ways. First, we give an example of a smooth projective surface over $\bb Z_2$ for which all \'etale cohomology groups are $2$-torsion-free, while $H^2_\crys$ has nontrivial $2$-torsion; thus, the inequality can be strict. Note that this example falls (just) outside the hypotheses of previous results like those of Caruso, \cite{Caruso}, which give conditions under which there is an abstract isomorphism $H^i_\crys(\frak X_k/W(k)) \cong H^i_\sub{\'et}(\frak X_{\bar{K}},\bb Z_p)\otimes_{\bb Z_p} W(k)$. Similar examples of smooth projective surfaces can also be constructed over (unramified extensions of) $\bb Z_p[\zeta_p]$, which shows the relevance of the bound $ie<p-1$. Secondly, we construct a smooth projective surface $\mathfrak{X}$ over $\roi_K$ where $H^2_{\sub{\'et}}(\mathfrak{X}_{\bar{K}},\mathbb{Z}_p)_{\sub{tor}} = \mathbb{Z}/p^2\bb Z$, while $H^2_\crys(\mathfrak{X}_k/W(k))_{\sub{tor}} = k \oplus k$; thus, the inequality in part (ii) cannot be upgraded to a subquotient relationship between the corresponding groups.
\end{remark}

Part (iii) implies that the crystalline cohomology of the special fibre (under the stated hypothesis) can be recovered from the generic fibre. The implicit functor in this recovery process relies on the theory of Breuil--Kisin modules, which were defined by Kisin, \cite{Kisin}, following earlier work of Breuil, \cite{Breuil}; for us, Kisin's observation that one can use the ring $\gS=W(k)[[T]]$ in place of Breuil's $S$ involving divided powers is critical. The precise statement of (iii) is the following. As $H^i_\sub{\'et}(\frak X_{C},\bb Z_p)$ is torsion-free by (ii) and the assumption, it is a lattice in a crystalline $G_K$-representation by (i). Kisin associates to any lattice in a crystalline $G_K$-representation a finite free $\gS = W(k)[[T]]$-module $\mathrm{BK}(H^i_\sub{\'et}(\frak X_{C},\bb Z_p))$ equipped with a Frobenius $\phi$, in such a way that $\mathrm{BK}(H^i_\sub{\'et}(\frak X_{C},\bb Z_p))\otimes_\gS W(k)[\tfrac 1p]$ (where the map $\gS\to W(k)$ sends $T$ to $0$ and is the Frobenius on $W(k)$) gets identified with
\[
(H^i_\sub{\'et}(\frak X_{C},\bb Z_p)\otimes_{\bb Z_p} B_\crys)^{G_K} = H^i_\crys(\frak X_k/W(k))[\tfrac 1p]\ .
\]
Then, under the assumptions of part (iii), we show that
\[
\mathrm{BK}(H^i_\sub{\'et}(\frak X_{C},\bb Z_p))\otimes_\gS W(k) = H^i_\crys(\frak X_k/W(k))
\]
as submodules of $\mathrm{BK}(H^i_\sub{\'et}(\frak X_{C},\bb Z_p))\otimes_\gS W(k)[\tfrac 1p]\cong H^i_\crys(\frak X_k/W(k))[\tfrac 1p]$.

As alluded to earlier, there is also a variant of Theorem~\ref{ThmA} (iii) if $K$ is algebraically closed. In fact, our approach is to reduce to this case; so, from now on, let $C$ be {\em any} complete algebraically closed nonarchimedean extension of $\bb Q_p$, with ring of integers $\roi$ and residue field $k$. In this situation, the literal statement of Theorem~\ref{ThmA} (iii) above is clearly false, as there is no Galois action. Instead, our variant says the following:

\begin{theorem}
\label{IntroThmAoverC}
Let $\frak X$ be a proper smooth formal scheme over $\roi$. Assume that $H^i_\crys(\frak X_k/W(k))$ and $H^{i+1}_\crys(\frak X_k/W(k))$ are $p$-torsion-free. Then $H^i_\crys(\frak X_k/W(k))$, with its $\phi$-action, can be recovered functorially from the rigid-analytic generic fibre $X$ of $\frak X$. More precisely, the $\mathbb{Z}_p$-module $H^i_{\sub{\'et}}(X, \mathbb{Z}_p)$ equipped with the de~Rham comparison isomorphism (as in Theorem~\ref{ThmRat} below) functorially recovers $H^i_{\crys}(\frak X_k/W(k))$.
\end{theorem}

The proof of this result (and the implicit functor) relies on a variant of Breuil--Kisin modules, due to Fargues, \cite{FarguesBK}, formulated in terms of Fontaine's period ring $A_\inf$ instead of the ring $\gS$. To explain this further, we recall the definitions first. The ring $A_\inf$ is defined as
\[
A_\inf = W(\roi^\flat)\ ,
\]
where $\roi^\flat = \projlim_\phi \roi/p$ is the ``tilt'' of $\roi$. Then $\roi^\flat$ is the ring of integers in a complete algebraically closed nonarchimedean field $C^\flat$ of characteristic $p$, the tilt of $C$; in particular, the Frobenius map on $\roi^\flat$ is bijective, and thus $A_\inf = W(\roi^\flat)$ has a natural Frobenius automorphism $\phi$, and $A_\inf/p = \roi^\flat$.

We will need certain special elements of $A_\inf$. Fix a compatible system of primitive $p$-power roots of unity $\zeta_{p^r}\in \roi$; then the system $(1,\zeta_p,\zeta_{p^2},\ldots)$ defines an element $\epsilon\in \roi^\flat$. Let $\mu = [\epsilon]-1\in A_\inf$ and
\[
\xi = \frac{\mu}{\varphi^{-1}(\mu)} = \frac{[\epsilon]-1}{[\epsilon]^{1/p}-1} = \sum_{i=0}^{p-1} [\epsilon]^{i/p}\ .
\]
There is a natural map $\theta: A_\inf\to \roi$ whose kernel is generated by the non-zero-divisor $\xi$. Then $A_\crys$ is defined as the $p$-adic completion of the PD envelope of $A_\inf$ with respect to the kernel of $\theta$; equivalently, one takes the $p$-adic completion of the $A_\inf$-algebra generated by the elements $\frac{\xi^n}{n!}$, $n\geq 1$, inside $A_\inf[\tfrac 1p]$. Witt vector functoriality gives a natural map $A_\inf \to W(k)$ that carries $\xi$ to $p$, and hence factors through $A_\crys$. Finally, the ring $B_\crys$ that appeared in Fontaine's functor is
\[
B_\crys = A_\crys[\tfrac 1\mu]\ .
\]
This is a $\bb Q_p$-algebra as $\mu^{p-1}\in pA_\crys$. We will also need $B_\dR^+$, defined as the $\xi$-adic completion of $A_\inf[\tfrac 1p]$; this is a complete discrete valuation ring with residue field $C$, uniformizer $\xi$, and quotient field $B_\dR := B_\dR^+[\frac{1}{\xi}]$. 

With this notation, the relevant category of modules is defined as follows:

\begin{definition} A Breuil--Kisin--Fargues module is a finitely presented $A_\inf$-module $M$ equipped with a $\phi$-linear isomorphism $\phi_M: M[\tfrac 1\xi]\cong M[\tfrac 1{\phi(\xi)}]$, such that $M[\tfrac 1p]$ is finite free over $A_\inf[\tfrac 1p]$.
\end{definition}

This is a suitable mixed-characteristic analogue of a Dieudonn\'e module; in fact, these objects intervene in the work \cite{ScholzeLectureNotes} of the third author as ``mixed-characteristic local shtukas''. We note that the relation to shtukas has been emphasized by Kisin from the start, \cite{Kisin}. For us, Fargues' classification of finite free Breuil--Kisin--Fargues modules is critical.

\begin{theorem}[Fargues]
\label{IntroThmFargues}
 The category of finite free Breuil--Kisin--Fargues modules is equivalent to the category of pairs $(T,\Xi)$, where $T$ is a finite free $\bb Z_p$-module, and $\Xi\subset T\otimes_{\bb Z_p} B_\dR$ is a $B_\dR^+$-lattice.
\end{theorem}

Let us briefly explain how to use Theorem~\ref{IntroThmFargues} to formulate Theorem~\ref{IntroThmAoverC}. Under the hypothesis of the latter, by Theorem~\ref{ThmA} (ii), the $\mathbb{Z}_p$-module $T := H^i_{\sub{\'et}}(X,\mathbb{Z}_p)$ is finite free. The de~Rham comparison isomorphism for $X$, formulated in Theorem~\ref{ThmRat} next, gives a $B_\dR^+$-lattice $\Xi := H^i_\crys(X/B_{\dR}^+)$ in $T \otimes_{\mathbb{Z}_p} B_\dR$. The pair $(T,\Xi)$ determines a Breuil--Kisin--Fargues module $(M,\phi_M)$ by Theorem~\ref{IntroThmFargues}. Then Theorem~\ref{IntroThmAoverC} states that the ``crystalline realization'' $(M,\phi_M) \otimes_{A_\inf} W(k)$ coincides with $(H^i_\crys(\mathfrak{X}_k/W(k)),\phi)$, which gives the desired reconstruction. 

The preceding formulation of Theorem~\ref{IntroThmAoverC} relies on the existence of a good de~Rham cohomology theory for proper smooth rigid-analytic spaces $X$ over $C$ that takes values in $B_\dR^+$-modules, and satisfies a de~Rham comparison theorem. Note that $H^i_\dR(X)$ is a perfectly well-behaved object: it is a finite dimensional $C$-vector space. However, it is inadequate for our needs as there is no sensible formulation of the de~Rham comparison theorem in terms of $H^i_\dR(X)$: there is no natural map $C \to B_\dR^+$ splitting the map $\theta:B_\dR^+ \to C$ (unlike the discretely valued case). Our next result shows that $H^i_\dR(X)$ nevertheless admits a canonical deformation across $\theta$, and that this deformation interacts well with $p$-adic comparison theorems. We regard this as an analogue of crystalline cohomology (with respect to the topologically nilpotent thickening $B_\dR^+\to C$ in place of the usual $W(k)\to k$).

\begin{theorem}\label{ThmRat}
Let $X$ be a proper smooth adic space over $C$. Then there are cohomology groups $H^i_\crys(X/B_\dR^+)$ which come with a canonical isomorphism
\[
H^i_\crys(X/B_\dR^+)\otimes_{B_\dR^+} B_\dR\cong H^i_\sub{\'et}(X,\bb Z_p)\otimes_{\bb Z_p} B_\dR\ .
\]
In case $X=X_0\hat{\otimes}_K C$ arises via base change from some complete discretely valued extension $K$ of $\bb Q_p$ with perfect residue field, this isomorphism agrees with the comparison isomorphism
\[
H^i_\dR(X_0)\otimes_K B_\dR\cong H^i_\sub{\'et}(X,\bb Z_p)\otimes_{\bb Z_p} B_\dR
\]
from \cite{ScholzePAdicHodge} under a canonical identification
\[
H^i_\crys(X/B_\dR^+) = H^i_\dR(X_0)\otimes_K B_\dR^+\ .
\]
Moreover, $H^i_\crys(X/B_\dR^+)$ is a finite free $B_\dR^+$-module, and we have the following:
\begin{enumerate}
\item (Conrad-Gabber \cite{ConradGabber}) The Hodge--de~Rham spectral sequence
\[
E_1^{ij} = H^j(X,\Omega_{X/C}^i)\Rightarrow H^{i+j}_\dR(X)
\]
degenerates at $E_1$.
\item The Hodge--Tate spectral sequence~\cite{ScholzeSurvey}
\[
E_2^{ij} = H^i(X,\Omega_{X/C}^j)(-j)\Rightarrow H^{i+j}_\sub{\'et}(X,\bb Z_p)\otimes_{\bb Z_p} C
\]
degenerates at $E_2$.
\end{enumerate}
\end{theorem}

We now turn to discussing the proof of Theorem~\ref{ThmA}. Our strategy is to construct a cohomology theory for proper smooth formal schemes over $\roi$ that is valued in Breuil--Kisin--Fargues modules. This new cohomology theory specializes to all other cohomology theories, as summarized next, and thus leads to explicit relationships between them, as in Theorem~\ref{ThmA}.

\begin{theorem}\label{ThmB} Let $\frak X$ be a proper smooth formal scheme over $\roi$, where $\roi$ is the ring of integers in a complete algebraically closed nonarchimedean extension $C$ of $\bb Q_p$. Then there is a perfect complex of $A_\inf$-modules
\[
R\Gamma_{A_\inf}(\frak X)\ ,
\]
equipped with a $\phi$-linear map $\phi: R\Gamma_{A_\inf}(\frak X)\to R\Gamma_{A_\inf}(\frak X)$ inducing a quasi-isomorphism
\[
R\Gamma_{A_\inf}(\frak X)[\tfrac 1\xi]\simeq R\Gamma_{A_\inf}(\frak X)[\tfrac 1{\phi(\xi)}]\ ,
\]
such that all cohomology groups are Breuil--Kisin--Fargues modules. Moreover, one has the following comparison results.
\begin{enumerate}
\item With crystalline cohomology of $\frak X_k$:
\[
R\Gamma_{A_\inf}(\frak X)\dotimes_{A_\inf} W(k)\simeq R\Gamma_\crys(\frak X_k/W(k))\ .
\]
\item With de~Rham cohomology of $\frak X$:
\[
R\Gamma_{A_\inf}(\frak X)\dotimes_{A_\inf} \roi\simeq R\Gamma_\dR(\frak X)\ .
\]
\item With crystalline cohomology of $\frak X_{\roi/p}$:
\[
R\Gamma_{A_\inf}(\frak X)\dotimes_{A_\inf} A_\crys\simeq R\Gamma_\crys(\frak X_{\roi/p}/A_\crys)\ .
\]
\item With \'etale cohomology of the rigid-analytic generic fibre $X$ of $\frak X$:
\[
R\Gamma_{A_\inf}(\frak X)\otimes_{A_\inf} A_\inf[\tfrac 1\mu]\simeq R\Gamma_\sub{\'et}(X,\bb Z_p)\otimes_{\bb Z_p} A_\inf[\tfrac 1\mu]\ .
\]
\end{enumerate}
\end{theorem}

We note that statement (iii) formally implies (i) and (ii) by standard facts about crystalline cohomology. Also, we note that (if one fixes a section $k\to \roi/p$) there is a canonical isomorphism
\[
R\Gamma_\crys(\frak X_{\roi/p}/A_\crys)[\tfrac 1p]\simeq R\Gamma_\crys(\frak X_k/W(k))\otimes_{W(k)} A_\crys[\tfrac 1p]\ ;
\]
this is related to a result of Berthelot--Ogus, \cite{BerthelotOgus2}. Thus, combining parts (iii) and (iv), we get the comparison
\[
R\Gamma_\crys(\frak X_k/W(k))\otimes_{W(k)} B_\crys\simeq R\Gamma_{A_\inf}(\frak X)\dotimes_{A_\inf} B_\crys\simeq R\Gamma_\sub{\'et}(X,\bb Z_p)\otimes_{\bb Z_p} B_\crys\ ,
\]
which proves Theorem~\ref{ThmA} (i); note that since each $H^i_{A_\inf}(\mathfrak{X})[\frac{1}{p}]$ is free over $A_\inf[\frac{1}{p}]$, the derived comparison statement above immediately yields one for the individual cohomology groups.

The picture here is that there is the cohomology theory $R\Gamma_{A_\inf}(\frak X)$ which lives over all of $\Spec A_\inf$, and which over various (big) subsets of $\Spec A_\inf$ can be described through other cohomology theories. These subsets often overlap, and on these overlaps one gets comparison isomorphisms. However, the cohomology theory $R\Gamma_{A_\inf}(\frak X)$ itself is a finer invariant which cannot be obtained by a formal procedure from the other known cohomology theories. In particular, the base change $R\Gamma_{A_\inf}(\frak X) \dotimes_{A_\inf} \roi^\flat$ does not admit a  description in classical terms, and gives a  specialization from the \'etale cohomology of $X$ with $\bb F_p$-coefficients to the de~Rham cohomology of $\mathfrak{X}_k$ (by Theorem~\ref{ThmB} (ii) and (iv)),  and is thus responsible for the inequality in Theorem~\ref{ThmA} (ii).

\begin{remark} It is somewhat surprising that there is a Frobenius acting on $R\Gamma_{A_\inf}(\frak X)$, as there is no Frobenius acting on $\frak X$ itself. This phenomenon is reminiscent of the Frobenius action on the de~Rham cohomology $R\Gamma_{\dR}(Y)$ of a proper smooth $W(k)$-scheme $Y$. However, in the latter case, the formalism of crystalline cohomology shows that $R\Gamma_{\dR}(Y)$ depends functorially on the special fibre $Y_k$; the latter lives in characteristic $p$, and thus carries a Frobenius. In our case, though, the theory $R\Gamma_{A_\inf}(\frak X)$ is {\em not} a functor of $\mathfrak{X}_{\roi/p}$ (see Remark \ref{rmk:NotFunctorSpecialFibre}), so there is no obvious Frobenius in the picture. Instead, in our construction, the Frobenius on $R\Gamma_{A_\inf}(\mathfrak{X})$ comes from the Frobenius action on the ``tilt'' of $X$. 
\end{remark}

Let us explain the definition of $R\Gamma_{A_\inf}(\frak X)$. We will construct a complex $A\Omega_{\frak X}$ of sheaves of $A_\inf$-modules on $\frak X_\sub{Zar}$, which will in fact carry the structure of a commutative $A_\inf$-algebra (in the derived category).\footnote{Our constructions can be upgraded to make $A\Omega_{\frak X}$ into a sheaf of $E_\infty$-$A_\inf$-algebras, but we will merely consider it as a commutative algebra in the derived category of $A_\inf$-modules on $\frak X$.} Then
\[
R\Gamma_{A_\inf}(\frak X) := R\Gamma(\frak X_\sub{Zar},A\Omega_{\frak X})\ .
\]
Let us remark here that, in the way constructucted in this paper, $A\Omega_{\frak X}$ depends on the map $\frak X\to \Spf \roi$, and so it would be better to write $A\Omega_{\frak X/\roi}$ instead. We write $A\Omega_{\frak X}$ to keep notation light.\footnote{In fact, by \cite{BMS2}, $A\Omega_{\frak X}$ only depends on $\frak X$ itself.}

The comparison results above are consequences of the following results on $A\Omega_{\frak X}$.

\begin{theorem}\label{ThmC} Let $\mathfrak{X}/\roi$ be as in Theorem~\ref{ThmB}. For the complex $A\Omega_{\frak X}$ of sheaves of $A_\inf$-modules defined below, there are canonical quasi-isomorphisms of complexes of sheaves on $\frak X_\sub{Zar}$ (compatible with multiplicative structures).
\begin{enumerate}
\item With crystalline cohomology of $\frak X_k$:
\[
A\Omega_{\frak X}\hat{\dotimes}_{A_\inf} W(k)\simeq W\Omega^\bullet_{\frak X_k/W(k)}\ .
\]
Here, the tensor product is $p$-adically completed, and the right side denotes the de~Rham--Witt complex of $\frak X_k$, which computes crystalline cohomology of $\frak X_k$.
\item With de~Rham cohomology of $\frak X$:
\[
A\Omega_{\frak X}\dotimes_{A_\inf} \roi\simeq \Omega^{\bullet,\cont}_{\frak X/\roi}\ ,
\]
where $\Omega^{i,\cont}_{\frak X/\roi} = \projlim_n \Omega^i_{(\frak X/p^n)/(\roi/p^n)}$.
\item With crystalline cohomology of $\frak X_{\roi/p}$: if $u: (\frak X_{\roi/p}/A_\crys)_\crys\to \frak X_\sub{Zar}$ denotes the projection, then
\[
A\Omega_{\frak X}\hat{\dotimes}_{A_\inf} A_\crys\simeq Ru_\ast \roi_{\frak X_{\roi/p}/A_\crys}^\crys\ .
\]
\item With (a variant of) \'etale cohomology of the rigid-analytic generic fibre $X$ over $\mathfrak{X}$: if $\nu: X_\sub{pro\'et}\to \frak X_\sub{Zar}$ denotes the projection, then
\[
A\Omega_{\frak X}\otimes_{A_\inf} A_\inf[\tfrac 1\mu]\simeq (R\nu_\ast \bb A_{\inf,X})\otimes_{A_\inf} A_\inf[\tfrac 1\mu]\ .
\]
\end{enumerate}
\end{theorem}

We note that Theorem~\ref{ThmC} implies Theorem~\ref{ThmB}. This is clear for parts (i), (ii) and (iii). For part (iv), one uses the following result from \cite{ScholzePAdicHodge} (cf.~\cite[\S 3, Theorem 8]{FaltingsAlmostEtale}): the canonical map
\[
R\Gamma_\sub{\'et}(X,\bb Z_p)\otimes_{\bb Z_p} A_\inf\to R\Gamma_\sub{pro\'et}(X,\bb A_{\inf,X})
\]
is an almost quasi-isomorphism; in particular, it is a quasi-isomorphism after inverting $\mu$. Here, $\bb A_{\inf,X}$ is a relative version of Fontaine's period ring $A_\inf$, obtained by repeating the construction of $A_\inf$ on the pro-\'etale site.

Theorem~\ref{ThmC} provides two different ways of looking at $A\Omega_{\frak X}$. On one hand, it can be regarded as a deformation of the de~Rham complex of $\frak X$ from $\roi$ to its pro-infinitesimal thickening $A_\inf\to \roi$, by (ii). This is very analogous to regarding crystalline cohomology of $\frak X_k$ as a deformation of the de~Rham complex of $\frak X_k$ from $k$ to its pro-infinitesimal thickening $W(k)\to k$. This turns out to be a fruitful perspective for certain problems; in particular, if one chooses coordinates on $\frak X$, then $A\Omega_{\frak X}$ can be computed explicitly, as a certain ``$q$-deformation of de~Rham cohomology''. This is very concrete, but unfortunately it depends on coordinates in a critical way, and we do not know how to see directly that $A\Omega_{\frak X}$ is independent of the choice of coordinates in this picture. 

\begin{remark}
This discussion raises an interesting question: is there a site-theoretic formalism, akin to crystalline cohomology, that realizes $A\Omega_{\frak X}$? Note that $A\Omega_{\frak X} \widehat{\dotimes}_{A_\inf} A_\crys$ does indeed arise by the crystalline formalism thanks to Theorem~\ref{ThmC} (iii). It is tempting to use the infinitesimal site to descend further to $A_\inf$; however, one can show that this approach does not work, essentially for the same reason that infinitesimal cohomology does not work well in characteristic $p$.\footnote{Footnote added in print: The question raised in this remark has been answered affirmatively by the construction of the prismatic site that shall appear in the forthcoming \cite{BhattScholzePrism}.}
\end{remark}

On the other hand, by Theorem~\ref{ThmC} (iv), one can regard $A\Omega_{\frak X}$ as being $R\nu_\ast \bb A_{\inf,X}$, up to some $\mu$-torsion, i.e.~as a variant of \'etale cohomology. It is this perspective with which we will define $A\Omega_{\frak X}$; this has the advantage of being obviously canonical. However, this definition is not very explicit, and much of our work goes into computing the resulting $A\Omega_{\frak X}$, and, in particular, getting the comparison to the de~Rham complex. It is this computation which builds the bridge between the apparently disparate worlds of \'etale cohomology and de~Rham cohomology.

\subsection{Strategy of the construction}

We note that computations relating \'etale cohomology and differentials, as alluded to above, have been at the heart of Faltings' approach to $p$-adic Hodge theory; however, they always had the problem of some unwanted ``junk torsion''. The main novelty of our approach is that we can get rid of the ``junk torsion'' by the following definition:

\begin{definition}\label{IntroAOmegaDef} Let
\[
\nu: X_\sub{pro\'et}\to \frak X_\sub{Zar}
\]
denote the projection (or the ``nearby cycles map'').  Then
\[
A\Omega_{\frak X}  := L\eta_\mu(R\nu_\ast \bb A_{\inf,X})\ .
\]
\end{definition}

\begin{remark} 
If one is careful with pro-sheaves, one can replace the pro-\'etale site with Faltings' site, \cite{FaltingsAlmostEtale}, \cite{AbbesGrosTopos}, in Definition \ref{IntroAOmegaDef}. 
\end{remark}

Here, $\mu = [\epsilon]-1\in A_\inf$ is the element introduced above. The critical new ingredient is the operation $L\eta_f$, defined on the derived category of $A$-modules\footnote{
In fact, we define $L\eta_f$ operation on any ringed topos, such as $(\frak X_\sub{Zar},A_\inf)$, which is the setup in which we are using it in Definition~\ref{IntroAOmegaDef}.}, for any non-zero-divisor $f\in A$. Concretely, if $D^\bullet$ is a complex of $f$-torsion-free $A$-modules, then $\eta_f D^\bullet$ is a subcomplex of $D^\bullet[\tfrac 1f]$ with terms
\[
(\eta_f D)^i = \{x\in f^i D^i\mid dx\in f^{i+1} D^{i+1}\}\ .
\]
One shows that this operation passes to an operation $L\eta_f$ on the derived category. This relies on the observation that
\[
H^i(\eta_f D^\bullet) = H^i(D^\bullet) / H^i(D^\bullet)[f]\ .
\]
In particular, the operation $\eta_f$ has the effect of killing some torsion on the level of cohomology groups, which is what makes it possible to kill the ``junk torsion'' mentioned above. We warn the reader that $L\eta_f$ is \emph{not} an exact operation.

\begin{remark}
\label{rmk:Bogus}
We note that the operation $L\eta_f$ appeared previously, notably in the work of Berthelot--Ogus, \cite[Section 8]{BerthelotOgus}. There, they prove that for an affine smooth scheme $\Spec R$ over $k$, $\phi$ induces a quasi-isomorphism
\[
R\Gamma_\crys(\Spec R/W(k)) \simeq L\eta_p R\Gamma_\crys(\Spec R/W(k))\ ,
\]
with applications to the relation between Hodge and Newton polygon. Illusie has strengthened this to an isomorphism of complexes
\[
W\Omega_{R/k}^\bullet\cong \eta_p W\Omega_{R/k}^\bullet\ ,
\]
cf.~\cite[I.3.21.1.5]{IllusieDRWitt}. 
\end{remark}

\begin{remark}
For any object $K$ in the derived category of $\mathbb{Z}_p$-modules equipped with a quasi-isomorphism $L\eta_p K \simeq K$, we show that the complex $K/p^n$ admits a canonical representative $K_n^\bullet$ for each $n$, with $K_n^i = H^i(K/p^n)$. In the case $K = R\Gamma_\crys(\Spec R/W(k))$, equipped with the Berthelot-Ogus quasi-isomorphism mentioned in Remark \ref{rmk:Bogus}, this canonical representative is the de~Rham--Witt complex; this amounts essentially to Katz's reconstruction of the de~Rham--Witt complex from crystalline cohomology via the Cartier isomorphism, cf. \cite[\S III.1.5]{IllusieRaynaud}.
\end{remark}

Next, we explain the computation of $A\Omega_{\frak X}$ when $\frak X=\Spf R$ is an affine formal scheme, which is ``small'' in Faltings' sense, i.e.~there exists an \'etale  map
\[
\square: \Spf R\to \widehat{\bb G}_m^d = \Spf \roi\langle T_1^{\pm 1},\ldots,T_d^{\pm 1}\rangle
\]
to some (formal) torus; this is always true locally on $\frak X_\sub{Zar}$. In that case, we define
\[
R_\infty = R\hat{\otimes}_{\roi\langle T_1^{\pm 1},\ldots,T_d^{\pm 1}\rangle}\roi\langle T_1^{\pm 1/p^\infty},\ldots,T_d^{\pm 1/p^\infty}\rangle\ ,
\]
on which the Galois group $\Gamma = \bb Z_p^d$ acts; here we use the choice of $p$-power roots of unity in $\roi$. Faltings' almost purity theorem implies that the natural map
\begin{equation}
\label{map:Faltingsqis}
R\Gamma_\sub{cont}(\Gamma,\bb A_\inf(R_\infty))\to R\Gamma_\sub{pro\'et}(X,\bb A_{\inf,X})
\end{equation}
is an {\em almost} quasi-isomorphism, in the sense of Faltings' almost mathematics (with respect to the ideal $[\frak m^\flat]\subset A_\inf$, where $\frak m^\flat\subset \roi^\flat$ is the maximal ideal). The key lemma is that the $L\eta$-operation converts the preceding map to an {\em honest} quasi-isomorphism: 

\begin{lemma}\label{KeyLemma} The induced map
\[
L\eta_\mu R\Gamma_\sub{cont}(\Gamma,\bb A_\inf(R_\infty))\to L\eta_\mu R\Gamma_\sub{pro\'et}(X,\bb A_{\inf,X})
\]
is a quasi-isomorphism.
\end{lemma}

This statement came as a surprise to us, and its proof relies on a rather long list of miracles; we have no good a priori reason to believe that this should be true. Part of the miracle is that the lemma can be proved by only showing that the left side is nice, without any extra knowledge of the right side than what follows from the almost quasi-isomorphism \eqref{map:Faltingsqis} above. In the announcement \cite{BMSAnnouncement}, we did not use this lemma, and instead had a more complicated definition of $A\Omega_{\frak X}$.

Moreover, the right side
\[
L\eta_\mu R\Gamma_\sub{pro\'et}(X,\bb A_{\inf,X})
\]
is equal to $A\Omega_R := R\Gamma(\Spf R,A\Omega_{\Spf R})$. This is not formal as $L\eta$ does not commute with taking global sections, but is also not the hard part of the argument.

Thus, one can compute $A\Omega_R$ as
\[
L\eta_\mu R\Gamma_\sub{cont}(\Gamma, \bb A_\inf(R_\infty))\ .
\]
This computation can be done explicitly, following the previous computations of Faltings. Before explaining the answer the general, we first give the description in the case of the torus; the result is best formulated using the so-called {\em $q$-analogue} $[i]_q := \tfrac{q^i-1}{q-1}$ of an integer $i \in \bb Z$.

\begin{theorem}\label{ThmD} If $R = \roi\langle T^{\pm 1}\rangle$, then $A\Omega_R$ is computed by the $q$-de~Rham complex
\[
A_\inf\langle T^{\pm 1}\rangle \xTo{\frac{\partial_q}{\partial_q T}} A_\inf \langle T^{\pm 1}\rangle\ : T^i\mapsto [i]_q T^{i-1}\ ,\ q=[\epsilon]\in A_\inf\ .
\]
In closed form,
\[
\frac{\partial_q}{\partial_q T}(f(T)) = \frac{f(qT)-f(T)}{qT-T}
\]
is a finite $q$-difference quotient.

In general, the formally \'etale map $\roi\langle T_1^{\pm 1},\ldots,T_d^{\pm 1}\rangle\to R$ deforms uniquely to a formally \'etale map
\[
A_\inf\langle T_1^{\pm 1},\ldots,T_d^{\pm 1}\rangle\to A(R)^\square\ .
\]
For each $i=1,\ldots,d$, one has an automorphism $\gamma_i$ of $A_\inf\langle T_1^{\pm 1},\ldots,T_d^{\pm 1}\rangle$ sending $T_i$ to $q T_i$ and $T_j$ to $T_j$ for $j\neq i$, where $q=[\epsilon]$. This automorphism lifts uniquely to an automorphism $\gamma_i$ of $A(R)^\square$ such that $\gamma_i\equiv 1\mod (q-1)$, so that one can define commuting ``$q$-derivations''
\[
\frac{\partial_q}{\partial_q T_i} := \frac{\gamma_i-1}{qT_i-T_i}: A(R)^\square\to A(R)^\square\ .
\]
Then $A\Omega_R$ is computed by the $q$-de~Rham complex
\[
0\to A(R)^\square\xTo{(\frac{\partial_q}{\partial_q T_1},\ldots,\frac{\partial_q}{\partial_q T_d})} (A(R)^\square)^d\to\ldots\to \bigwedge^i (A(R)^\square)^d\to \ldots\to \bigwedge^d (A(R)^\square)^d\to 0\ ,
\]
where all higher differentials are exterior powers of the first differential.
\end{theorem}

In particular, after setting $q=1$, this becomes the usual de~Rham complex, which is related to part (ii) of Theorem~\ref{ThmC}. In fact, already in $A_\crys$, the elements $[i]_q$ and $i$ differ by a unit, which is related to part (iii) of Theorem~\ref{ThmC}.

Interestingly, the $q$-de~Rham complex admits a natural structure as a differential graded algebra, but a noncommutative one: when commuting a function past a differential, one must twist by one of the automorphisms $\gamma_i$. Concretely, the Leibniz rule for $\frac{\partial_q}{\partial_q T}$ reads
\[
\frac{\partial_q}{\partial_q T} (f(T)g(T)) = f(T) \frac{\partial_q}{\partial_q T}(g(T)) + g(qT) \frac{\partial_q}{\partial_q T}(f(T))\ ,
\]
where $g(qT)$ appears in place of $g(T)$. (Note that this is not symmetric in $f$ and $g$, so there are really two different formulas.) If one wants to rewrite this as the Leibniz rule
\[
\frac{\partial_q}{\partial_q T} (f(T)g(T)) = f(T) \frac{\partial_q}{\partial_q T}(g(T)) + \frac{\partial_q}{\partial_q T}(f(T)) g(T)\ ,
\]
one has to introduce noncommutativity when multiplying the $q$-differential $\frac{\partial_q}{\partial_q T}(f(T))$ by the function $g(T)$; this can be done in a consistent way. Nevertheless, one can show that the $q$-de~Rham complex is an $E_\infty$-algebra (over $A_\inf$), so the commutativity is restored up to consistent higher homotopies.

\begin{remark}The occurrence of the perhaps less familiar (and more general) notion of an $E_\infty$-algebra, instead of the stricter and more hands-on notion of a commutative differential graded algebra, is not just an artifact of our construction, but a fundamental feature of the output: even when $R = \roi \langle T^{\pm 1} \rangle$, the $E_\infty$-$A_\inf$-algebra $A\Omega_R$ (or even $A\Omega_R/p$) cannot be represented by a commutative differential graded algebra (see Remark \ref{rmk:qdRcdga}).
\end{remark}

Finally, let us say a few words about the proof of Lemma~\ref{KeyLemma}. Its proof relies on a relation to the de~Rham--Witt complex of Langer--Zink \cite{LangerZink}. First, recall that there is an alternative definition of $A_\inf$ as
\[
A_\inf = \varprojlim_F W_r(\roi)\ ;
\]
similarly, we have
\[
\bb A_\inf(R_\infty) = \varprojlim_F W_r(R_\infty)\ .
\]
Roughly, Lemma~\ref{KeyLemma} follows by taking the inverse limit over $r$, along the $F$ maps, of the following variant.

\begin{lemma}\label{KeyLemmaWr} For any $r\geq 1$, the natural map
\[
L\eta_\mu R\Gamma_\sub{cont}(\Gamma,W_r(R_\infty))\to L\eta_\mu R\Gamma_\sub{pro\'et}(X,W_r(\hat{\roi}_X^+))
\]
is a quasi-isomorphism; let $\widetilde{W_r\Omega}_R$ denote their common value. Then (up to the choice of roots of unity) there are canonical isomorphisms
\[
H^i(\widetilde{W_r\Omega}_R)\cong W_r\Omega^{i,\cont}_{R/\roi}\ ,
\]
where the right side denotes $p$-adically completed versions of the de~Rham--Witt groups of Langer--Zink, \cite{LangerZink}.
\end{lemma}

\begin{remark} It is also true that $\widetilde{W_r\Omega}_R\cong A\Omega_R\dotimes_{A_\inf} W_r(\roi)$, and $A\Omega_R = \projlim_r \widetilde{W_r\Omega}_R.$
\end{remark}

Here, the strategy is the following. One first computes the cohomology groups of the explicit left side
\[
L\eta_\mu R\Gamma_\sub{cont}(\Gamma,W_r(R_\infty))
\]
and matches those with the de~Rham--Witt groups. These are made explicit by Langer--Zink, and we match their description with ours; this is not very hard but a bit cumbersome, as the descriptions are quite combinatorially involved. In fact, we can a priori give the cohomology groups the structure of a ``pro-$F$-$V$-complex'' (using a Bockstein operator as the differential), so that by the universal property of the de~Rham--Witt complex, they receive a map from the de~Rham--Witt complex; it is this canonical map that we prove to be an isomorphism. In particular, the isomorphism is compatible with natural $d$, $F$, $V$, $R$ and multiplication maps.

After this computation of the left side, one proves a lemma that if $D_1\to D_2$ is an almost quasi-isomorphism of complexes such that $D_1$ is sufficiently nice, then $L\eta_\mu D_1\to L\eta_\mu D_2$ is a quasi-isomorphism, see Lemma~\ref{lem:LetaActualIsom}. In fact, this argument only needs a qualitative description of the left side, and one can prove the main results of our paper without establishing the link to de~Rham--Witt complexes.

We note that the complexes $\widetilde{W_r\Omega}_R$ provide a partial lift of the Cartier isomorphism to mixed characteristic. More precisely, $A_\inf$ admits two different maps $\tilde\theta_r: A_\inf\to W_r(\roi)$ and $\theta_r = \tilde\theta_r \phi^r: A_\inf\to W_r(\roi)$ to $W_r(\roi)$, the first of which comes from the description $A_\inf = \projlim_F W_r(\roi)$; the map $\theta_1$ agrees with Fontaine's map $\theta$ used above. Then formal properties of the $L\eta$-operation (Proposition~\ref{prop:LetaBock}, Lemma~\ref{lem:compositionLeta}) show that
\[
A\Omega_{\frak X}\dotimes_{A_\inf,\theta_r} W_r(\roi)
\]
is computed by a complex whose terms are the cohomology groups $W_r\Omega^{i,\cont}_{\frak X/\roi}$ of
\[
\widetilde{W_r\Omega}_{\frak X} = A\Omega_{\frak X}\dotimes_{A_\inf,\tilde\theta_r} W_r(\roi)\ .
\]
By the crystalline comparison, $A\Omega_{\frak X}\dotimes_{A_\inf,\theta_r} W_r(\roi)$ computes the crystalline cohomology of $\frak X/W_r(\roi)$ (equivalently, of $\frak X_{\roi/p}/W_r(\roi)$). Thus, this reproves in this setup that Langer--Zink's de~Rham--Witt complex computes crystalline cohomology. On the other hand, after base extension from $A_\inf$ to $W(k)$, the maps $\theta_r$ and $\tilde\theta_r$ agree up to a power of Frobenius on $W(k)$. Thus, reformulating this from a slightly different perspective, there are two different deformations of $A\Omega_{\frak X}\dotimes_{A_\inf} W_r(k)\simeq W_r\Omega^\bullet_{\frak X_k/k}$ to mixed characteristic: one is the de~Rham--Witt complex $A\Omega_{\frak X}\dotimes_{A_\inf,\theta_r} W_r(\roi)\simeq W_r\Omega^{\bullet,\cont}_{\frak X/\roi}$, the other is the complex $A\Omega_{\frak X}\dotimes_{A_\inf,\tilde\theta_r} W_r(\roi) = \widetilde{W_r\Omega}_{\frak X}$ whose cohomology groups are the de~Rham--Witt groups $W_r\Omega^{\bullet,\cont}_{\frak X/\roi}$. From this point of view, the fact that these two specialize to the same complex over $W_r(k)$ recovers the Cartier isomorphism.

\subsection{The genesis of this paper}

We comment briefly on the history of this paper. The starting point for this work was the question whether one could geometrically construct Breuil--Kisin modules, which had proved to be a powerful tool in \emph{abstract} integral $p$-adic Hodge theory. A key point was the introduction of Fargues' variant of Breuil--Kisin modules, which does not depend on any choices, contrary to the classical theory of Breuil--Kisin modules (which depends on the choice of a uniformizer). The search for a natural $A_\inf$-valued cohomology theory took off ground after we read a paper of Hesselholt, \cite{HesselholtCp}, that computed the topological cyclic homology (or rather topological Frobenius homology) of $\roi = \roi_{\bb C_p}$, with the answer being given by the Breuil--Kisin--Fargues version of Tate twists. This made it natural to guess that in general, (a suitable graded piece of) topological Frobenius homology should produce the sought-after cohomology theory. A computation of the homotopy groups of $TR^r(R;p,\bb Z_p)$ then suggested the existence of complexes $\widetilde{W_r\Omega}_R$ with cohomology groups given by $W_r\Omega^{i,\cont}_{R/\roi}$, as in Lemma~\ref{KeyLemmaWr}. The naive guess $R\Gamma_\sub{pro\'et}(X,W_r(\hat{\roi}_X^+))$ for these complexes is correct up to some small torsion. In fact, it gets better as $r\to \infty$, and in the limit $r=\infty$, the naive guess can be shown to be almost correct; this gives an interpretation of the ``junk torsion'' as coming from the non-integral terms of the de~Rham--Witt complex, cf.~Proposition~\ref{prop:interpretjunk}. Analyzing the expected properties of $A\Omega_R$ then showed that one needed an operation like $L\eta$ with the property of Proposition~\ref{prop:LetaBock} below: the naive guess $D=R\Gamma_\sub{pro\'et}(X,\bb A_{\inf,X})$ has the property that $H^i(D/\mu)$ is almost given by $W\Omega^{i,\cont}_{R/\roi}$, whereas the correct complex $A\Omega_R$ should have the property that $A\Omega_R/\mu$ is (almost) quasi-isomorphic to the de~Rham--Witt complex of $R$. In this context we rediscovered the $L\eta$-operation. Thus, although topological Hochschild homology has played a key role in the genesis of this paper, it does not play any role in the paper itself (although it may become important for future developments). In particular, we do not prove that our new cohomology theory is actually related to topological Hochschild homology in the expected way\footnote{Footnote added in print: the reconstruction of the $A\Omega$ complexes via topological Hochschild homology as suggested in this paragraph has appeared in \cite{BMS2}.}.

\subsection{Outline}

Finally, let us explain the content of the different sections. As it is independent of the rest of the paper, we start in Section 2 by giving some examples of smooth projective surfaces illustrating the sharpness of our results.

In Sections 3 through 7, we collect various foundations. In Section 3, we recall a few facts about perfectoid algebras. This contains much more than we actually need in the paper, but we thought that it may be a good idea to give a summary of the different approaches and definitions of perfectoid rings in the literature, notably the original definition, \cite{ScholzeThesis}, the definition of Kedlaya--Liu, \cite{KedlayaLiu}, the results of Davis--Kedlaya, \cite{DavisKedlaya}, and the very general definition of Gabber--Ramero, \cite{GabberRamero2}. Next, in Section 4, we recall a few facts from the theory of Breuil--Kisin modules, and the variant notion over $A_\inf$ defined by Fargues. In particular, we state Fargues' classification theorem for finite free Breuil--Kisin--Fargues modules. This classification is in terms of data that can be easily defined using rational $p$-adic Hodge theory (using only the generic fibre). We recall some relevant facts about rational $p$-adic Hodge theory in Section 5, including a brief reminder on the pro-\'etale site. In Section 6, we define the $L\eta$-operation in great generality, and prove various basic properties. In Section 7, we recall that in some situations, one can use Koszul complexes to compute group cohomology, and discuss some related questions, such as multiplicative structures.

In Sections 8 through 14, we construct the new cohomology theory, and prove the geometric results mentioned above. As a toy case of the general statements that will follow, we construct in Section 8 the complex $\widetilde{\Omega}_R = \widetilde{W_1\Omega}_R$. All statements can be proved directly in this case, but the arguments are already indicative of the general case. After dealing with this case, we define and study $A\Omega_R$ in Section 9. In that section, we prove Lemma~\ref{KeyLemmaWr}, and deduce Lemma~\ref{KeyLemma}, except for the identification with de~Rham--Witt groups. In Section 10, we recall Langer--Zink's theory of the relative de~Rham--Witt complex. In Section 11, we show how to build an ``$F$-$V$-procomplex'' from the abstract structures of the pro-\'etale cohomology groups, and use this to prove the identification with de~Rham--Witt groups. It remains to prove the comparison with crystalline cohomology, which is the content of Section 12. Our approach here is very hands-on: we build explicit functorial models of both $A\Omega_R$ and crystalline cohomology, and an explicit functorial map. There should certainly be a more conceptual argument. In Section 13, we give a similar hands-on presentation of a de~Rham comparison isomorphism for rigid-analytic varieties over $\bb C_p$, and show that it is compatible with the result from~\cite{ScholzePAdicHodge}. We use this to prove Theorem~\ref{ThmRat}. In the final Section 14, we assemble everything and deduce the main results.

\subsection{Acknowledgements}

We would like to thank Ahmed Abbes, Sasha Beilinson, Chris Davis, Johan de Jong, Laurent Fargues, Ofer Gabber, Lars Hesselholt, Kiran Kedlaya, Jacob Lurie, Michael Ra\-po\-port and Jared Weinstein for useful conversations. During the course of this work, Bhatt was partially supported by NSF Grants DMS \#1501461 and DMS \#1522828, and a Packard fellowship, Morrow was funded by the Hausdorff Center for Mathematics, and Scholze was a Clay Research Fellow. Both Bhatt and Scholze would like to thank the University of California, Berkeley, and the MSRI for their hospitality during parts of the project. All the authors are grateful to the Clay Foundation for funding during parts of the project. 

Comments from the anonymous referee, Teruhisa Koshikawa and especially K\c{e}stutis \v{C}esnavi\v{c}ius were very helpful in improving the exposition of the paper.

\newpage

\section{Some examples}

In this section, we record some examples proving our results are sharp. First, in \S \ref{ss:ExampleSurface}, we give an example of a smooth projective surface over $\mathbb{Z}_2$ where there is no torsion in \'etale cohomology of the generic fibre (in a fixed degree), but there is torsion in crystalline cohomology of the special fibre (in the same degree); thus, the last implication in Theorem \ref{ThmA} (ii) cannot be reversed. Secondly, in \S \ref{ss:ExampleDegeneratingTorsion}, we record an example of a smooth projective surface over a (ramified) extension of $\mathbb{Z}_p$ such that the torsion in the \'etale cohomology of the generic fibre is {\em not} a subquotient of the torsion in the crystalline cohomology of the special fibre; this shows that the length inequality in Theorem \ref{ThmA} (ii) cannot be upgraded to an inclusion of the corresponding groups.

We note that both constructions rely on the interesting behaviour of finite flat group schemes in mixed characteristic: In the first example, a map of finite flat group schemes degenerates, while in the second example a finite flat group scheme itself degenerates.

\subsection{A smooth projective surface over $\bb Z_2$}
\label{ss:ExampleSurface}

The goal of this section is to prove the following result.

\begin{theorem}\label{thm:Example} There is a smooth projective geometrically connected (relative) surface $X$ over $\bb Z_2$ such that
\begin{enumerate}
\item the \'etale cohomology groups $H^i_\sub{\'et}(X_{\overline{\bb Q}_2},\bb Z_2)$ are free over $\bb Z_2$ for all $i\in \bb Z$, and
\item the second crystalline cohomology group $H^2_\sub{crys}(X_{\bb F_2}/\bb Z_2)$ has nontrivial $2$-torsion given by $H^2_\sub{crys}(X_{\bb F_2}/\bb Z_2)_\sub{tor} = \bb F_2$.
\end{enumerate}
\end{theorem}

We are not aware of any such example in the literature. In fact, we are not aware of any example in the literature of a proper smooth scheme $X$ over the ring of integers $\roi$ in a $p$-adic field for which there is not an abstract isomorphism
\[
H^i_\crys(X_k/W(k))\cong H^i_\sub{\'et}(X_{\bar{K}},\bb Z_p)\otimes_{\bb Z_p} W(k)\ .
\]
For example, Illusie, \cite[Proposition 7.3.5]{IllusieDRWitt} has proved that the crystalline cohomology of any Enriques surface in characteristic $2$ ``looks like'' the \'etale cohomology of an Enriques surfaces in characteristic $0$, and all other examples we found were of a similar nature.

We will construct $X$ as a generic hypersurface inside a smooth projective $3$-fold with similar (but slightly weaker) properties. Let us describe the construction of this $3$-fold first. We start with a ``singular'' smooth Enriques surface $S$ over $\bb Z_2$; here, singular means that $\mathrm{Pic}^\tau(S) \cong \mu_2$ as a group scheme, and it is equivalent to the condition that $\pi_1(S_{\overline{\bb F}_2})\cong \bb Z/2\bb Z$. For existence of $S$, we note that there are singular Enriques surfaces over $\bb F_2$ (see below), and all of those lift to $\bb Z_2$ by a theorem of Lang and Ogus, \cite[Theorem 1.3, 1.4]{LangEnriques}. In particular, there is a double cover $\tilde{S}\to S$, and in fact $\tilde{S}$ is a K3 surface. Explicitly, cf.~\cite[pp.~222--223]{BombieriMumford3}, one can take for $\tilde{S}_{\bb F_2}$ the smooth intersection of three quadrics in $\bb P^5_{\bb F_2}$ (with homogeneous coordinates $x_1,x_2,x_3,y_1,y_2,y_3$) given by the equations
\[\begin{aligned}
x_1^2+x_2x_3+y_1^2+x_1y_1&=0\ ,\\
x_2^2+x_1x_3+y_2^2+x_2y_2&=0\ ,\\
x_3^2+x_1x_2+y_3^2+x_3y_3&=0\ .
\end{aligned}\]
This admits a free action of $\bb Z/2\bb Z$ given by $(x_i:y_i)\mapsto (x_i:x_i+y_i)$. Then $\tilde{S}_{\bb F_2}$ is a K3 surface, and $S_{\bb F_2}=\tilde{S}_{\bb F_2}/(\bb Z/2\bb Z)$ is a singular Enriques surface.\footnote{The $\bb Z/2\bb Z$-action is free away from $x_1=x_2=x_3=0$, which would intersect $\tilde{S}_{\bb F_2}$ only when $y_1=y_2=y_3=0$, which is impossible. To check smoothness, use the Jacobian criterion to compute possible singular points. The minor for the differentials of $y_1,y_2,y_3$ shows $x_1x_2x_3=0$; assume wlog $x_1=0$. Then the minor for $x_1,x_2,y_2$ shows $x_2^2x_3=0$, so wlog $x_2=0$. Then the first equation gives $y_1=0$, and the second $y_2=0$. Now the minor for $x_1,x_2,x_3$ shows $x_3^2y_3=0$, which together with the third equation shows $x_3=y_3=0$.}

Moreover, we fix an ordinary elliptic curve $E$ over $\bb Z_2$. This contains a canonical subgroup $\mu_2\subset E$, and we get a nontrivial map
\[
\eta: \bb Z/2\bb Z\to \mu_2\to E\ .
\]
Note that $\eta_{\bb Q_2}$ is nonzero, while $\eta_{\bb F_2}$ is zero. Finally, we let $\pi: D\to S$ be the $E$-torsor which is the pushout of the $\bb Z/2\bb Z$-torsor $\tilde{S}\to S$ along $\eta$; then $D$ is a smooth projective geometrically connected $3$-fold.

\begin{proposition} The smooth projective $3$-fold $D$ over $\bb Z_2$ has the following properties.
\begin{enumerate}
\item The \'etale cohomology groups $H^i_\sub{\'et}(D_{\overline{\bb Q}_2},\bb Z_2)$ are free over $\bb Z_2$ for $i=0,1,2$.
\item The crystalline cohomology group $H^2_\sub{crys}(D_{\bb F_2}/\bb Z_2)$ has nontrivial $2$-torsion, given by $\bb F_2$.
\end{enumerate}
\end{proposition}

\begin{proof} We start with part (ii). Let $k=\bb F_2$. Then $D_k = S_k\times E_k$ is the trivial $E_k$-torsor by construction. Thus, the K\"unneth formula and Illusie's computation of $H^\ast_\crys(S_k/W(k))$, \cite[Proposition 7.3.5]{IllusieDRWitt}, show that $H^2_\crys(D_k/W(k))_\sub{tor} = k$.

Now we deal with part (i). Let $C=\overline{\bb Q}_2$. It is a general fact that $H^i_\sub{\'et}(D_C,\bb Z_2)$ is free over $\bb Z_2$ for $i=0,1$. Let $\pi_1(D_C)^{\sub{ab},2}$ be the maximal abelian pro-$2$-quotient of $\pi_1(D_C)$; equivalently, $\pi_1(D_C)^{\sub{ab},2} = H_{1,\sub{\'et}}(D_C,\bb Z_2)$. Then it is again a general fact that $H^2_\sub{\'et}(D_C,\bb Z_2)$ is free over $\bb Z_2$ if and only if $\pi_1(D_C)^{\sub{ab},2} = H_{1,\sub{\'et}}(D_C,\bb Z_2)$ is free over $\bb Z_2$. Indeed, this follows from the short exact sequence
\[
0\to \Ext^1(H_{1,\sub{\'et}}(D_C,\bb Z_2),\bb Z_2)\to H^2_\sub{\'et}(D_C,\bb Z_2)\to \Hom(H_{2,\sub{\'et}}(D_C,\bb Z_2),\bb Z_2)\to 0\ .
\]
Thus, it suffices to prove that $\pi_1(D_C)^{\sub{ab},2}$ is free over $\bb Z_2$. We can, in fact, compute the whole fundamental group $\pi_1(D_C)$ of $D_C$. Namely, pulling back $\tilde{S}_C\to S_C$ along $\pi_C: D_C\to S_C$ gives a $\bb Z/2\bb Z$-cover $\tilde{D}_C\to D_C$, and $\tilde{D}_C = \tilde{S}_C\times E_C$ decomposes as a product, which implies that
\[
\pi_1(\tilde{D}_C) = \pi_1(E_C)\cong \hat{\bb Z}\times \hat{\bb Z}\ .
\]
Thus, $\pi_1(D_C)$ is an extension of (not necessarily commutative) groups
\[
0\to \pi_1(E_C)\to \pi_1(D_C)\to \bb Z/2\bb Z\to 0\ .
\]
On the other hand, we have the map $\tilde{D}_C\to E_C$, which is by construction equivariant for the $\bb Z/2\bb Z$-action which is the covering action of $\tilde{D}_C\to D_C$ on the left, and is translation by $\eta: \bb Z/2\bb Z\to E_C$ on the right. As this action is nontrivial we may pass to the quotient and get a map $D_C\to E_C/\eta = E^\prime_C$, where $E^\prime_C$ is another elliptic curve over $C$. We get a commutative diagram with exact rows:
\[\xymatrix{
0\ar[r] & \pi_1(E_C)\ar@{=}[d]\ar[r] & \pi_1(D_C)\ar[d]\ar[r] & \bb Z/2\bb Z\ar@{=}[d]\ar[r] & 0\\
0\ar[r] & \pi_1(E_C)\ar[r] & \pi_1(E^\prime_C)\ar[r] & \bb Z/2\bb Z\ar[r] & 0.
}\]
This shows that $\pi_1(D_C) = \pi_1(E^\prime_C)\cong \hat{\bb Z}\times \hat{\bb Z}$, so that in particular $\pi_1(D_C)^{\sub{ab},2}\cong \bb Z_2\times \bb Z_2$ is free over $\bb Z_2$.
\end{proof}

\begin{proof}{\it (of Theorem~\ref{thm:Example})} Let $D$ over $\bb Z_2$ be the smooth projective $3$-fold constructed above. Let $X\subset D$ be a smooth and (sufficiently) ample hypersurface; this can be chosen over $\bb Z_2$: One has to arrange smoothness only over $\bb F_2$, so the result follows from the Bertini theorem over finite fields due to Gabber, \cite{GabberBertini}, and more generally Poonen, \cite{PoonenBertini}.

Let $C=\overline{\bb Q}_2$ as above. First, we check that $H^i_\sub{\'et}(X_C,\bb Z_2)$ is free over $\bb Z_2$ for all $i\in \bb Z$. Clearly, only $i=0,1,2,3,4$ are relevant, and by Poincar\'e duality it is enough to consider $i=0,1,2$, and again $i=0,1$ are always true. Let $U = D\setminus X$, which is affine. Then we have a long exact sequence
\[
H_{c,\sub{\'et}}^2(U_C,\bb Z_2)\to H^2_\sub{\'et}(D_C,\bb Z_2)\to H^2_\sub{\'et}(X_C,\bb Z_2)\to H_{c,\sub{\'et}}^3(U_C,\bb Z_2)\to \ldots\ .
\]
Recall that as $U$ is affine, smooth and $3$-dimensional, $H_{c,\sub{\'et}}^i(U_C,\bb Z_2) = H_{c,\sub{\'et}}^i(U_C,\bb Z/2\bb Z)=0$ for $i<3$ by Artin's cohomological bounds. In particular, $H_{c,\sub{\'et}}^3(U_C,\bb Z_2)$ is free over $\bb Z_2$, and so the displayed long exact sequence implies that $H^2_\sub{\'et}(X_C,\bb Z_2)$ is free over $\bb Z_2$, as desired.

Let $k=\bb F_2$. We claim that the map
\[
H^i_\crys(D_k/W(k))\to H^i_\crys(X_k/W(k))
\]
is an isomorphism for $i=0,1$ and is injective for $i=2$ with torsion-free cokernel, if $X$ was chosen sufficiently ample. This follows from a general weak Lefschetz theorem for crystalline cohomology by Berthelot, \cite{BerthelotLefschetz}, but can also be readily checked by hand by reducing to the similar question for $H^i_\dR(D_k)\to H^i_\dR(X_k)$, cf.~Lemma~\ref{lem:weaklefschetz} below.
\end{proof}

\begin{remark} In this example, the cospecialization map
\[
H^2_\sub{\'et}(X_{\bar{\bb F}_2},\bb Z_2)\to H^2_\sub{\'et}(X_{\bar{\bb Q}_2},\bb Z_2)
\]
is not injective. Indeed, the left side contains a torsion class coming from the pullback of the $\bb Z/2\bb Z$-cover $\tilde{S}_{\bar{\bb F}_2}\to S_{\bar{\bb F}_2}$, whereas the right side is torsion-free.
\end{remark}

\begin{remark}
\label{rmk:NotFunctorSpecialFibre}
In this example, the $3$-fold $D$  provides one  lift of the smooth projective $k$-scheme $D_k \simeq S_k \times E_k$ to $\mathbb{Z}_2$, and has $H^2_{\sub{\'et}}(D_{\bar{\mathbb{Q}}_2},\mathbb{Z}_2)$ being torsion-free. On the other hand, the $3$-fold $D' := S \times E$ gives another lift of $D_k$ to $\mathbb{Z}_2$ such that $H^2_{\sub{\'et}}(D'_{\bar{\mathbb{Q}}_2},\mathbb{Z}_2)$ contains $2$-torsion coming from $S$. Thus, the torsion in the \'etale cohomology of the generic fibre of a smooth and proper $\mathbb{Z}_2$-scheme is not a functor of the special fibre. In particular, the theory $R\Gamma_{A_\inf}(\mathfrak{X})$ from Theorem~\ref{ThmB} is not a functor of the special fibre $\mathfrak{X}_{\roi/p}$; in fact, not even $R\Gamma_{A_\inf}(\mathfrak X)/p$ is.
\end{remark}

\subsection{An example of degenerating torsion in cohomology}
\label{ss:ExampleDegeneratingTorsion}

Let $\mathcal{O}$ be the ring of integers in a complete nonarchimedean algebraically closed extension $C$ of $\mathbf{Q}_p$.\footnote{One can also realize the example over some sufficiently ramified finite extension of $\bb Q_p$.} Let $k$ be the residue field of $\mathcal{O}$. The goal of this section is to give an example of a smooth projective surface $H/\mathcal{O}$ such that the torsion in cohomology degenerates from $\mathbb{Z}/p^2\bb Z$ (in the \'etale cohomology of $H_C$) to $k \oplus k$ (in the crystalline cohomology of $H_k$); the precise statement is recorded in Theorem \ref{thm:ExampleDegeneratingTorsion}. 

\subsubsection{The construction}
\label{ss:ExampleConstruction}

The strategy of the construction is to first produce an example of the desired phenomenon in the world of algebraic stacks by using an interesting degeneration of group schemes; later, we will push the example to varieties. The basic idea is to degenerate the constant group scheme $\mathbb{Z}/p^2\bb Z$ to a group scheme that is killed by $p$; this is not possible in characteristic $0$, but can be accomplished over a mixed characteristic base.

\begin{lemma}
\label{lem:ConstructFFGS} Let $E/\roi$ be an elliptic curve with supersingular reduction. Let $x\in E(C)$ be a point of exact order $p^2$, and let $G\subset E$ be the flat closure of the subgroup generated by $x$. Then $G_C \simeq \mathbb{Z}/p^2\bb Z$ and $G_k=E_k[p]$.
\end{lemma}

\begin{proof} We only need to identify $G_k\subset E_k$; but $E_k$ has a unique subgroup of order $p^r$ for any $r$, given by the kernel of the $r$-fold Frobenius. Thus, $G_k=E_k[p]$ as both are subgroups of order $p^2$.
\end{proof}

\begin{remark}
\label{rmk:BGisExample} With suitable definitions of \'etale and crystalline cohomology for stacks, the classifying stack $BG$ of the group scheme constructed in Lemma \ref{lem:ConstructFFGS} is a proper smooth stack over $\mathcal{O}$, and satisfies: $H^2_\sub{\'et}(BG_C,\mathbb{Z}_p) \simeq \mathbb{Z}/p^2\bb Z$, while $H^2_{\crys}(BG_k/W(k)) \simeq k \oplus k$; this follows from the computations given later in the section.
\end{remark}

We now fix a finite flat group scheme $G$ sitting in an elliptic curve $E$ with supersingular reduction as above. Our goal is to approximate $BG$ by a smooth projective variety in a way that reflects the phenomenon in Remark \ref{rmk:BGisExample}. First, we find a convenient action of $G$ on a projective space. (In fact, the construction below applies to any finite flat group scheme $G$.)

\begin{lemma}
\label{lem:ConstructFFGSSpace}
There exists a projective space $P/\mathcal{O}$ with an action of $G$ such that the locus $Z_P \subset P$ of points with non-trivial stabilizers has codimension $>2$ on the special fibre.
\end{lemma}

\begin{remark} The number $2$ in Lemma \ref{lem:ConstructFFGSSpace} can be replaced by any positive integer.
\end{remark}

The closed set $Z_P \subset P$ mentioned above is (by definition) the complement of the maximal open $U_P \subset P$ with the following property: the base change $b:F \to P$ of the action map $a:G \times P \to P \times P$ given by $(g,x) \mapsto (gx, x)$ along the diagonal $\Delta:P \to P \times P$ is an isomorphism over $U_P$. As $b$ is finite surjective, one can alternately characterize the closed subset $Z_P \subset P$ by the following two equivalent conditions: 
\begin{enumerate}
\item $Z_P$ is the set of those $x \in P$ such that the fibre of $b$ over $\kappa(x)$ has length $> 1$.
\item $Z_P$ is the support of $b_* \mathcal{O}_F / \mathcal{O}_P$.
\end{enumerate}
In particular, the formation of $U_P$ and $Z_P$ (as subsets of $P$) commutes with taking fibres over points of $\mathrm{Spec}(\mathcal{O})$, and they are both $G$-stable subsets of $P$.

\begin{proof}
Choose a faithful representation $G \to \mathrm{GL}(V)$, inducing a $G$-action on $\mathbf{P}(V)$. By replacing $V$ if necessary, we may also assume that the $G$-action on $\mathbf{P}(V)$ is faithful on each fibre. In particular, there is a maximal $G$-stable open $U \subset \mathbf{P}(V)$ that is fiberwise dense such that the $G$-action on $U$ has no stabilizers (constructed as $U_P$ above). The complement $Y \subset \mathbf{P}(V)$ is a closed subset that has codimension $\geq 1$ on each fibre. Now fix an integer $c > 2$, and consider the induced $G$-action on $W := \prod_{i=1}^c V$. Set $P := \mathbf{P}(W)$. We claim that this satisfies the conclusion of the lemma.

Let $\tilde{U} \subset V - \{0\}$ be the inverse image of $U$ under $V - \{0\} \to \mathbf{P}(V)$, and let $\tilde{Y} = V - \tilde{U}$, so $\tilde{Y} - \{0\}$ is the inverse image of $Y$. Note that $\tilde{Y}$, equipped with its reduced structure, is a $\mathbf{G}_m$-equivariant closed subset of $V$ with codimension $\geq 1$ on each fibre. Now consider $\tilde{Z'} := \prod_{i=1}^c \tilde{Y}  \subset W := \prod_{i=1}^c V$. Then $\tilde{Z'}$ (say with its reduced structure) defines a $\mathbf{G}_m$-equivariant closed subset of $W$ of codimension $\geq c$ on each fibre. Removing $0$ and quotienting by $\mathbf{G}_m$ defines a proper closed subset $Z' \subset P$ of codimension $\geq c$ on each fibre. It is easy to see that the locus $Z_P \subset P$ of points with non-trivial stabilizers is contained in $Z'$, so $Z_P$ also has codimension $\geq c>2$ on each fibre.
\end{proof}

Choose $P$ and $G$ as  in Lemma \ref{lem:ConstructFFGSSpace}. We can use this action to approximate $BG$ by passing to the quotient as follows. Let $h: P\to X = P/G$ be the scheme-theoretic quotient, so that $X$ is a projective scheme, flat over $\roi$. Inside $X$, we have the open subset $U_X\subset X$ defined as the quotient $U_P/G$, with complement $Z_X = X\setminus U_X$. 
 
\begin{lemma} The construction satisfies the following properties.
\begin{enumerate}
\item The closed subset $Z_X\subset X$ has codimension $>2$ on the special fibre. 
\item The map $X\to \mathrm{Spec}(\mathcal{O})$ is smooth over $U_X$. 
\end{enumerate}
\end{lemma}

\begin{proof}
The map $h$ is finite surjective and $G$-equivariant. Our construction shows that $h(Z_P) = Z_X$, giving (i). For (ii), observe that $U_P \to U_X$ is a $G$-torsor, and thus faithfully flat. Moreover, the formation of this map is compatible by base change. Thus, since $U_P$ is smooth, so is $U_X$: It is enough to check that $U_{X,k}$ is regular (by the fibral criterion of smoothness), equivalently of finite Tor-dimension, which follows from the existence of the faithfully flat map $U_{P,k}\to U_{X,k}$ from the regular scheme $U_{P,k}$.
\end{proof}

We now fix a very ample line bundle $L$ on $X$ once and for all. Let $H\subset X$ be a smooth complete intersection of $\dim(P) - 2$ hypersurfaces of sufficiently large degree such that $H\subset U_X$. Such $H$ exist, as $Z_X \subset X$ has codimension $>2$ on the special fibre, so a general complete intersection surface $H$ will miss $Z_X$, i.e., $H \cap Z_X = \emptyset$ (first on the special fibre, and thus globally by properness); thus, $H \subset U_X$. Since $U_X$ is smooth, the general such $H$ will also be smooth by Bertini.

We will check that $H$ is a sufficiently good approximation to $BG$ for our purposes. More precisely:

\begin{theorem}
\label{thm:ExampleDegeneratingTorsion}
The above construction gives a smooth projective (relative) surface $H$ over $\Spec(\mathcal{O})$ such that $H^2_\sub{\'et}(H_C, \mathbb{Z}_p)_{\mathrm{tor}} \simeq \mathbb{Z}/p^2\bb Z$, while $H^2_{\crys}(H_k/W(k))_{\mathrm{tor}} \simeq k \oplus k$.
\end{theorem}

\begin{remark} In this example, one can also show that $H^1_\sub{\'et}(H_C, \mathbb{Z}/p) \simeq \mathbb{Z}/p$, while $H^1_{\dR}(H_k) \simeq k \oplus k$. Thus, the inequality $\dim_{\mathbb{F}_p} H^i(H_C, \mathbb{F}_p) \leq \dim_k H^i_{\dR}(H_k)$ coming from Theorem~\ref{ThmA} (ii) can be strict.
\end{remark}

\begin{proof} For \'etale cohomology, let $\tilde{H}\subset P$ be the preimage of $H$, so $\tilde{H}\to H$ is a $G$-torsor. As $\tilde{H}_C\subset P_C$ is a smooth complete intersection of ample hypersurfaces, the weak Lefschetz theorem implies that $H^i_\sub{\'et}(\tilde{H}_C,\bb Z_p)$ is given by $\bb Z_p$, $0$, and a torsion-free group, in degrees $0$, $1$, and $2$, respectively. Now we use the Leray spectral sequence for the $G_C\cong \bb Z/p^2\bb Z$-cover $\tilde{H}_C\to H_C$,
\[
H^i(\bb Z/p^2 \bb Z,H^j_\sub{\'et}(\tilde{H}_C,\bb Z_p))\Rightarrow H^{i+j}_\sub{\'et}(H_C,\bb Z_p)\ .
\]
This implies that
\[
H^2_\sub{\'et}(H_C,\bb Z_p)_\sub{tor} = H^2(\bb Z/p^2 \bb Z,\bb Z_p) = \bb Z/p^2\bb Z\ .
\]

For crystalline cohomology, consider the quotient $P\times E\to (P\times E)/G=:X_E$. As $G$ acts freely on $E$, and thus on $P\times E$, this is a $G$-torsor. We have a projection $X_E=(P\times E)/G\to X=P/G$, which is an $E$-torsor $U_{X_E}\to U_X$ over the open subset $U_X$. In particular, over $H\subset U_X$, we get an $E$-torsor $H_E\to H$.

Now note that $H_E\subset X_E=(P\times E)/G$ is a smooth intersection of $\dim P - 2$ sections of $L_E^{\otimes n}$, for sufficiently large $n$, where $L_E$ on $X_E$ is the pullback of the ample line bundle $L$ on $X=P/G$. Note that $L_E$ is not ample, but it has the weakened Serre vanishing property that for any coherent sheaf $\cal F$ on $X_E$, $H^i(X_E,\cal F\otimes L_E^{\otimes n})=0$ for all sufficiently large $n$ and $i>1$. Indeed, this follows from Serre vanishing on $X$ and the Leray spectral sequence for $X_E\to X$. By a version of the weak Lefschetz theorem in crystalline cohomology, cf.~Lemma~\ref{lem:weaklefschetz} below, we see that the map
\[
H^i_\crys(X_{E,k}/W(k))\to H^i(H_{E,k}/W(k))
\]
is an isomorphism for $i=0,1$, and injective with torsion-free cokernel for $i=2$. The left side can be computed by using the Leray spectral sequence for the projection $X_{E,k}=(P_k\times E_k)/G_k\to E_k/G_k\cong E_k$, with fibres given by $P_k$. The result is that for $i=0,1$, the composite map $H_{E,k}\to X_{E,k}\to E_k/G_k\cong E_k$ induces an isomorphism
\[
H^i_\crys(E_k/W(k))\isoto H^i_\crys(H_{E,k}/W(k))\ ,
\]
and $H^2_\crys(H_{E,k}/W(k))$ is torsion-free. 

Now we consider the $E_k$-torsor $f: H_{E,k}\to H_k$, and the associated Leray spectral sequence
\[
H^i_\crys(H_k,R^j f_{\crys\ast} \roi_{H_{E,k}})\Rightarrow H^{i+j}_\crys(H_{E,k}/W(k))\ .
\]
In particular, in low degrees, we get a long exact sequence
\begin{equation}\label{eq:specseq}
\begin{aligned}
0\to H^1_\crys(H_k/W(k))&\to H^1_\crys(H_{E,k}/W(k))\xTo{a} H^0_\crys(H_k,R^1 f_{\crys\ast} \roi_{H_{E,k}})\\
&\to H^2_\crys(H_k/W(k))\to H^2_\crys(H_{E,k}/W(k))\to \ldots\ .
\end{aligned}
\end{equation}
Fix a point $x\in H_k$; then the map $a$ can be analyzed through the composition
\[
H^1_\crys(E_k/W(k))\isoto H^1_\crys(H_{E,k}/W(k))\xTo{a} H^0_\crys(H_k,R^1 f_{\crys\ast} \roi_{H_{E,k}})\buildrel {x^\ast}\over\hookrightarrow H^1_\crys(E_k/W(k))\ .
\]
Here $x^\ast$ is the map given by restriction to the fibre $E_k$ of $H_{E_k}\to H_k$ over $x$. The induced endomorphism of $H^1_\crys(E_k/W(k))$ is induced by the map $E_k\to E_k/G_k = E_k/E_k[p]\cong E_k$, and is thus given by multiplication by $p$. This is injective, so it follows that $a$ is injective. Moreover, the image of $x^\ast$ is saturated, which forces $x^\ast$ to be an isomorphism. It follows that $a$ is injective, with cokernel given by $H^1_\crys(E_k/W(k))/p\cong k\oplus k$.

Coming back to the sequence~\eqref{eq:specseq}, we find $H^1_\crys(H_k/W(k))=0$, while $H^2_\crys(H_k/W(k))_\sub{tor} = k\oplus k$, as desired.
\end{proof}

The following version of weak Lefschetz was used in the proof.

\begin{lemma}\label{lem:weaklefschetz} Let $k$ be a perfect field of characteristic $p$, and let $X$ be a smooth projective variety of dimension $d$ over $k$, with a line bundle $L$. Let $i_L\geq 0$ be an integer such that for any coherent sheaf $\cal F$ on $X$, the cohomology group $H^i(X,\cal F\otimes L^{\otimes n})$ vanishes if $n$ is sufficiently large and $i>i_L$.

Then there exists some integer $n_0$ such that for all $n\geq n_0$ and any smooth hypersurface $H\subset X$ with divisor $L^{\otimes n}$, the map
\[
H^j_\crys(X/W(k))\to H^j_\crys(H/W(k))
\]
is an isomorphism for $j<d-i_L-1$, and injective with torsion-free cokernel for $j=d-i_L-1$.
\end{lemma}

\begin{proof} Berthelot, \cite{BerthelotLefschetz}, proved this when $L$ is ample, i.e.~$i_L=0$. His proof immediately gives the general result: Let $K$ be the cone of $R\Gamma_\crys(X/W(k))\to R\Gamma_\crys(H/W(k))$. It suffices to show that $K\in D^{\geq d-i_L-1}$, with $H^{d-i_L-1}(K)$ torsion-free. As $K$ is $p$-complete, this is equivalent to proving that $K/p\in D^{\geq d-i_L-1}$. But $K/p$ is the cone of $R\Gamma_\dR(X)\to R\Gamma_\dR(H)$. Thus, it suffices to prove that for any $j\geq 0$, the cone $K_j$ of
\[
R\Gamma(X,\Omega^j_X)\to R\Gamma(H,\Omega^j_H)
\]
lies in $D^{\geq d-i_L-j-1}$. Let $\cal I\subset \roi_X$ be the ideal sheaf of $H$; then $\cal I\cong L^{\otimes -n}$. Now we have a short exact sequence
\[
0\to \cal I\otimes_{\roi_X} \Omega^{j-1}_H\to \Omega^j_X/\cal I\to \Omega^j_H\to 0\ .
\]
As $R\Gamma(X,\cal I\otimes_{\roi_X} \Omega^i_X)$ is Serre dual to $R\Gamma(X,L^{\otimes n}\otimes _{\roi_X} \Omega^{d-i}_X)$, it lies in $D^{\geq d-i_L}$ if $n$ is large enough. It remains to see that
\[
R\Gamma(X,\cal I\otimes_{\roi_X} \Omega^{j-1}_H)\in D^{\geq d-i_L-j}\ ,
\]
if $n$ is large enough; we will prove more generally that for any fixed $r\geq 1$,
\[
R\Gamma(X,\cal I^{\otimes r}\otimes_{\roi_X} \Omega^{j-1}_H)\in D^{\geq d-i_L-j}\ ,
\]
if $n$ is large enough. For this, we induct on $j$. If $j=1$, we use the short exact sequence
\[
0\to \cal I^{\otimes (r+1)}\to \cal I^{\otimes r}\to \cal I^{\otimes r}\otimes_{\roi_X} \roi_H\to 0
\]
to reduce to $R\Gamma(X,L^{\otimes -rn})\in D^{\geq d-i_L}$ (and with $r+1$ in place of $r$) for sufficiently large $n$, which follows from Serre duality and the assumption on $L$. For $j>1$, we have a short exact sequence
\[
0\to \cal I^{\otimes (r+1)}\otimes_{\roi_X} \Omega^{j-2}_H\to \cal I^{\otimes r}\otimes_{\roi_X} \Omega^{j-1}_X\to \cal I^{\otimes r}\otimes_{\roi_X} \Omega^{j-1}_H\to 0\ .
\]
By induction, $R\Gamma$ of the first term lies in $D^{\geq d-i_L-j+1}$, and $R\Gamma$ of the second term lies in $D^{\geq d-i_L}$; this gives the required bound on the last term.
\end{proof}

\newpage

\section{Algebraic preliminaries on perfectoid rings}\label{sec:perfectoid}

The goal of this section is to record some facts about perfectoid rings. In \S \ref{ss:theta}, we recall a slightly non-standard perspective on Fontaine's ring $A_\inf$. In particular, we introduce the $\theta_r$ and $\tilde{\theta}_r$ maps which play a crucial role in the rest of the paper; the construction applies to a fairly large class of rings. In \S \ref{ss:PerfectoidOverview}, we specialize these constructions to perfectoid rings; with an eye towards our intended application, we analyze the kernel of the $\theta_r$ and $\tilde{\theta}_r$ maps in the case of perfectoid rings with enough roots of unity. Along the way, we try to summarize the definitions and relations between various classes of perfectoid rings in the literature. Finally, in \S \ref{ss:PerfectoidField}, we collect some results on perfectoid fields; notably, we prove in Proposition~\ref{prop:WrCoherent} that $W_r(\roi)$ is coherent for the ring of integers $\roi$ in a perfectoid field. 

\subsection{Fontaine's ring $A_\inf$}
\label{ss:theta}

Fix a prime number $p$, and let $S$ be a commutative ring which is $\pi$-adically complete and separated for some element $\pi\in S$ dividing $p$ (note that it follows that $S$ is $p$-adically complete by, for example, \cite[Tag 090T]{StacksProject}). Denoting by $\phi:S/pS\to S/pS$ the absolute Frobenius, let $S^\flat:=\projlim_\phi S/pS$ be the tilt of $S$, which is a perfect $\bb F_p$-algebra on which we will continue to denote the Frobenius by $\phi$. In this situation, we have Fontaine's ring $\bb A_\inf(S)$.

\begin{definition} Fontaine's ring is given by
\[
\bb A_\inf(S) = W(S^\flat)\ ,
\]
which is equipped with a Frobenius automorphism $\phi$.
\end{definition}

We start by recalling a slightly nonstandard perspective on $\bb A_\inf(S)$.

\begin{lemma}\label{lemma_witt_alg_1}
Let $S$ be as above, i.e., a ring which is $\pi$-adically complete with respect to some element $\pi\in S$ dividing $p$.
\begin{enumerate}
\item The canonical maps \[\projlim_{x\mapsto x^p}S\To S^\flat=\projlim_\phi S/p S\To\projlim_\phi S/\pi S\] are isomorphisms of monoids/rings.
\item For any $f\in S$, the following inclusions hold: $W_r(f^{p^{r-1}}S)\subset [f]W_r(S)\subset W_r(f S)$; also $[p]^2\in pW_r(S)$ and $p^rW_r(S)\subseteq W_r(pS)$. It follows that the rings $W_r(S)$ and $W(S)$ are complete for the $[\pi]$, $[p]$, and $p$-adic topologies.
\item The homomorphism
\[
\phi^\infty:\projlim_F W_r(S^\flat)\To \projlim_R W_r(S^\flat)\ ,
\]
induced by the homomorphisms $\phi^r:W_r(S^\flat)\to W_r(S^\flat)$ for $r\geq 1$, is an isomorphism.
\item The homomorphism
\[
\projlim_F W_r(S^\flat)\To \projlim_F W_r(S/\pi S)\ ,
\]
induced by the canonical map $S^\flat\to S/\pi S$, is an isomorphism.
\item The canonical homomorphism
\[
\projlim_F W_r(S)\To \projlim_F W_r(S/\pi S)
\]
is an isomorphism.
\end{enumerate}

In particular, there is a canonical isomorphism
\[
\bb A_\inf(S)\cong \projlim_F W_r(S)\ .
\]
Under this identification, the restriction operator $R$ on the right side gets identified with $\phi^{-1}$ on the left side; in particular, $R$ is an automorphism of $\projlim_F W_r(S)$.
\end{lemma}

\begin{proof} Parts (i) and (ii) are standard: for example, the proof of (i) is just as in \cite[Lem.~3.4(i)]{ScholzeThesis}; to see that $[p]^2\in pW_r(S)$ (which is true already for $S=\bb Z$), note that $[p]\in VW_{r-1}(S) + pW_r(S)$, and $VW_{r-1}(S)^2\subset pW_r(S)$ as follows from the identity
\[
V^i[x]V^j[y] = V^i([x]\cdot F^iV^j[y]) = p^j V^i([xy^{p^{i-j}}])
\]
(using $V(aF(b))=V(a)b$ and $FV=p$) for $i\geq j$. Also, $p^r=0$ in $W_r(S/p)$ (as $W_r(\mathbb F_p)=\mathbb Z/p^r\mathbb Z$), so $p^rW_r(S)\subset W_r(pS)$. Part (iii) is a trivial consequence of $S^\flat$ being perfect.

For part (iv), note that since $W_r$ commutes with inverse limits of rings we have, using (i), \[\projlim_F W_r(S^\flat)=\projlim_F\projlim_\phi W_r(S/\pi S)=\projlim_\phi \projlim_F W_r(S/\pi S)\isoto  \projlim_F W_r(S/\pi S),\] where the final projection is an isomorphism since $\phi$ induces an automorphism of the ring $\projlim_F W_r(S/\pi S)$ (thanks to the formulae $R\phi=\phi R=F$ in characteristic $p$).

Finally, for part (v): For any fixed $s\ge1$ we claim first that the canonical morphism of pro-rings
\[
\{W_r(S/\pi^s S)\}_{r\sub{ wrt }F}\to \{W_r(S/\pi S)\}_{r\sub{ wrt }F}
\]
is an isomorphism. As it is surjective, it is sufficient to show that the kernel $\{W_r(\pi S/\pi^s S)\}_s$ is pro-isomorphic to zero; fix $r\ge 1$. By (ii), there is some $c$ such that $p^c$ is zero in $W_r(S/\pi^s S)$, and we claim that $F^{s+c}:W_{r+s+c}(S/\pi^s S)\to W_r(S/\pi^s S)$ kills the kernel $W_{r+s+c}(\pi S/\pi^s S)$. Indeed, the kernel is generated by elements $V^i[a]$ for $i\geq 0$, $a\in\pi S/\pi^sS$, and $F^{s+c} V^i[a] =0\in W_r(S/\pi^sS)$ as either $i\ge c$, in which case $F^{s+c} V^i[a]=p^cF^sV^{i-c}[a]=0$, or else $i<c$, in which case $F^{s+c} V^i[a]= p^i [a]^{p^{s+c-i}}=0$. This proves the desired pro-isomorphism, from which it follows that
\[
\projlim_F W_r(S/\pi^s S)\isoto \projlim_F W_r(S/\pi S)\ .
\]
Taking the limit over $s\ge 1$, exchanging the order of the limits, and using $W_r(S)=\projlim_sW_r(S/\pi^sS)$ completes the proof.
\end{proof}

Continue to let $S$ be as in the previous lemma. According to the lemma there is a chain of isomorphisms
\[
\bb A_\inf(S)=\projlim_R W_r(S^\flat)\stackrel{\phi^\infty}{\longleftarrow}\projlim_F W_r(S^\flat)\To \projlim_F W_r(S/\pi S)\longleftarrow \projlim_F W_r(S)\ ,
\]
through which each canonical projection $\projlim_F W_r(S)\To W_r(S)$ induces a homomorphism
\[
\tilde\theta_r:\bb A_\inf(S)\to W_r(S)\ .
\]
Denoting by $\phi$ the Frobenius on $\bb A_\inf(S)$, we define
\[
\theta_r:=\tilde\theta_r\phi^r:\bb A_\inf(S)\To W_r(S)
\]
for each $r\ge 1$. The maps $\theta_r$ and especially $\tilde\theta_r$ are of central importance in the comparison between the theory developed in this paper, and the theory of de~Rham--Witt complexes.

Explicitly, identifying $\projlim_{x\mapsto x^p}S$ and $S^\flat$ as monoids by Lemma \ref{lemma_witt_alg_1}(i) and following the usual convention of denoting an element $x$ of $S^\flat$ as $x=(x^{(0)},x^{(1)},\dots)\in \projlim_{x\mapsto x^p}S$, these maps are described as follows.

\begin{lemma}\label{lemma_theta}
For any $x\in S^\flat$ we have $\theta_r([x])=[x^{(0)}]\in W_r(S)$ and $\tilde\theta_r([x]) = [x^{(r)}]$ for $r\ge 1$.
\end{lemma}

\begin{proof}
This follows from a straightforward chase through the above isomorphisms.
\end{proof}

In particular Lemma \ref{lemma_theta} implies that $\theta:=\theta_1:\bb A_\inf(S)\to S$ (and not $\tilde\theta_1$) is the usual map of $p$-adic Hodge theory, and also shows that the diagram \[\xymatrix{
\bb A_\inf(S)\ar[r]^-{\theta_r}\ar[d]_R & W_r(S)\ar[d]\\
W_r(S^\flat)\ar[r] & W_r(S/pS)
}\]
commutes, where the bottom arrow is induced by the canonical map $S^\flat=\projlim_\phi S/p S\to S/p S$, $x\mapsto x^{(0)}$. Indeed, by $p$-adic continuity it is sufficient to check commutativity of the diagram on Teichm\"uller lifts, for which it follows immediately from the previous lemma.

Further functorial properties of the maps $\theta_r$ are presented in the following lemma.

\begin{lemma}\label{lemma_theta_r_diagrams}
Continue to let $S$ be as in the previous two lemmas. Then the following diagrams commute:
\[\xymatrix{
\bb A_\inf(S)\ar[d]_{\op{id}}\ar[r]^-{\theta_{r+1}}& W_{r+1}(S)\ar[d]^R\\
\bb A_\inf(S)\ar[r]^-{\theta_r}& W_r(S)\\
}\quad
\xymatrix{
\bb A_\inf(S)\ar[d]_{\phi}\ar[r]^-{\theta_{r+1}}& W_{r+1}(S)\ar[d]^F\\
\bb A_\inf(S)\ar[r]^-{\theta_r}& W_r(S)\\
}\quad
\xymatrix{
\bb A_\inf(S)\ar[r]^-{\theta_{r+1}}& W_{r+1}(S)\\
\bb A_\inf(S)\ar[r]^-{\theta_r}\ar[u]^{\lambda_{r+1}\phi^{-1}}& W_r(S)\ar[u]_V\\
}
\]
where the third diagram requires an element $\lambda_{r+1}\in \bb A_\inf(S)$ satisfying $\theta_{r+1}(\lambda_{r+1})=V(1)$ in $W_{r+1}(S)$.

Equivalently, the following diagrams involving $\tilde\theta_r$ commute.
\[\xymatrix{
\bb A_\inf(S)\ar[d]_{\phi^{-1}}\ar[r]^-{\tilde\theta_{r+1}}& W_{r+1}(S)\ar[d]^R\\
\bb A_\inf(S)\ar[r]^-{\tilde\theta_r}& W_r(S)\\
}\quad
\xymatrix{
\bb A_\inf(S)\ar[d]_{\op{id}}\ar[r]^-{\tilde\theta_{r+1}}& W_{r+1}(S)\ar[d]^F\\
\bb A_\inf(S)\ar[r]^-{\tilde\theta_r}& W_r(S)\\
}\quad
\xymatrix{
\bb A_\inf(S)\ar[r]^-{\tilde\theta_{r+1}}& W_{r+1}(S)\\
\bb A_\inf(S)\ar[r]^-{\tilde\theta_r}\ar[u]^{\tilde\lambda_{r+1}}& W_r(S)\ar[u]_V\\
}
\]
Here, $\tilde\lambda_{r+1} = \varphi^{r+1}(\lambda_{r+1})\in \bb A_\inf(S)$ is an element satisfying $\tilde\theta_{r+1}(\tilde\lambda_{r+1}) = V(1)\in W_{r+1}(S)$.
\end{lemma}

\begin{proof} We check that the second set of squares commute. Under the above chain of isomorphisms $\bb A_\inf(S)\cong\projlim_F W_r(S)$, we showed in Lemma \ref{lemma_witt_alg_1} that the action of $\phi^{-1}$ on $\bb A_\inf(S)$ corresponds to that of the restriction map $R$ on $\projlim_F W_r(S)$; hence the diagram
\[\xymatrix{
\bb A_\inf(S)\ar[d]_{\phi^{-1}}\ar[r]^-{\tilde\theta_{r+1}}& W_{r+1}(S)\ar[d]^R\\
\bb A_\inf(S)\ar[r]^-{\tilde\theta_r}& W_r(S)\\
}\]
commutes. Commutativity of the second diagram follows from the definition of the maps $\tilde\theta_r$.

Finally, using commutativity of the second diagram, the commutativity of the third diagram follows from the fact that $VF$ is multiplication by $V(1)$ on $W_{r+1}(S)$.
\end{proof}

By the first diagram in the previous lemma, we may let $r\to\infty$ to define a map $\theta_\infty:\bb A_\inf(S)\to W(S)$ satisfying $\theta_\infty([x])=[x^{(0)}]$ for any $x\in S^\flat$. We will analyze this map further in Lemma~\ref{lem:mapAinfWitt} below.

\subsection{Perfectoid rings}
\label{ss:PerfectoidOverview}

We will be interested in the following class of rings.

\begin{definition} A ring $S$ is {\em perfectoid} if and only if it is $\pi$-adically complete for some element $\pi\in S$ such that $\pi^p$ divides $p$, the Frobenius map $\phi: S/pS\to S/pS$ is surjective, and the kernel of $\theta: \bb A_\inf(S)\to S$ is principal.
\end{definition}

\begin{example} 
\label{ex:PerfectoidEasy}
The following rings are examples of perfectoid algebras. First, any perfect $\bb F_p$-algebra is perfectoid (where we take $\pi=0$); here, perfect means that the Frobenius map is an isomorphism. Moreover, the $p$-adic completion $\bb Z_p^\cycl$ of $\bb Z_p[\zeta_{p^\infty}]$ is perfectoid; one may also take the $p$-adic completion of the ring of integers of any other algebraic extension of $\bb Q_p$ containing the cyclotomic extension. Another example is $\bb Z_p^\cycl\langle T^{1/p^\infty}\rangle$, and there are many obvious variants.
\end{example}

\begin{remark} The original definition, \cite{ScholzeThesis}, of a perfectoid $K$-algebra, where $K$ is a perfectoid field, was in a slightly different context. We refer to Lemma~\ref{lemma_Tate_perfectoid} below for the relation.
\end{remark}

\begin{remark} In \cite{GabberRamero2}, Gabber and Ramero define a ``perfectoid'' condition for a complete topological ring $S$ carrying the $I$-adic topology for some finitely generated ideal $I$. In fact, $S$ is perfectoid in their sense if and only if $S$ (as a ring without topology) is perfectoid in the sense of the definition above: From \cite[Proposition~14.2.9]{GabberRamero2}, it already follows that their definition is independent of the topology (which can be taken to be the $p$-adic topology). Now \cite[Lemma~14.1.16 (iv)]{GabberRamero2} shows that if $S$ is perfectoid in their sense, then there exists a $\pi\in S$ and a unit $u\in S^\times$ such $\pi^p = pu$, and $\phi: S/p S\to S/pS$ is surjective. The last condition that $\ker \theta$ is principal is part of their definition of a perfectoid ring. Conversely, if $S$ is perfectoid in our sense and we endow it with the $p$-adic topology, then by Lemma \ref{lemma_frobenius_surjectivity} below, there exists $\pi\in S$ and a unit $u\in S^\times$ such that $\pi^p = pu$; taking $I=(\pi)$ shows that $S$ is a P-ring in the sense of \cite[Definition 14.1.14]{GabberRamero2}. Among P-rings, perfectoid rings in their sense are singled out by having the property that $\ker \theta$ is principal, \cite[Definition 14.2.1]{GabberRamero2}, which is also part of our definition.
\end{remark}

In relation to this, let us discuss surjectivity properties of the Frobenius:

\begin{lemma}\label{lemma_frobenius_surjectivity}
Let $S$ be a ring which is $\pi$-adically complete with respect to some element $\pi\in S$ such that $\pi^p$ divides $p$. Then the following are equivalent:
\begin{enumerate}\itemsep0pt
\item Every element of $S/\pi pS$ is a $p^\sub{th}$-power.
\item Every element of $S/pS$ is a $p^\sub{th}$-power.
\item Every element of $S/\pi^pS$ is a $p^\sub{th}$-power.
\item The Witt vector Frobenius $F:W_{r+1}(S)\to W_r(S)$ is surjective for all $r\ge1$.
\item The map $\theta_r:\bb A_\inf(S)\to W_r(S)$ is surjective for all $r\ge1$.
\end{enumerate}
Moreover, if these equivalent conditions hold then there exist $u, v\in S^\times$ such that $u\pi$ and $vp$ admit systems of $p$-power roots in $S$.
\end{lemma}

\begin{proof}
The implications (i)$\Rightarrow$(ii)$\Rightarrow$(iii) are trivial since $\pi pS\subset pS\subset\pi^pS$.

Assuming (iii), a simple inductive argument allows us to write any given element $x\in S$ as an infinite sum $x=\sum_{i=0}^\infty x_i^p\pi^{pi}$ for some $x_i\in S$; but then $x\equiv(\sum_{i=0}^\infty x_i\pi^i)^p$ mod $p\pi S$, establishing~(i).


Condition (iv) states that the transition maps in the inverse system $\projlim_F W_r(S)$ are surjective, which implies that each map $\tilde\theta_r$ is surjective, and hence that each map $\theta_r$ is surjective, i.e., (v).

Next, (v) implies (ii) since any element of $S$ in the image of $\theta=\theta_1$ is a $p^\sub{th}$-power mod $p$.

It remains to show that (ii) implies (iv), but we will first prove the ``moreover'' assertion using only (i). Applying Lemma \ref{lemma_witt_alg_1}(i) to both $S$ and $S/\pi p$ implies that the canonical map $\projlim_{x\mapsto x^p}S\to \projlim_{x\mapsto x^p}S/\pi p$ is an isomorphism. Applying (i) repeatedly, there therefore exists $\omega\in\projlim_{x\mapsto x^p}S$ such that $\omega^{(0)}\equiv \pi$ mod $\pi pS$ (resp.~ $\equiv p$ mod $\pi pS$). Writing $\omega^{(0)}=\pi+\pi px$ (resp.~$\omega^{(0)}=p+\pi px$) for some $x\in S$, the proof of the ``moreover'' claim is completed by noting that $1+p x\in S^\times$ (resp.~$1+\pi x\in S^\times$).

Finally, assuming (ii) (which we have shown implies (i)), the ``moreover'' assertion implies that there exist $\pi'\in S$ and $v\in S^\times$ satisfying $\pi'^p=vp$. Note that $S$ is $\pi'$-adically complete, and so we may apply the implication (ii)$\Rightarrow$(i) for the element $\pi'$ to deduce that every element of $S/\pi'pS$ is a $p^\sub{th}$-power; it follows that every element of $S/Ip$ is a $p^\sub{th}$-power, where $I$ is the ideal $\{a\in S:a^p\in pS\}$. Now apply implication ``(xiv)$^\prime\Rightarrow$(ii)'' of Davis--Kedlaya \cite{DavisKedlaya} to complete the proof.
\end{proof}

Next, we analyze injectivity of the Frobenius map.

\begin{lemma}\label{lemma_inj_of_Frob}
Let $S$ be a ring which is $\pi$-adically complete with respect to some element $\pi\in S$ such that $\pi^p$ divides $p$, and assume that $\phi: S/\pi S\to S/\pi^p S$ is surjective.
\begin{enumerate}
\item If $\ker\theta$ is a principal ideal of $\bb A_\inf(S)$, then $\phi: S/\pi S\to S/\pi^p S$ is an isomorphism, and any generator of $\ker\theta$ is a non-zero-divisor.
\item Conversely, if $\phi: S/\pi S\to S/\pi^p S$ is an isomorphism and $\pi$ is a non-zero-divisor, then $\ker\theta$ is a principal ideal (and hence $S$ is perfectoid).
\end{enumerate}
\end{lemma}

\begin{proof} Since multiplying $\pi$ by a unit does not affect any of the assertions, we may assume by the previous lemma that $\pi$ admits a compatible sequence of $p$-power roots, i.e., that there exists $\pi^\flat\in S^\flat$ satisfying $\pi^{\flat(0)}=\pi$.

We begin by constructing a certain element of $\ker\theta$ (a ``distinguished'' or ``primitive'' element, cf.~Remark~\ref{remark_distinguished} below). By the hypothesis that $\pi^p$ divides $p$, and Lemma \ref{lemma_frobenius_surjectivity}, it is possible to write $p=\pi^p\theta(-x)$ for some $x\in \bb A_\inf(S)$, whence $\xi=p+[\pi^\flat]^px$ belongs to $\ker\theta$ (recall here that $\theta([\pi^\flat])=\pi$). Then there is a commutative diagram
\[\xymatrix{
\bb A_\inf(S)/\xi\ar[r]^-{\theta}\ar[d]&S\ar[d]\\
\bb A_\inf(S)/(\xi,[\pi^\flat]^p)\ar[r]& S/\pi^p S
}\]
in which the lower left entry identifies with $\bb A_\inf(S)/(p,[\pi^\flat]^p)=S^\flat/\pi^{\flat p} S^\flat$ and the lower horizontal arrow identifies with the map $S^\flat/\pi^{\flat p} S^\flat\to S/\pi^p S$ induced by the canonical projection $S^\flat=\projlim_\phi S/\pi^p S\to S/\pi^p S$.

Suppose first that $\ker\theta$ is principal and let $\xi^\prime$ be a generator; we claim that $\ker\theta$ is actually generated by the element $\xi$. Let $\xi^\prime = (\xi_0^\prime,\xi_1^\prime,\dots)\in \bb A_\inf(S)$ be the Witt vector expansion. Write $\xi=\xi^\prime a$ for some $a\in \bb A_\inf(S)$, and consider the resulting Witt vector expansions:
\[
(\pi^{\flat p}x_0,1+\pi^{\flat p^2}x_1,\dots)=p+[\pi^\flat]^px=\xi=\xi^\prime a=(\xi_0^\prime,\xi_1^\prime,\dots)(a_0,a_1,\dots)=(\xi_0^\prime a_0,\xi_0^{\prime p}a_1+\xi_1^\prime a_0^p,\dots)
\]
It follows that $\xi_1^\prime a_0^p=1+\pi^{\flat p^2}x_1-\xi_0^{\prime p}a_1$. We claim that this is a unit of $S^\flat$. To check this, using that $S^\flat = \varprojlim_\phi S/\pi S$, it is enough to check that the image of $\xi_1^\prime a_0^p$ in $S/\pi S$ is a unit. But this image is simply $1$, as both $\pi^\flat$ and $\xi_0^\prime$ have trivial image in $S/\pi S$. So both $\xi_1^\prime$ and $a_0$ are units of $S^\flat$; in particular, this implies that $a\in \bb A_\inf(S)^\times$, thereby proving that $\xi=\xi^\prime a$ is also a generator of $\ker\theta$, as required.

Now, for part (i), if $\theta: \bb A_\inf(S)/\xi\to S$ is an isomorphism, then so is $S^\flat/\pi^{\flat p}S^\flat\to S/\pi^p S$ by the displayed diagram above. The map $\phi: S/\pi S\to S/\pi^p S$ gets identified with $\phi: S^\flat/\pi^\flat S^\flat\to S^\flat/\pi^{\flat p} S^\flat$, which is an isomorphism. We also need to check that $\xi$ is a non-zero-divisor (as then any other generator of $\ker \theta$ differs from $\xi$ by a unit). So suppose that $b\in \bb A_\inf(S)$ satisfies $(p+[\pi^\flat]^px)b=0$. Then also $(p^r+[\pi^\flat]^{pr}x^r)b=0$ for any odd $r\ge1$, since $p+[\pi^\flat]^px$ divides $p^r+[\pi^\flat]^{pr}x^r$, and so $p^rb\in[\pi^\flat]^{pr}\bb A_\inf(S)$. Using this to examine the Witt vector expansion of $b=(b_0,b_1,\dots)$ shows that $b_i^{p^r}\in\pi^{\flat p^{r+i+1}r} S^\flat$ for each $i\ge0$; hence $b_i\in\pi^{\flat p^{i+1}r} S^\flat$ since $S^\flat$ is perfect. As this holds for all odd $r\ge1$, and as $S^\flat$ is $\pi^\flat$-adically complete and separated, it follows that $b_i=0$ for all $i\ge0$, i.e., $b=0$.

Conversely, for part (ii), assume that $S/\pi S\to S/\pi^p S$ is an isomorphism, and that $\pi$ is a non-zero-divisor in $S$. Note first that the first condition implies that for all $n\geq 0$, $S/\pi^{1/p^n} S\to S/\pi^{1/p^{n-1}} S$ is an isomorphism, by taking the quotient modulo $\pi^{1/p^n}$. This implies that the kernel of $S^\flat\to S/\pi S$ is generated by $\pi^\flat$: Indeed, given $x=(x^{(0)},x^{(1)},\ldots)\in S^\flat = \varprojlim_{x\mapsto x^p} S$ with $x^{(0)}\in \pi S$, one inductively checks that $x^{(n)}$ is divisible by $\pi^{1/p^n}$, using that $\phi: S/\pi^{1/p^n} S\to S/\pi^{1/p^{n-1}} S$ is an isomorphism. This implies that $x$ is divisible by $\pi^\flat$. Thus, we see that $S^\flat/\pi^\flat S^\flat\to S/\pi S$ is an isomorphism. Now let $x\in \bb A_\inf(S)$ satisfy $\theta(x)=0$. Then one can write $x=\xi y_0 + [\pi^\flat] x_1$, where $\pi \theta(x_1)=\theta([\pi^\flat] x_1)=0$. As $\pi$ is a non-zero-divisor, this implies $\theta(x_1)=0$, so we can inductively write $x=\xi(y_0+[\pi^\flat] y_1 + \ldots)$, showing that $\ker\theta$ is generated by $\xi$.
\end{proof}

\begin{remark}[Distinguished elements]\label{remark_distinguished}
Let $S$ be a perfectoid ring, and let $\xi\in\ker\theta$. Then $\xi$ is said to be {\em distinguished} if and only if its Witt vector expansion $\xi=(\xi_0,\xi_1,\dots)$ has the property that $\xi_1$ is a unit of $S^\flat$. The argument in Lemma \ref{lemma_inj_of_Frob} shows that $\xi$ generates $\ker\theta$ if and only if it is distinguished.

For example, let $\xi\in \bb A_\inf(S)$ satisfy $\theta_r(\xi)=V(1)$ in $W_r(S)$ for some $r>1$ (for any fixed $r>1$, such an element $\xi$ does exist by Lemma \ref{lemma_frobenius_surjectivity}(v)). We claim that $\xi$ is a distinguished element of $\ker\theta$, whence it is a generator. Indeed, noting that $V(1)=(0,1,0,\dots,0)$, the first diagram of Lemma \ref{lemma_theta_r_diagrams} shows that $\theta(\xi)=0$, while the commutative diagram immediately before Lemma \ref{lemma_theta_r_diagrams} shows that $\xi_1^{(0)}\equiv1$  mod $p S$, whence $\xi_1$ is a unit of $S^\flat$.
\end{remark}

We return to the maps $\theta_r$, describing their kernels in the case of a perfectoid ring:

\begin{lemma}\label{lemma_ker_theta_r}
Suppose that $S$ is a perfectoid ring, and let $\xi\in \bb A_\inf(S)$ be any element generating $\ker\theta$. Then $\ker\theta_r$ is generated by the non-zero-divisor
\[
\xi_r:=\xi\phi^{-1}(\xi)\cdots\phi^{-(r-1)}(\xi)
\]
for any $r\ge1$. Equivalently, $\ker \tilde\theta_r$ is generated by
\[
\tilde\xi_r := \phi^r(\xi_r) = \phi(\xi)\phi^2(\xi)\cdots\phi^r(\xi)\ .
\]
\end{lemma}

\begin{proof} We prove the result by induction on $r\ge1$, the case $r=1$ being covered by the hypotheses; so fix $r\ge1$ for which the result is true. By the previous remark we may, after multiplying $\xi$ by a unit (depending on the fixed $r\ge1$), assume that $\theta_{r+1}(\xi)=V(1)$. Hence Lemma \ref{lemma_theta_r_diagrams} implies that there is a commutative diagram
\[\xymatrix{
0\ar[r]&\bb A_\inf(S)\ar[r]^-{\xi\phi^{-1}}\ar[d]^{\theta_r}&\bb A_\inf(S)\ar[r]^-\theta\ar[d]^{\theta_{r+1}} & S\ar@{=}[d]\ar[r] & 0\\
0\ar[r] & W_r(S)\ar[r]_-V & W_{r+1}(S)\ar[r]_-{R^r} & S\ar[r]&0
}\]
in which both rows are exact. Since $\ker\theta_r$ is generated by $\xi\phi^{-1}(\xi)\cdots\phi^{-(r-1)}(\xi)$, it follows that $\ker\theta_{r+1}$ is generated by $\xi\phi^{-1}(\xi)\cdots\phi^{-r}(\xi)$, as desired.
\end{proof}

Henceforth we will often identify $\bb A_\inf(S)/\tilde\xi_r$ with $W_r(S)$ via $\tilde\theta_r$. Some Tor-independence assertions related to this identification are summarised in the following lemma:

\begin{lemma}\label{lemma_tor_independence_1}
Let $S\to S^\prime$ be a map between perfectoid rings. Then the canonical maps
\[W_j(S)\dotimes_{\bb A_\inf(S)} \bb A_\inf(S^\prime)\To W_j(S^\prime),\qquad W_j(S)\dotimes_{W_r(S)}W_r(S^\prime)\To W_j(S^\prime)\] are quasi-isomorphisms for all $1\le j\le r$. Here, $W_j(S)$ is considered as a $W_r(S)$-module along either the Frobenius or restriction map.
\end{lemma}

\begin{proof} Let $\xi\in \bb A_\inf(S)$ be a generator of $\ker\theta$, and let $\tilde\xi_j$ be as in the previous lemma, which is a non-zero-divisor of $\bb A_\inf(S)$. The image of $\xi$ in $\bb A_\inf(S^\prime)$ is still a generator of $\ker\theta$, as the condition of being distinguished passes through ring homomorphisms. Thus, we may apply Lemma~\ref{lemma_ker_theta_r} to both $S$ and $S^\prime$ to see that
\[
W_j(S)\dotimes_{\bb A_\inf(S)}\bb A_\inf(S^\prime)=\bb A_\inf(S)/\tilde\xi_j\dotimes_{\bb A_\inf(S)}\bb A_\inf(S^\prime)=\bb A_\inf(S^\prime)/\tilde\xi_j=W_j(S^\prime).
\]
Note that this argument also works with $\tilde\xi_j$ replaced by $\xi_j$.

Using this result also with $r$ in place of $j$, we get
\[
W_j(S)\dotimes_{W_r(S)}W_r(S^\prime)=W_j(S)\dotimes_{W_r(S)}W_r(S)\dotimes_{\bb A_\inf(S)} \bb A_\inf(S^\prime)
=W_j(S)\dotimes_{\bb A_\inf(S)}\bb A_\inf(S^\prime)=W_j(S^\prime),
\]
as required; this works with either the restriction or Frobenius map (using either the $\theta$ or the $\tilde\theta$-maps implicitly).
\end{proof}

An important property of perfectoid rings is the automatic vanishing of the cotangent complex.

\begin{lemma}\label{lem:perfectoidcotangent} Let $S\to S^\prime$ be a map between perfectoid rings. Then $\bb L_{S^\prime/S}\dotimes_{\bb Z} \bb F_p\simeq 0$; in particular, the (derived) $p$-adic completion $\widehat{\bb L}_{S^\prime/S}\simeq 0$.
\end{lemma}

\begin{proof} Note that $S^\prime = S\dotimes_{\bb A_\inf(S)} \bb A_\inf(S^\prime)$; thus, by base change for the cotangent complex, it is enough to show that $\bb L_{\bb A_\inf(S^\prime)/\bb A_\inf(S)}\dotimes_{\bb Z} \bb F_p\simeq 0$. But $\bb L_{\bb A_\inf(S^\prime)/\bb A_\inf(S)}\dotimes_{\bb Z} \bb F_p\simeq \bb L_{S^{\prime\flat}/S^\flat}$. But for any perfect ring $R$ of characteristic $p$, $\bb L_{R/\bb F_p}\simeq 0$ (as Frobenius is both an isomorphism and zero \cite[Lem.~6.5.13(i)]{GabberRamero1}), so that a transitivity triangle shows $\bb L_{S^{\prime\flat}/S^\flat}\simeq 0$.
\end{proof}

\begin{example}[Perfect rings of characteristic $p$]
Suppose that $S$ is a ring of characteristic $p$. Then $S$ is perfectoid if and only if it is perfect. Indeed, if $S$ is perfect, then the kernel of $\theta: W(S)\to S$ is generated by $p$, and the other conditions are clear. For the converse, by assumption $\phi: S\to S$ is surjective. The element $p\in \ker (\theta: \bb A_\inf(S)\to S)$ is distinguished, and thus a generator. Therefore, $S=\bb A_\inf(S)/p=S^\flat$ is perfect.

In particular, in this case $S^\flat=S$, $\theta_\infty: \bb A_\inf(S)\to W(S)$ is an isomorphism, and the maps $\theta_r:\bb A_\inf(S)\to W_r(S)$ identify with the canonical Witt vector restriction maps.
\end{example}

\begin{example}[Roots of unity]\label{examples_roots_of_unity}
Suppose that $S$ is a perfectoid ring which contains a compatible system $\zeta_{p^r}$, $r\ge1$, of $p$-power roots of unity, where $\zeta_p$ is a ``primitive $p$-th root of unity'' in the sense that $1+\zeta_p+\ldots+\zeta_p^{p-1}=0$. Note that this includes the case that $S$ is of characteristic $p$, and all $\zeta_{p^r} = 1$.

Define $\ep:=(1,\zeta_p,\zeta_{p^2},\dots)\in S^\flat=\projlim_{x\mapsto x^p} S$. We claim that
\[
\xi:=1+[\ep^{1/p}]+[\ep^{1/p}]^2+\cdots+[\ep^{1/p}]^{p-1}
\]
is a generator of $\ker\theta$ satisfying $\theta_r(\xi)=V(1)$ for all $r>1$. Note that
\[
\theta(\xi) = 1+\zeta_p+\ldots+\zeta_p^{p-1} = 0
\]
by assumption on $\zeta_p$. It will then follow from Lemma \ref{lemma_ker_theta_r} that $\ker \tilde\theta_r$ is generated by
\[
\tilde\xi_r=\phi(\xi)\phi^2(\xi)\cdots\phi^r(\xi)=\sum_{i=0}^{p^r-1}[\ep]^i \ .
\]

According to Remark \ref{remark_distinguished} it is sufficient to check that $\theta_r(\xi)=V(1)$ for all $r\geq 1$. By functoriality it is sufficient to prove this in the special case that $S := \mathbb{Z}_p^{\mathrm{cycl}}$ as in Example~\ref{ex:PerfectoidEasy}, which has the advantage that $S$ is now $p$-torsion free. So the ghost map $\op{gh}:W_r(S)\to S^r$ is now injective and it is sufficient to prove that $\op{gh}(\theta_r(\xi))=\op{gh}(V(1))$. But it follows easily from Lemma \ref{lemma_theta} that the composition $\op{gh}\circ\theta_r:\bb A_\inf(S)\to S^r$ is given by $(\theta,\theta\phi,\dots,\theta\phi^{r-1})$, and so in particular that \[\op{gh}(\theta_r(\xi))=(\theta(\xi),\theta\phi(\xi),\dots,\theta\phi^{r-1}(\xi)).\] Since $\theta(\xi)=0$ and $\op{gh}(V(1))=(0,p,p,p,\dots)$, it remains only to check that $\theta\phi^i(\xi)=p$ for all $i\ge 1$, which is straightforward: \[\theta\phi^i(\xi)=\theta(1+[\ep^{p^{i-1}}]+[\ep^{p^{i-1}}]^2+\cdots+[\ep^{p^{i-1}}]^{p-1})=1+1+\cdots+1=p.\] This completes the proof of the assertions about $\xi$.
\end{example}

The most important case of perfectoid rings for the paper are those which are flat over $\bb Z_p$ and contain enough $p$-power roots of unity, for which we summarise in the following result some additional properties of $\bb A_\inf(S)$.

\begin{proposition}\label{proposition_roots_of_unity}
Let $S$ be a perfectoid ring which is flat over $\bb Z_p$ and contains a compatible sequence $\zeta_p,\zeta_{p^2},\dots$ of primitive $p$-power roots of unity; let $\ep\in S^\flat$ and $\xi, \tilde\xi_r\in \bb A_\inf(S)$ be as in Example \ref{examples_roots_of_unity}, and set $\mu:=[\ep]-1\in \bb A_\inf(S)$. Then, for any $r\ge 0$:
\begin{enumerate}
\item The element $\tilde\theta_r(\mu)=[\zeta_{p^r}]-1\in W_r(S)$ is a non-zero-divisor;
\item The element $\mu\in \bb A_\inf(S)$ is a non-zero-divisor;
\item The element $\mu$ divides $\phi^r(\mu)=[\ep^{p^r}]-1$, and $\tilde\xi_r=\phi^r(\mu)/\mu$.
\item The element $\mu$ divides $\tilde{\xi_r} - p^r$.
\end{enumerate}
\end{proposition}

\begin{proof} The identity $\tilde\theta_r(\mu)=[\zeta_{p^r}]-1$ follows from Lemma~\ref{lemma_theta}. To check that $[\zeta_{p^r}]-1$ is a non-zero-divisor of $W_r(S)$ for all $r\ge 1$, we note that since $S$ is $p$-torsion-free, the ghost map is injective and so we may check this by proving that \[\op{gh}([\zeta_{p^r}]-1)=(\zeta_{p^r}-1,\zeta_{p^{r-1}}-1,\dots,\zeta_p-1)\] is a non-zero-divisor of $S^r$; i.e., we must show that $\zeta_{p^r}-1$ is a non-zero-divisor in $S$ for all $r\ge 1$. But $\zeta_{p^r} - 1$ divides $p$, and $S$ is flat over $\bb Z_p$.

This proves (i). We get (ii) by noting that $\bb A_\inf(S) = \projlim_F W_r(S)$. Now (iii) is immediate from the definitions. For (iv), observe that $\tilde{\xi_r} = \frac{[\ep]^{p^r} - 1}{[\ep] - 1} = \sum_{i=1}^{p^r } [\ep]^{i-1}$ by (iii). If we set $\mu = 0$, then $[\ep] = 1$, so $\tilde{\xi_r}$ is congruent to $\sum_{i=1}^{p^r} 1 = p^r$ modulo $\mu$, as wanted.
\end{proof}

\begin{corollary}\label{corollary_roots_of_unity}
Let $S$ be a perfectoid ring which is flat over $\bb Z_p$ and contains a compatible sequence $\zeta_p,\zeta_{p^2},\dots$ of primitive $p$-power roots of unity. Then, for any $0\le j\le r$: 
\begin{enumerate}
\item The following ideals of $W_r(S)$ are equal:
\[
\Ann_{W_r(S)}V^j(1),\quad\op{ker}(W_r(S)\xto{F^j} W_{r-j}(S)),\quad\tfrac{[\zeta_{p^j}]-1}{[\zeta_{p^r}]-1}W_r(S)\ .
\]
\item The following ideals of $W_r(S)$ are equal:
\[
\Ann_{W_r(S)}\left(\tfrac{[\zeta_{p^j}]-1}{[\zeta_{p^r}]-1}\right),\quad V^j(1)W_r(S),\quad V^jW_{r-j}(S)\ .
\]
\item The map $F^j$ and multiplication by $V^j(1)$ induce isomorphisms of $W_r(S)$-modules
\[
F_\ast^j W_{r-j}(S)\isofrom W_r(S)/\tfrac{[\zeta_{p^j}]-1}{[\zeta_{p^r}]-1}\isoto \Ann_{W_r(S)}\left(\tfrac{[\zeta_{p^j}]-1}{[\zeta_{p^r}]-1}\right)\ .
\]
\end{enumerate}
\end{corollary}

\begin{remark}\label{rem:idealsgeneralcase} The proof also shows that if $S$ is any perfectoid ring, then
\[
\Ann_{W_r(S)}V^j(1) = \op{ker}(W_r(S)\xto{F^j} W_{r-j}(S))\ ,\ V^j(1)W_r(S) = V^jW_{r-j}(S)\ ,
\]
and via $F^j$ and multiplication by $V^j(1)$,
\[
W_{r-j}(S)\isofrom W_r(S) / \Ann_{W_r(S)}V^j(1)\isoto V^j W_{r-j}(S)\ .
\]
This is a partial analogue of the statement that for perfect rings $S$ of characteristic $p$, $W_r(S)$ admits a filtration (by $p^jW_r(S)$) where all graded pieces are $S$.
\end{remark}

\begin{proof} (i): Injectivity of $V^j:W_{r-j}(S)\to W_r(S)$ and the identity $xV^j(1)=V^j(F^j(x))$, for $x\in W_r(S)$, show that the stated annihilator and kernel are equal. As $W_r(S) = \bb A_\inf/\tilde\xi_r$ and $W_{r-j}(S) = \bb A_\inf/\tilde\xi_{r-j}$ (compatible with the transition map $F^j$), it follows that the kernel is generated by $\tilde\theta_r(\tilde\xi_{r-j}) = \frac{[\zeta_{p^j}]-1}{[\zeta_{p^r}]-1}$.

(ii): Surjectivity of $F^j:W_r(S)\to W_{r-j}(S)$ (Lemma \ref{lemma_frobenius_surjectivity}) implies that $V^j(1)$ generates the ideal $V^jW_{r-j}(S)$, since $V^j(F^j(x))=xV^j(1)$ for $x\in W_r(S)$. Since $[\zeta_{p^r}]-1$ is a non-zero-divisor of $W_r(S)$ by the previous proposition, the elements $[\zeta_{p^j}]-1$ and $\tfrac{[\zeta_{p^j}]-1}{[\zeta_{p^r}]-1}$ have the same annihilator. Clearly $V^j(1)$ annihilates $[\zeta_{p^j}]-1$, since $([\zeta_{p^j}]-1)V^j(1)=V^jF^j[\zeta_{p^j}]-V^j(1)=V^j(1)-V^j(1)=0$. Finally, if $x$ annihilates $[\zeta_{p^j}]-1$ then $R^{r-j}(x)=0$ since $R^{r-j}([\zeta_{p^j}]-1)$ is a non-zero-divisor, and so $x\in V^jW_{r-j}(S)$.

(iii): This follows from (i) and (ii).
\end{proof}

Let us now compare the notion of a perfectoid ring introduced above with another notion, that of a perfectoid Tate ring. Let $R$ be a {\em complete Tate ring}, i.e., a complete topological ring $R$ containing an open subring $R_0\subset R$ on which the topology is $\pi$-adic for some $\pi\in R_0$ such that $R=R_0[\tfrac1\pi]$. Recall that {\em a ring of integral elements} $R^+\subset R$ is an open and integrally closed subring of powerbounded elements. For example, the subring $R^\circ\subset R$ of all powerbounded elements is a ring of integral elements.

In the terminology of Fontaine \cite{Fontaine2013}, extending the original definition \cite{ScholzeThesis}, $R$ is said to be {\em perfectoid} if and only if it is {\em uniform} (i.e., its subring $R^\circ$ of powerbounded elements is bounded) and there is a topologically nilpotent unit $\pi\in R$ such that $\pi^p$ divides $p$ in $R^\circ$, and the Frobenius is surjective on $R^\circ/\pi^p R^\circ$.

\begin{lemma}\label{lemma_Tate_perfectoid}
Let $R$ be a complete Tate ring with a ring of integral elements $R^+\subset R$. If $R$ is perfectoid in Fontaine's sense, then $R^+$ is perfectoid. Conversely, if $R^+$ is perfectoid and bounded in $R$, then $R$ is perfectoid in Fontaine's sense.
\end{lemma}

We remark that perfectoid $K$-algebras in the sense of \cite{ScholzeThesis} (as well as perfectoid $\bb Q_p$-algebras in the sense of \cite{KedlayaLiu}) are complete Tate rings which are perfectoid in Fontaine's sense (and conversely a complete Tate ring which is perfectoid in Fontaine's sense and is a $K$-, resp.~$\bb Q_p$-, algebra is a perfectoid $K$-, resp.~$\bb Q_p$-, algebra in the sense of \cite{ScholzeThesis}, resp.~\cite{KedlayaLiu}).

\begin{proof} Assume that $R$ is perfectoid in Fontaine's sense. First, we check that $R^\circ$ is perfectoid. As $R^\circ$ is bounded, it follows that $R^\circ$ is $\pi$-adically complete. By Lemma~\ref{lemma_inj_of_Frob}, to show that $R^\circ$ is perfectoid, we need to see that the surjective map $\phi: R^\circ/\pi R^\circ\to R^\circ / \pi^p R^\circ$ is an isomorphism. But if $x\in R^\circ$ is such that $x^p = \pi^p y$ for some $y\in R^\circ$, then $z=x/\pi\in R$ has the property that $z^p=y$ is powerbounded, which implies that $z$ itself is powerbounded, i.e. $x\in \pi R^\circ$. Thus, $R^\circ$ is perfectoid.

Now we want to see that then also $R^+$ is perfectoid. Note that $\pi R^\circ$ consists of topologically nilpotent elements, and so $\pi R^\circ\subset R^+$ as the right side is open and integrally closed. By Lemma~\ref{lemma_frobenius_surjectivity} we know that any element of $R^\circ/p\pi R^\circ$ is a $p$-th power. Take any element $x\in R^+$, and write $x=y^p+p\pi z$ for some $y,z\in R^\circ$. Then $z^\prime = \pi z\in R^+$, so that $x=y^p + pz^\prime$. It follows that $y^p = x - pz^\prime\in R^+$, and so $y\in R^+$. Thus, the equation $x=y^p + pz^\prime$ shows that $\phi: R^+/p\to R^+/p$ is surjective, and in particular so is $\phi: R^+/\pi R^+\to R^+/\pi^p R^+$. For injectivity, we argue as for $R^\circ$. Using Lemma~\ref{lemma_inj_of_Frob} again, this implies that $R^+$ is perfectoid.

For the converse, note first that since $R^+\subset R$ is by assumption bounded, so is $R^\circ\subset R$, as $\pi R^\circ\subset R^+$; thus, the first part of Fontaine's definition is verified. It remains to see that there is some topologically nilpotent unit $\pi\in R$ such that $\pi^p$ divides $p$ in $R^\circ$, and the Frobenius is surjective on $R^\circ/\pi^p R^\circ$. Let assume for the moment that there is some topologically nilpotent unit $\pi\in R$ such that $\pi^p$ divides $p$ in $R^\circ$. Given $x\in R^\circ$, $\pi x\in R^+$ can be written as $\pi x = y^p + p\pi z$ with $y,z\in R^+$, by Lemma~\ref{lemma_frobenius_surjectivity}. Note that $\pi\in R^+$ can be assumed to have a $p$-th root $\pi^{1/p}\in R^+$ by changing it by a unit; then $y^\prime = y/\pi^{1/p}\in R$ actually lies in $R^\circ$ as $y^{\prime p} = x - pz\in R^\circ$. But then $x= y^{\prime p} + pz$ with $y',z\in R^\circ$, so Frobenius is surjective on $R^\circ/p R^\circ$, and a fortiori on $R^\circ/\pi^p R^\circ$.

It remains to see that if $R^+$ is perfectoid, then there is some topologically nilpotent unit $\pi\in R$ such that $\pi^p$ divides $p$ in $R^\circ$. The problem here is to ensure the condition that $\pi$ is a unit in $R$.

Pick any topologically nilpotent unit $\pi_0\in R$, so $\pi_0\in R^+$. We have the surjection $\theta: \bb A_\inf(R^+)\to R^+$ whose kernel is generated by a distinguished element $\xi\in \bb A_\inf(R^+)$. From \cite[Lemma 5.5]{kedlayanonarchwitt}, it follows that there is some $\pi^\flat\in (R^+)^\flat$ and a unit $u\in (R^+)^\times$ such that $\theta([\pi^\flat]) = u \pi_0$. Now $\pi=\theta([\pi^{\flat 1/p^n}])$ for $n$ sufficiently large has the desired property.
\end{proof}

A related lemma is the following.

\begin{lemma}\label{lem:genericfiberperfectoid} Let $R_0$ be a perfectoid ring which is $\pi$-adically complete for some non-zero-divisor $\pi$ such that $\pi^p$ divides $p$. Then $R=R_0[\tfrac 1\pi]$, endowed with the $\pi$-adic topology on $R_0$, is a complete Tate ring which is perfectoid in Fontaine's sense. Moreover, $\pi R^\circ\subset R_0$.
\end{lemma}

More precisely, $R_0\subset R^\circ$, and the cokernel is killed by any fractional power of $\pi$.

\begin{proof} Argue as in~\cite[Lemma 5.6]{ScholzeThesis}.
\end{proof}

\subsection{The case of a perfectoid field}
\label{ss:PerfectoidField}

Finally, we add some additional results in the case that $S=\roi=\roi_K$ is the ring of integers in a perfectoid field $K$ of characteristic $0$ containing all $p$-power roots of unity. In this section, we abbreviate $A_\inf = \bb A_\inf(\roi)$.

We let $\epsilon=(1,\zeta_p,\zeta_{p^2},\ldots)\in \roi^\flat$, and consider the elements $\mu=[\epsilon]-1\in A_\inf$ and $\xi = \frac{\mu}{\phi^{-1}(\mu)}$, which generates the kernel of $\theta$. We also have $\xi_r = \frac{\mu}{\phi^{-r}(\mu)}$ which generates the kernel of $\theta_r$, and $\tilde\xi_r = \frac{\phi^r(\mu)}{\mu}$ which generates the kernel of $\tilde\theta_r$, as in Proposition~\ref{proposition_roots_of_unity}.

Before going on, let us recall some more of Fontaine's period rings.

\begin{definition} Consider the following rings associated with $K$.
\begin{enumerate}
\item Let $A_\crys$ be the $p$-adic completion of the $A_\inf$-subalgebra of $A_\inf[\tfrac 1p]$ generated by all $\frac{\xi^m}{m!}$, $m\geq 0$. This is the universal $p$-adically complete PD thickening (compatible with the PD structure on $\bb Z_p$) of $\roi$, or equivalently of $\roi/p$.
\item Let $B_\crys^+ = A_\crys[\tfrac 1p]$, and $B_\crys = A_\crys[\frac 1\mu] = B_\crys^+[\frac 1\mu]$, noting that $\mu^{p-1}\equiv \xi^p\mod p\in A_\inf$, and thus $\mu^{p-1}\in pA_\crys$.
\item Let $B_\dR^+$ be the $\xi$-adic completion of $B_\crys^+$, which is a complete discrete valuation ring with residue field $K$, and $B_\dR = \mathrm{Frac} B_\dR^+ = B_\dR^+[\tfrac 1\xi]$.
\end{enumerate}
\end{definition}

\begin{lemma}\label{lem:mapAinfWitt} The kernel of the natural map
\[
\theta_\infty: A_\inf\to W(\roi) = \varprojlim_R W_r(\roi)\ ,
\]
given as the limit of the maps $\theta_r$, is generated by $\mu$. Equivalently,
\[
\bigcap_r \frac{\mu}{\varphi^{-r}(\mu)} A_\inf = \mu A_\inf\ .
\]
In particular, the ideal $(\mu)\subset A_\inf$ is independent of the choice of roots of unity.

The cokernel of $\theta_\infty$ is killed by $W(\frak m^\flat)$. If $K$ is spherically complete, then $\theta_\infty$ induces an isomorphism
\[
A_\inf/\mu \cong W(\roi)\ .
\]
\end{lemma}

Recall that a nonarchimedean field is spherically complete if any decreasing sequence of discs $a_1+I_1\supset a_2+I_2\supset\cdots$ has nonempty intersection, or equivalently, ${\projlim}^1_r I_r = 0$. This condition is stronger than completeness as one does not ask that the radii of the discs goes to $0$, and for example $\bb C_p = \widehat{\overline{\bb Q}}_p$ is not spherically complete. However, any nonarchimedean field $K$ admits an extension $\tilde{K}/K$ which is spherically complete.

\begin{proof} The kernel of $\theta_\infty$ is the intersection of the kernels of the maps $\theta_r$, which are generated by $\xi_r = \frac{\mu}{\phi^{-r}(\mu)}$. To check that
\[
\bigcap_r \frac{\mu}{\varphi^{-r}(\mu)} A_\inf = \mu A_\inf\ ,
\]
it suffices, since $(p, \xi_r)$ is a regular sequence, to check that
\[
\bigcap_r \frac{\epsilon-1}{\epsilon^{1/p^r}-1} \roi^\flat = (\epsilon-1) \roi^\flat\ ,
\]
which follows from a consideration of valuations.

For each $r\geq 1$, we have a short exact sequence
\[
0\to \xi_r A_\inf\to A_\inf\to W_r(\roi)\to 0\ .
\]
Passing to the limit gives a long exact sequence
\[
0\to \mu A_\inf\to A_\inf\to W(\roi)\to {\projlim_r}^1 \xi_r A_\inf\to 0\ .
\]
Thus, it remains to prove that ${\projlim}^1_r \xi_r A_\inf$ is killed by $W(\frak m^\flat)$, and is $0$ if $K$ is spherically complete. Writing down the similar sequences modulo $p^s$ for any $s\geq 1$ (which are still exact), one sees that
\[
{\projlim_r}^1 \xi_r A_\inf = \projlim_s {\projlim_r}^1 \xi_r A_\inf/p^s\ ,
\]
and one reduces to proving that
\[
{\projlim_r}^1 \xi_r \roi^\flat,
\]
which is $0$ if $K^\flat$ is spherically complete by the observation before the proof, is always killed by $\frak m^\flat$. But for any $m\in \frak m^\flat$, multiplication by $m$ on the system $(\xi_r \roi^\flat)_r$ factors, for sufficiently large $r$, through the constant system $(\mu \roi^\flat)_r$, which has trivial ${\projlim}^1$. It remains only to observe that if $K$ is spherically complete then so is $K^\flat$. Given a decreasing sequence of ideals $I_r$ of $\roi^\flat$ with radii not going to zero, we may rescale to assume that $I_r\supseteq\pi^\flat\roi^\flat$ for all $r$, where $\pi^\flat\in\roi^\flat$ satisfies $\roi^\flat/\pi^\flat=\roi/\pi$ for some $\pi\in\roi$; let $J_r\subset\roi$ be the corresponding ideal such that $J_r/\pi\roi=I_r/\pi^\flat\roi^\flat$. Then ${\projlim}^1_r I_r = {\projlim}^1_r I_r/\pi^\flat\roi^\flat={\projlim}^1_r J_r/\pi\roi={\projlim}^1_r J_r=0$ by spherical completeness of $K$.
\end{proof}

Another result we will need is the following coherence result. For this, let $\roi=\roi_K$ be the ring of integers in any perfectoid field $K$.

\begin{proposition}\label{prop:WrCoherent} For any $r\geq 1$, the ring $W_r(\roi)$ is coherent.
\end{proposition}

Unfortunately, in general $A_\inf$ is not coherent, cf.~\cite{KedlayaAinf}. We start with some reminders on coherent rings \cite[Tag 05CU]{StacksProject}. Recall that a ring $R$ is coherent if every finitely generated ideal is finitely presented. Equivalently, any finitely generated submodule of a finitely presented module is finitely presented. Then the category of finitely presented $R$-modules is stable under extensions, kernels and cokernels.

\begin{lemma}\label{QuotCoherent} Let $R$ be a ring and $I\subset R$ a finitely generated ideal. 
\begin{enumerate}
\item An $R/I$-module $M$ is finitely presented as an $R/I$-module if and only if $M$ is finitely presented as an $R$-module.
\item If $R$ is coherent, then $R/I$ is coherent.
\end{enumerate}
\end{lemma}

\begin{proof} For part (i), if $M$ is finitely presented as an $R$-module, then taking $\otimes_R R/I$ of any finite presentation of $M$ as an $R$-module shows that $M$ is finitely presented as an $R/I$-module. Conversely, take a finite presentation
\[
(R/I)^n\to (R/I)^m\to M\to 0\ .
\]
This gives an exact sequence
\[
R^n\oplus I^m\to R^m\to M\to 0\ ,
\]
giving finite presentation of $M$ as an $R$-module, as $I$ is finitely generated.

For part (ii), let $J\subset R/I$ be any finitely generated ideal, with preimage $\tilde{J}\subset R$. As $I$ and $J$ are finitely generated (as an $R$-modules), $\tilde{J}$ is finitely generated. As $R$ is coherent, $\tilde{J}$ is finitely presented, so we can find an exact sequence
\[
R^n\to R^m\to \tilde{J}\to 0\ .
\]
This gives an exact sequence
\[
(R/I)^n\to (R/I)^m\to \tilde{J}/I\tilde{J}\to 0\ ,
\]
so that $\tilde{J}/I\tilde{J}$ is finitely presented as an $R/I$-module. On the other hand, we have an exact sequence
\[
I/I^2\to \tilde{J}/I\tilde{J}\to J\to 0
\]
of $R/I$-modules, where $I/I^2$ is finitely generated. This makes $J$ a quotient of a finitely presented $R/I$-module by a finitely generated $R/I$-module, thus $J$ is finitely presented as an $R/I$-module.
\end{proof}

\begin{lemma}\label{LiftCoherence} Let $S\to R$ be a surjective map of rings with square-zero kernel $I\subset S$. Assume that $R$ is coherent and $I$ is a finitely presented $R$-module. Then $S$ is coherent.
\end{lemma}

\begin{proof} Let $J\subset S$ be a finitely generated ideal. One has an exact sequence
\[
0\to J\cap I\to J\to J_R\to 0\ ,
\]
where $J_R\subset R$ is the image of $J$. Then $J_R$ is a finitely generated ideal of $R$, and therefore finitely presented as an $R$-module. By Lemma~\ref{QuotCoherent} (i), it is also finitely presented as $S$-module. As $J$ is finitely generated and $J_R$ is finitely presented, it follows that $J\cap I$ is finitely generated (as an $S$-module, and thus as an $R$-module). Now $J\cap I\subset I$ is a finitely generated $R$-submodule of the finitely presented $R$-module $I$, making $J\cap I$ finitely presented (as an $R$-module, and thus as an $S$-module). Therefore, $J$ is an extension of finitely presented $S$-modules, and hence itself finitely presented.
\end{proof}

\begin{lemma}\label{SplitCoherence} Let $R$ be a ring, $f\in R$ a non-zero-divisor. Assume that $(R,f)$ satisfy the Artin-Rees property, i.e. for every inclusion $M\subset N$ of finitely generated $R$-modules, the restriction of the $f$-adic topology on $N$ to $M$ is the $f$-adic topology of $M$. Then $R$ is coherent if $R[f^{-1}]$ and $R/f$ are coherent.
\end{lemma}

\begin{proof} First, observe that by Lemma~\ref{LiftCoherence} (and the assumption that $f$ is a non-zero-divisor) coherence of $R/f$ implies coherence of $R/f^n$ for all $n\geq 1$. Let $I\subset R$ be a finitely generated ideal, and choose a surjection $R^n\to I$ with kernel $K\subset R^n$. We have to prove that $K$ is finitely generated. By assumption $K[f^{-1}]$ is finitely generated, so we may find a map $R^m\to K$ with cokernel $C$ being $f$-torsion. Now $C$ embeds into the cokernel of $R^m\to R^n$; it follows from the Artin-Rees property that the $f$-torsion-part of the cokernel of $R^m\to R^n$ is of bounded exponent. (There is some $N$ such that the preimage of $f^N R^n$ lies in the image of $fR^m$; then, if $x$ is such that $f^Nx$ is in the image of $R^m$, it is in fact in the image of $fR^m$, so that $f^{N-1}x$ is already in the image of $R^m$.) This means that $C$ is of bounded exponent: $f^NC=0$ for some $N$. Thus, it is enough to prove that $K/f^N$ is finitely generated, or even that $K/f$ is finitely generated.

Note that as $I\subset R$ has no $f$-torsion, $K/f$ occurs in a short exact sequence
\[
0\to K/f\to R^n/f\to I/fI\to 0\ .
\]
Therefore, it is enough to prove that $I/fI$ is finitely presented as an $R/f$-module.

Now, by the Artin-Rees property again, there is some $M$ such that $I\cap f^MR\subset fI$. As $R/f^M$ is coherent, $I/(I\cap f^MR)\subset R/f^M$ is finitely presented as an $R/f^M$-module. As $I/fI$ is a quotient of $I/(I\cap f^MR)$ by the finitely generated module $fI$, it follows that $I/fI$ is a finitely presented $R/f^M$-module. By Lemma \ref{QuotCoherent}, it follows that $I/fI$ is also finitely presented as an $R/f$-module.
\end{proof}

\begin{lemma}\label{ARIsogeny} Let $g: R\to S$ be an injective map of rings, $f\in R$ such that both $R$ and $S$ are $f$-torsion free. Assume moreover that the cokernel of $g$ (as a map of $R$-modules) is killed by some power $f^n$ of $f$. Then $(R,f)$ satisfies the Artin-Rees property if and only if $(S,f)$ does. 
\end{lemma}

\begin{proof} The functors $M\mapsto M\otimes_R S$ and $N\mapsto N$ induce inverse equivalences of categories between the category of $R$-modules up to bounded $f$-torsion and the category of $S$-modules up to bounded $f$-torsion. As the Artin-Rees property does not depend on bounded $f$-torsion, one easily checks the lemma.
\end{proof}

After these preparations, we can prove that $W_r(\roi)$ is coherent.

\begin{proof}{(\it of Proposition~\ref{prop:WrCoherent})} Assume first that $K$ is of characteristic $p$. Then $\roi$ is a perfect valuation ring of characteristic $p$, and in particular coherent. Moreover, $W_r(\roi)\to \roi$ is a successive square-zero extension by a copy of $\roi$, which shows that $W_r(\roi)$ is coherent by Lemma~\ref{LiftCoherence}.

Thus, assume now that $K$ is of characteristic $0$. Note that as $\roi$ is $p$-torsion free, the map $W_r(\roi)\to \prod_{i=1}^r \roi$ given by the ghost components is injective, with cokernel bounded $p$-torsion. Note that $\roi$, and thus $\prod_{i=1}^r \roi$, is coherent and satisfies the Artin-Rees property with respect to $f=p$. By Lemma~\ref{SplitCoherence} and Lemma~\ref{ARIsogeny}, it is enough to prove that $W_r(\roi)/p$ is coherent. But $W_r(\roi)/p = W_r(\roi/p^N)/p$ for $N$ big enough, so that it is enough to prove that $W_r(\roi/p^N)$ is coherent.

Now we argue by induction on $r$, so assume $W_{r-1}(\roi/p^N)$ is coherent. For any $i=0,\ldots,N$, consider $R_i = W_r(\roi/p^N) / V^{r-1}(p^i\roi/p^N)$. Then $R_0 = W_{r-1}(\roi/p^N)$ and $R_N = W_r(\roi/p^N)$. We claim by induction on $i$ that $R_i$ is coherent. Note that $R_{i+1}\to R_i$ is a square zero extension by $p^i\roi/p^{i+1}\roi$ regarded as an $R_i$-module via $R_i\to \roi/p^N\to \roi/p\xto{\phi^{r-1}} \roi/p$. This is finitely presented as an $R_i$-module, so the result follows from Lemma~\ref{LiftCoherence}.
\end{proof}

\begin{corollary}\label{cor:WrFinPresNoAlmostZero} Let $M$ be a finitely presented $W_r(\roi)$-module. Then there are no non-zero elements of $M$ which are killed by $W_r(\frak m)$.
\end{corollary}

Note that $W_r(\frak m)\subset W_r(\roi)$ defines an almost setting, of the nicest possible sort: that is, $W_r(\frak m)$ is an increasing union of principal ideals generated by non-zero-divisors, cf.~Corollary~\ref{corollary_almost_Witt}.

\begin{proof} Assume that $x\in M$ is killed by $W_r(\frak m)$. The submodule $M^\prime\subset M$ generated by $x$ is a finitely generated submodule of the finitely presented $W_r(\roi)$-module $M$, thus by coherence of $W_r(\roi)$, $M^\prime$ is finitely presented. Thus, $M^\prime = W_r(\roi) / I$ for some finitely generated ideal $I\subset W_r(\roi)$. On the other hand, as $x$ is killed by $W_r(\frak m)$, we have $W_r(\frak m)\subset I$. Thus, $M^\prime$ is a quotient of $W_r(\roi)/W_r(\frak m) = W_r(k)$, where $k$ is the residue field of $\roi$. As such, $M^\prime = W_s(k)$ for some $0\leq s\leq r$. But the kernel $I$ of $W_r(\roi)\to W_s(k)$ is not finitely generated: if it were, then the kernel $\frak m$ of $\roi\to k$ would also be finitely generated.
\end{proof}

\newpage

\section{Breuil--Kisin--Fargues modules}

The goal of this section is to study the mixed characteristic analogue of Dieudonn\'e modules, i.e., Breuil--Kisin modules \cite{Breuil,Kisin} (for discretely valued fields) and Breuil--Kisin--Fargues modules \cite{FarguesBK} (for perfectoid fields). We begin in \S \ref{subsec:BKmodules} by recalling facts about Breuil--Kisin modules; the most important results here are the structure theorem in Proposition~\ref{prop:BKFree} and Kisin's theorem Theorem~\ref{ThmKisin} about lattices in crystalline Galois representations. The perfectoid analogue of Kisin's theorem is Fargues' classification of finite free Breuil--Kisin--Fargues modules in Theorem~\ref{ThmFargues}, which forms the highlight of \S \ref{subsec:BKFmodules}. In between, in \S \ref{subsec:Ainfmodules}, we study the algebraic properties of the $A_\inf$-modules that arise as Breuil--Kisin--Fargues modules; this discussion includes an analogue of the structure theorem mentioned above in Proposition \ref{prop:splitAinfmodule} (which rests on a classification result of Kedlaya, see Lemma \ref{lem:VectAinf}), and the length estimate in Corollary \ref{cor:ElementaryDivisorsSpecialization}, which is crucial to our eventual applications.

\subsection{Breuil--Kisin modules}
\label{subsec:BKmodules}

Let us start by recalling the ``classical'' theory of Breuil--Kisin modules. Here, we start with a complete discretely valued extension $K$ of $\bb Q_p$ with perfect residue field $k$, and let $\roi = \roi_K$ be its ring of integers. Moreover, we fix a uniformizer $\pi\in K$.

In this situation, we have a natural surjection
\[
\tilde\theta: \gS = W(k)[[T]]\to \roi
\]
sending $T$ to $\pi$. We call this map $\tilde\theta$ as it plays the role of $\tilde\theta$ over $A_\inf$. The kernel of $\tilde\theta$ is generated by an Eisenstein polynomial $E=E(T)\in W(k)[[T]]$ for the element $\pi$. There is a Frobenius $\phi$ on $\gS$ which is the Frobenius on $W(k)$, and sends $T$ to $T^p$.

\begin{definition} A Breuil--Kisin module is a finitely generated $\gS$-module $M$ equipped with an isomorphism
\[
\phi_M: M\otimes_{\gS,\phi} \gS[\tfrac 1E]\cong M[\tfrac 1E]\ .
\]
\end{definition}

The definition may differ slightly from other definitions in the literature. With our definition, the category of Breuil--Kisin modules forms an abelian tensor category.

\begin{example}[Tate twist] There is a ``Tate twist'' in the Breuil--Kisin setup. This is given by $\gS\{1\}$ with underlying $\gS$-module $\gS$ and Frobenius given by $\phi_{\gS\{1\}}(x) = \tfrac uE \phi(x)$, where $x\in \gS\{1\}=\gS$ and $u\in \gS$ is some explicit unit depending on the choice of $E$. This object is $\otimes$-invertible in the category of Breuil--Kisin modules. It can be defined as follows. For each $r$, consider the map
\[
\tilde\theta_r: \gS\to \gS/E_r\ ,
\]
where $E_r := E\phi(E)\cdots \phi^{r-1}(E)$ (so $E_1=E$). Let
\[
(\gS/E_r)\{1\} := \bb L_{(\gS/E_r)/\gS}[-1]\cong E_r \gS / E_r^2 \gS\ ,
\]
which is a free $\gS/E_r$-module of rank $1$. Here, as everywhere else in the paper, we use cohomological indexing. We claim that for $r>s$, there is a natural isomorphism $(\gS/E_r)\{1\}\otimes_{\gS/E_r} \gS/E_s\cong (\gS/E_s)\{1\}$. Indeed, there is an obvious map
\[
E_r \gS / E_r^2 \gS\to E_s \gS / E_s^2 \gS\ ,
\]
and the image is precisely $p^{r-s} E_s \gS / E_s^2 \gS$, as $\frac{E_r}{E_s}$ is congruent to a unit times $p^{r-s}$ modulo $E_s$. Thus, dividing the obvious map by $p^{r-s}$, we get the desired natural isomorphism
\[
(\gS/E_r)\{1\}\otimes_{\gS/E_r} \gS/E_s\cong (\gS/E_s)\{1\}\ .
\]
We may now define $\gS\{1\}=\projlim_r (\gS/E_r)\{1\}$, which becomes a free $\gS = \projlim_r \gS/E_r$-module of rank $1$. Concretely,
\[
\gS\{1\} = \{(a_1E_1,a_2E_2,\ldots)\in \prod_i E_i \gS/E_i^2 \gS\mid a_{i+1} E_{i+1}\equiv p a_i E_i\mod E_i^2\}\ ,
\]
which maps isomorphically to
\[
\{(\overline{a}_1E_1,\overline{a}_2E_2,\ldots)\in \prod E_i \gS/E_i\phi(E_{i-1})\gS\mid \overline{a}_{i+1}E_{i+1}\equiv p \overline{a}_i E_i\mod E_i\phi(E_{i-1})\}\ .
\]

There is a map
\[\begin{aligned}
E\phi_{\gS\{1\}}: \gS\{1\}\to \gS\{1\}: (a_1E_1,a_2E_2,\ldots)\mapsto &(0,E\phi(a_1)\phi(E_1),E\phi(a_2)\phi(E_2),\ldots)\\
 =\, &(0,\phi(a_1)E_2,\phi(a_2)E_3,\ldots)\in \prod E_i \gS/E_i\phi(E_{i-1})\gS\ ,
\end{aligned}\]
where on the target, we use the second description of $\gS\{1\}$. In particular, we get a map
\[
\phi_{\gS\{1\}}: \gS\{1\}\otimes_{\gS,\phi} \gS[\tfrac 1E]\to \gS\{1\}[\tfrac 1E]\ .
\]

For any integer $n$, we define $M\{n\} = M\otimes_\gS \gS\{1\}^{\otimes n}$.
\end{example}

We have the following structural result. One reason that we state this is to motivate our definition of Breuil--Kisin--Fargues modules later, which will have the condition that $M[\tfrac 1p]$ is finite free as an assumption (as it is not automatic in that setup).

\begin{proposition}\label{prop:BKFree} Let $(M,\phi_M)$ be a Breuil--Kisin module. Then there is a canonical exact sequence of Breuil--Kisin modules
\[
0\to (M_\sub{tor},\phi_{M_\sub{tor}})\to (M,\phi_M)\to (M_\sub{free},\phi_{M_\sub{free}})\to (\bar{M},\phi_{\bar{M}})\to 0\ ,
\]
where:
\begin{enumerate}
\item The module $M_\sub{tor}\subset M$ is the torsion submodule, and is killed by a power of $p$.
\item The module $M_\sub{free}$ is a finite free $\gS$-module.
\item The module $\bar{M}$ is a torsion $\gS$-module, killed by a power of $(p,T)$.
\end{enumerate}

In particular, $M[\tfrac 1p]\cong M_\sub{free}[\tfrac 1p]$ is a finite free $\gS[\tfrac 1p]$-module.
\end{proposition}

\begin{proof} Let $M_\sub{tor}\subset M$ be the torsion submodule. Then $M^\prime = M/M_\sub{tor}$ is a torsion-free $\gS$-module. As such, it is projective in codimension $1$, i.e. $M^\prime$ defines a vector bundle $\mathcal{E}$ on $\Spec \gS\setminus \{s\}$, where $s\in \Spec \gS$ is the closed point. As $\gS$ is a $2$-dimensional regular local ring, this implies that $M_\sub{free} = H^0(\Spec \gS\setminus \{s\},\mathcal{E})$ is a vector bundle on $\Spec \gS$, i.e.~a finite free $\gS$-module. The map $M^\prime\to M_\sub{free}$ is injective, and the cokernel is supported set-theoretically at $\{s\}\subset \Spec \gS$, i.e.~killed by a power of $(p,T)$. All constructions are functorial, and thus there are induced Frobenii on all modules considered.

It remains to prove that $M_\sub{tor}$ is killed by a power of $p$. Let $I=\mathrm{Fitt}_i(M)\subset \gS$ be any Fitting ideal of $M$. We have to show that if $I\neq 0$, then a power of $p$ lies in $I$; equivalently, we must check that $\gS/I$ vanishes after inverting $p$. First, we remark that the existence of $\phi_M$ and the base change compatibility of Fitting ideals imply that
\[
I\otimes_{\gS,\phi} \gS[\tfrac 1E] = I[\tfrac 1E],
\]
and therefore
\begin{equation}
\label{eq:FittingIdealFrobeniusStable}
 (\gS/I)[\tfrac{1}{E}] = (\gS/\phi^* I)[\tfrac{1}{E}]
 \end{equation}
as quotients of $\gS[\frac{1}{E}]$. On the other hand, applying the Iwasawa classification of modules over $\gS$, we find 
\[A := (\gS/I)[\tfrac{1}{p}] \cong \prod_{i=1}^n K_0[T]/(f_i(T)^{n_i}),\]
 where $f_i(T) \in W(k)[T]$ is a monic irreducible polynomial congruent to $T^{d_i}$ modulo $p$, $n_i \geq 1$ is an integer, and $f_i \neq f_j$ for $i \neq j$. We will show that $A = 0$. Fix an algebraic closure $C$ of $K$, and consider the finite set $Z := \Spec(A)(C)$ of the $C$-valued points of $A$. By the condition on $f_i$, this set can be identified with a finite subset of the maximal ideal $\mathfrak{m} \subset \roi_C \subset C$, i.e., of the $C$-points of the open unit disc of radius $1$ about $0$. Now equation \eqref{eq:FittingIdealFrobeniusStable} shows that if we set $Z' = \{x \in \mathfrak{m} \mid x^p \in Z\}$, then $Z \cap U = Z' \cap U$ where $U = \mathfrak{m} - \{\pi_1,...,\pi_e\}$ with the $\pi_i$'s being the distinct roots of $E$ in $C$ (with $\pi_1 = \pi$, our chosen uniformizer). We will show that this leads to a contradiction unless $A = 0$ (or, equivalently, $Z = \emptyset$). If $Z \neq \emptyset$, choose $x \in Z$ with $|x|$ maximal. Then there exists some $y \in Z'$ with $y^p = x$. If $|x| \geq |\pi|$, then $|y| > |x| \geq |\pi|$, so $y \in Z' \cap U = Z \cap U$, and thus we obtain $y \in Z$ with $|y| > |x|$, contradicting the maximality in the choice of $x$. Thus $|x| < |\pi|$ for all $x \in Z$. But then $x \in Z \cap U = Z' \cap U$, so $x^p \in Z$ as well. Continuing this way, we obtain that $x^{p^n} \in Z$ for all $n \geq 0$. As $Z$ is finite and $|x| < 1$, this is impossible unless $x = 0$. Thus, $Z = \{0\}$, which translates to $A = K_0[T]/(T^d)$ for some $d \geq 0$. Equation \eqref{eq:FittingIdealFrobeniusStable} then tells us that $K_0[T]/(T^d) \simeq K_0[T]/(T^{dp})$. By considering lengths, we see that $d = 0$, and thus $A = 0$. This shows that $(\gS/I)[\frac{1}{p}] = 0$, so $p^n \in I$ for some $n \gg 0$.
\end{proof}

Let us now recall the relation to crystalline representations of $G_K$. Fix an algebraic closure $\bar{K}$ of $K$ with fixed $p$-power roots $\pi^{1/p^n}\in \bar{K}$ of $\pi$, and let $K_\infty=K(\pi^{1/p^\infty})\subset \bar{K}$. Let $C$ be the completion of $\bar{K}$ with ring of integers $\roi_C\subset C$, and $A_\inf = A_\inf(\roi_C)$, with corresponding $A_\crys$, $B_\crys$. In particular, there is an element $\pi^\flat = (\pi,\pi^{1/p},\ldots)\in \bar{K}^\flat$, with $[\pi^\flat]\in A_\inf$. We have a map $\gS\to A_\inf$ which sends $T$ to $[\pi^\flat]^p$. Thus, the diagram
\[\xymatrix{
\gS\ar[r]^-{\tilde\theta}\ar[d] & \roi\ar[d]\\
A_\inf\ar[r]^-{\tilde\theta} & \roi_C
}\]
commutes. This diagram is equivariant for the action of $G_{K_\infty} = \mathrm{Gal}(\bar{K}/K_\infty)$ (but not for $G_K = \mathrm{Gal}(\bar{K}/K)$).

If $V$ is a crystalline $G_K$-representation on a $\bb Q_p$-vector space, we recall that there is an associated (rational) Dieudonn\'e module
\[
D_\crys(V) = (V\otimes_{\bb Q_p} B_\crys)^{G_K}\ ,
\]
which comes with a $\phi,G_K$-equivariant identification
\[
D_\crys(V)\otimes_{W(k)[\frac 1p]} B_\crys = V\otimes_{\bb Q_p} B_\crys\ .
\]

\begin{theorem}[\cite{Kisin}]\label{ThmKisin} There is a natural fully faithful tensor functor $T\mapsto M(T)$ from $\bb Z_p$-lattices $T$ in crystalline $G_K$-representations $V$ to finite free Breuil--Kisin modules. Moreover, given $T$, $M(T)$ is characterized by the existence of a $\phi,G_{K_\infty}$-equivariant identification
\[
M(T)\otimes_\gS W(C^\flat)\cong T\otimes_{\bb Z_p} W(C^\flat)\ .
\]
\end{theorem}

We warn the reader that the functor is not exact: One critical part of the construction is the extension of a vector bundle on the punctured spectrum $\Spec \gS\setminus \{s\}$, where $s\in \Spec R$ is the closed point, to a vector bundle on $\Spec \gS$, and this functor is not exact.

\begin{remark} We will check below in the discussion around Proposition~\ref{prop:compbkbkf} that $M(T)$ actually satisfies the following statements.
\begin{enumerate}
\item There is an identification
\[
M(T)\otimes_\gS A_\inf[\tfrac 1\mu]\cong T\otimes_{\bb Z_p} A_\inf[\tfrac 1\mu]
\]
which is equivariant for the $\phi$ and $G_{K_\infty}$-actions.
\item There is an equality
\[
M(T)\otimes_\gS B_\crys^+ = D_\crys(V)\otimes_{W(k)[\frac 1p]} B_\crys^+
\]
as submodules of
\[
M(T)\otimes_\gS B_\crys = T\otimes_{\bb Z_p} B_\crys = D_\crys(V)\otimes_{W(k)[\frac 1p]} B_\crys\ .
\]
\end{enumerate}
In particular, there is an identification of rational Dieudonn\'e modules $M(T)\otimes_{\gS} W(k)[\tfrac 1p] = D_\crys(V)$ by tensoring the second identification with $W(\bar{k})[\tfrac 1p]$, and passing to $G_{K_\infty}$-invariants. Thus,
\[
M(T)\otimes_{\gS} W(k)\subset M(T)\otimes_{\gS} W(k)[\tfrac 1p]= D_\crys(V)
\]
defines a natural lattice in crystalline cohomology, functorially associated with the $G_K$-stable lattice $T\subset V$. A main goal of this paper is to show that, at least under suitable torsion-freeness assumptions, this algebraic construction is compatible with the geometry.
\end{remark}

\begin{proof}{\it (of Theorem~\ref{ThmKisin})} The existence of the functor and the identification are stated in~\cite[Theorem 1.2.1]{KisinShimura}. Assume that $M(T)^\prime$ is any other module with the stated property. By~\cite[Proposition 2.1.12]{Kisin}, to check that $M(T)=M(T)^\prime$ equivariantly for $\phi$, it suffices to check this after base extension to the $p$-adic completion $\gS[\tfrac 1T]^\wedge_p$ of $\gS[\tfrac 1T]$. There, it follows from the equivalence between finite free $\phi$-modules over $\gS[\tfrac 1T]^\wedge_p$ and finite free $\bb Z_p$-modules with $G_{K_\infty}$-action, see \cite[Proposition 4.1.1]{Katzpadic} or \cite[Proposition 2.32]{FontaineOuyang}. (Implicit here is that the functor from crystalline $G_K$-representations to $G_{K_\infty}$-representations is fully faithful.) But this $G_{K_\infty}$-representation is in both cases $T$, by the displayed identification for $M(T)$ and $M(T)^\prime$.
\end{proof}

In Corollary~\ref{cor:identtatetwist}, we will check that $\bb Z_p(1)$ is sent to $\gS\{1\}$ under this functor.

\subsection{Some commutative algebra over $A_\inf$}
\label{subsec:Ainfmodules}

In order to prepare for the definition of Fargues' variant over $A_\inf$, we study commutative algebra over the nonnoetherian ring $A_\inf$.

We fix a perfectoid field $K$ with ring of integers $\roi$. Let $\roi^\flat \subset K^\flat$ be the tilt of $\roi \subset K$, and fix an element $x\in A_\inf = W(\roi^\flat)$ which is the Teichm\"uller lift of a nonzero noninvertible element of $\roi^\flat$. We study modules over $A_\inf$, and show that they behave somewhat analogously to modules over a $2$-dimensional regular local ring (such as $\gS$).

We begin by proving an analogue over $A_\inf$ of the well-known fact that all vector bundles on the punctured spectrum of a $2$-dimensional regular local ring are trivial. In fact, the proof below can be easily adapted to show the latter. This result is due to Kedlaya, and the proof below was first explained in a lecture course at UC Berkeley in 2014, \cite{ScholzeLectureNotes}.

\begin{lemma}
\label{lem:VectAinf} Let $s \in \Spec(A_\inf)$ denote the closed point, and let $U := \Spec A_\inf \setminus \{s\}$ be the punctured spectrum. Then restriction induces an equivalence of categories between vector bundles on $\Spec(A_\inf)$ and vector bundles on $U$. In particular, all vector bundles on $U$ are free.
\end{lemma}

\begin{proof} Let $R = A_\inf$, $R_1 = R[\frac{1}{p}]$,  $R_2 = R[\frac{1}{x}]$, and $R_{12} = R[\frac{1}{xp}]$. If we set $U_i = \Spec(R_i)$ for $i \in \{1,2,12\}$, then $U = U_1 \cup U_2$, and $U_1 \cap U_2 = U_{12}$. 

To show the restriction functor is fully faithful, it suffices to show that $A_\inf \to \mathcal{O}(U)$ is an isomorphism, since all vector bundles on $A_\inf$ are free. Using the preceding affine open cover of $U$, and viewing all rings in sight as subrings of $R_{12}$, it suffices to show $R = R_1 \cap R_2 \subset R_{12}$. This follows easily by combining the following observations: The element $x$ is a Teichm\"uller lift, the Teichm\"uller lift is multiplicative, and  each element of $A_\inf$ can be written uniquely as a  power series $\sum_{i \geq 0} a_i \cdot p^i$ with $a_i$ being a Teichm\"uller lift.

For essential surjectivity, we can identify vector bundles $\mathcal{M}$ on $U$ with triples $(M_1,M_2,h)$, where $M_i$ is a finite projective $R_i$-module, and $h:M_1 \otimes_{R_1} R_{12} \simeq M_2 \otimes_{R_2} R_{12}$ is an isomorphism of $R_{12}$-modules; write $M_{12}$ for the latter common value, and let $d$ be the rank of any of these finite projective modules. Let $M := \mathrm{ker} (M_1 \oplus M_2 \to M_{12})$. As a quasi-coherent sheaf on $\Spec(A_\inf)$, this is simply $j_* \mathcal{M}$ where $j:U \to \Spec(A_\inf)$ is the defining quasi-compact open immersion. In particular, we have  $M \otimes_R R_i \simeq M_i$ for $i \in \{1,2,12\}$. We will check that $M$ is a finite projective $A_\inf$-module of rank $d$.

First, we claim that $M$ is contained in a finitely generated $A_\inf$-submodule $M' \subset M_1$ with $M'/M$ killed by a power of $p$. Write $M_1$ as a direct summand of a free module $F_1$ over $R_1$, and let $F_1^\circ \subset F_1$ be the corresponding free module over $R$; let $\psi:F_1 \to M_1$ be the resulting map. As $n \in \mathbb{Z}$ varies, the images $\psi(p^{-n} F_1^\circ) \subset M_1$ give a filtering family of finitely generated $A_\inf$-submodules of $M_1$, and we will show that $M$ lies inside one of these. Let $F_{12} := F_1 \otimes_{R_1} R_{12}$ be the corresponding free $R_{12}$-module, and let $F_{12}^\circ \subset F_{12}$ be the corresponding free $R_2$-module. Then we have an induced projection $\psi:F_{12} \to M_{12}$. Also, we know $p^{-n} F_1^\circ = F_1 \cap p^{-n} F_{12}^\circ \subset F_{12}$ for all $n$, so it is enough to show that $M \subset M_{12}$ is contained in some $\psi(p^{-n} F_{12}^\circ) \subset F_{12}$.  As $M = M_1 \cap M_2$, it suffices to check that $M_2 \subset \psi(p^{-n} F_{12}^\circ)$. But this is immediate as $M_2$ is finitely generated, and $\cup_n \psi(p^{-n} F_{12}^\circ) = M_{12}$. Thus, if we set $M' := \psi(p^{-n} F_1^\circ)$ for $n \gg 0$, then $M'$ is finitely generated and $M \subset M'$. To verify that $M'/M$ is killed by a power of $p$, note that $M[\frac{1}{p}] = M'[\frac{1}{p}] = M_1$. Thus, $M'/M$ is a finitely generated $A_\inf$-module killed by inverting $p$, and so it must be killed by a finite power of $p$.

Next, we show $\dim_k(M \otimes_{A_\inf} k) \geq d$. For this, let $W = W(k)$, and $L = W[\frac{1}{p}]$. The inclusion $M \subset M_1$ then defines a map $M \otimes_{A_\inf} W \to M_1 \otimes_{A_\inf} W \simeq L^{\oplus d}$. The image of this map generates the target as a vector space (since $M[\frac{1}{p}] = M_1$) and is contained in a finitely generated $W$-submodule of $L^{\oplus d}$ by the previous paragraph. As $W$ is noetherian, this image is free of rank $d$, so the claimed inequality follows immediately by further tensoring with $k$. 

Next, we claim that $M$ is $p$-adically complete and separated.  Note that $M_2$ is $p$-adically separated as it is a finite projective module over the $p$-adically separated ring $R_2$. As $M \subset M_2$, it follows that $M$ is $p$-adically separated. For completeness, take any elements $m_i\in M$; we want to form the sum $\sum_{i\geq 0} p^i m_i$. Choose a surjection $A_\inf^r\to M^\prime$, and fix elements $\tilde{m}_i\in A_\inf^r$ lifting the image of $m_i$ in $M^\prime$. Then we can form the sum $\tilde{s}=\sum_{i\geq 0} p^i \tilde{m}_i\in A_\inf^r$, and the image $s\in M^\prime$ of $\tilde{s}$ maps to $0$ in $M^\prime/M$, as $M^\prime/M$ is killed by a power of $p$; thus, $s\in M$, and is the desired limit of the partial sums.

As $M$ is $p$-adically complete and $p$-torsion free, we immediately reduce to checking that $M/p$ is finite free of rank $d$: any map $A_{\inf}^d \to M$ that is an isomorphism after reduction modulo $p$ is an isomorphism (by arguing inductively with the $5$-lemma modulo $p^n$, and then passing to the inverse limit). Now consider the exact sequence 
\[0 \to M \to M_1 \oplus M_2 \to C \to 0,\]
where $C$ is defined as the cokernel. Then $C \subset M_{12}$ is $p$-torsion free, so it follows that $M/p \to M_1/p \oplus M_2/p$ is injective. But $M_1/p = 0$, so $M \to M_2/p$ is injective. Now $M_2/p \simeq M_2 \otimes_{R_2} R_2/p \simeq M_2 \otimes_{R_2} K^\flat$ is a $K^\flat$-vector space of dimension $d$. So we are reduced to checking that $M/p \subset M_2/p \simeq (K^\flat)^d$ is a finite free $\roi^\flat$-module of rank $d$. We already know that $\dim_k(M/p \otimes_{R_2} k) = \dim_k(M \otimes_{A_\inf} k) \geq d$. By Lemma \ref{lem:SubmodulesValuationRing}, we have $\dim_k(M/p \otimes_{R_2} k) = d$. Lemma \ref{lem:CriterionFreenessValuationRings} then gives the claim.
\end{proof}

The following two facts concerning modules over valuation rings were used above: 

\begin{lemma}
\label{lem:SubmodulesValuationRing}
Any $\roi^\flat$-submodule $E$ of $(K^\flat)^d$ satisfies $\dim_k(E \otimes k) \leq d$.
\end{lemma}

\begin{proof}
Assume towards contradiction that there exists a map $f:F \to E$ with $F$ finite free of rank $>d$ such that $f \otimes k$ is injective. Then the image $F'$ of $f$ is a finitely generated torsion free $\roi^\flat$-submodule of $E$. As $\roi^\flat$ is a valuation ring, any finitely generated torsion free module is free, so $F'$ is finite free of rank $\leq d$. But then the composite $f \otimes k:F \otimes k \to F' \otimes k\to E \otimes k$ has image of dimension $\leq d$, which contradicts the assumption.
\end{proof}

\begin{lemma}
\label{lem:CriterionFreenessValuationRings}
If $D \subset (K^\flat)^d$ is an $\roi^\flat$-submodule with $\dim_k(D \otimes k) = d$, then $D$ is finite free of rank $d$. 
\end{lemma}

\begin{proof}
We show this by induction on $d$.  If $d = 1$, then $D$ is one of three possible modules: a principal fractional ideal, a fractional ideal of the form $\mathfrak{m^\flat} \otimes J = \mathfrak{m^\flat} \cdot J$ for a principal fraction field $J$, or $K^\flat$ itself. One easily checks that the second and third possibility cannot occur: one has $D \otimes k = 0$ for both those cases (using $\mathfrak{m}^\flat \otimes \mathfrak{m}^\flat \simeq \mathfrak{m}^\flat$ for the second case), contradicting $\dim_k(D \otimes k) = 1$. Thus, $D$ is a principal fractional ideal, and thus finite free of rank $1$.

For $d>1$, choose any map $\roi^\flat \to D$ that hits a basis element $v$ after applying $- \otimes k$, and is thus injective. Saturating the resulting inclusion $\roi^\flat \subset D$ defines an injective map $g:J \to D$ with torsion free cokernel such that $J$ has generic rank $1$, and the image $g \otimes k$ has dimension $\geq 1$ (as it contains $v$). In fact, since $\dim_k(J \otimes k) \leq 1$ (by the $d=1$ analysis above), it follows that $\dim_k(J \otimes k) = 1$, and that $g \otimes k$ is injective with image of dimension $1$. This gives a short exact sequence
\[ 0 \to J \to D \to D/J \to 0\]
where $J$ and $D/J$ are torsion free of ranks $1$ and $d-1$ respectively. Applying $- \otimes k$ and calculating dimensions gives $\dim_k(D/J \otimes k) = d-1$. By induction, both $J$ and $D/J$ are then free, and thus so is $D$.
\end{proof}

Next, we observe that finitely presented modules over $A_\inf$ are sometimes perfect, i.e.~admit a finite resolution by finite projective modules. Some of the subtleties here arise because (in general) $A_\inf$ is not coherent.

\begin{lemma}
\label{lem:AinfModulePerfect}
Let $M$ be a finitely presented $A_\inf$-module such that $M[\tfrac{1}{p}]$ is finite free over $A_\inf[\tfrac 1p]$. Then:
\begin{enumerate}
\item The $A_\inf$-module $M$ is perfect as an $A_\inf$-complex.
\item The submodule $M_{\sub{tor}} \subset M$ of torsion elements is killed by $p^n$ for $n \gg 0$, and finitely presented and perfect over $A_\inf$.
\item $M$ has $\Tor$-dimension $\leq 2$, and $\Tor^{A_\inf}_2(M, W(k)) = 0$. Moreover, if $M$ has no $x$-torsion, then $\Tor_i^{A_\inf}(M, W(k)) = 0$ for $i > 0$. 
\end{enumerate}
\end{lemma}

We freely use Lemma \ref{QuotCoherent} and Lemma \ref{LiftCoherence} in the proof below.

\begin{proof}
For (i), assume first $M[\frac{1}{p}] = 0$. Then, by finite generation, $M$ is killed by $p^n$ for some $n > 0$, and thus is a finitely presented $W_n(\roi^\flat)$-module. By induction on $n$, we will show that any finitely presented $W_n(\roi^\flat)$-module $M$ is perfect over $A_\inf$. If $n = 1$, then $M$ is a finitely presented $\roi^\flat$-module. But then $M$ is perfect over $\roi^\flat$ (as $\roi^\flat$ is a valuation ring), and thus also over $A_\inf$ (as $\roi^\flat = A_\inf/p$ is perfect as an $A_\inf$-module). In general, for a finitely presented $W_n(\roi^\flat)$-module $M$, we have a short exact sequence
\[ 0 \to pM \to M \to M/pM \to 0.\]
Then $M/pM$ is finitely presented over $\roi^\flat$, and thus perfect over $A_\inf$ by the $n=1$ case. Also, since $W_n(\roi^\flat)$ is coherent, $pM \subset M$ is finitely presented over $W_n(\roi^\flat)$. Moreover, $p^{n-1} \cdot pM = p^n M = 0$, so $pM$ is a finitely presented $W_{n-1}(\roi^\flat)$-module. By induction, $pM$ is also perfect over $A_\inf$. The exact sequence then shows that $M$ is perfect over $A_\inf$. 

For general $M$, by clearing denominators on generators of $M[\frac{1}{p}]$, we can find a free $A_\inf$-module $N$ and an inclusion $N \subset M$ that is an isomorphism after inverting $p$. The quotient $Q$ is then a finitely presented $A_\inf$-module killed by inverting $p$, so $Q$ is perfect by the preceding argument. Also, $N$ is perfect as it is finite free; it formally follows that $M$ is perfect as well.

For (ii), choose $N$ and $Q$ as in the previous paragraph. Then $M_{\sub{tor}} \cap N = 0$ as $N$ has no torsion. Thus, $M_{\sub{tor}} \to Q$ is injective, so $M_{\sub{tor}}$ is killed by $p^n$ for some $n > 0$, and so $M_{\sub{tor}} = M[p^n]$. Now consider $K := M \otimes_{A_\inf}^L A_{\inf}/p^n$. This is perfect over $A_\inf/p^n = W_n(\roi^\flat)$ by (i) and base change. As $W_n(\roi^\flat)$ is coherent, each $H^i(K)$ is finitely presented. But $H^{-1}(K) = M[p^n]$, so $M[p^n]$ is finitely presented over $W_n(\roi^\flat)$, and thus also over $A_\inf$. The perfectness now follows from (i) applied to $M[p^n]$.

For (iii), the fact that $M$ has $\Tor$-dimension $\leq 2$ follows easily from the previous arguments using the fact that any finitely presented $\roi^\flat$-module has $\Tor$-dimension $\leq 1$ over $\roi^\flat$, and thus $\Tor$-dimension $\leq 2$ over $A_\inf$. For the rest, let $\tilde{W} = \varinjlim_n A_\inf/(x^{1/p^n})$, so we have a short exact sequence
\[ 0 \to Q \to \tilde{W} \to W(k) \to 0. \]
The last map in this sequence is the $p$-adic completion map and $\tilde{W}$ is $p$-torsion-free. Thus, $Q$ is an $A_\inf[\frac{1}{p}]$-module, and thus $\Tor_i^{A_\inf}(M,Q) = 0$ for $i > 0$ as $M[\frac{1}{p}]$ is finite free. Also, since $x \in A_\inf$ is a non-zero-divisor, the $A_\inf$-module $\tilde{W}$ has $\Tor$-dimension $1$; it follows from the long exact sequence on $\Tor$ that $\Tor^{A_\inf}_2(M,W(k)) = 0$. Now if $M$ is further assumed to have no $x$-torsion, then $\Tor^{A_\inf}_i(M,\tilde{W}) = 0$ for $i > 0$. Thus, we have a short exact sequence
\[ 0 \to \Tor_1^{A_\inf}(M,W(k)) \to M \otimes_{A_\inf} Q \to M \otimes_{A_\inf} \tilde{W} \to M \otimes_{A_\inf} W(k) \to 0.\]
As $M[\frac{1}{p}]$ is finite free, the first term above is killed after inverting $p$. On the other hand, $p$ acts invertibly on $Q$ and thus on the second term above; thus $\Tor_1^{A_\inf}(M,W(k)) = 0$, as wanted.
\end{proof}

Next, we give a criterion for an $A_\inf$-module to define a vector bundle on $U=\Spec A_\inf \setminus \{s\}$. This is a weak analogue over $A_\inf$ of the fact that a finitely generated torsion free module over a $2$-dimensional regular local ring gives a vector bundle on the punctured spectrum. 

\begin{lemma}\label{lem:AinfModuleVect}
Let $M$ be a finitely generated $p$-torsion-free $A_\inf$-module such that $M[\tfrac{1}{p}]$ is finite projective over $A_\inf[\tfrac 1p]$. Then the quasi-coherent sheaf associated to $M$ restricts to a vector bundle on $U$.
\end{lemma}

\begin{proof} It is enough to check that $M\otimes_{A_\inf} A_{\inf,(p)}$ is finite free, where $A_{\inf,(p)}$ is the localization at the prime ideal $(p)\subset A_\inf$. But $A_{\inf,(p)}$ is a discrete valuation ring: The function sending $\sum_{i\geq 0} [a_i] p^i\in A_\inf$ with $a_i\in \roi^\flat$ to the minimal integer $i$ for which $a_i\neq 0$ defines a discrete valuation on $A_\inf$, with corresponding prime ideal $(p)$, and corresponding discrete valuation ring $A_{\inf,(p)}$. As $M\otimes_{A_\inf} A_{\inf,(p)}$ is a finitely generated $p$-torsion-free module, it is thus finite free, as desired.
\end{proof}

\begin{remark} In Lemma~\ref{lem:AinfModuleVect}, it is unreasonable to hope that $M$ itself is finite projective. For example, if $M$ is the ideal $(x,p) \subset A_\inf$, then $M$ is not finite projective over $A_\inf$, and yet restricts to the trivial line bundle over $U$.
\end{remark}

\begin{corollary}\label{cor:VectAinf1p} Let $N$ be a finite projective $A_\inf[\tfrac 1p]$-module. Then $N$ is free.
\end{corollary}

\begin{proof} Let $M\subset N$ be a finitely generated $A_\inf$-submodule such that $M[\tfrac 1p]=N[\tfrac 1p]$. Then $M$ satisfies the hypothesis of Lemma~\ref{lem:AinfModuleVect}, and thus by Lemma~\ref{lem:VectAinf} there is some finite free $A_\inf$-module $M^\prime$ such that the vector bundles corresponding to $M$ and $M^\prime$ agree on $U=\Spec A_\inf\setminus \{s\}$. In particular, $M^\prime[\tfrac 1p]=N[\tfrac 1p]$, which is therefore finite free.
\end{proof}

Putting the above results together, we obtain the following structural result:

\begin{proposition}\label{prop:splitAinfmodule}
Let $M$ be a finitely presented $A_\inf$-module such that $M[\tfrac{1}{p}]$ is finite projective (equivalently, free) over $A_\inf[\tfrac 1p]$. Then there is a functorial exact sequence
\[ 0 \to M_{\sub{tor}} \to M \to M_{\sub{free}} \to \overline{M} \to 0\]
satisfying:
\begin{enumerate}
\item $M_{\sub{tor}}$ is finitely presented and perfect as an $A_\inf$-module, and is killed by $p^n$ for $n \gg 0$.
\item $M_{\sub{free}}$ is a finite free $A_\inf$-module.
\item $\overline{M}$ is finitely presented and perfect as an $A_\inf$-module, and is supported at the closed point $s \in \Spec(A_\inf)$, i.e., it is killed by some power of $(x,p)$.
\end{enumerate}
Moreover, $M$ is a finite free $A_\inf$-module if either $M\otimes_{A_\inf} W(k)$ is $p$-torsion-free, or if $K$ has characteristic zero and $M \otimes_{A_\inf} \roi$ is $p$-torsion-free.
\end{proposition}

\begin{proof}
Let $M_{\sub{tor}} \subset M$ be the torsion submodule of $M$. Then (i) is immediate from Lemma~\ref{lem:AinfModulePerfect}. Let $N = M/M_{\sub{tor}}$, so $N$ is a finitely presented $A_\inf$-module (by (i)) that is free after inverting $p$ (as $M$ is so) and has no $p$-torsion. Lemma~\ref{lem:AinfModuleVect} then implies that $N$ defines a vector bundle on $U$. Lemma~\ref{lem:VectAinf} implies that $M_{\sub{free}} := H^0(U, N)$ is a finite free $A_\inf$-module, giving (ii). Also, since $N$ had no $p$-torsion, the induced map $N \to M_{\sub{free}}$ is injective and an isomorphism over $U$. Thus, the cokernel $\overline{M}$ is a finitely presented $A_\inf$-module supported at the closed point $s \in \Spec(A_\inf)$, proving most of (iii); the perfectness of $\overline{M}$ follows from the perfectness of the other $3$ terms.

For the final statement, we first note that in general, if $R$ is a local integral domain with residue field $k_s$ and quotient field $k_\eta$, and $M$ is a finitely generated $R$-module such that
\[
\dim_{k_s}(M\otimes_R k_s) = \dim_{k_\eta}(M\otimes_R k_\eta)\ ,
\]
then $M$ is finite free. Indeed, any nonzero Fitting ideal $I\subset R$ of $M$ has to be all of $R$, as otherwise the rank of $M\otimes_R k_\eta$ would differ from the rank of $M\otimes_R k_s$, since $k_\eta\not\in \Spec(R/I)$ while $k_s\in \Spec(R/I)$. Applying this to $R=A_\inf$ and the given module $M$ gives the conclusion, as the dimension at the generic point agrees with the dimension at $W(k)[\tfrac 1p]$ and $\roi[\tfrac 1p]$ because $M[\tfrac 1p]$ is finite free over $A_\inf[\tfrac 1p]$, and this dimension agrees with the dimension of $M\otimes_{A_\inf} k$ by assumption.
\end{proof}

We record an inequality stating roughly that rank goes up under specialization for finitely presented modules. 

\begin{lemma}
\label{lem:InequalitySpecialization}
Let $M$ be a finitely presented $W_n(\roi^\flat)$-module. Let $M_\eta$ and $M_s$ be the base change of $M$ along $W_n(\roi^\flat) \to W_n(K^\flat)$ and $W_n(\roi^\flat) \to W_n(k)$ respectively. Then $M_\eta$ and $M_s$ have finite length over the corresponding local rings, and we have: 
\[\ell(M_\eta) \leq \ell(M_s).\]
\end{lemma}

In the proof below, the length function $\ell(-)$ applied to certain perfect complexes $K$ over $W_n(k)$ simply means the usual alternating sum $\sum_i (-1)^i \ell(H^i(K))$.

\begin{proof}
Note that $M \dotimes_{W_n(\roi^\flat)} W_n(k) \simeq M \dotimes_{A_{\inf}} W(k)$. By Lemma~\ref{lem:AinfModulePerfect}, it follows that each $\Tor_i^{W_n(\roi^\flat)}(M, W_n(k))$ has finite length, and vanishes for $i>1$.

We now show the more precise statement
\[\ell(M_\eta) = \ell(M_s) - \ell(\Tor_1^{W_n(\roi^\flat)}(M, W_n(k))).\]
The left hand side is $\ell(M \dotimes_{W_n(\roi^\flat)} W_n(K^\flat))$ as $W_n(\roi^\flat) \to W_n(K^\flat)$ is flat, while the right hand side is $\ell(M \dotimes_{W_n(\roi^\flat)} W_n(k))$ by the vanishing shown above. With this reformulation, both sides above are additive in short exact sequences in $M$. Writing $M$ as an extension of $M/p^{n-1}M$ by $p^{n-1}M/p^nM$, we inductively reduce down to the case $n=1$; here we use the identification $M \dotimes_{W_n(\roi^\flat)} W_n(k) \simeq M \dotimes_{\roi^\flat} k$ when $M$ is killed by $p$. By the classification of finitely presented modules over valuation rings, we may assume $M = \roi^\flat$ or $M = \roi^\flat / (x^r)$ for suitable non-zero $r$ in the value group of $K^\flat$. Both these cases can be checked directly: the relevant lengths are both $1$ in the first case, and $0$ in the second case. Thus, we are done.
\end{proof}

Using this, we arrive at an inequality relating the specializations of certain $A_\inf$-modules over $W(k)$ and $W(K^\flat)$:

\begin{corollary}
\label{cor:ElementaryDivisorsSpecialization}
Let $M$ be a finitely presented $A_\inf$-module such that $M[\tfrac{1}{p}]$ is free over $A_\inf[\tfrac 1p]$. Let $M_1 := M \otimes_{A_\inf} W(K^\flat)$ and $M_2 := M \otimes_{A_\inf} W(k)$ be the displayed scalar extensions. Then:
\begin{enumerate}
\item The modules $M_1$ and $M_2$ have the same rank.
\item For all $n\geq 1$, $\ell(M_2/p^n)\geq \ell(M_1/p^n)$.
\end{enumerate}
\end{corollary}

\begin{proof} The first assertion is immediate as both $M_1[\frac{1}{p}]$ and $M_2[\frac{1}{p}]$ are base changes of the finite free module $M[\frac{1}{p}]$. Part (ii) follows by applying Lemma~\ref{lem:InequalitySpecialization} to $M/p^n$.
\end{proof}

The next lemma will help in understanding the crystalline specialization.

\begin{lemma}
\label{lem:CrysSpecializationInjects}
Let $C \in D(A_\inf)$ such that $H^j(C)[\tfrac{1}{p}]$ is free for each $j$. Fix some index $i$. Then the natural map $H^i(C) \otimes_{A_\inf} W(k) \to H^i(C \dotimes_{A_\inf} W(k))$ is injective, and bijective after inverting $p$. Furthermore, if $H^{i+1}(C)$ has no $x$-torsion, then this map is bijective.
\end{lemma}

\begin{proof}
The bijectivity after inverting $p$ is formal from the assumption on the $H^j(C)[\frac{1}{p}]$. For the rest, let $\tilde{W} = \varinjlim A_\inf/(x^{1/p^n})$, so $W(k)$ is the $p$-adic completion of $\tilde{W}$. We first observe that
\[ H^i(C) \otimes_{A_\inf} \tilde{W} \to H^i(C \dotimes_{A_\inf} \tilde{W})\]
is injective: by compatibility of both sides with filtered colimits, this reduces to the corresponding statement for $A_{\inf}/(x^{1/p^n})$, which can be checked easily using the Koszul presentation for the latter ring over $A_\inf$. This analysis also shows that if $H^{i+1}(C)$ has no $x$-torsion, then the above map is bijective.
 
 Now let $Q = \ker(\tilde{W} \to W(k))$, so there is a short exact sequence
\[ 0 \to Q \to \tilde{W} \to W(k) \to 0.\]
Since $W(k)$ is the $p$-adic completion of the $p$-torsion-free module $\tilde{W}$, it follows that $Q$ is an $A_\inf[\frac{1}{p}]$-module. In particular, by the hypothesis that all the $H^j(C)[\frac{1}{p}]$ are free, we have 
\[ H^i(C) \otimes_{A_\inf} Q \simeq H^i(C \dotimes_{A_\inf} Q).\]
Now consider the following  diagram of canonical maps:
\[ \xymatrix{ H^i(C) \otimes_{A_\inf} Q \ar[r] \ar[d]^-{a} & H^i(C) \otimes_{A_\inf} \tilde{W} \ar[r] \ar[d]^-{b} & H^i(C) \otimes_{A_\inf} W(k) \ar[d]^-{c} \ar[r] & 0 \\
		   H^i(C \dotimes_{A_\inf} Q) \ar[r] & H^i(C \dotimes_{A_\inf} \tilde{W}) \ar[r]^-{d} & H^i(C \dotimes_{A_\inf} W(k)). & }\]
Here both rows are exact, the map $a$ is bijective, and the map $b$ is injective (as explained above for both). A diagram chase then shows that the map $c$ is injective, as wanted. 

Furthermore, we claim that the map labelled $d$ then must be surjective. Indeed, the obstruction to surjectivity is the boundary map $H^i(C \dotimes_{A_\inf} W(k)) \to H^{i+1}(C \dotimes_{A_\inf} Q)$ extending the bottom row to a long exact sequence; but this map must be $0$ since the target is an $A_{\inf}[\frac{1}{p}]$-module, and we know that $d[\frac{1}{p}]$ is surjective, as $c[\tfrac 1p]$ is. The diagram now shows that the surjectivity of $c$ follows from the surjectivity of $b$. But the latter was shown above under the hypothesis that $H^{i+1}(C)$ has no $x$-torsion, so we are done.
\end{proof}

Combining Proposition~\ref{prop:splitAinfmodule} with Lemma~\ref{lem:CrysSpecializationInjects}, we essentially obtain:

\begin{corollary}\label{cor:crysspec}
Let $C \in D(A_\inf)$ be a perfect complex such that $H^j(C)[\tfrac{1}{p}]$ is free over $A_\inf[\tfrac 1p]$ for all $j\in \bb Z$. Then, for every $j$, $H^j(C)$ is a finitely presented $A_\inf$-module. Moreover, for fixed $i$, if $H^i(C\dotimes_{A_\inf} W(k))$ is $p$-torsion-free, then $H^i(C)$ is a finite free $A_\inf$-module, and in particular $H^i(C\dotimes_{A_\inf} W(K^\flat)) = H^i(C)\otimes_{A_\inf} W(K^\flat)$ is $p$-torsion-free. If moreover $H^{i+1}(C)\otimes_{A_\inf} W(k)$ is $p$-torsion-free (e.g., by Lemma~\ref{lem:CrysSpecializationInjects}, this happens if $H^{i+1}(C\dotimes_{A_\inf} W(k))$ is $p$-torsion-free), then
\[
H^i(C)\otimes_{A_\inf} W(k) = H^i(C\dotimes_{A_\inf} W(k))\ .
\]
\end{corollary}

We remark here that the equality  $H^i(C\dotimes_{A_\inf} W(K^\flat)) = H^i(C)\otimes_{A_\inf} W(K^\flat)$ invoked above follows from the flatness of $A_{\inf} \to W(K^\flat)$, see proof of Lemma~\ref{lem:AinfModuleVect}.

\begin{proof} First, we check that $H^j(C)$ is a finitely presented $A_\inf$-module for all $j$. We prove this by descending induction on $j$, noting that it is trivially true for all $j\gg 0$ as then $H^j(C)=0$. If it is true for all $j^\prime>j$, then $H^{j^\prime}(C)$ is perfect for all $j^\prime>j$ by Lemma~\ref{lem:AinfModulePerfect}. This implies that $\tau^{\leq j} C$ is still perfect, so that $H^j(C)$ is the top cohomology group of a perfect complex, which is always finitely presented.

Now if $H^{i+1}(C) \otimes_{A_{\inf}} W(k)$ is $p$-torsionfree, then $H^{i+1}(C)$ is finite free by the last statement in Proposition~\ref{prop:splitAinfmodule}, and thus has no $x$-torsion. The last statement in Lemma~\ref{lem:CrysSpecializationInjects} now yields the desired equality.
\end{proof}

The next lemma implies that torsion-freeness conditions on the de~Rham or crystalline specializations are equivalent. Here, as well as in Lemma \ref{lem:FreeoverAinf1p} and Corollary \ref{cor:crysspecbetter}, we assume that $K$ is of characteristic $0$ and contains all $p$-power roots of unity.

\begin{lemma}
\label{lem:TorsionFreeWO}
Let $C \in D(A_\inf)$ be a perfect complex such that $H^j(C)[\tfrac{1}{p}]$ is free over $A_\inf[\tfrac 1p]$ for all $j\in \bb Z$. Fix some index $i$. Then $H^i(C \dotimes_{A_\inf} W(k))$ is $p$-torsion-free if and only if $H^i(C \dotimes_{A_\inf} \roi)$ is $p$-torsion-free. 
\end{lemma}
\begin{proof}
Note that the stated hypothesis imply that each $H^j(C)$, and hence each truncation of $C$, is perfect over $A_\inf$ by the previous corollary and Lemma~\ref{lem:AinfModulePerfect}. Assume first that $H^i(C \dotimes_{A_\inf} W(k))$ is $p$-torsion-free. Then $H^i(C) \otimes_{A_\inf} W(k)$ is $p$-torsion-free by Lemma \ref{lem:CrysSpecializationInjects}, and then $H^i(C)$ is finite free by Proposition \ref{prop:splitAinfmodule}. As $\Tor^{A_\inf}_i(H^j(C), W(k)) = 0$ for all $j$ and $i > 1$ by Lemma \ref{lem:AinfModulePerfect} (iii), this implies $(\tau^{\geq i} C) \dotimes_{A_\inf} W(k) \simeq \tau^{\geq i} (C \dotimes_{A_\inf} W(k))$. Now $\tau^{\geq i}(C \dotimes_{A_\inf} W(k)) \dotimes_{W(k)} k \in D^{\geq i}(k)$ by the assumption that $H^i(C \dotimes_{A_\inf} W(k))$ has no torsion, so $\tau^{\geq i} C \dotimes_{A_\inf} k \in D^{\geq i}(k)$ as well. Rewriting, we see $(\tau^{\geq i} C \dotimes_{A_\inf} \roi) \dotimes_{\roi} k \in D^{\geq i}(k)$. This implies the following: (a) the perfect complex $\tau^{\geq i} C \dotimes_{A_\inf} \roi \in D^{b}(\roi)$ must lie in $D^{\geq i}(\roi)$, and (b) $H^i$ of this last complex is free; here we use the following fact: a finitely presented $\roi$-module is free if and only if $\Tor^{\roi}_1(M,k) = 0$ (see the end of the proof of Proposition \ref{prop:splitAinfmodule}). The first of these properties implies that $\tau^{\geq i} C \dotimes_{A_\inf} \roi \simeq \tau^{\geq i}(C \dotimes_{A_\inf} \roi)$, and the second then implies that $H^i(C \dotimes_{A_\inf} \roi)$ is $p$-torsion-free, as wanted. The converse is established in exactly the same way. 
\end{proof}

We record a criterion for $M[\frac{1}{p}]$ being finite projective. 

\begin{lemma}\label{lem:FreeoverAinf1p}
Let $M$ be a finitely presented $A_\inf$-module. Let $\mu = [\epsilon] - 1$, with $\epsilon$ as in Example~\ref{examples_roots_of_unity}. Assume the following:
\begin{enumerate}
\item $M[\frac{1}{p\mu}]$ is finite projective over $A_\inf[\frac{1}{p \mu}]$.
\item $M\otimes_{A_\inf} B_{\crys}^+$ is finite projective over $B_{\crys}^+$. 
\end{enumerate}
Then $M[\frac{1}{p}]$ is finite free over $A_\inf[\frac{1}{p}]$.
\end{lemma}

\begin{proof}
Let $R = A_\inf[\frac{1}{p}]$, and let $N = M[\frac{1}{p}]$. Then $\mu \in R$ is a non-zero-divisor; let $\widehat{R}$ be the $\mu$-adic completion of $R$. We first show that the canonical map $R\to \widehat{R}$ factors through $B_\crys^+$. To check this, we need to produce a canonical map $A_\crys\to A_\inf[\tfrac 1p]/\mu^n$ for all $n$ (which will then factor through $B_\crys^+ = A_\crys[\tfrac 1p]$). Fix some such $n$. It suffices to show that the images of $\frac{\xi^m}{m!}\in A_\inf[\tfrac 1p]$ for varying $m$ belong to a bounded subalgebra of $A_\inf[\tfrac 1p]/\mu^n$. Note that the cokernel of the map $A_\inf/\mu^n\to A_\inf/\xi^n\oplus A_\inf/\phi^{-1}(\mu)^n$ is bounded $p$-torsion: this cokernel is finitely presented over $A_\inf$, and acyclic after inverting $p$ (since $p \equiv \xi \mod (\phi^{-1} \mu)$). Certainly, $\frac{\xi^m}{m!}$ maps to $0$ in $A_\inf[\tfrac 1p]/\xi^n$ for $m\geq n$, so it remains to handle the second factor. For this, note that $\xi\equiv \phi^{-1}(\mu)^{p-1}\mod p$ in $A_\inf$, so adjoining all $\frac{\xi^m}{m!}$ is equivalent to adjoining all $\frac{\phi^{-1}(\mu)^{(p-1)m}}{m!}$. Now these elements have trivial image in $A_\inf/\phi^{-1}(\mu)^n$ for $m\geq n$, finishing the proof that $R\to \widehat{R}$ factors through $B_\crys^+$.

By the Beauville--Laszlo lemma, \cite{BeauvilleLaszlo}, and Corollary \ref{cor:VectAinf1p}, it is enough to check that $N[\frac{1}{\mu}]$ is finite projective over $R[\frac{1}{\mu}]$, that $N \otimes_R \widehat{R}$ is finite projective over $\widehat{R}$, and that $N$ has no $\mu$-torsion. The first part is true by assumption (i). The second part follows from assumption (ii) as the map $R \to \widehat{R}$ factors through the canonical map $R \to B_\crys^+$, as shown in the previous paragraph. It remains to show that $N$ has no $\mu$-torsion. For this, observe that we have a short exact sequence
\[
0 \to R \to \widehat{R} \to Q \to 0
\]
with $Q$ being an $R[\frac{1}{\mu}]$-module. Tensoring this with $N$, and using that
\[
\Tor_1^R(N,Q) = \Tor_1^{R[\frac 1\mu]}(N[\tfrac 1\mu],Q) = 0
\]
by projectivity of $N[\tfrac 1\mu]$, we get an injection $N\hookrightarrow N\otimes_R \widehat{R}$, which is $\mu$-torsion-free.
\end{proof}

Let us state a corresponding version of Corollary~\ref{cor:crysspec}.

\begin{corollary}\label{cor:crysspecbetter}
Let $C \in D(A_\inf)$ be a perfect complex such that for all $j\in \bb Z$, $H^j(C)[\tfrac{1}{p\mu}]$ is free over $A_\inf[\tfrac 1{p\mu}]$, and $H^j(C\dotimes_{A_\inf} B_\crys^+)$ is free over $B_\crys^+$. Then, for every $j$, $H^j(C)$ is a finitely presented $A_\inf$-module with $H^j(C)[\tfrac 1p]$ free over $A_\inf[\tfrac 1p]$.

Moreover, for fixed $i$, if $H^i(C\dotimes_{A_\inf} W(k))$ is $p$-torsion-free, then $H^i(C)$ is a finite free $A_\inf$-module, and in particular $H^i(C\dotimes_{A_\inf} W(K^\flat)) = H^i(C)\otimes_{A_\inf} W(K^\flat)$ is $p$-torsion-free. If moreover $H^{i+1}(C)\otimes_{A_\inf} W(k)$ is $p$-torsion-free (e.g., if $H^{i+1}(C\dotimes_{A_\inf} W(k))$ is $p$-torsion-free), then
\[
H^i(C)\otimes_{A_\inf} W(k) = H^i(C\dotimes_{A_\inf} W(k))\ .
\]
\end{corollary}

\begin{proof} We only need to prove that $H^j(C)[\tfrac 1p]$ is finite free over $A_\inf[\tfrac 1p]$; the rest is Corollary~\ref{cor:crysspec}. For this, one argues again by decreasing induction on $j$, so one can assume that $j$ is maximal with $H^j(C)\neq 0$. Then $H^j(C)$ satisfies the hypothesis of Lemma~\ref{lem:FreeoverAinf1p}, which gives the conclusion.
\end{proof}

\begin{remark}
\label{rmk:TorsionFreedRCrysEquivAbstract}
Using Lemma \ref{lem:TorsionFreeWO}, the hypothesis on $H^i(C \dotimes_{A_\inf} W(k))$ in Corollary \ref{cor:crysspec} and Corollary \ref{cor:crysspecbetter} can be replaced by the same hypothesis on $H^i(C \dotimes_{A_\inf} \roi)$.
\end{remark}

\subsection{Breuil--Kisin--Fargues modules}
\label{subsec:BKFmodules}

Let $K$ be a perfectoid field with ring of integers $\roi=\roi_K=K^\circ\subset K$. We get the ring $A_\inf=A_\inf(\roi) = W(\roi^\flat)$ equipped with a Frobenius automorphism $\phi$, where $\roi^\flat = \varprojlim_\phi \roi/p$ as usual. Fix a generator $\xi$ of $\ker(\theta: A_\inf\to\roi)$, and let $\tilde\xi = \phi(\xi)$.

\begin{definition} A \emph{Breuil--Kisin--Fargues module} is a finitely presented $A_\inf$-module $M$ with an isomorphism
\[
\varphi_M: M \otimes_{A_\inf,\phi} A_\inf [\tfrac{1}{\tilde\xi}]\isoto M[\tfrac{1}{\tilde\xi}]\ ,
\]
such that $M[\tfrac 1p]$ is a finite projective (equivalently, free) $A_\inf[\tfrac 1p]$-module.
\end{definition}

This should be regarded as a mixed-characteristic version of a Dieudonn\'e module. The next example illustrates why we impose the condition that $M[\frac{1}{p}]$ is finite free.

\begin{example}\label{ex:nonbkfmodule}
Let $K = C$ where $C$ is a completed algebraic closure of $\mathbb{Q}_p$. Let $\mu = [\epsilon] - 1$, with notation as in Example~\ref{examples_roots_of_unity}. Set $I = (\mu)$, and $M := A_\inf/I$; this is a finitely presented $A_\inf$-module. As $\mu \mid \phi(\mu)$, we have $\phi^*(I) \subset I$, which induces a map $\phi_M:\phi^* M \to M$. Moreover, as $\phi^*(I) \subset I$ becomes an equality after inverting $\tilde\xi$, so does $\phi_M$.
\end{example}

Again, there is a version of the Tate twist.

\begin{example}[Tate twist]\label{ex:bkftwist} There is a Breuil--Kisin--Fargues module $A_\inf\{1\}$ given by
\[
A_\inf\{1\} = \frac 1\mu (A_\inf\otimes_{\bb Z_p} \bb Z_p(1))
\]
if $K$ is of characteristic $0$ and contains all $p$-power roots of unity. Here, $\mu=[\epsilon]-1$ as usual. The Frobenius on $A_\inf\{1\}$ is induced by the usual Frobenius on $A_\inf$. More canonically, we have the following description. Recall the maps
\[
\tilde\theta_r: A_\inf\to W_r(\roi)\ ,
\]
with kernel generated by $\tilde\xi_r$. Then the cotangent complex
\[
W_r(\roi)\{1\} := \widehat{\bb L}_{W_r(\roi)/\bb Z_p}[-1] = \bb L_{W_r(\roi)/A_\inf}[-1] = \tilde\xi_r A_\inf / \tilde\xi_r^2 A_\inf
\]
is free of rank $1$ over $W_r(\roi)$. As in the Breuil--Kisin case, for $r>s$ the obvious map
\[
W_r(\roi)\{1\} = \tilde\xi_r A_\inf / \tilde\xi_r^2 A_\inf\to \tilde\xi_s A_\inf / \tilde\xi_s^2 A_\inf = W_s(\roi)\{1\}
\]
has image $p^{r-s} W_s(\roi)\{1\}$; thus, dividing it by $p^{r-s}$, we can take the inverse limit
\[
A_\inf\{1\} = \projlim_r W_r(\roi)\{1\}
\]
to get an $A_\inf$-module which is free of rank $1$. Again, it is equipped with a natural Frobenius. Moreover, if $K$ contains all $p$-power roots of unity and we fix a choice of roots of unity and the standard choice $\xi = \frac{\mu}{\phi^{-1}(\mu)}$ with $\mu=[\epsilon]-1$, then the system of elements $\tilde\xi_r\in W_r(\roi)\{1\}$ define a compatible system of elements (using that $\phi(\xi)\equiv p\mod \xi$), inducing a basis element $e\in A_\inf\{1\}$, on which $\phi$ acts by $e\mapsto \tfrac 1{\tilde\xi} e$. More canonically, there is a map
\[
d\log: W_r(\roi)^\times\to \Omega^1_{W_r(\roi)/\bb Z_p}\ ,
\]
which on $p$-adic Tate modules induces a map
\[
d\log: \bb Z_p(1)\to T_p(\Omega^1_{W_r(\roi)/\bb Z_p}) = W_r(\roi)\{1\}\ .
\]
These maps are compatible for varying $r$, inducing a map $\bb Z_p(1)\to A_\inf\{1\}$ which is equivariant for the trivial $\phi$-action on $\bb Z_p(1)$, and thus a map $A_\inf\otimes_{\bb Z_p} \bb Z_p(1)\to A_\inf\{1\}$, which can be checked to have image $\mu A_\inf\{1\}$. More concretely, this amounts to checking that the elements
\[
\left(\frac 1{[\zeta_{p^r}]-1} \frac{d([\zeta_{p^s}])}{[\zeta_{p^s}]}\right)_{s\geq 1}\in T_p(\Omega^1_{W_r(\roi)/\bb Z_p})  = W_r(\roi)\{1\}
\]
are generators.

If $M$ is any $A_\inf$-module, we set $M\{n\} = M\otimes_{A_\inf} A_\inf\{1\}^{\otimes n}$ for $n\in \bb Z$.
\end{example}

\begin{remark} Assuming again that $K$ contains the $p$-power roots of unity, there is a nonzero map $A_\inf\cong A_\inf\otimes_{\bb Z_p} \bb Z_p(1)\to A_\inf\{1\}$, whose cokernel is the module from Example~\ref{ex:nonbkfmodule} above. Thus, the category of Breuil--Kisin--Fargues modules is not stable under cokernels. It is still an exact tensor category, where the Tate twist is invertible.
\end{remark}

Let us discuss the \'etale specialization of a Breuil--Kisin--Fargues module. For this, we assume that $K=C$ is algebraically closed of characteristic $0$, and fix $p$-power roots of unity giving rise to $\epsilon\in C^\flat$ and $\mu=[\epsilon]-1$ as usual.

\begin{lemma}\label{lem:bkfmoduleetale} Let $(M,\phi_M)$ be a Breuil--Kisin--Fargues module, where the base field $K=C$ is algebraically closed of characteristic $0$. Then
\[
T=(M\otimes_{A_\inf} W(C^\flat))^{\phi_M = 1}
\]
(where $\phi_M$ means $\phi_M\otimes\phi$) is a finitely generated $\bb Z_p$-module which comes with an identification
\[
M\otimes_{A_\inf} W(C^\flat) = T\otimes_{\bb Z_p} W(C^\flat)\ .
\]
Moreover, one has
\[
M\otimes_{A_\inf} A_\inf[\tfrac 1\mu] = T\otimes_{\bb Z_p} A_\inf[\tfrac 1\mu]
\]
as submodules of the common base extension to $W(C^\flat)$.
\end{lemma}

Geometrically, $T$ corresponds to \'etale cohomology.

\begin{proof} As $C^\flat$ is an algebraically closed field of characteristic $p$, finitely generated $W(C^\flat)$-modules with a Frobenius automorphism are equivalent to finitely generated $\bb Z_p$-modules; this proves that $T$ is finitely generated and comes with an identification
\[
M\otimes_{A_\inf} W(C^\flat) = T\otimes_{\bb Z_p} W(C^\flat)\ .
\]
To prove the statement
\[
M\otimes_{A_\inf} A_\inf[\tfrac 1\mu] = T\otimes_{\bb Z_p} A_\inf[\tfrac 1\mu]\ ,
\]
one can formally reduce to the case where $M$ is finite free, using Proposition~\ref{prop:splitAinfmodule} and the observation that if $M$ is $p$-torsion, then $M\otimes_{A_\inf} W(C^\flat) = M\otimes_{A_\inf} A_\inf[\tfrac 1\mu]$. Thus, we assume from now on that $M$ is finite free.

First, we claim that $T\subset M\otimes_{A_\inf} A_\inf[\tfrac 1\mu]$. To prove this, we may replace $M$ by $M\{n\}$ for sufficiently large $n$ so that $\phi_M^{-1}$ maps $M$ into $M$. In that case, we claim the stronger statement $T\subset M$. Let $r\ge1$, take any element $t\in (M\otimes_{A_\inf} W_r(C^\flat))^{\phi_M = 1}$, and look at an element $x\in \roi^\flat$ of minimal valuation for which $[x] t\in M/p^r$. Assume that $x$ is not a unit. We have $\phi_M(t)=t$, or equivalently $t=\phi_M^{-1}(t)$, so $[x] t = \phi_M^{-1}([x]^p t)$. But then $[x]^p t\in [x]^{p-1} M/p^r$, and thus $\phi_M^{-1}([x]^p t)\in [x]^{(p-1)/p} M/p^r$, as $\phi_M^{-1}$ preserves $M$ by assumption. Thus, $[x]t\in [x]^{(p-1)/p} M/p^r$, which contradicts the choice of $x$. Thus, $x$ is a unit, so that $t\in M/p^r$. Passing to the limit over $r$ shows that $T\subset M$, as desired.

Applying the result $T\otimes_{A_\inf} A_\inf[\tfrac 1\mu]\subset M\otimes_{A_\inf} A_\inf[\tfrac 1\mu]$ also for the dual module $M^\ast$ and dualizing again shows the reverse inclusion, finishing the proof that $M\otimes_{A_\inf} A_\inf[\tfrac 1\mu] = T\otimes_{\bb Z_p} A_\inf[\tfrac 1\mu]$.
\end{proof}

Let us also mention the following result concerning the crystalline specialization, which works whenever $K$ is of characteristic $0$.

\begin{lemma}\label{lem:bkfmodulecrys} Let $M$ be a Breuil--Kisin--Fargues module. Then $M' = M\otimes_{A_\inf} W(k)$ is a finitely generated $W(k)$-module equipped with a Frobenius automorphism after inverting $p$. Fix a section $k\to \roi_K/p$, which induces a section $W(k)\to A_\inf$. Then there is a (noncanonical) $\phi$-equivariant isomorphism
\[
M\otimes_{A_\inf} B_\crys^+\cong M'\otimes_{W(k)} B_\crys^+
\]
reducing to the identity over $W(k)[\tfrac 1p]$.
\end{lemma}

\begin{proof} This follows from a result of Fargues--Fontaine, \cite[Corollaire 11.1.14]{FarguesFontaine}.
\end{proof}

We will see that in geometric situations, the $\phi$-equivariant isomorphism
\[
M\otimes_{A_\inf} B_\crys^+\cong M'\otimes_{W(k)} B_\crys^+
\]
is canonical, cf.~Proposition~\ref{prop:BerthelotOgusAcrys}. One can check using Lemma~\ref{lem:FreeoverAinf1p} that for Breuil--Kisin--Fargues modules equipped with the choice of such an isomorphism, and morphisms respecting those, the kernel and cokernel are again Breuil--Kisin--Fargues modules, so that this variant category is an abelian tensor category in which the objects coming from geometry live. However, the constructions for proper smooth (formal) schemes of this paper have analogues for $p$-divisible groups where the resulting identification is not canonical. In that case, the phenomenon that the category of Breuil--Kisin--Fargues modules is not abelian is related to the existence of the morphism of $p$-divisible groups
\[
\bb Q_p/\bb Z_p\to \mu_{p^\infty}
\]
over $\roi_K$, if $K$ contains all $p$-power roots of unity, which does not have any reasonable kernel or cokernel as it is $0$ in the special fibre, but an isomorphism in the generic fibre.

The main theorem about Breuil--Kisin--Fargues modules is Fargues' classification; we refer to~\cite{ScholzeLectureNotes} for a proof.

\begin{theorem}[Fargues]\label{ThmFargues} Assume that $K=C$ is algebraically closed of characteristic $0$. The category of finite free Breuil--Kisin--Fargues modules is equivalent to the category of pairs $(T,\Xi)$, where $T$ is a finite free $\bb Z_p$-module, and $\Xi$ is a $B_\dR^+$-lattice in $T\otimes_{\bb Z_p} B_\dR$. Here, the functor is given by sending a finite free Breuil--Kisin--Fargues module $(M,\phi_M)$ to the pair $(T,\Xi)$, where
\[
T=(M\otimes_{A_\inf} W(C^\flat))^{\phi_M=1}
\]
and
\[
\Xi = M\otimes_{A_\inf} B_\dR^+\subset M\otimes_{A_\inf} B_\dR\cong T\otimes_{\bb Z_p} B_\dR\ .
\]
\end{theorem}

\begin{remark} For the proof of our main theorems, we only need fully faithfulness of the functor $M\mapsto (T,\Xi)$, which is easy to prove directly. Indeed, faithfulness follows directly from Lemma~\ref{lem:bkfmoduleetale} and the observation that $A_\inf\to A_\inf[\tfrac 1\mu]$ is injective. Now, given two Breuil--Kisin--Fargues modules $(M,\phi_M)$ and $(N,\phi_N)$ and a map $T(M)\to T(N)$ mapping $\Xi(M)$ into $\Xi(N)$, Lemma~\ref{lem:bkfmoduleetale} gives a canonical $\phi$-equivariant map $M[\tfrac 1\mu]\to N[\tfrac 1\mu]$. We have to see that this maps $M$ into $N$. By finite generation of $M$, $M$ maps into $\mu^{-n} N$ for some $n\geq 0$, which we assume to be minimal. Assume $n>0$; replacing $N$ by $\mu^{-n+1} N$, we may reduce to the case $n=1$. We claim that $M$ maps into $\phi^{-r}(\mu)^{-1} N$ for all $r\geq 0$, by induction on $r$. For this, we need to see that the induced maps
\[
M\to \phi^{-r}(\mu)^{-1} N / \phi^{-r-1}(\mu)^{-1} N \cong N / \phi^{-r}(\xi) N = N\otimes_{A_\inf,\theta\circ \phi^r} \roi\hookrightarrow N\otimes_{A_\inf,\theta\circ \phi^r} C
\]
are zero, where the isomorphism is multiplication by $\phi^{-r}(\mu)$. But note that by assumption, $\Xi(M) = M\otimes_{A_\inf} B_\dR^+$ maps into $\Xi(N) = N\otimes_{A_\inf} B_\dR^+$, and so by the diagram
\[\xymatrix{
M\otimes_{A_\inf,\phi^r} B_\dR^+\ar[d]^{\phi_M^r}_\cong \ar[r] & N\otimes_{A_\inf,\phi^r} B_\dR^+\ar[d]^{\phi_N^r}_\cong\\
M\otimes_{A_\inf} B_\dR^+\ar[r] & N\otimes_{A_\inf} B_\dR^+
}\]
also $M\otimes_{A_\inf,\phi^r} B_\dR^+$ maps into $N\otimes_{A_\inf,\phi^r} B_\dR^+$ for all $r\geq 0$. Therefore, the map $M\to N\otimes_{A_\inf,\theta\circ\phi^r} C$ induced by multiplication by $\phi^{-r}(\mu)$ is zero, showing that indeed $M$ maps into $\phi^{-r}(\mu)^{-1} N$ for all $r\geq 0$. But now $N=\bigcap_{r\geq 0} \phi^{-r}(\mu)^{-1} N$ by Lemma~\ref{lem:mapAinfWitt}, so $M$ maps into $N$, as desired.
\end{remark}

We warn the reader that, like in Theorem~\ref{ThmKisin}, this equivalence of categories is not exact. More precisely, the functor from Breuil--Kisin--Fargues modules to pairs $(T,\Xi)$ is exact, but the inverse is not.

As an easy example, $A_\inf\{1\}$ corresponds to $T=\bb Z_p(1)$ and $\Xi = \xi^{-1}( T\otimes_{\bb Z_p} B_\dR^+)$.

\subsection{Relating Breuil--Kisin and Breuil--Kisin--Fargues modules}

Let us observe that any Breuil--Kisin module defines a Breuil--Kisin--Fargues module. For this, we start again with a complete discretely valued extension $K$ of $\bb Q_p$ with perfect residue field $k$ and fixed uniformizer $\pi\in K$, and let $C$ be a completed algebraic closure of $K$ with fixed roots $\pi^{1/p^n}\in C$, giving an element $\pi^\flat\in C^\flat$. Then $\gS = W(k)[[T]]$ is equipped with a Frobenius automorphism $\phi$, and the map $\tilde{\theta}:\gS \to \roi_K$ given by $T \mapsto \pi$. The constructions over $K$ and $C$ are related by the map  $\gS \to A_\inf$ that sends $T$ to $[\pi^\flat]^p$ and is the Frobenius on $W(k)$; note that this map commutes with $\phi$ and $\tilde{\theta}$. We first check that this map is flat: 

\begin{lemma}
\label{lem:SigmaAinfFlat}
The map $\gS \to A_\inf$ above is flat.
\end{lemma}
\begin{proof}
We must check that $M \dotimes_{\gS} A_\inf$ is concentrated in degree $0$ for any $\gS$-module $M$. By approximation, we may assume $M$ is finitely generated. As $\gS$ is regular, any such $M$ is perfect. Thus, $M \dotimes_{\gS} A_\inf$ is also perfect. In particular, it is derived $p$-adically complete, so we can write $M \dotimes_{\gS} A_\inf \simeq R\lim(M/p^n \dotimes_{\gS/p^n} A_\inf/p^n)$; here we implicitly use the Artin--Rees lemma over $\gS$ to replace the pro-system $\{M \dotimes_{\gS} \gS/p^n\}$ with $\{M/p^n\}$. It is now enough to check that $\gS/p^n \to A_\inf/p^n$ is flat. As both rings are flat over $\mathbb{Z}/p^n$, we may assume $n=1$, i.e., we need to show $\gS/p \to A_\inf/p \simeq \roi^\flat$ is flat; this is clear as the source is a discrete valuation ring, and the target is torsionfree.
\end{proof}

\begin{remark}
More generally, one has: if $A \to B$ is a map of $p$-adically complete $p$-torsionfree rings with $A$ noetherian and $A/p \to B/p$ flat, then $A \to B$ is flat. To prove this, one simply replaces perfect complexes with pseudo-coherent complexes in the proof above.
\end{remark}

Base change along this map relates Breuil--Kisin modules to Breuil--Kisin--Fargues modules:

\begin{proposition}\label{prop:bktobkf}
The association $M \mapsto M \otimes_{\gS} A_\inf$ defines an exact tensor functor from Breuil--Kisin modules over $\gS$ to Breuil--Kisin--Fargues modules over $A_\inf$.
\end{proposition}

\begin{proof}
Let $(M,\phi_M)$ be a Breuil--Kisin module over $\gS$, i.e., $M$ is a finitely presented $\gS$-module equipped with an identification $\phi_M:(\phi^* M)[\frac{1}{E}] \simeq M[\frac{1}{E}]$, where $E(T) \in \gS$ is the Eisenstein polynomial defining $\pi$. We claim that $N := M \otimes_{\gS} A_\inf$ equipped with the identification $(\phi_M \otimes \mathrm{id}): (\phi^* N)[\frac{1}{f(E)}] \simeq N[\frac{1}{f(E)}]$ is a Breuil--Kisin--Fargues modules. For this, first note that $\tilde\xi := f(E)$ is a generator of the kernel of $\tilde{\theta}:A_\inf \to \roi_C$. Moreover, $M[\frac{1}{p}]$ is free by Proposition \ref{prop:BKFree}, so $N[\frac{1}{p}]$ is free as well; this verifies that we obtain a Breuil--Kisin--Fargues module. The resulting functor is clearly symmetric monoidal, and exactness follows from Lemma \ref{lem:SigmaAinfFlat}.
\end{proof}

\begin{corollary}\label{cor:identtatetwist} Under the functor of Theorem~\ref{ThmKisin}, $\bb Z_p(1)$ is sent to $\gS\{1\}$.
\end{corollary}

\begin{proof} From the definition in terms of cotangent complexes, we see that $\gS\{1\}\otimes_{\gS} A_\inf\cong A_\inf\{1\}$ as Breuil--Kisin--Fargues modules, compatibly with the $G_{K_\infty}$-action. As there is a canonical identification
\[
A_\inf\{1\} = \tfrac 1\mu (\bb Z_p(1)\otimes_{\bb Z_p} A_\inf )\ ,
\]
in particular we get a $\phi,G_{K_\infty}$-equivariant identification
\[
\gS\{1\}\otimes_\gS W(C^\flat)\cong \bb Z_p(1)\otimes_{\bb Z_p} W(C^\flat) \ ,
\]
which by Theorem~\ref{ThmKisin} proves that $\bb Z_p(1)$ is sent to $\gS\{1\}$.
\end{proof}

Finally, we want to relate Theorem~\ref{ThmKisin} with Theorem~\ref{ThmFargues}. Thus, let $T$ be a lattice in a crystalline $G_K$-representation $V$. We get
\[
D_\crys(V) = (V\otimes_{\bb Q_p} B_\crys)^{G_K}\ ,
\]
which comes with a $\phi,G_K$-equivariant identification
\[
D_\crys(V)\otimes_{W(k)[\frac 1p]} B_\crys = V\otimes_{\bb Q_p} B_\crys\ .
\]
On the other hand, by Theorem~\ref{ThmKisin}, we have the finite free Breuil--Kisin module $M(T)$ over $\gS$, which gives rise to a finite free Breuil--Kisin--Fargues module $M(T)\otimes_\gS A_\inf$. By Theorem~\ref{ThmKisin} and Lemma~\ref{lem:bkfmoduleetale}, we have a $G_{K_\infty}$-equivariant identification
\[
M(T)\otimes_\gS A_\inf[\tfrac 1\mu] = T\otimes_{\bb Z_p} A_\inf[\tfrac 1\mu]\ .
\]

\begin{proposition}\label{prop:compbkbkf} One has an equality
\[
M(T)\otimes_\gS B_\crys^+ = D_\crys(V)\otimes_{W(k)[\frac 1p]} B_\crys^+
\]
as submodules of
\[
M(T)\otimes_\gS B_\crys = T\otimes_{\bb Z_p} B_\crys = D_\crys(V)\otimes_{W(k)[\frac 1p]} B_\crys\ .
\]

In particular, under Fargues' classification, $M(T)\otimes_\gS A_\inf$ corresponds to the pair $(T,\Xi)$, where
\[
\Xi = D_\crys(V)\otimes_{W(k)[\frac 1p]} B_\dR^+\subset T\otimes_{\bb Z_p} B_\dR\ ;
\]
equivalently,
\[
\Xi = D_\dR(V)\otimes_K B_\dR^+\subset T\otimes_{\bb Z_p} B_\dR\ ,
\]
where $D_\dR(V) = (T\otimes_{\bb Z_p} B_\dR)^{G_K}$.
\end{proposition}

The moral of the story here is that if one does $p$-adic Hodge theory over $C$, there is no Galois action on $T$ anymore, and instead one should keep track of a $B_\dR^+$-lattice in $T\otimes_{\bb Z_p} B_\dR$, which is a shadow of the de~Rham comparison isomorphism. In Section~\ref{sec:ratPAdicHodgeNew} below we will give a geometric construction of a $B_\dR^+$-lattice in \'etale cohomology tensored with $B_\dR$ for any proper smooth rigid-analytic variety over $C$ (in a way compatible with the usual de~Rham comparison isomorphism).

\begin{proof} This follows from Kisin's construction of $M(T)$, which starts with the crystalline side and an isomorphism between $M(T)$ and $D_\crys(V)\otimes \gS[\frac 1p]$ on some rigid-analytic open of the generic fibre of $\Spf \gS$, cf.~\cite[Section 1.2, Lemma 1.2.6]{Kisin}.
\end{proof}

\newpage

\section{Rational $p$-adic Hodge theory}

In this section, we recall a few facts from rational $p$-adic Hodge theory, in the setting of~\cite{ScholzePAdicHodge}. Let $K$ be some complete discretely valued extension of $\bb Q_p$ with perfect residue field $k$, and let $X$ be a proper smooth rigid-analytic variety over $K$, considered as an adic space. Let $C$ be a completed algebraic closure of $K$ with absolute Galois group $G_K$, and let $B_\dR = B_\dR(C)$ be Fontaine's field of $p$-adic periods.

\begin{theorem}[\cite{ScholzePAdicHodge}]\label{thm:ratPAdicHodge} The $p$-adic \'etale cohomology groups $H^i_\sub{\'et}(X_C,\bb Z_p)$ are finitely generated $\bb Z_p$-modules, and there is a comparison isomorphism
\[
H^i_\sub{\'et}(X_C,\bb Z_p)\otimes_{\bb Z_p} B_\dR\cong H^i_\dR(X)\otimes_K B_\dR\ ,
\]
compatible with the $G_K$-action, and natural filtrations. In particular, $H^i_\sub{\'et}(X_C,\bb Q_p)$ is de~Rham as a $G_K$-representation.
\end{theorem}

In particular, the theorem gives a natural $B_\dR^+$-lattice
\[
H^i_\dR(X)\otimes_K B_\dR^+\subset H^i_\sub{\'et}(X_C,\bb Z_p)\otimes_{\bb Z_p} B_\dR\ ,
\]
where $H^i_\dR(X) = (H^i_\sub{\'et}(X_C,\bb Z_p)\otimes_{\bb Z_p} B_\dR)^{G_K}$. Thus, by Theorem~\ref{ThmFargues}, the torsion-free quotient of $H^i_\sub{\'et}(X_C,\bb Z_p)$ and this $B_\dR^+$-lattice given by de~Rham cohomology define a finite free Breuil--Kisin--Fargues module, which we will call
\[
\mathrm{BKF}(H^i_\sub{\'et}(X_C,\bb Z_p))\ .
\]

\begin{remark} Assume that the torsion-free quotient of $H^i_\sub{\'et}(X_C,\bb Z_p)$ is crystalline as a Galois representation. Then, by Theorem~\ref{ThmKisin}, there is an associated Breuil--Kisin module
\[
\mathrm{BK}(H^i_\sub{\'et}(X_C,\bb Z_p))\ .
\]
By Proposition~\ref{prop:compbkbkf}, we then have
\[
\mathrm{BKF}(H^i_\sub{\'et}(X_C,\bb Z_p)) = \mathrm{BK}(H^i_\sub{\'et}(X_C,\bb Z_p))\otimes_{\gS} A_\inf\ .
\]
\end{remark}

In fact, the $B_\dR^+$-lattice
\[
H^i_\dR(X)\otimes_K B_\dR^+\subset H^i_{\sub{\'et}}(X_C,\bb Z_p)\otimes_{\bb Z_p} B_\dR
\]
depends only on $X_C$. We postpone discussion of this point until later, see Section~\ref{sec:ratPAdicHodgeNew}. This implies that the construction of $\mathrm{BKF}(H^i_\sub{\'et}(X,\bb Z_p))$ works for any proper smooth rigid-analytic space $X$ over $C$.

The goal of this paper is to show that if $X$ is the generic fibre of some proper smooth formal scheme $\frak X/\roi_K$, then this Breuil--Kisin--Fargues module can be constructed geometrically, and deduce comparisons between the Breuil--Kisin--Fargues module and the crystalline cohomology of the special fibre.\\

In this section, we will recall aspects of pro-\'etale cohomology following \cite{ScholzePAdicHodge}, and then briefly recall the strategy of the proof of Theorem~\ref{thm:ratPAdicHodge}.

\subsection{The pro-\'etale site of an adic space}
We first recall the pro-\'etale site \cite[Definition~3.9]{ScholzePAdicHodge} associated to a locally Noetherian adic space $X$. Let $\rm{pro-}X_\sub{\'et}$ be the category of pro-objects associated to the category $X_\sub{\'et}$ of adic spaces which are \'etale over $X$. Objects of $\rm{pro-}X_\sub{\'et}$ will be denoted by $\projlimf_{i\in I}U_i$, where $I$ is a small cofiltered category and $I\to X_\sub{\'et},\;i\mapsto U_i$ is a functor. The underlying topological space of $\projlimf_{i\in I}U_i$ is by definition $\projlim_{i\in I}|U_i|$, where $|U_i|$ is the underlying topological space of $U_i$.

An object $U\in \rm{pro-}X_\sub{\'et}$ is said to be {\em pro-\'etale over $X$} if and only if $U$ is isomorphic in $\rm{pro-}X_\sub{\'et}$ to an object $\projlimf_{i\in I}U_i$ with the property that all transition maps $U_j\to U_i$ are finite \'etale and surjective.

The {\em pro-\'etale site} $X_\sub{pro\'et}$ of $X$ is the full subcategory $\rm{pro-}X_\sub{\'et}$ consisting of those objects which are pro-\'etale over $X$; a collection of maps $\{f_i:U_i\to U\}$ in $X_\sub{pro\'et}$ is defined to be a covering if and only if the collection $\{|U_i|\to|U|\}$ is a pointwise covering of the topological space $|U|$, and moreover each $f_i$ satisfies the following assumption (which is stronger than asking that $f_i$ is pro-\'etale in the sense of~\cite[Definition 3.9]{ScholzePAdicHodge}, but the notions agree for countable inverse limits).\footnote{Cf.~\cite{ScholzePAdicHodgeErratum}.} One can write $U_i\to U$ as an inverse limit $U_i=\varprojlim_{\mu<\lambda} U_\mu$ of $U_\mu\in X_\sub{pro\'et}$ along some ordinal $\lambda$, such that $U_0\to U$ is \'etale (i.e.~the pullback of a map in $X_\sub{\'et}$), and for all $\mu>0$, the map
\[
U_\mu\to U_{<\mu} := \varprojlim_{\mu^\prime<\mu} U_{\mu^\prime}
\]
is finite \'etale and surjective, i.e.~the pullback of a finite \'etale and surjective map in $X_\sub{\'et}$ (cf.~\cite[Lemma 3.10 (v)]{ScholzePAdicHodge}).

There is a natural projection map of sites
\[
\nu:X_\sub{pro\'et}\to X_\sub{\'et}\ ,
\]
with the property that
\[
H^j(U,\nu^*\cal F)=\indlim_{i\in I}H^j(U_i,\cal F)
\]
for any abelian sheaf $\cal F$ on $X_\sub{\'et}$, $j\geq 0$, and any $U=\projlimf_{i\in I}U_i\in X_\sub{pro\'et}$ for which $|U|$ is quasi-compact and quasi-separated, \cite[Lemma~3.16]{ScholzePAdicHodge}.

Suppose now that $X$ is a locally Noetherian adic space over $\Spa(\bb Q_p,\bb Z_p)$. An object $U\in X_\sub{pro\'et}$ is said to be {\em affinoid perfectoid} \cite[Definition~4.3]{ScholzePAdicHodge} if and only if it is isomorphic in $X_\sub{pro\'et}$ to an object $\projlimf_{i\in I}U_i$ with the following three properties:
\begin{enumerate}\itemsep0pt
\item the transition maps $U_j\to U_i$ are finite \'etale surjective whenever $j\ge i$; 
\item $U_i=\Spa(R_i, R_i^+)$ is affinoid for each $i$;
\item the complete Tate ring $R:=(\indlim_iR_i^+)_p^\wedge\otimes_{\bb Z_p}\bb Q_p$ is perfectoid.
\end{enumerate}

We note that the final condition implies that $R^+:=(\indlim_i R_i^+)^\wedge$ is a perfectoid ring, by Lemma~\ref{lemma_Tate_perfectoid}.

Continuing to assume that $X$ is a locally Noetherian adic space over $\Spa(\bb Q_p,\bb Z_p)$, it is known that the affinoid perfectoid objects in $X_\sub{pro\'et}$ form a basis for the topology \cite[Proposition~4.8]{ScholzePAdicHodge}. We will only require this result when $X$ is smooth over $\Spa(C,\roi)$, where $C$ is a perfectoid field of mixed characteristic and $\roi=\roi_C=C^\circ\subset C$ is its ring of integers; in this case we recall some details of the proof (see \cite[Example~4.4, Lemma~4.6, Corollary~4.7]{ScholzePAdicHodge}). Locally, $X$ admits an \'etale map to the $d$-dimensional torus
\[
\bb T^d:=\Spa(C\langle T_1^{\pm 1},\dots,T_d^{\pm 1}\rangle,\roi\langle T_1^{\pm 1},\dots, T_d^{\pm 1}\rangle)
\]
that factors as a composite of rational embeddings and finite \'etale covers. In this case, we have the following lemma.

\begin{lemma}[{\cite[Lemma 4.5]{ScholzePAdicHodge}}]\label{lem:perfectoidtower} Let $X\to \bb T^d$ be an \'etale map that factors as a composite of rational embeddings and finite \'etale maps. For $r\geq 1$, let
\[
X_r = X\times_{\bb T^d} \bb T_r^d\ ,
\]
where
\[
\bb T_r^d = \Spa(C\langle T_1^{\pm 1/p^r},\dots,T_d^{\pm 1/p^r}\rangle,\roi\langle T_1^{\pm 1/p^r},\dots, T_d^{\pm 1/p^r}\rangle)\ .
\]
Then $\projlimf_r X_r\in X_\sub{pro\'et}$ is affinoid perfectoid.
\end{lemma}

We recall the main sheaves of interest on $X_\sub{pro\'et}$, and explicitly state their values on an affinoid perfectoid $U=\projlimf_{i\in I}\Spa(R_i,R_i^+)$.

\begin{definition}
\label{def:PeriodSheaves}
 Consider the following sheaves on $X_\sub{pro\'et}$.
\begin{enumerate}
\item The integral structure sheaf $\roi_X^+ = \nu^\ast \roi_{X_\sub{\'et}}^+$.
\item The structure sheaf $\roi_X = \nu^\ast \roi_{X_\sub{\'et}}$.
\item The completed integral structure sheaf $\hat\roi_X^+ = \projlim_r \roi_X^+/p^r$.
\item The completed structure sheaf $\hat\roi_X = \hat\roi_X^+[\tfrac 1p]$.
\item The tilted (completed) integral structure sheaf $\hat\roi_{X^\flat}^+ = \projlim_\phi \roi_X^+/p$.
\item Fontaine's period sheaf $\bb A_{\inf,X}$, the derived $p$-adic completion of $W(\hat\roi_{X^\flat}^+)$.
\end{enumerate}
\end{definition}

\begin{remark} The sheaves $\hat\roi_X^+$ and $W_r(\hat\roi_X^+)$ are derived $p$-adically complete (cf.~Section~\ref{ss:Completions} for the definition of derived completeness). This follows from the next lemma and its natural version for $W_r(\hat\roi_X^+)$ and the observation that almost zero modules are always derived complete in these contexts. However, it is not clear to us whether $W(\hat\roi_{X^\flat}^+)$ is derived $p$-adically complete. (Its failure to be derived $p$-adically complete is $[\frak m^\flat]$-torsion.) This is the reason that we define $\bb A_{\inf,X}$ as the derived $p$-adic completion of $W(\hat\roi_{X^\flat}^+)$ (which actually makes it a sheaf of complexes).
\end{remark}

\begin{lemma}[{\cite[Lemma 4.10, Lemma 5.10, Theorem 6.5]{ScholzePAdicHodge}}]\label{lem:descrperiodsheaves} Let $U=\projlimf_i U_i\in X_\sub{pro\'et}$ be affinoid perfectoid, where the $U_i=\Spa(R_i,R_i^+)$ are affinoid, such that the $p$-adically completed direct limit $(R,R^+)$ of the $(R_i,R_i^+)$ is perfectoid. Then
\[\begin{aligned}
\roi_X^+(U) &= \varinjlim_i R_i^+\ ,\ \roi_X(U) = \varinjlim_i R_i\ ,\ \hat\roi_X^+(U) = R^+\ ,\\ \hat\roi_X(U) &= R\ ,\ \hat\roi_{X^\flat}^+(U) = R^{+ \flat}\ ,\ H^0(U,\bb A_{\inf,X}) = \bb A_\inf(R^+)\ .
\end{aligned}\]
Moreover, for $i>0$, the groups
\[
H^i(U,\roi_X) = H^i(U,\hat\roi_X) = 0
\]
vanish, the $\roi$-modules $H^i(U,\roi_X^+)$ and $H^i(U,\hat\roi_X^+)$ are killed by $\frak m$, the $\roi^\flat$-module $H^i(U,\hat\roi_{X^\flat}^+)$ is killed by $\frak m^\flat$, and the $A_\inf$-module $H^i(U,\bb A_{\inf,X})$ is killed by $[\frak m^\flat]$.
\end{lemma}

We note that using the argument from the proof of Theorem~\ref{thm:etalevsAinfcohom} below, it follows that $H^i(U,\bb A_{\inf,X})$ is actually killed by $W(\frak m^\flat)$.

Also, using the same formulae as Lemma~\ref{lemma_witt_alg_1}, there is a chain of natural morphisms of sheaves on $X_\sub{pro\'et}$:
\[
W(\hat\roi^+_{X^\flat})=\projlim_R W_r(\hat\roi^+_{X^\flat})\stackrel{\phi^\infty}{\longleftarrow}\projlim_F W_r(\hat\roi^+_{X^\flat})\To \projlim_F W_r(\roi^+_X/p) \longleftarrow\projlim_F W_r(\hat\roi^+_X)\ .
\]
Each of these morphisms is an isomorphism of sheaves; this follows from sheafifying the proof of Lemma~\ref{lemma_witt_alg_1}. Therefore, there are induced morphisms
\[
\tilde \theta_r:\bb A_{\inf,X}\to W_r(\hat\roi_X^+)\ ,\ \theta_r:=\tilde\theta_r\phi^r: \bb A_{\inf,X}\to W_r(\hat\roi_X^+)\ ,
\]
and $\theta:=\theta_1: \bb A_{\inf,X}\to \hat\roi_X^+$, as the target is in all cases derived $p$-adically complete already. By checking on affinoid perfectoids, the results of Section~\ref{sec:perfectoid} imply similar results on the level of sheaves on $X_\sub{pro\'et}$.

We will need the following result.

\begin{theorem}[{\cite[proof of Theorem 8.4]{ScholzePAdicHodge}}]\label{thm:etalevsAinfcohom} Assume that $C$ is algebraically closed, and $X$ is a proper smooth adic space over $C$. Then the inclusion $A_\inf\hookrightarrow \bb A_{\inf,X}$ induces an almost quasi-isomorphism
\[
R\Gamma_\sub{\'et}(X,\bb Z_p)\otimes_{\bb Z_p} A_\inf\to R\Gamma(X_\sub{pro\'et},\bb A_{\inf,X})\ ;
\]
more precisely, the cohomology of the cone is killed by $W(\frak m^\flat)$.
\end{theorem}

\begin{proof} The cohomology of the cone is killed by $[\frak m^\flat]$, and derived $p$-complete (cf.~Lemma~\ref{lem:derivedcompletecohom}). Thus, it becomes a module over the derived $p$-completion of $A_\inf/[\frak m^\flat]$, which is given by $W(k) = A_\inf/W(\frak m^\flat)$. In particular, it is killed by $W(\frak m^\flat)$.
\end{proof}

Let us now briefly recall the proof of Theorem~\ref{thm:ratPAdicHodge}. Let $X/K$ be a proper smooth rigid-analytic variety. Theorem~\ref{thm:etalevsAinfcohom} implies that
\[
R\Gamma_\sub{\'et}(X_C,\bb Z_p)\otimes_{\bb Z_p} B_\dR^+\cong R\Gamma_\sub{pro\'et}(X_C,\bb B_{\dR,X}^+)\ ,
\]
where $\bb B_{\dR,X}^+$ is the relative period sheaf defined in \cite{ScholzePAdicHodge}. On the other hand, one can define a sheaf $\roi\bb B_{\dR,X}^+$ as a suitable completion of $\roi_X\otimes_{W(k)} \bb B_{\dR,X}^+$,\footnote{The original definition was slightly wrong, cf.~\cite{ScholzePAdicHodgeErratum}.} which comes with a connection $\nabla$ (induced from the $\roi_X$-factor), and there is a Poincar\'e lemma:
\[
0\to \bb B_{\dR,X}^+\to \roi \bb B_{\dR,X}^+\xTo{\nabla} \roi \bb B_{\dR,X}^+\otimes_{\roi_X} \Omega^1_X\to \ldots
\]
is exact; this is inspired by work of Andreatta--Iovita, \cite{AndreattaIovita}. One finishes by observing that the cohomology of $\roi \bb B_{\dR,X}^+[\xi^{-1}]$ is the same as the cohomology of $\roi_X\hat{\otimes}_K B_\dR$, which follows from a direct Galois cohomology computation, due to Brinon, \cite{BrinonRepresentations}.

\newpage

\section{The $L\eta$-operator}

Consider a ring $A$ and non-zero-divisor $f\in A$, and denote by $D(A)$ the derived category of $A$-modules. If $M^\blob$ is a cochain complex such that $M^i$ is $f$-torsion free for all $i\in \bb Z$, we denote by $\eta_fM^\blob$ the subcomplex of $M^\blob[\tfrac 1f]$ defined as
\[
(\eta_f M)^i:=\{x\in f^iM^i:dx\in f^{i+1}M^{i+1}\}\ .
\]
In \S \ref{subset:LetaCons}, we show that the functor $\eta_f(-)$ descends to the derived category, inducing a (non-exact!) functor $L\eta_f: D(A)\to D(A)$, and study various properties of the resulting construction. In \S \ref{ss:Completions}, we recall some basic properties of completions in the derived category, and study their commutation with $L\eta$.

\subsection{Construction and properties of $L\eta$}
\label{subset:LetaCons}

For applications, it will be important to have the $L\eta_f$-operation also on a ringed site (or topos), so let us work in this generality. Let $(T,\roi_T)$ be a ringed topos. Let $D(\roi_T)$ be the derived category of $\roi_T$-modules. Recall that $D(\roi_T)$ is, by definition, the localization of the category $K(\roi_T)$ of complexes of $\roi_T$-modules (up to homotopy) at the quasi-isomorphisms.

Recall that a complex $C^\blob\in K(\roi_T)$ is $K$-flat if for every acyclic complex $D^\blob\in K(\roi_T)$, the total complex $\mathrm{Tot}(C^\blob\otimes_{\roi_T} D^\blob)$ is acyclic, cf.~\cite[Tag 06YN]{StacksProject}. Let us say that $C^\blob$ is strongly $K$-flat if in addition each $C^i$ is a flat $\roi_T$-module.

\begin{lemma}\label{lem:flatresolution} For every complex $D^\blob\in K(\roi_T)$, there is a strongly $K$-flat complex $C^\blob\in K(\roi_T)$ and a quasi-isomorphism $C^\blob\to D^\blob$.

In particular, $D(\roi_T)$ is the localization of the full subcategory of strongly $K$-flat complexes in $K(\roi_T)$, along the quasi-isomorphisms.
\end{lemma}

\begin{proof} The first sentence follows from \cite[Tag 077J]{StacksProject} (and its proof to see that the complex is strongly $K$-flat, noting that filtered colimits of flat modules are flat). The second sentence is a formal consequence.
\end{proof}

Now let $\cal I\subset \roi_T$ be an invertible ideal sheaf. Weakening the notion of strongly $K$-flat complexes, we say that $C^\blob$ is $\cal I$-torsion-free if the map $\cal I\otimes C^i\to C^i$ is injective for all $i\in \bb Z$; we denote its image by $\cal I\cdot C^i\subset C^i$.

\begin{definition} Let $C^\blob\in K(\roi_T)$ be an $\cal I$-torsion-free complex. Define a new ($\cal I$-torsion-free) complex $\eta_{\cal I} C^\blob = (\eta_{\cal I} C)^\blob\in K(\roi_T)$ with terms
\[
(\eta_{\cal I} C)^i = \{x\in C^i\mid dx\in \cal I\cdot C^{i+1}\}\otimes_{\roi_T} \cal I^{\otimes i}
\]
and differential
\[
d_{(\eta_{\cal I} C)^i}: (\eta_{\cal I} C)^i\to (\eta_{\cal I} C)^{i+1}
\]
making the following diagram commute:
\[\xymatrix{
(\eta_{\cal I} C)^i\ar^{d_{C^i}\otimes \cal I^{\otimes i}}[r] \ar_{d_{(\eta_{\cal I} C)^i}}[d] & \cal I\cdot C^{i+1}\otimes \cal I^{\otimes i}\ar^{\cong}[d] \\
(\eta_{\cal I} C)^{i+1}\ar@{^(->}[r] & C^{i+1}\otimes \cal I^{\otimes (i+1)}\ .
}\]
\end{definition}

\begin{remark}
\label{rmk:LetaPrincipal} The definition is phrased to depend only on the ideal $\cal I$, and not on a chosen generator $f\in \cal I$. If $f\in \cal I$ is a generator (assuming it exists), then one has
\[
(\eta_{\cal I} C)^i = (\eta_f C)^i := \{x\in f^i C^i\mid dx\in f^{i+1} C^{i+1}\}\ ,
\]
and the differential is compatible with the differential on $C^\blob[\tfrac 1f]$. Moreover, in this case, there is an isomorphism $\eta_{\cal I} (C[1]) \simeq (\eta_{\cal I}C)[1]$ given by multiplication by $f$ in each degree. 
\end{remark}

One can describe the effect of this operation on cohomology as killing the $\cal I$-torsion:

\begin{lemma}\label{lem:CohomLeta} Let $C^\blob\in K(\roi_T)$ be an $\cal I$-torsion-free complex. Then there is a canonical isomorphism
\[
H^i(\eta_{\cal I} C^\blob) = \left(H^i(C^\blob) / H^i(C^\blob)[\cal I]\right)\otimes_{\roi_T} \cal I^{\otimes i}
\]
for all $i\in \bb Z$. Here,
\[
H^i(C^\blob)[\cal I] = \ker(H^i(C^\blob)\to H^i(C^\blob)\otimes_{\roi_T} \cal I^{\otimes -1})\subset H^i(C^\blob)
\]
is the $\cal I$-torsion.

In particular, if $\alpha: C^\blob\to D^\blob$ is a quasi-isomorphism of $\cal I$-torsion free complexes, then so is $\eta_{\cal I} \alpha: \eta_{\cal I} C^\blob\to \eta_{\cal I} D^\blob$.
\end{lemma}

\begin{proof} Let $Z^i(C^\blob)\subset C^i$, $Z^i(\eta_{\cal I} C^\blob)\subset (\eta_{\cal I} C)^i$ be the cocycles. Then there is a natural isomorphism
\[
Z^i(C^\blob)\otimes_{\roi_T} \cal I^{\otimes i}\cong Z^i(\eta_{\cal I} C^\blob)\ ,
\]
inducing a surjection
\[
H^i(C^\blob)\otimes_{\roi_T} \cal I^{\otimes i}\to H^i(\eta_{\cal I} C^\blob)\ .
\]
Unraveling the definitions, one sees that if $x\in Z^i(C^\blob)\otimes_{\roi_T} \cal I^{\otimes i}$ is a cocycle, then its image in $H^i(\eta_{\cal I} C^\blob)$ vanishes if and only if there is an element $y\in C^{i-1}\otimes_{\roi_T} \cal I^{\otimes (i-1)}$ such that
\[
dy\in Z^i\otimes_{\roi_T} \cal I^{\otimes (i-1)}\cong \Hom_{\roi_T}(\cal I,Z^i\otimes_{\roi_T} \cal I^{\otimes i})
\]
agrees with the map $\cal I\subset \roi_T\buildrel x\over\to Z^i(C^\blob)\otimes_{\roi_T} \cal I^{\otimes i}$. This happens precisely when $x$ gives an $\cal I$-torsion element of $H^i(C^\blob)$. The final statement follows formally.
\end{proof}

In particular, the following corollary follows.

\begin{corollary}
\label{cor:LetaExists}The functor $\eta_{\cal I}$ from strongly $K$-flat complexes in $K(\roi_T)$ to $D(\roi_T)$ factors canonically over a functor $L\eta_{\cal I}: D(\roi_T)\to D(\roi_T)$. The functor $L\eta_{\cal I}$ commutes with all filtered colimits.

Moreover, $L\eta_{\cal I}: D(\roi_T)\to D(\roi_T)$ commutes with canonical truncations, i.e.~for all $a\leq b$ in $\bb Z\cup \{-\infty,\infty\}$ and any $C\in D(\roi_T)$, one has
\[
L\eta_{\cal I}(\tau^{[a,b]} C)\cong \tau^{[a,b]} L\eta_{\cal I}(C)\ .
\]
\end{corollary}

We repeat a warning made earlier:

\begin{remark}
The functor $L\eta_{\cal I}:D(\roi_T) \to D(\roi_T)$ constructed above is {\em not} exact. For example, when $T$ is the punctual topos and $\cal I = (p) \subset \mathbb{Z}$, then $L\eta_{\cal I}(\mathbb{Z}/p\bb Z) = 0$, but $L\eta_{\cal I}(\mathbb{Z}/p^2\bb Z) = \mathbb{Z}/p\bb Z \neq 0$.
\end{remark}

The operation $L\eta_{\cal I}$ interacts well with the $\otimes$-structure:

\begin{proposition}\label{prop:Letalaxsymmmon} There is a natural lax symmetric monoidal structure on $L\eta_{\cal I}: D(\roi_T)\to D(\roi_T)$, i.e.~for all $C,D\in D(\roi_T)$, there is a natural map
\[
L\eta_{\cal I} C\dotimes_{\roi_T} L\eta_{\cal I} D\to L\eta_{\cal I}(C\dotimes_{\roi_T} D)\ ,
\]
functorial in $C$ and $D$, and symmetric in $C$ and $D$.
\end{proposition}

\begin{proof} Let $C^\blob$, $D^\blob$ be strongly $K$-flat representatives of $C$ and $D$. Then one has a natural map
\[
\mathrm{Tot}((\eta_{\cal I} C)^\blob\otimes_{\roi_T} (\eta_{\cal I} D)^\blob)\to \eta_{\cal I} \mathrm{Tot}(C^\blob\otimes_{\roi_T} D^\blob)\ ,
\]
given termwise by the map
\[
(\eta_{\cal I} C)^i\otimes_{\roi_T} (\eta_{\cal I}D)^j\to \eta_{\cal I} \mathrm{Tot}(C^\blob\otimes_{\roi_T} D^\blob)^{i+j}
\]
compatible with
\[
(C^i\otimes_{\roi_T} \cal I^{\otimes i})\otimes_{\roi_T} (D^j\otimes_{\roi_T} \cal I^{\otimes j})\to (C^i\otimes_{\roi_T} D^j)\otimes_{\roi_T} \cal I^{\otimes (i+j)}\ ,
\]
observing that if $x\in C^i$ and $y\in D^j$ have the property $dx\in \cal I\cdot C^{i+1}$ and $dy\in \cal I\cdot D^{j+1}$, then
\[
d(x\otimes y) = dx\otimes y + (-1)^i x\otimes dy\in \cal I\cdot (C^{i+1}\otimes_{\roi_T} D^j\oplus C^i\otimes_{\roi_T} D^{j+1})\ .
\]

This map gives the structure of a lax symmetric monoidal functor $\eta_{\cal I}: K_{\sub{strongly }K\sub{-flat}}(\roi_T)\to D(\roi_T)$, which factors uniquely over a lax symmetric monoidal functor $L\eta_{\cal I}: D(\roi_T)\to D(\roi_T)$.
\end{proof}

In an important special case, this operation even commutes with the $\otimes$-product:

\begin{proposition}\label{prop:Letasymmmon} Assume that $T$ is the punctual topos, and $R=\roi_T$ is a valuation ring. Let $f\in R$ be any generator of $\cal I$. Then $L\eta_f$ is symmetric monoidal.
\end{proposition}

\begin{proof} As everything commutes with filtered colimits, it is enough to check that if $C$ and $D$ are perfect complexes, then the natural map
\[
L\eta_f C\dotimes_R L\eta_f D\to L\eta_f(C\dotimes_R D)
\]
is a quasi-isomorphism. Note that $R$ is coherent, so that all cohomology groups of $C$ and $D$ are finitely presented. Moreover, finitely presented modules over valuation rings are finite direct sums of modules of the form $R/g$ for elements $g\in R$, by the elementary divisor theorem. These are of projective dimension $1$, so that both $C$ and $D$ split as a direct sum $\bigoplus_i H^i(C)[-i]$, $\bigoplus_i H^i(D)[-i]$. Thus, we can reduce to the case $C=(R/g)[i]$, $D=(R/h)[j]$ for some elements $g,h\in R$, $i,j\in \bb Z$. We may assume $i=j=0$ as all operations commute with shifts (see Remark \ref{rmk:LetaPrincipal}). If either $g$ or $h$ divides $f$, then we claim that both sides are trivial. Indeed, assume without loss of generality that $g$ divides $f$. Then $L\eta_f C=0$, and all cohomology groups of $C\dotimes_R D$ are killed by $g$, and thus by $f$, so that $L\eta_f(C\dotimes_R D)=0$ as well. Finally, if $f$ divides both $g$ and $h$, then $L\eta_f C=R/(g/f)$, $L\eta_f D=R/(h/f)$, and one verifies that
\[
L\eta_f(R/g\dotimes_R R/h) = R/(g/f)\dotimes_R R/(h/f)\ ,
\]
cf.~Lemma~\ref{lemma_on_Koszul_1} below for a more general statement.
\end{proof}

The next lemma bounds how far $L\eta_{\cal I}$ is from the identity.

\begin{lemma}\label{lem:mapLetabackforth} For any integer $m$, is a natural transformation
\[
\cal I^{\otimes m}\otimes_{\roi_T} \tau^{\leq m}\to \tau^{\leq m} L\eta_{\cal I}
\]
of functors on $D(\roi_T)$. For any integer $n$, there is a natural transformation
\[
\tau^{\geq n} L\eta_{\cal I}\to \cal I^{\otimes n}\otimes_{\roi_T} \tau^{\geq n}
\]
of functors on the full subcategory of those $C \in D(\roi_T)$ with $H^n(C)$ being $\cal I$-torsion-free. On this subcategory, if $n\leq m$, then the composites
\[\begin{aligned}
\cal I^{\otimes (m-n)} \otimes_{\roi_T} \tau^{[n,m]} L\eta_{\cal I} \to \cal I^{\otimes m}\otimes_{\roi_T} \tau^{[n,m]}\to \tau^{[n,m]} L\eta_{\cal I}\ ,\\
\cal I^{\otimes m}\otimes_{\roi_T} \tau^{[n,m]} \to \tau^{[n,m]} L\eta_{\cal I}\to \cal I^{\otimes n}\otimes_{\roi_T} \tau^{[n,m]}
\end{aligned}\]
are the identity maps tensored with the inclusions $\cal I^{\otimes (m-n)}\hookrightarrow \roi_T$ resp.~$\cal I^{\otimes m}\to \cal I^{\otimes n}$.
\end{lemma}

\begin{proof} It suffices to construct similar natural transformations on the category of $\cal I$-torsion-free complexes, so let $C^\blob$ be an $\cal I$-torsion-free complex. For the first transformation, it suffices to construct a map
\[
\cal I^{\otimes m}\otimes_{\roi_T} \tau^{\leq m} C^\blob\to \eta_{\cal I} C^\blob\ .
\]
But for $i<m$, $(\eta_{\cal I} C)^i$ contains $\cal I^{\otimes m}\otimes C^i$ (where we regard $\cal I^{\otimes m}$ as embedded into $\cal I^{\otimes i}$ by regarding both as ideal sheaves), and if $i=m$, it still contains $\cal I^{\otimes m}\otimes Z^m$, where $Z^m\subset C^m$ denotes the cocycles.

For the second transformation, let $C^\bullet$ be an $\cal I$-torsion-free complex with $H^n(C^\bullet)[\cal I] = 0$. We will show that there is a canonical map
\[
\eta_{\cal I} C^\blob \to \cal I^{\otimes n}\otimes_{\roi_T} \tau^{\geq n} C^\blob.
\]
For this, note that $(\eta_{\cal I} C)^i$ is contained in $\cal I^{\otimes n}\otimes_{\roi_T} C^i$ for $i\geq n$, when both sides are viewed as subsheaves of $C^i[\frac{1}{\cal I}]$ in the usual way; this defines the preceding map in degrees $> n$. To get the map in degree $\leq n$ by the same recipe, it is enough to show that the sheaf $\cal I^{\otimes n} \otimes_{\roi_T} C^{n-1}$ contains (and is thus equal to) the sheaf $(\eta_{\cal I} C)^{n-1}$, as subsheaves in $C^{n-1}[\frac{1}{\cal I}]$. But this immediate for us: the quotient $(\eta_{\cal I} C)^{n-1} / \big(\cal I^{\otimes n} \otimes_{\roi_T} C^{n-1}\big)$ is easily identified with $\cal I^{\otimes n-1} \otimes H^n(C)[\cal I]$, which vanishes by hypothesis. This gives the desired natural transformation on the subcategory. 

The identification of the composites is immediate from the definition.
\end{proof}

The following special case will come up repeatedly in the sequel:

\begin{lemma}
\label{lem:LetaCoconnectiveMap}
Let $C \in D^{\geq 0}(\mathcal{O}_T)$ such that $H^0(C)[\cal I] = 0$. Then there is a canonical map $L\eta_{\cal I} C \to C$.
\end{lemma}
\begin{proof}
This map is obtained by applying the second natural transformation constructed in Lemma~\ref{lem:mapLetabackforth} for $n=0$ to $C$.
\end{proof}

Composing two such operations behaves as expected:

\begin{lemma}\label{lem:compositionLeta} Let $\cal I,\cal J\subset \roi_T$ be two invertible ideal sheaves, with product $\cal I\otimes_{\roi_T} \cal J\isoto \cal I\cdot \cal J\subset \roi_T$. There is a canonical equivalence of functors
\[
L\eta_{\cal I\cdot \cal J}\cong L\eta_{\cal I}\circ L\eta_{\cal J}: D(\roi_T)\to D(\roi_T)\ .
\]
\end{lemma}

\begin{proof} Consider the category of $\cal I\cdot \cal J$-torsion-free complexes; this category is preserved by both $\eta_{\cal I}$ and $\eta_{\cal J}$, and $\eta_{\cal I\cdot \cal J} = \eta_{\cal I}\circ \eta_{\cal J}$ on this category. Deriving gives the desired equivalence.
\end{proof}

A crucial property is the following observation.\footnote{It is this property of the $L\eta$-operation that had initially led us to rediscover it.}

\begin{proposition}\label{prop:LetaBock} If $C\in D(\roi_T)$, construct a complex $H^\bullet(C/\cal I)$ with terms
\[
H^i(C/\cal I) = H^i(C\dotimes_{\roi_T} \roi_T/\cal I)\otimes_{\roi_T} \cal I^{\otimes i}
\]
and with differential induced by the Bockstein-type boundary map corresponding to the short exact sequence
\[
0\to \cal I/\cal I^2\to \roi_T/\cal I^2\to \roi_T/\cal I\to 0\ .
\]
Then there is a natural quasi-isomorphism
\[
L\eta_{\cal I} C\dotimes_{\roi_T} \roi_T/\cal I\quis H^\bullet(C/\cal I)\ .
\]

More precisely, if $C^\blob$ is an $\cal I$-torsion-free representative of $C$, then there is a natural map of complexes
\[
\eta_{\cal I} C^\blob\otimes_{\roi_T} \roi_T/\cal I\to H^\bullet(C/\cal I)\ ,
\]
which is a quasi-isomorphism; moreover, the left side represents the derived tensor product.
\end{proposition}

Note that even when $C$ does not have a distinguished representative in $K(\roi_T)$, the proposition shows that $L\eta_{\cal I} C\dotimes_{\roi_T} \roi_T/\cal I$ does have a distinguished representative as a complex, namely $H^\bullet(C/\cal I)$. As we will see, this is related to the canonical representative (given by the de~Rham--Witt complex) of the complex computing crystalline cohomology.

\begin{proof} It is enough to prove the assertion about $C^\blob$. Note that $\eta_{\cal I} C^\blob$ is $\cal I$-torsion-free, and for $\cal I$-torsion-free complexes, the underived tensor product with $\roi_T/\cal I$ represents the derived product.

Note that there is a natural map
\[
(\eta_{\cal I} C)^n\to Z^n(C^\blob/\cal I)\otimes_{\roi_T} \cal I^{\otimes n}
\]
from the definition of $(\eta_{\cal I} C)^n$. One gets an induced map
\[
(\eta_{\cal I} C^\blob)/\cal I = \eta_{\cal I} C^\blob\otimes_{\roi_T} \roi_T/\cal I\to H^\bullet(C/\cal I)\ ,
\]
and one checks that this is compatible with the differentials.

Now we check that this map of complexes is a quasi-isomorphism; it suffices to check that one gets an isomorphism on $H^0$ (as the situation at $H^n$ is just a twist and shift). First, we check injectivity of
\[
H^0((\eta_{\cal I} C^\blob)/\cal I)\to H^0(H^\bullet(C/\cal I))\ .
\]
Let $\bar{\alpha}\in H^0((\eta_{\cal I} C^\blob)/\cal I)$. We can lift $\bar{\alpha}$ to an element
\[
\alpha\in (\eta_{\cal I} C)^0=\{\gamma\in C^0\mid d\gamma\in \cal I\cdot C^1\}\ ,
\]
with $d\alpha\in \cal I\cdot (\eta_{\cal I} C)^1$ (so that $\alpha$ is a cocycle modulo $\cal I$), and we have to show if $\bar{\alpha}$ maps to $0$ in $H^0(H^\bullet(C/\cal I))$, then there is some
\[
\beta\in (\eta_{\cal I} C)^{-1} = \{\gamma\in C^{-1}\mid d\gamma\in \cal I\cdot C^0\}\otimes_{\roi_T} \cal I^{\otimes -1}
\]
such that $d\beta-\alpha\in \cal I\cdot (\eta_{\cal I} C)^0$. The assumption that $\bar{\alpha}$ maps to $0$ in $H^0(H^\bullet(C/\cal I))$ means that there is some
\[
\bar{\beta}\in H^{-1}(C/\cal I)\otimes_{\roi_T} \cal I^{\otimes -1}
\]
which maps to $\bar{\alpha}$ under the Bockstein. We may lift $\bar{\beta}$ to an element $\beta\in C^{-1}\otimes_{\roi_T} \cal I^{\otimes -1}$. The property that it is a cocycle modulo $\cal I$ means that $d\beta\in C^0$, and the property that the Bockstein is $\bar{\alpha}$ means that $d\beta-\alpha\in \cal I\cdot C^0$. Thus, in fact, implies $d\beta-\alpha\in \cal I\cdot (\eta_f C)^0$. Indeed, twisting the defining equation of $(\eta_f C)^0$ by $\cal I$, we have
\[
\cal I\cdot (\eta_f C)^0 = \{\gamma\in \cal I\cdot C^0\mid d\gamma\in \cal I^{\otimes 2}\cdot C^1\}\ ,
\]
and $d(d\beta-\alpha) = -d\alpha\in \cal I\cdot (\eta_f C)^1\subset \cal I^{\otimes 2}\cdot C^1$.

It remains to check surjectivity of
\[
H^0((\eta_{\cal I} C^\blob)/\cal I)\to H^0(H^\bullet(C/\cal I))\ .
\]
Thus, take an element $\bar{\alpha}\in H^0(C/\cal I)$ which is killed by the Bockstein. This means that $\bar{\alpha}$ lifts to an element of $H^0(C/\cal I^2)$, and so we can lift $\bar{\alpha}$ to an element $\alpha\in C^0$ with $d\alpha\in \cal I^{\otimes 2}\cdot C^1$. But this implies that $\alpha\in (\eta_{\cal I} C)^0$ and satisfies
\[
d\alpha\in \cal I\cdot (\eta_f C)^1\ ,
\]
as it lies in $\cal I^{\otimes 2}\cdot C^1$ and is killed by $d$. Thus, $\alpha$ defines a cocycle of $(\eta_{\cal I} C^\blob)/\cal I$, giving an element of $H^0((\eta_{\cal I} C^\blob)/\cal I)$ mapping to $\bar{\alpha}$.
\end{proof}

We observe that $\eta_{\cal I}$ preserves $\cal I$-torsion-free differential graded algebras, and that this structure is compatible with the isomorphism from Proposition~\ref{prop:LetaBock}.

\begin{lemma}\label{lem:Letadga} Let $R^\bullet$ be a differential graded $\roi_T$-algebra with $\cal I$-torsion-free terms. Then $\eta_{\cal I} R^\bullet$ is naturally a differential graded algebra, with $\cal I$-torsion-free terms. Moreover, $H^\bullet(R^\bullet/\cal I)$ has a natural structure of differential graded algebra, where multiplication is given by the cup product. The quasi-isomorphism
\[
\eta_{\cal I} R^\bullet\otimes_{\roi_T} \roi_T/\cal I\to H^\bullet(R^\bullet/\cal I)
\]
is a morphism of differential graded algebras.
\end{lemma}

\begin{proof} Easy and left to the reader.
\end{proof}

Finally, we observe that the $L\eta$-operation commutes with pullback along a flat morphism of topoi. More precisely, let $f: (T^\prime,\roi_{T^\prime})\to (T,\roi_T)$ be a flat map of ringed topoi. Two important cases are the case where $f$ is a point of $(T,\roi_T)$, and the case where $T=T^\prime$, which amounts to a flat change of rings. Let $\cal I\subset \roi_T$ be an invertible ideal sheaf with pullback $\cal I^\prime = f^\ast \cal I\subset \roi_{T^\prime}$, which is still an invertible ideal sheaf.

\begin{lemma}\label{lem:LEtaChangeTopos} The diagram
\[\xymatrix{
D(\roi_T)\ar[r]^-{f^\ast}\ar[d]_{L\eta_{\cal I}} & D(\roi_{T^\prime})\ar[d]^{L\eta_{\cal I^\prime}}\\
D(\roi_T)\ar[r]^-{f^\ast}& D(\roi_{T^\prime})
}\]
commutes, i.e.~there is a natural quasi-isomorphism $L\eta_{\cal I^\prime} f^\ast C\cong f^\ast L\eta_{\cal I} C$ for all $C\in D(\roi_T)$.
\end{lemma}

\begin{proof} Represent $C$ by an $\cal I$-torsion-free complex $C^\blob$. Then $f^\ast C^\blob$ is $\cal I^\prime$-torsion-free as $f^\ast$ is exact, by flatness of $f$. One then verifies immediately that $\eta_{\cal I^\prime} f^\ast C^\blob\cong f^\ast \eta_{\cal I} C^\blob$.
\end{proof}

\subsection{Completions}
\label{ss:Completions}

In this section, we make a few remarks about completions, and their commutation with $L\eta$. The discussion works in a replete topos, \cite[Definition 3.1.1]{BhattScholze}, but the only relevant case for us is the case of the punctual topos, so the reader is invited to forget about all topoi. Throughout this section, we assume that $T$ is replete.

Assume that $\cal J\subset \roi_T$ is a locally finitely generated ideal, as in \cite[\S 3.4]{BhattScholze}. Recall that by \cite[Lemma 3.4.12]{BhattScholze}, a complex $K\in D(\roi_T)$ is {\em derived $\cal J$-complete} if $K\isoto \hat{K}$, where the completion $\hat{K}$ is given locally by
\[
\hat{K}|_U = R\projlim_n (K|_U\dotimes_{\bb Z[f_1,\ldots,f_r]} \bb Z[f_1,\ldots,f_r]/(f_1^n,\ldots,f_r^n))\ ,
\]
if $\cal J|_U$ is generated by $f_1,\ldots,f_r$.

Perhaps surprisingly, this condition on a complex can be checked on its cohomology groups.

\begin{lemma}[{\cite[Proposition 3.4.4, Lemma 3.4.14]{BhattScholze}}]\label{lem:derivedcompletecohom} A complex $K\in D(\roi_T)$ is derived $\cal J$-complete if and only if each $\roi_T$-module $H^i(K)$ is derived $\cal J$-complete.

The category of derived $\cal J$-complete $\roi_T$-modules is an abelian Serre subcategory of the category of all $\roi_T$-modules, i.e.~closed under kernels, cokernels, and extensions.
\end{lemma}

\begin{remark} We pause to remark that this statement is already interesting (and not very well-known) in the simplest case of the punctual topos, $\roi_T=\bb Z$ and $\cal J = (p)$. In this case, it says that a complex $K\in D(\bb Z)$ is derived $p$-complete, i.e.
\[
K\simeq R\projlim_n (K\dotimes_{\bb Z} \bb Z/p^n\bb Z)\ ,
\]
if and only if each $H^i(K)$ satisfies
\[
H^i(K)\simeq R\projlim_n (H^i(K)\dotimes_{\bb Z} \bb Z/p^n \bb Z)\ .
\]
Note that any complex $K$ whose terms are $p$-torsion-free and $p$-adically complete is derived $p$-complete. However, its cohomology groups may not be separated, as for example in the case of
\[
K = [\widehat{\bigoplus_{n\geq 1} \bb Z_p}\xTo{(1,p,p^2,\ldots)} \widehat{\bigoplus_{n\geq 1} \bb Z_p}]\ .
\]
Here, the differential is injective, but $H^1(K)$ is not $p$-adically separated: The element $(1,p,p^2,\ldots)$ projects to a nonzero element of $H^1(K)$, which is divisible by any power of $p$. Surprisingly, the $\bb Z_p$-module $M=H^1(K)$ still has some intrinsic property, namely it is derived $p$-complete.
\end{remark}

Recall that an $\roi_T$-module $M$ is {\em classically $\cal J$-complete} if the natural map
\[
M\to \projlim_k M/\cal J^k
\]
is an isomorphism.

\begin{lemma}[{\cite[Proposition 3.4.2]{BhattScholze}}]\label{lem:derivedvsclassicalcomplete} Let $M$ be an $\roi_T$-module. Then $M$ is classically $\cal J$-complete if and only if it is derived $\cal J$-complete and $\cal J$-adically separated, i.e.~$\bigcap \cal J^k M = 0$.
\end{lemma}

We will often use the following lemma, identifying the cohomology groups of a (derived) completed direct sum.

\begin{lemma}\label{lem:cohomcompleteddirectsum} Let $C_i\in D(\roi_T)$, $i\in I$, be derived $\cal J$-complete complexes, and assume that $\cal J$ is locally generated by one element.

Assume that for each $i\in I$, $H^0(C_i)$ is classically $\cal J$-complete, and $H^0(C_i)[\cal J^\infty]=H^0(C_i)[\cal J^n]$ for some $n\geq 0$ independent of $i$. Let $C$ be the derived $\cal J$-completion of $\bigoplus_{i\in I} C_i$. Then $H^0(C)$ is the classical $\cal J$-adic completion of $\bigoplus_{i\in I} H^0(C_i)$,
\[
H^0(C) = \projlim_k \bigoplus_{i\in I} H^0(C_i)/\cal J^k\ .
\]
\end{lemma}

\begin{proof} First, we observe that if $M_i$, $i\in I$, are derived $\cal J$-complete modules, then the derived $\cal J$-completion of $\bigoplus_{i\in I} M_i$ is again concentrated in degree $0$. This may be done locally, so let $f$ be a local generator of $\cal J$. Then the only possible obstruction comes from the term $\projlim_k \bigoplus_{i\in I} M_i[f^k]$ (where the transition maps are multiplication by $f$), which however embeds into
\[
\projlim_k \prod_{i\in I} M_i[f^k] = \prod_{i\in I} \projlim_k M_i[f^k] = 0\ ,
\]
as each $M_i$ is derived $f$-complete.

In particular, the spectral sequence computing the cohomology of $C$ in terms of the derived completions of the direct sums of the cohomology groups of the $C_i$ collapses, saying that $H^0(C)$ is the derived completion of $\bigoplus_{i\in I} H^0(C_i)$.

Using the assumption $H^0(C_i)[f^\infty] = H^0(C_i)[f^n]$, one sees that $\projlim_k^1 \bigoplus_{i\in I} H^0(C_i)[f^k] = 0$. Thus, the derived inverse limit of $\{\bigoplus_{i\in I} H^0(C_i)[f^k]\}_k$ vanishes, so that in fact $H^0(C)$ is the classical completion of $\bigoplus_{i\in I} H^0(C_i)$.
\end{proof}

Now we turn to relations between $L\eta$ and completions.

\begin{lemma}\label{lem:Letapreservecompleteness} Let $\cal I\subset \roi_T$ be an invertible ideal sheaf, and let $C\in D(\roi_T)$ be derived $\cal J$-complete. Then $L\eta_{\cal I} C$ is derived $\cal J$-complete.
\end{lemma}

\begin{proof} We have to see that
\[
H^i(L\eta_{\cal I} C) = H^i(C)/H^i(C)[\cal I]
\]
is derived $\cal J$-complete. But $H^i(C)$ is derived $\cal J$-complete by assumption, and hence so is $H^i(C)[\cal I]$ as the kernel of a map of derived $\cal J$-complete modules, and thus also $H^i(L\eta_{\cal I} C)$ as a cokernel.
\end{proof}

Note that the lemma does not say that $L\eta_{\cal I}$ commutes with $\cal J$-adic completions. This is, in fact, not true in general. However, it is true in the important case $\cal J = \cal I$.

\begin{lemma}\label{lem:Letacommutecompletion} Assume that $\cal I\subset \roi_T$ is an invertible ideal sheaf which is locally free of rank $1$. Let $C\in D(\roi_T)$ with derived $\cal I$-adic completion $\hat{C}$. Then the natural maps
\[
\widehat{L\eta_{\cal I} C}\to L\eta_{\cal I} \hat{C}\to R\projlim_n L\eta_{\cal I} (C\dotimes_{\roi_T} \roi_T/\cal I^n)
\]
are quasi-isomorphisms. Here, the first map exists because $L\eta_{\cal I} \hat{C}$ is $\cal I$-adically complete.
\end{lemma}

\begin{proof} We may work locally, and assume $\cal I$ is generated by a non-zero-divisor $f\in \roi_T$. Moreover, all three complexes are derived $f$-complete. Thus, to prove that the maps are quasi-isomorphisms, it suffices to check that they are quasi-isomorphisms after reduction modulo $f$. Now Proposition~\ref{prop:LetaBock} shows that the first map is a quasi-isomorphism, as $H^i(C/f) = H^i(\hat{C}/f)$, and the Bockstein stays the same.

Applying similar reasoning for the second map, it is enough to prove that
\[
H^i(C/f)\to \{H^i((C\dotimes_{\roi_T} \roi_T/f^n)/f)\}_n
\]
is a pro-isomorphism. But in fact for any complex $D$ of $\roi_T/f$-modules (like $D=C/f$), the map
\[
D\to \{D\dotimes_{\roi_T} \roi_T/f^n\}_n
\]
is a pro-quasi-isomorphism.
\end{proof}

\newpage

\section{Koszul complexes}\label{sec:koszul}

In this section, we collect various useful facts about Koszul complexes.

\begin{definition}\label{def:koszul} Let $M$ be an abelian group with commuting endomorphisms $f_i: M\to M$, $i=1,\ldots,d$. The Koszul complex
\[
K_M(f_1,\ldots,f_d)
\]
is defined as
\[
M\xTo{(f_1,\ldots,f_d)} \bigoplus_{1\leq i\leq d} M\to \bigoplus_{1\leq i_1<i_2\leq d} M\to \ldots\to \bigoplus_{1\leq i_1<\ldots<i_k\leq d} M\to \ldots\ ,
\]
where the differential from $M$ in spot $i_1<\ldots<i_k$ to $M$ in spot $j_1<\ldots<j_{k+1}$ is nonzero only if $\{i_1,\ldots,i_k\}\subset \{j_1,\ldots,j_{k+1}\}$, in which case it is given by $(-1)^{m-1} f_{j_m}$, where $m\in\{1,\ldots,k+1\}$ is the unique integer such that $j_m\not\in \{i_1,\ldots,i_k\}$.
\end{definition}

In other words,
\[
K_M(f_1,\ldots,f_d) = M\otimes_{\bb Z[f_1,\ldots,f_d]} \bigotimes_{i=1}^d (\bb Z[f_1,\ldots,f_d]\xTo{f_i} \bb Z[f_1,\ldots,f_d])\ ,
\]
where the tensor product is taken over $\bb Z[f_1,\ldots,f_d]$, and the complex sits in nonnegative cohomological degrees. Note that this presentation shows that $K_M(f_1,\ldots,f_d)$ is canonically independent of the order of the $f_i$, as the tensor product on cochain complexes is symmetric monoidal. Also, $K_M(f_1,\ldots,f_d)$ computes $M\dotimes_{\bb Z[f_1,\ldots,f_d]} \bb Z$ up to a shift by $|I|$. We give one example of this construction that will be quite useful in the sequel:

\begin{example}
Let $A$ be a commutative ring, and let $R = A[x_1,\ldots,x_d]$. For $i=1,\ldots,d$, let $f_i:R \to R$ be the $A$-linear endomorphism given by $\frac{\partial}{\partial x_i}$. Then $K_R(f_1,\ldots,f_d)$ is simply the de~Rham complex $\Omega^\bullet_{R/A}$.
\end{example}

\begin{lemma}\label{lem:koszulcohom} Let $\Gamma_\mathrm{disc} = \prod_{i=1}^d \bb Z$ be the free abelian group on generators $\gamma_1,\ldots,\gamma_d$, and $\Gamma= \prod_{i=1}^d \bb Z_p$ its $p$-adic completion.
\begin{enumerate}
\item Let $M$ be a $\Gamma_\mathrm{disc}$-module. The group cohomology $R\Gamma(\Gamma_\mathrm{disc},M)$ is computed by $K_M(\gamma_1-1,\ldots,\gamma_d-1)$.
\item Let $N$ be a continuous $\Gamma$-module that can be written as an inverse limit $N=\varprojlim_{k\geq 1} N_k$ of continuous discrete $\Gamma$-modules $N_k$ killed by $p^k$. Then the natural map
\[
R\Gamma_\cont(\Gamma,N)\to R\Gamma(\Gamma_\mathrm{disc},N)
\]
is a quasi-isomorphism, and thus $R\Gamma_\cont(\Gamma,N)$ is computed by $K_N(\gamma_1-1,\ldots,\gamma_d-1)$.
\end{enumerate}
\end{lemma}

\begin{proof} The first part is standard: One has a free resolution
\[
\bigotimes_i (\bb Z[\Gamma_\mathrm{disc}]\xTo{\gamma_i-1} \bb Z[\Gamma_\mathrm{disc}])\to \bb Z
\]
of $\bb Z$ as $\bb Z[\Gamma_\mathrm{disc}]$-module, and taking homomorphisms into $M$ gives a resolution of $M$ by acyclic $\Gamma_\mathrm{disc}$-modules, leading to the Koszul complex.

For the second, we may assume that $N\to N_k$ is surjective for any $k$ (by replacing $N_k$ by the image of $N\to N_k$). Then
\[
R\Gamma_\cont(\Gamma,N)\to R\projlim_k R\Gamma_\cont(\Gamma,N_k)
\]
is a quasi-isomorphism, as follows from the description by continuous cocycles. The similar result holds true for the cohomology of $\Gamma_\mathrm{disc}$ by part (i). Thus, we can assume that $N$ itself is a discrete $\Gamma$-module killed by a power of $p$. In that case, we have a similar free resolution
\[
\bigotimes_i (\bb Z_p[[\Gamma]]\xTo{\gamma_i-1} \bb Z_p[[\Gamma]])\to \bb Z_p\ ,
\]
which leads to the same result.
\end{proof}

We will often implicitly use the following remark to see that our constructions are independent of the choice of roots of unity.

\begin{remark} In part (ii), if one changes the basis $\gamma_i\in \Gamma$ into $c(i)\gamma_i$ for $c(i)\in \bb Z_p^\times$, the resulting Koszul complexes are canonically isomorphic. Indeed, let $J_i\subset \bb Z_p[[\Gamma]]$ be the ideal generated by $\gamma_i-1$; this is the kernel of $\bb Z_p[[\Gamma]]\to \bb Z_p[[\Gamma/\bb \gamma_i^{\bb Z_p}]]$, and so depends only on $\gamma_i$ up to scalar. Then one has the free resolution
\[
\bigotimes_i (J_i\to \bb Z_p[[\Gamma]])\to \bb Z_p\ ;
\]
mapping this into $M$ gives a resolution by acyclic $\Gamma$-modules, leading to a complex computing $R\Gamma_\cont(\Gamma,N)$. Once one fixes the generators $\gamma_i-1\in J_i$, this becomes identified with the Koszul complex above.
\end{remark}

Next, we analyze the multiplicative structure.

\begin{lemma}\label{lem:koszuldga} Let $R$ be a (not necessarily commutative) $A$-algebra, for some commutative ring $A$, and let $\Gamma_\mathrm{disc}=\prod_{i=1}^d \bb Z$ be a free abelian group acting on $R$ by $A$-algebra automorphisms. Then
\[
K_R(\gamma_1-1,\ldots,\gamma_d-1)
\]
has a natural structure as a differential graded algebra over $A$ such that the quasi-isomorphism
\[
K_R(\gamma_1-1,\ldots,\gamma_d-1)\simeq R\Gamma(\Gamma_\mathrm{disc},R)
\]
is a quasi-isomorphism of algebra objects in the derived category $D(A)$. In particular, on cohomology groups, it induces the cup product.
\end{lemma}

\begin{remark} Even if $R$ is commutative, the resulting differential graded algebra will not be commutative. However, if there is some element $f\in A$ such that the action of $\Gamma$ on $R/f$ is trivial, then $K_R(\gamma_1-1,\ldots,\gamma_d-1)/f$ is commutative.
\end{remark}

\begin{proof} We give a presentation of a differential graded algebra $K_R^\prime$ over $A$, and then check that as a complex of $A$-modules, it is given by $K_R(\gamma_1-1,\ldots,\gamma_d-1)$, and is quasi-isomorphic to $R\Gamma(\Gamma_{\mathrm{disc}},R)$ compatibly with the multiplication.

Consider the differential graded algebra $K_R^\prime$ over $A$ which is generated by $R$ in degree $0$ and an additional variable $x_i$ of cohomological degree $1$ for each $i=1,\ldots,d$, subject to the following relations.
\begin{enumerate}
\item Anticommutation: $x_ix_j=-x_jx_i$, $x_i^2=0$ for all $i,j\in \{1,\ldots,d\}$.
\item Commutation with $R$: For all $r\in R$ and $i=1,\ldots,d$,
\[
x_i r = \gamma_i(r) x_i\ .
\]
\item Differential: $dx_i=0$ for $i=1,\ldots,d$, and
\[
dr = \sum_{i=1}^d (\gamma_i(r)-r)x_i\ .
\]
\end{enumerate}

We observe that the Leibniz rule $d(rr^\prime) = r\cdot dr^\prime + dr \cdot r^\prime$ for $r,r^\prime\in R$ is automatically satisfied:
\[\begin{aligned}
r \cdot dr^\prime + dr \cdot r^\prime &= \sum_{i=1}^d\big(r(\gamma_i(r^\prime)-r^\prime)x_i + (\gamma_i(r)-r)x_i r^\prime\big)\\
&= \sum_{i=1}^d\big((r\gamma_i(r^\prime)-rr^\prime)x_i + (\gamma_i(r)\gamma_i(r^\prime) - r\gamma_i(r^\prime))x_i\big)\\
&= \sum_{i=1}^d(\gamma_i(rr^\prime)-rr^\prime)x_i\\
&= d(rr^\prime)\ .
\end{aligned}\]
This, in fact, essentially dictates the rule $x_i r = \gamma_i(r) x_i$ (which introduces noncommutativity even when $R$ is commutative).

It follows that in degree $k$, $K_R^\prime$ is a free $R$-module on the elements $x_{i_1}\wedge\ldots\wedge x_{i_k}$, $i_1<\ldots<i_k$. The corresponding identification of the terms of $K_R^\prime$ with $K_R(\gamma_1-1,\ldots,\gamma_d-1)$ is compatible with the differential. We leave it to the reader to check that it is compatible with the multiplication on group cohomology.
\end{proof}

Let us discuss an example.

\begin{example}[The $q$-de~Rham complex] 
\label{ex:qdR}
Let $A$ be a commutative ring with a unit $q\in A^\times$, and consider the $A$-algebra $R=A[T^{\pm 1}]$. This admits an action of $\Gamma_\mathrm{disc}=\gamma^{\bb Z}$, where $\gamma$ acts by $T\mapsto qT$. In that case $R\Gamma(\Gamma_\mathrm{disc},R)$ is computed by the complex
\[
C^\blob: R\xTo{\gamma-1} R = \left(A[T^{\pm 1}]\to A[T^{\pm 1}]x\right)\ : T^n\mapsto (q^n-1) T^n x\ .
\]
Here, we have used a formal symbol $x$ for the generator in degree $1$. In this case, the multiplication is given as follows. In degree $0$, the multiplication is the usual commutative multiplication of $A[T^{\pm 1}]$. It remains to describe the products $f(T)\cdot (g(T)x)$ and $(g(T)x)\cdot f(T)$, where $f(T),g(T)\in A[T^{\pm 1}]$. These are given by
\[
f(T)\cdot (g(T)x) = f(T)g(T) x\ ,\ (g(T)x)\cdot f(T) = g(T)f(qT)x\ .
\]
In other words, the only interesting thing happens when one commutes $x$ past the function $f(T)$, which amounts to replacing $f(T)$ by $f(qT)$.

We note that we can now also apply the operator $\eta_{q-1}$ to $C^\blob$. This leads to the complex
\[
\eta_{q-1} C^\blob: A[T^{\pm 1}]\to A[T^{\pm 1}]d\log_q T: T^n\mapsto [n]_q T^n d\log_q T\ .
\]
Here, we use the formal symbol $d\log_q T$ (=$(q-1)x$) for the generator in degree $1$, and $[n]_q=\frac{q^n-1}{q-1}\in A$ is the $q$-deformation of the integer $n$. We call this the $q$-de~Rham complex $q\op-\Omega^\blob_{A[T^{\pm 1}]/A}$. We stress that this complex depends critically on the choice of coordinates: there is no well-defined complex $q\op-\Omega^\blob_{R/A}$ for any smooth $A$-algebra $R$. In closed form, the differential in the $q$-de~Rham complex is given by
\[
f(T)\mapsto \frac{f(qT)-f(T)}{q-1} d\log_q T = \frac{f(qT)-f(T)}{qT-T} d_q T\ ,
\]
where we have formally set $d_q T= Td\log_q T$. Note that if one sets $q = 1$, this finite $q$-difference quotient becomes the derivative. Again, this is a differential graded algebra, and the interesting multiplication rule is
\[
d\log_q T\cdot f(T) = f(qT) \cdot d\log_q T\ .
\]

One can also define the $q$-de~Rham complex in several variables
\[
q\op-\Omega^\blob_{A[T_1^{\pm 1},\ldots,T_d^{\pm 1}]/A} = \bigotimes_{i=1}^d q\op-\Omega^\blob_{A[T_i^{\pm 1}]/A}\ ,
\]
where the tensor product is taken over $A$. This can be written as
\[
\eta_{q-1} K_{A[T_1^{\pm 1},\ldots,T_d^{\pm 1}]}(\gamma_1-1,\ldots,\gamma_d-1)\ ,
\]
where $\gamma_i$ acts by sending $T_i$ to $qT_i$, and $T_j$ to $T_j$ for $j\neq i$. In particular, this computes
\[
L\eta_{q-1} R\Gamma(\Gamma_\mathrm{disc},A[T_1^{\pm 1},\ldots,T_d^{\pm 1}])\ .
\]
The $q$-de~Rham complex is still a differential graded algebra. In degree $1$, it has elements $d\log_q T_i$ for $i=1,\ldots,d$, and we have the multiplication rule
\[
d\log_q T_i \cdot f(T_1,\ldots,T_d) = f(T_1,\ldots,qT_i,\ldots,T_d) \cdot d\log_q T_i\ .
\]
\end{example}

We briefly discuss (using some $\infty$-categorical language) why the $q$-de Rham complex does not admit the structure of a commutative differential graded algebra.

\begin{remark}
\label{rmk:qdRcdga}
Take $A = \mathbb{F}_2[q^{\pm 1}]$ in Example \ref{ex:qdR}, and $R = A[T^{\pm 1}]$. Set $E_2 := R\Gamma(\Gamma_{\mathrm{disc}}, R)$ and $E_1 = L\eta_{q-1} E_1$, viewed as objects in the derived $\infty$-category of $A$-modules. In this remark, we freely use the following: (a) $E_2$ admits an $E_\infty$-$A$-algebra structure as $R\Gamma(\Gamma_{\mathrm{disc}},-)$ is lax symmetric monoidal, (b) $E_1$ admits an $E_\infty$-$A$-algebra structure as $L\eta_{q-1}$ is lax symmetric monoidal, and (c) the map $E_1 \to E_2$ lifts to a map of $E_\infty$-$A$-algebras. Granting these, we claim that the $E_\infty$-$\mathbb{F}_2$-algebra $E_1$ cannot be modeled by a commutative differential graded algebra over $\mathbb{F}_2$. 

Recall that the cohomology groups $H^*(E)$ of an $E_\infty$-$\bb F_2$-algebra $E$ carry a functorial Steenrod operation $\mathrm{Sq}^0:H^*(E) \to H^*(E)$ which acts as the identity on $H^*(X, \mathbb{F}_2)$ for any space $X$, and vanishes on $H^i(D)$ for $i > 0$ when $D$ is a commutative differential graded algebra over $\mathbb{F}_2$. Now observe that $\mathrm{Sq}^0(x) = x$ for the element $x \in H^1(E_2)$ coming from $x \in C^1$ (with notation as in the previous example); this can be seen by using the canonical map $C^*(S^1, \mathbb{F}_2) \simeq R\Gamma(\Gamma_{\mathrm{disc}}, \mathbb{F}_2) \to R\Gamma(\Gamma_{\mathrm{disc}}, R) =: E_2$, which carries the generator in $H^1(S^1, \mathbb{F}_2)$ to $x \in H^1(E_2)$. Since $d\log_q T \in H^1(E_1)$ maps to $(q-1) x \in H^1(E_2)$ and $\mathrm{Sq}^0$ is $\phi$-linear on $E_\infty$-$A$-algebras, it follows that $\mathrm{Sq}^0(d\log_qT) \in H^1(E_1)$ maps to $(q-1)^2 x \in H^1(E_2)$. As the latter is non-zero, so is $\mathrm{Sq}^0(d\log_q T)$. In particular, $\mathrm{Sq}^0$ acts non-trivially on $H^1(E_1)$, so $E_1$ cannot be represented by a commutative differential graded algebra over $\mathbb{F}_2$. 
\end{remark}

Moreover, we need a lemma about the behaviour of $L\eta$ on Koszul complexes.

\begin{lemma}\label{lemma_on_Koszul_1}
Let $f$ be a non-zero-divisor of a ring $R$, let $M^\blob$ be a complex of $f$-torsion-free $R$-modules, and let $g_1,\dots,g_m\in R$ be non-zero-divisors, each of which is either divisible by $f$ or divides $f$.

If there is some $i$ such that $g_i$ divides $f$, then
\[
\eta_f(M^\blob\otimes_R K_R(g_1,\ldots,g_m))
\]
is acyclic.

On the other hand, if $f$ divides $g_i$ for all $i$, then there is an isomorphism of complexes
\[
\eta_f K_R(g_1,\dots,g_m)\cong K_R(g_1/f,\dots,g_m/f)\ ,
\]
and more generally an isomorphism of complexes
\[
\eta_f(M^\blob\otimes_R K_R(g_1,\dots,g_m))\cong \eta_f M^\blob\otimes_R K_R(g_1/f,\dots,g_m/f)\ .
\]
\end{lemma}

\begin{proof} Arguing inductively, we may assume that $i=1$, and let $g:=g_1$. Assume first that $g$ divides $f$. Note that on any complex of the form $M^\blob\otimes_R K_R(g)$, multiplication by $g$ is homotopic to $0$. As $g$ divides $f$ by assumption, it follows that multiplication by $f$ is homotopic to $0$, and in particular all $H^i(M^\blob\otimes_R K_R(g))$ are killed by $f$. This implies that $\eta_f(M^\blob\otimes_R K_R(g))$ is acyclic by Lemma~\ref{lem:CohomLeta}.

Now assume that $f$ divides $g$. We embed $K(g/f)=(R\xTo{g/f} R)$ into $K(g)=(R\xTo{g} R)$ by using multiplication by $f$ in degree $1$. The complex $M^\blob\otimes_R K(g)$ is given explicitly (in degree $n$ and $n+1$) by
\begin{align*}
\cdots\To M^n\oplus M^{n-1}&\To M^{n+1}\oplus M^n\To\cdots\\
(x,y)&\mapsto (dx,dy+(-1)^n g x)
\end{align*}

One can realize $\eta_f(M^\blob\otimes_R K(g))$ as the subcomplex of $(M^\blob\otimes_R K(g))[f^{-1}]$ which in degree $n$ consists of those elements $(x,y)\in f^n M^n\oplus f^n M^{n-1}$ with $(dx,dy + (-1)^n gy)\in f^{n+1} M^{n+1}\oplus f^{n+1} M^n$. Using the similar model for $\eta_f M^\blob$, this implies that $x\in (\eta_f M)^n$, and also $y\in f(\eta_f M)^{n-1}$, as $dy+(-1)^n gx\in f^{n+1} M^n$, where $gx\in gf^n M^n\subset f^{n+1} M^n$ since $f$ divides $g$. Conversely, if $x\in (\eta_f M)^n$ and $y\in f(\eta_f M)^{n-1}$, then $(dx,dy + (-1)^n gy)\in f^{n+1} M^{n+1}\oplus f^{n+1} M^n$, so that we have identified $\eta_f(M^\blob\otimes_R K(g))$ with the complex
\begin{align*}
\cdots\To (\eta_f M)^n\oplus f(\eta_f M)^{n-1}&\To (\eta_f M)^{n+1}\oplus f(\eta_f M)^n\To\cdots\\
(x,y)&\mapsto (dx,dy+(-1)^n g x)
\end{align*}

But this complex is precisely $\eta_f M^\blob\otimes_R K(g/f)$, under the fixed embedding $K(g/f)\to K(g)$.
\end{proof}

In some situations, one can compute the cohomology of Koszul complexes.

\begin{lemma}\label{lemma_on_Koszul_2}
Let $g$ be an element of a ring $R$. 
\begin{enumerate}
\item Let $M^\blob$ be a complex of $R$-modules. If multiplication by $g$ on $M^\blob$ is homotopic to $0$, then the long exact cohomology sequence
\[
\cdots H^{n-1}(M^\blob)\xto{g}H^{n-1}(M^\blob)\to H^n(M\otimes_R K_R(g))\to H^n(M^\blob)\xto{g} H^n(M^\blob)\to\cdots
\]
for $M^\blob\otimes_R K_R(g)$ breaks into short exact sequences, 
\[
0 \to H^{n-1}(M^\blob)\to H^n(M^\blob\otimes_R K_R(g))\to H^n(M^\blob) \to 0,
\]
which are moreover split.
\item Let $M$ be an $R$-module. If $g_1,\dots,g_m\in R$ are all divisible by $g$, and $g_i$ is $g$ times a unit for some $i$, then there is an isomorphism of $R$-modules
\[
H^n(K_M(g_1,\dots,g_m))\cong \Ann_M(g)^{\binom{m-1}n}\oplus M/gM^{\binom{m-1}{n-1}}\ .
\]
\end{enumerate}
\end{lemma}

\begin{proof}
(i): Given a cocycle $x\in M^n$, the assumption implies that $gx=dx'$ for some $x'\in M^{n-1}$ depending on $x$ via a homomorphism (given by the homotopy); the association $x\mapsto (x,x')\in M^n\oplus M^{n-1}$ induces a well-defined homomorphism $H^n(M^\blob)\to H^n(M^\blob\otimes_R K_R(g))$ which splits the canonical map.

(ii): Without loss of generality, we may assume $g_1 = g$. Then this follows by induction from (i) applied to $K_M(g_1,\ldots,g_{i-1})$ as $g_i$ is homotopic to $0$ on $K_M(g_1,\ldots,g_{i-1})$ for each $i$.
\end{proof}

\newpage

\section{The complex $\widetilde{\Omega}_{\frak X}$}

Fix a perfectoid field $K$ of characteristic $0$ that admits a system of primitive $p$-power roots $\zeta_{p^r}$, $r\geq 1$, which we will fix for convenience, although our constructions are independent of this choice. Let $\cal O = \cal O_K = K^\circ$ be the ring of integers, which is endowed with the $p$-adic topology.

Now let $\mathfrak X/\cal O$ be a smooth $p$-adic formal scheme, i.e.~$\mathfrak X$ is locally of the form $\Spf R$, where $R$ is a $p$-adically complete flat $\cal O$-algebra such that $R/p$ is a smooth $\cal O/p$-algebra; equivalently, by a theorem of Elkik, \cite{Elkik}, $R$ is the $p$-adic completion of a smooth $\cal O$-algebra. We will simply call such $R$ formally smooth $\roi$-algebras below. Let $X$ be the generic fibre of $\mathfrak X$, which is a smooth adic space over $K$. We have the projection
\[
\nu: X_\sub{pro\'et}\to \mathfrak X_\sub{Zar}\ .
\]
In everything we do, we may as well replace $\nu$ by the projection $X_\sub{pro\'et}\to \mathfrak X_\sub{\'et}$, but the Zariski topology is enough.

\begin{definition} The complex $\widetilde{\Omega}_{\frak X}\in D(\mathfrak X_\sub{Zar})$ is given by
\[
\widetilde{\Omega}_{\frak X} = L\eta_{\zeta_p-1} (R\nu_\ast \hat\roi^+_X)\ ,
\]
where $\hat\roi^+_X$ is defined in Definition \ref{def:PeriodSheaves}.
\end{definition}

As the ideal $(\zeta_p-1)$ is independent of the choice of $\zeta_p$, so is $\widetilde{\Omega}_{\frak X}$. In this paper, we consider $\widetilde{\Omega}_{\frak X}$ merely as an object of the derived category (and not an $\infty$-categorical enhancement). Then $\widetilde{\Omega}_{\frak X}$ is naturally a commutative $\roi_{\frak X}$-algebra object in $D(\mathfrak X_\sub{Zar})$, as both $R\nu_\ast$ and $L\eta_{\zeta_p-1}$ are lax symmetric monoidal.

Our goal is to identify the cohomology groups of this complex with differential forms on $\frak X$;  this identification involves a Tate twist (or, rather, a Breuil--Kisin--Fargues twist), so we define that first, cf.~Example~\ref{ex:bkftwist}.
 
\begin{definition} Set
\[
\roi\{1\}:= T_p(\Omega^1_{\cal O/\bb Z_p})=\widehat{\bb L}_{\cal O/\bb Z_p}[-1] = \bb L_{\roi/A_\inf}[-1] = \tilde\xi A_\inf/\tilde\xi^2 A_\inf\ ;
\]
this is a free $\cal O=A_\inf/\tilde\xi$-module of rank $1$ that canonically contains the Tate twist $\cal O(1)$ as a free submodule with cokernel annihilated exactly by $(\zeta_p-1)=(p^{1/(p-1)})$.
\end{definition}

Explicitly, if we regard the $\zeta_{p^r}$ as fixed, one gets a trivialization $\roi\{1\}\cong \roi$ with generator given by the system of
\[
\left(\frac 1{\zeta_p-1} \frac{d(\zeta_{p^r})}{\zeta_{p^r}}\right)_r\in T_p(\Omega^1_{\cal O/\bb Z_p})\ .
\]
For any $\roi$-module $M$ and $n\in \bb Z$, we write $M\{n\} = M\otimes_\roi \roi\{n\}$. Our main result here is:

\begin{theorem}\label{thm:IntegralCartier}
There is a natural isomorphism
\[
H^i(\widetilde{\Omega}_{\frak X})\cong \Omega^{i,\cont}_{\frak X/\cal O}\{-i\}
\]
of sheaves on $\frak X_\sub{Zar}$. Here, $\Omega^{i,\cont}_{\frak X/\cal O} = \projlim \Omega^i_{(\frak X/p^n)/(\roi/p^n)}$ denotes the $\roi_{\frak X}$-module of continuous differentials.

In particular, $\widetilde{\Omega}_{\frak X}$ is a perfect complex of $\roi_{\frak X}$-modules.
\end{theorem}

Note that $R\nu_\ast \hat\roi^+_X$ is a complex that is only almost (in the technical sense) understood, using Faltings' almost purity theorem. It is thus surprising that in the theorem, we can identify the cohomology sheaves of $\widetilde{\Omega}_{\frak X} = L\eta_{\zeta_p-1} R\nu_\ast \hat\roi^+_X$ on the nose. This is possible as $L\eta_{\zeta_p-1}$ turns certain (but not all) almost quasi-isomorphisms into actual quasi-isomorphisms, cf.~Lemma~\ref{lem:LetaActualIsom} below.

The theorem can be regarded as a version of the Cartier isomorphism in mixed characteristic, except that $\widetilde{\Omega}_{\frak X}$ is not the de~Rham complex; however, we will later see that its reduction to the residue field $k$ of $\roi$ agrees with the de~Rham complex of $R\otimes_\roi k$.

\begin{remark} In Proposition~\ref{prop:OmegavsL}, we also prove that the complex $\tau^{\leq 1} \widetilde{\Omega}_{\frak X}$ is canonically identified with the $p$-adic completion of $\bb L_{\frak X/\bb Z_p}[-1]\{-1\}$. Now the $p$-adic completion of $\bb L_{\frak X/\bb Z_p}$ gives an extension of $\Omega^{1,\cont}_{\frak X/\cal O}$ by $\roi\{1\}[1]$; the corresponding $\Ext^2$-class measures the obstruction to lifting $\frak X$ to $A_\inf/\tilde\xi^2$. Thus, $\tau^{\leq 1} \widetilde{\Omega}_{\frak X}$ also measures the same obstruction; this gives an integral lift of the analogous Deligne-Illusie identification \cite[Theorem 3.5]{DeligneIllusie} over the special fibre. In particular, if $\frak X$ does not lift to $A_\inf/\tilde{\xi}^2$, then $\widetilde{\Omega}_{\frak X}$ does not split as a direct sum of its cohomology sheaves.
\end{remark}

The rest of the section is dedicated to proving Theorem~\ref{thm:IntegralCartier}. It will be useful to prove a stronger local result, which we will now formulate. The following definition is due to Faltings.

\begin{definition} Let $R$ be a formally smooth $\roi$-algebra. Then $R$ is called \emph{small} if there is an \'etale map
\[
\Spf R\to \widehat{\bb G}_m^d = \Spf \roi\langle T_1^{\pm 1},\ldots,T_d^{\pm 1}\rangle \ .
\]
\end{definition}

Let $\frak X = \Spf R$ with generic fiber $X=\Spa(R[\tfrac 1p],R)$. We will denote such ``framing'' maps
\[
\square: \frak X\to \widehat{\bb G}_m^d = \Spf \roi\langle T_1^{\pm 1},\ldots,T_d^{\pm 1}\rangle
\]
to the torus by the symbol $\square$. Given a framing, we let
\[
R_\infty = R\hat{\otimes}_{\roi\langle T_1^{\pm 1},\ldots,T_d^{\pm 1}\rangle} \roi\langle T_1^{\pm 1/p^\infty},\ldots,T_d^{\pm 1/p^\infty}\rangle\ ,
\]
which is a perfectoid ring, integrally closed in the perfectoid $K$-algebra $R_\infty[\tfrac 1p]$. In particular, the corresponding tower
\[
\projlimf_r \mathrm{Spa}(R\otimes_{\roi \langle T_1^{\pm 1},\ldots,T_d^{\pm 1}\rangle} K\langle T_1^{\pm 1/p^r},\ldots,T_d^{\pm 1/p^r}\rangle, R\otimes_{\roi \langle T_1^{\pm 1},\ldots,T_d^{\pm 1}\rangle} \roi\langle T_1^{\pm 1/p^r},\ldots,T_d^{\pm 1/p^r}\rangle)
\]
in $X_\sub{pro\'et}$ is affinoid perfectoid, with limit $\mathrm{Spa}(R_\infty[\tfrac 1p],R_\infty)$, and so Lemma~\ref{lem:descrperiodsheaves} applies. There is an action of $\Gamma = \bb Z_p(1)^d$ on $R_\infty$, where after a choice of roots of unity, a generator $\gamma_i$, $i=1,\ldots,d$, acts by $T_i^{1/p^m}\mapsto \zeta_{p^m} T_i^{1/p^m}$, $T_j^{1/p^m}\mapsto T_j^{1/p^m}$ for $j\neq i$.

On the other hand, assume for the moment that $\Spf R$ is connected. Then we can consider the completion $\widehat{\overline{R}}$ of the normalization $\overline{R}$ of $R$ in the maximal (pro-)finite \'etale extension of $R[\tfrac 1p]$, on which $\Delta = \mathrm{Gal}(\overline{R}[\tfrac 1p]/R[\tfrac 1p])$ acts. Again, $\widehat{\overline{R}}$ is perfectoid. Then $R_\infty\subset \widehat{\overline{R}}$ and $\Delta$ surjects onto $\Gamma$. By Faltings' almost purity theorem, the map
\[
R\Gamma_{\mathrm{cont}}(\Gamma,R_\infty)\to R\Gamma_{\mathrm{cont}}(\Delta,\widehat{\overline{R}})
\]
is an almost quasi-isomorphism, i.e.~all cohomology groups of the cone are killed by the maximal ideal $\frak m$ of $\roi$.

Using \cite[Proposition 3.5, Proposition 3.7 (iii), Corollary 6.6]{ScholzePAdicHodge}, one can identify the cohomology groups on the pro-finite \'etale site with continuous group cohomology groups, to see that
\[
R\Gamma_{\mathrm{cont}}(\Delta,\widehat{\overline{R}}) = R\Gamma(X_\sub{prof\'et},\hat\roi^+_X)\ .
\]
Note that the right side is well-defined even if $\mathrm{Spf} R$ is not connected.

In this situation, we can consider the following variants of $\widetilde{\Omega}_{\frak X}$.

\begin{definition} Let $R$ be a small formally smooth $\roi$-algebra as above, and let
\[
\square: \Spf R\to \widehat{\bb G}_m^d = \Spf \roi\langle T_1^{\pm 1},\ldots,T_d^{\pm 1}\rangle \ .
\]
be a framing. Define the following complexes:
\[\begin{aligned}
\widetilde{\Omega}_R^\square &= L\eta_{\zeta_p-1} R\Gamma_{\mathrm{cont}}(\Gamma,R_\infty) \\
\widetilde{\Omega}_R^\sub{prof\'et} &= L\eta_{\zeta_p-1} R\Gamma(X_\sub{prof\'et},\hat\roi^+_X) \\
\widetilde{\Omega}_R^\sub{pro\'et} &= L\eta_{\zeta_p-1} R\Gamma(X_\sub{pro\'et},\hat\roi^+_X)\ .
\end{aligned}\]
\end{definition}

Note that there are obvious maps
\[
\widetilde{\Omega}_R^\square\to \widetilde{\Omega}_R^\sub{prof\'et}\to \widetilde{\Omega}_R^\sub{pro\'et}\ .
\]
By the almost purity theorem, more precisely by \cite[Lemma 4.10 (v)]{ScholzePAdicHodge}, and the observation that $L\eta_{\zeta_p-1}$ takes almost quasi-isomorphisms to almost quasi-isomorphisms, they are almost quasi-isomorphisms. Finally, there is a map
\[
\widetilde{\Omega}_R^\sub{pro\'et}\to R\Gamma(\frak X,\widetilde{\Omega}_{\frak X})\ .
\]

\begin{theorem}\label{thm:IntegralCartierLocal} Let $R$ be a small formally smooth $\roi$-algebra. The maps
\[
\widetilde{\Omega}_R^\square\to \widetilde{\Omega}_R^\sub{prof\'et}\to \widetilde{\Omega}_R^\sub{pro\'et}\to R\Gamma(\frak X,\widetilde{\Omega}_{\frak X})
\]
are quasi-isomorphisms; write $\widetilde{\Omega}_R$ for their common value. Then there are natural isomorphisms
\[
H^i(\widetilde{\Omega}_R)\cong \Omega_{R/\roi}^{i,\cont}\{-i\}\ ,
\]
where $\Omega^{i,\cont}_{R/\roi}$ denotes the locally free $R$-module $\Omega^{i,\cont}_{R/\roi} = \projlim \Omega^i_{(R/p^n)/(\roi/p^n)}$ of continuous differentials.
\end{theorem}

\begin{proof}[Proof that Theorem~\ref{thm:IntegralCartierLocal} implies Theorem~\ref{thm:IntegralCartier}]
As any sufficiently small Zariski open of $\frak X$ is of the form $\Spf R$ for a small formally smooth $\roi$-algebra $R$, it suffices to check that the isomorphisms $H^i(\widetilde{\Omega}_R)\cong \Omega_{R/\roi}^{i,\cont}\{-i\}$ constructed in the proof of Theorem~\ref{thm:IntegralCartierLocal} are compatible with localization. As these isomorphisms are multiplicative (Corollary~\ref{cor:alltildeOmegasame} (ii)), we reduce to the case $i=1$. In this case, the isomorphism $\Omega^{1,\cont}_{R/\roi} \to H^1(\widetilde{\Omega}_R)$ is described in co-ordinate free terms in Proposition~\ref{prop:OmegavsL} and the following discussion.
\end{proof}

\begin{remark} Without the assumption that $R$ is small, one can still define $\widetilde{\Omega}_R^\sub{prof\'et}$ and $\widetilde{\Omega}_R^\sub{pro\'et}$. However, we do not know whether the maps
\[
\widetilde{\Omega}_R^\sub{prof\'et}\to \widetilde{\Omega}_R^\sub{pro\'et}\to R\Gamma(\frak X,\widetilde{\Omega}_{\frak X})
\]
are quasi-isomorphisms without the assumption that $R$ is small. (One can check that they are almost quasi-isomorphisms.)
\end{remark}

\subsection{The local computation}
\label{ss:localcomptildeOmega}

Let $R$ be a small formally smooth $\roi$-algebra with a fixed framing
\[
\square: \frak X=\Spf R\to \widehat{\bb G}_m^d\ .
\]
Let
\[
R_\infty = R\hat{\otimes}_{\roi\langle T_1^{\pm 1},\ldots,T_d^{\pm 1}\rangle} \roi\langle T_1^{\pm 1/p^\infty},\ldots,T_d^{\pm1/p^\infty}\rangle
\]
which has a $\Gamma=\bb Z_p(1)^d$-action as above. We start by recalling the computation of the cohomology groups of the complex
\[
R\Gamma_{\mathrm{cont}}(\Gamma,R_\infty)\ ,
\]
in a presentation which uses the choice of the framing $\square$ and a choice of roots of unity $\zeta_{p^r}$. Note that $\Omega^{1,\cont}_{R/\roi}$ is a free $R$-module with basis $d\log(T_1),\ldots,d\log(T_d)$, and thus
\[
\Omega^{i,\cont}_{R/\roi}\cong \bigwedge^i R^d\cong R^{\binom di}\ .
\]

\begin{proposition}\label{prop:compcontcohom} For all $i\geq 0$, the map
\[
\bigwedge^i R^d = H^i_{\mathrm{cont}}(\Gamma,R)\to H^i_{\mathrm{cont}}(\Gamma,R_\infty)
\]
is split injective, with cokernel killed by $\zeta_p - 1$. Moreover, $H^i_{\mathrm{cont}}(\Gamma,R_\infty)$ and $H^i_\cont(\Gamma,R_\infty)/(\zeta_p-1)$ have no almost zero elements.
\end{proposition}

Recall that an element $m$ in an $\roi$-module $M$ is called almost zero if it is killed by $\frak m$.

\begin{proof} Note that $R\to R_\infty$ admits a $\Gamma$-equivariant section, as $R_\infty$ is the $p$-adic completion of
\[
\bigoplus_{a_1,\ldots,a_d\in \bb Z[\tfrac 1p]\cap [0,1)} R\cdot T_1^{a_1}\dots T_d^{a_d}\ ;
\]
this shows that the induced map on cohomology is split injective. By~\cite[Lemma 5.5]{ScholzePAdicHodge}, the cokernel is killed by $\zeta_p-1$. In fact, the cokernel is given by
\[
R\otimes_\roi \bigoplus_{(0,\ldots,0)\neq (a_1,\ldots,a_d)\in (\bb Z[\tfrac 1p]\cap [0,1))^d} H^i_{\mathrm{cont}}(\Gamma,\roi\cdot T_1^{a_1}\dots T_d^{a_d})\ .
\]
To check whether $H^i_\cont(\Gamma,R_\infty)$ and $H^i_\cont(\Gamma,R_\infty)/(\zeta_p-1)$ have almost zero elements, it remains to check whether the displayed module has almost zero elements (as $\bigwedge^i R^d$ and $\bigwedge^i (R/(\zeta_p-1))^d$ have no almost zero elements, using Lemma~\ref{lem:topfree} below). As $R$ is topologically free over $\roi$, cf.~Lemma~\ref{lem:topfree} below, it is enough to see that the big direct sum has no almost zero elements, for which it is enough to see that each term in the direct sum has no almost zero elements. But each direct summand is a cohomology group of a perfect complex of $\roi$-modules, which (as $\roi$ is coherent) implies that all cohomology groups are finitely presented $\roi$-modules. Now, it remains to recall that finitely presented $\roi$-modules do not have almost zero elements, cf.~Corollary~\ref{cor:WrFinPresNoAlmostZero}.
\end{proof}

The following lemma was used in the proof.

\begin{lemma}\label{lem:topfree} Any formally smooth $\roi$-algebra $R$ is the $p$-adic completion of a free $\roi$-module.
\end{lemma}

\begin{proof} Let $k$ be the residue field of $\roi$, and fix a section $k\to \roi/p$. Then, as $R/p$ is a smooth $\roi/p$-algebra and in particular finitely presented, we see that for $r$ large enough, $R/(\zeta_{p^r}-1)$ is isomorphic to $R_k\otimes_k \roi/(\zeta_{p^r}-1)$, where $R_k = R\otimes_\roi k$ is the special fiber. Thus, as $R_k$ is a free $k$-module, $R/(\zeta_{p^r}-1)$ is a free $\roi/(\zeta_{p^r}-1)$-module. Picking any lift of the basis of $R/(\zeta_{p^r}-1)$ to $R$ gives a topological basis of $R$.
\end{proof}

To check that the maps
\[
\widetilde{\Omega}_R^\square\to \widetilde{\Omega}_R^\sub{prof\'et}\to \widetilde{\Omega}_R^\sub{pro\'et}
\]
are quasi-isomorphisms, we use the following lemma.

\begin{lemma}\label{lem:LetaActualIsom} Let $A$ be a ring with an ideal $I\subset A$. Let $f\in I$ be a non-zero-divisor.
\begin{enumerate}
\item Let $M$ be an $A$-module such that both $M$ and $M/f$ have no non-zero elements killed by $I$. Let $\alpha: M\to N$ be a map of $A$-modules such that the kernel and cokernel are killed by $I$. Then the induced map $\beta: M/M[f]\to N/N[f]$ is an isomorphism.
\item Let $g: C\to D$ be a map in $D(A)$ such that for all $i\in \bb Z$, the kernel and cokernel of the map $H^i(C)\to H^i(D)$ are killed by $I$. Assume moreover that for all $i\in \bb Z$, $H^i(C)$ and $H^i(C)/f$ have no non-zero elements killed by $I$. Then $L\eta_f g: L\eta_f C\to L\eta_f D$ is a quasi-isomorphism.
\end{enumerate}
\end{lemma}

\begin{remark} The lemma is wrong without some assumptions on $C$. For example, in the case $A=\roi$, $I=\frak m$, $f=\zeta_p-1$, the almost isomorphism $\frak m\to \roi$ does not become a quasi-isomorphism after applying $L\eta_{\zeta_p-1}$; here $\frak m/(\zeta_p-1)\frak m$ has almost zero elements. Similarly, $\roi/(\zeta_p-1)\frak m\to \roi/(\zeta_p-1)\roi$ does not become a quasi-isomorphism; here $\roi/(\zeta_p-1)\frak m$ has almost zero elements.

It is a bit surprising that, in (ii), it is enough to put assumptions on $C$, and none on $D$. 
\end{remark}

\begin{proof} Part (ii) follows from part (i) and Lemma~\ref{lem:CohomLeta}. For part (i), as the kernel of $\alpha$ is killed by $I$ but $M$ has no non-zero elements killed by $I$, $\alpha$ is injective. As $M/M[f]\cong fM$ via multiplication by $f$, this implies that $\beta: fM\to fN$ is injective. On the other hand, we have the inclusions $IfN\subset fM\subset fN\subset M$ as submodules of $N$. Thus, $fN/fM\hookrightarrow M/fM$ consists of elements killed by $I$, and thus vanishes by assumption. Thus, $fN=fM$, and $\beta$ is an isomorphism.
\end{proof}

The following corollary proves the first half of Theorem~\ref{thm:IntegralCartierLocal}; the natural identification of the cohomology groups with differentials will be proved as a consequence of Proposition~\ref{prop:OmegavsL} below.

\begin{corollary}\label{cor:alltildeOmegasame} Let $R$ be a small formally smooth $\roi$-algebra with framing $\square$.
\begin{enumerate}
\item The maps
\[
\widetilde{\Omega}_R^\square\to \widetilde{\Omega}_R^\sub{prof\'et}\to \widetilde{\Omega}_R^\sub{pro\'et}
\]
are quasi-isomorphisms.

\emph{From now on, we will write $\widetilde{\Omega}_R$ for any of $\widetilde{\Omega}_R^\square$, $\widetilde{\Omega}_R^\sub{prof\'et}$ and $\widetilde{\Omega}_R^\sub{pro\'et}$, using superscripts only when the distinction becomes important.}

\item For all $i\geq 0$,  there is an isomorphism (depending on our choice of framing)
\[
R^d\buildrel\simeq\over\to H^1(\widetilde{\Omega}_R)\ ,
\]
whose exterior powers induce isomorphisms
\[
\bigwedge^i R^d\buildrel\simeq\over\to H^i(\widetilde{\Omega}_R)\ .
\]

\item For any formally \'etale map $R\to R^\prime$ of small formally smooth $\roi$-algebras, the natural map
\[
\widetilde{\Omega}_R\dotimes_R R^\prime\to \widetilde{\Omega}_{R^\prime}
\]
is a quasi-isomorphism.

\item The map
\[
\widetilde{\Omega}_R\to R\Gamma(\frak X,\widetilde{\Omega}_{\frak X})
\]
is a quasi-isomorphism.
\end{enumerate}
\end{corollary}

\begin{proof} For part (i), let $C=R\Gamma_{\mathrm{cont}}(\Gamma,R_\infty)$, and let $D$ be either of
\[
R\Gamma(X_\sub{prof\'et},\hat\roi^+_X)\ ,\ R\Gamma(X_\sub{pro\'et},\hat\roi^+_X)\ .
\]
Then we have a map $g: C\to D$ which is an almost quasi-isomorphism, and $C$ satisfies the hypothesis of Lemma~\ref{lem:LetaActualIsom} (with $A=\roi$, $I=\frak m$, $f=\zeta_p-1$) by Proposition~\ref{prop:compcontcohom}. It follows that the map
\[
L\eta_{\zeta_p-1} g: L\eta_{\zeta_p-1} C\to L\eta_{\zeta_p-1} D
\]
is a quasi-isomorphism, as desired.

Part (ii) follows from Proposition~\ref{prop:compcontcohom} and the formula
\[
H^i(L\eta_{\zeta_p-1} C) = H^i(C)/H^i(C)[\zeta_p-1]\ ,
\]
which is compatible with cup products. Using this identification of the cohomology groups, part (iii) follows.

For part (iv), note that there is an induced map
\[
\widetilde{\Omega}_R\otimes_R \roi_{\frak X}\to \widetilde{\Omega}_{\frak X}\ ,
\]
and it is enough to show that this is a quasi-isomorphism in $D(\frak X_\sub{Zar})$, as the left side defines a coherent complex whose $R\Gamma$ is $\widetilde{\Omega}_R$. Note that for any affine open $\frak U=\Spf R^\prime\subset \Spf R$ with generic fibre $U$, by part (iii) the left side evaluated on $\frak U$ is given by
\[
L\eta_{\zeta_p-1} R\Gamma(U_\sub{pro\'et},\hat\roi_X^+)\ .
\]
To check whether the map
\[
\widetilde{\Omega}_R\otimes_R \roi_{\frak X}\to \widetilde{\Omega}_{\frak X}\ ,
\]
is a quasi-isomorphism, we can check on stalks at points, so let $x\in \frak X$ be any point. The stalk of the left side is
\[
\varinjlim_{\frak U\ni x} L\eta_{\zeta_p-1} R\Gamma(U_\sub{pro\'et},\hat\roi_X^+)\ ,
\]
and (using that $L\eta$ commutes with taking stalks by Lemma~\ref{lem:LEtaChangeTopos}), the stalk of the right side is
\[
L\eta_{\zeta_p-1} \varinjlim_{\frak U\ni x} R\Gamma(U_\sub{pro\'et},\hat\roi_X^+)\ .
\]
But $L\eta$ commutes with filtered colimits, so the result follows.
\end{proof}

Note that by functoriality of the pro-\'etale (or pro-finite \'etale) site, the assocation $R\mapsto \widetilde{\Omega}_R$ is functorial in $R$. We end this section by observing a K\"unneth formula for $\widetilde{\Omega}$.

\begin{proposition}\label{prop:kuenneth} Let $R_1$ and $R_2$ be two small formally smooth $\roi$-algebras, and let $R=R_1\hat{\otimes}_\roi R_2$. Then the natural map
\[
\widetilde{\Omega}_{R_1}\hat{\otimes}_\roi \widetilde{\Omega}_{R_2}\to \widetilde{\Omega}_R
\]
is a quasi-isomorphism.
\end{proposition}

\begin{proof} Choose framings $\square_1$ and $\square_2$ for $R_1$ and $R_2$, and endow $R$ with the product framing $\square = \square_1\times \square_2$. We may, using part (iii) of Corollary~\ref{cor:alltildeOmegasame}, reduce to the case $R_1=\roi\langle T_1^{\pm 1},\ldots,T_{d_1}^{\pm 1}\rangle$, $R_2=\roi\langle T_{d_1+1}^{\pm 1},\ldots,T_{d_1+d_2}^{\pm 1}\rangle$. In that case, one has
\[
R_\infty = \roi\langle T_1^{\pm 1/p^\infty},\ldots,T_{d_1+d_2}^{\pm1/p^\infty}\rangle = R_{1,\infty}\hat{\otimes}_\roi R_{2,\infty}\ .
\]
As the continuous group cohomology of $\Gamma = \bb Z_p(1)^{d_1+d_2} = \Gamma_1\times \Gamma_2$ is given by a Koszul complex, one deduces that
\[
R\Gamma_{\mathrm{cont}}(\Gamma,R_\infty) = R\Gamma_{\mathrm{cont}}(\Gamma_1,R_{1,\infty})\hat{\otimes}_\roi R\Gamma_{\mathrm{cont}}(\Gamma_2,R_{2,\infty})\ .
\]
It remains to see that $L\eta_{\zeta_p-1}$ behaves in a symmetric monoidal way in this case, i.e.~the induced natural map
\[
\widetilde{\Omega}_{R_1}^{\square_1}\hat{\otimes}_\roi \widetilde{\Omega}_{R_2}^{\square_2}\to \widetilde{\Omega}_R^\square
\]
is a quasi-isomorphism. This follows from Proposition~\ref{prop:Letasymmmon} and Lemma~\ref{lem:Letacommutecompletion} (noting that $p$-adic and $\zeta_p-1$-adic completion agree).
\end{proof}

\subsection{The identification of $\tau^{\leq 1} \widetilde{\Omega}_R$} 

As before, let $R$ be a small formally smooth $\roi$-algebra, with $\frak X=\Spf R$ and $X=\Spa(R[\tfrac 1p],R)$. In this subsection (and the next), we want to get a canonical identification of $\tau^{\leq 1} \widetilde{\Omega}_R$ with the $p$-adic completion of
\[
\bb L_{R/\bb Z_p}[-1]\{-1\}\ .
\]

First, we construct the map. Consider the transitivity triangle
\[
\widehat{\bb L}_{\cal O/\bb Z_p}[-1] \otimes_{\cal O} \hat\roi^+_X\to \widehat{\bb L}_{\hat\roi^+_X/\bb Z_p} [-1]\to \widehat{\bb L}_{\hat\roi^+_X/\cal O}[-1]
\]
of $p$-completed cotangent complexes on $X_\sub{pro\'et}$. Observe that $\widehat{\bb L}_{\hat\roi^+_X/\cal O} \simeq 0$ as in fact $\widehat{\bb L}_{S/\cal O}\simeq 0$ for any perfectoid $\cal O$-algebra $S$, see Lemma~\ref{lem:perfectoidcotangent}. We obtain a map
\[
\widehat{\bb L}_{R/\bb Z_p}[-1] \to R\Gamma(X_\sub{pro\'et},\widehat{\bb L}_{\hat\roi^+_X/\bb Z_p} [-1]) = R\Gamma(X_\sub{pro\'et},\widehat{\bb L}_{\cal O/\bb Z_p}[-1] \otimes_{\cal O} \hat\roi^+_X) \cong R\Gamma(X_\sub{pro\'et},\hat\roi^+_X)\{1\}\ .
\]

\begin{proposition}\label{prop:OmegavsL} The map
\[
\widehat{\bb L}_{R/\bb Z_p}[-1]\{-1\}\to R\Gamma(X_\sub{pro\'et},\hat\roi^+_X)
\]
just constructed factors uniquely over a map
\[
\widehat{\bb L}_{R/\bb Z_p}[-1]\{-1\}\to L\eta_{\zeta_p-1} R\Gamma(X_\sub{pro\'et},\hat\roi^+_X)=\widetilde{\Omega}_R\ ,
\]
and this induces an equivalence
\[
\widehat{\bb L}_{R/\bb Z_p}[-1]\{-1\}\buildrel\simeq\over\to \tau^{\leq 1} \widetilde{\Omega}_R\ .
\]
\end{proposition}

Note that from the transitivity triangle
\[
\widehat{\bb L}_{\cal O/\bb Z_p}[-1]\otimes_{\cal O} R\to \widehat{\bb L}_{R/\bb Z_p} [-1]\to \widehat{\bb L}_{R/\cal O}[-1]\ ,
\]
one sees that the cohomology groups of
\[
\widehat{\bb L}_{R/\bb Z_p}[-1]\{-1\}
\]
are given by $R$ in degree $0$ and $\Omega^{1,\cont}_{R/\roi}\{-1\}$ in degree $1$. Thus, the proposition gives a canonical identification
\[
\Omega^{1,\cont}_{R/\roi}\{-1\}\cong H^1(\widetilde{\Omega}_R)\ ,
\]
and combining this with Corollary~\ref{cor:alltildeOmegasame} finishes the proof of the canonical identification
\[
\Omega^{i,\cont}_{R/\roi}\{-i\}\cong H^i(\widetilde{\Omega}_R)\ ,
\]
thereby also finishing the proof of Theorem~\ref{thm:IntegralCartierLocal}, and thus of Theorem~\ref{thm:IntegralCartier}.

\begin{proof} First, we check that the factorization is unique. This is the content of the following lemma.

\begin{lemma}\label{lem:uniquefactoverLeta} Let $A$ be a ring with a non-zero-divisor $f$, and let $\alpha: C\to D$ be a map in $D(A)$ such that $H^i(C)=0$ for $i>1$, $H^i(D)=0$ for $i<0$, and $H^0(D)$ is $f$-torsion-free. Then there is at most one factorization of $\alpha$ as the composite of a map $\beta: C\to L\eta_f D$ and the natural map $L\eta_f D\to D$ from Lemma~\ref{lem:LetaCoconnectiveMap}, and it exists if and only if the map
\[
H^1(C\dotimes_A A/f)\to H^1(D\dotimes_A A/f)
\]
is zero, which happens if and only if the map $H^1(C)\to H^1(D)$ factors through $fH^1(D)$.
\end{lemma}

\begin{proof} First, we make the elementary verification that $H^1(C\dotimes_A A/f)\to H^1(D\dotimes_A A/f)$ is zero if and only if $H^1(C)\to H^1(D)$ factors through $fH^1(D)$. Note that $H^1(C)$ surjects onto $H^1(C\dotimes_A A/f)$, and $H^1(D)/f$ injects into $H^1(D\dotimes_A A/f)$. Thus, the claim follows from observing that in the diagram
\[\xymatrix{
H^1(C)\ar[r]\ar@{->>}[d] & H^1(D)/f\ar@{^(->}[d]\\
H^1(C\dotimes_A A/f)\ar[r] & H^1(D\dotimes_A A/f)\ ,
}\]
the lower arrow is zero if and only if the upper arrow is zero.

Note that $H^i(D)=0$ for $i<0$ and $H^0(D)$ is $f$-torsion free. By Lemma \ref{lem:LetaCoconnectiveMap}, there is a natural map $L\eta_f D\to D$. We may assume that $H^i(D)=0$ for $i>1$, as $\alpha$ factors through $\tau^{\leq 1} D$, and $L\eta_f$ commutes with truncations (so that any factorization $\beta: C\to L\eta_f D$ factors uniquely over $\tau^{\leq 1} L\eta_f D = L\eta_f(\tau^{\leq 1} D)$). For such $D$ with $H^i(D)=0$ for $i>1$ or $i<0$ and $H^0(D)$ being $f$-torsion-free, there is a distinguished triangle
\[
L\eta_f D\to D\to H^1(D/f)[-1]\ ,
\]
where the second map is the tautological map $D\to D/f\to \tau^{\geq 1} D/f = H^1(D/f)[-1]$.  Applying $\Hom(C,-)$ gives an exact sequence
\[\Hom(C, H^1(D/f)[-2]) \to \Hom(C, L\eta_f D) \to \Hom(C, D)  \to \Hom(C, H^1(D/f)[-1]).\]
Now $\Hom(C, H^1(D/f)[-2]) = 0$ since $C \in D^{\leq 1}(A)$. This shows that there is at most one factorization of $\alpha:C \to D$ through a map $\beta:C \to L\eta_f D$. Moreover, such a $\beta$ exists if and only if the composite $C \stackrel{\alpha}{\to} D \to H^1(D/f)[-1]$ vanishes. This composite is identified with the composite $C \to C/f \to H^1(C/f)[-1] \stackrel{H^1(\alpha/f)}{\to} H^1(D/f)[-1]$. Thus, such a $\beta$ exists if and only if $H^1(\alpha/f) = 0$; this gives everything but the last phrase of the lemma. For the last phrase, it is enough to observe that $H^1(C/f) = H^1(C)/f$ and $H^1(D/f) = H^1(D)/f$ since $C, D \in D^{\leq 1}(A)$.
\end{proof}

This applies in particular in our situation to imply that the factorization in the proposition is unique if it exists.

Now we do a local computation, so fix a framing $\square: \frak X\to \widehat{\bb G}_m^d$. Let $S=\roi\langle T_1^{\pm 1},\ldots,T_d^{\pm 1}\rangle$, so we have a formally \'etale map $S\to R$. Then by Corollary~\ref{cor:alltildeOmegasame}, we have a quasi-isomorphism
\[
\widetilde{\Omega}_S\otimes_S R\to \widetilde{\Omega}_R\ .
\]
Similarly, there is a quasi-isomorphism
\[
\widehat{\bb L}_{S/\bb Z_p}\otimes_S R\to \widehat{\bb L}_{R/\bb Z_p}\ ,
\]
by the transitivity triangle and the vanishing of $\widehat{\bb L}_{R/S}$. Thus, if we can prove the proposition for $S$, giving an equivalence
\[
\widehat{\bb L}_{S/\bb Z_p}[-1]\{-1\}\buildrel\simeq\over\to \tau^{\leq 1} \widetilde{\Omega}_S\ ,
\]
then the result for $R$ follows by base extension.

Thus, we may assume that $R=\roi\langle T_1^{\pm 1},\ldots,T_d^{\pm 1}\rangle$. Also, we may replace the map
\[
\widehat{\bb L}_{R/\bb Z_p}[-1]\{-1\}\to R\Gamma(X_\sub{pro\'et},\hat\roi^+_X)
\]
by the map
\[
\widehat{\bb L}_{R/\bb Z_p}[-1]\{-1\}\to R\Gamma_\sub{cont}(\bb Z_p(1)^d,\roi\langle T_1^{\pm 1/p^\infty},\ldots,T_d^{\pm 1/p^\infty}\rangle)
\]
constructed similarly, as the resulting $\widetilde{\Omega}_R$-complexes agree. By Lemma~\ref{lem:uniquefactoverLeta}, to check that the desired factorization exists and gives the desired quasi-isomorphism, we have to see that
\[
R = H^0(\widehat{\bb L}_{R/\bb Z_p}[-1]\{-1\})\to H^0_\sub{cont}(\bb Z_p(1)^d,\roi\langle T_1^{\pm 1/p^\infty},\ldots,T_d^{\pm 1/p^\infty}\rangle)
\]
is an isomorphism and
\[
\Omega^{1,\cont}_{R/\roi}\{-1\}\to H^1_\sub{cont}(\bb Z_p(1)^d,\roi\langle T_1^{\pm 1/p^\infty},\ldots,T_d^{\pm 1/p^\infty}\rangle)
\]
is an isomorphism onto $(\zeta_p-1) H^1_\sub{cont}$. The first statement follows directly from the definitions. For the second statement, we note that both $\Omega^{1,\cont}_{R/\roi}\{-1\}$ and $(\zeta_p-1)H^1_\sub{cont}$ are isomorphic to $R^d$ with bases on either side coming from the choice of coordinates (and the choice of roots of unity). It is enough to check that basis elements match, which by functoriality reduces to the case $d=1$. We finish the proof of Proposition~\ref{prop:OmegavsL} in the next subsection.
\end{proof}

\subsection{The key case $\frak X = \widehat{\bb G}_m$}

Assume now that $\frak X = \widehat{\bb G}_m = \Spf R$, where $R = \roi \langle T^{\pm 1} \rangle$. Set $R_\infty = \roi \langle T^{\pm 1/p^\infty} \rangle$, and let $\Gamma = \bb Z_p(1)$ be the natural group acting $R$-linearly on $R_\infty$.

We recall the map considered above. We start with the map
\[
\widehat{\bb L}_{R/\bb Z_p} \to \widehat{\bb L}_{R_\infty/\bb Z_p}
\]
induced by $p$-completion of the pullback. Since $R \to R_\infty$ is $\Gamma$-equivariant, this induces a map
\[
\widehat{\bb L}_{R/\bb Z_p} \to R\Gamma_\sub{cont}(\Gamma,\widehat{\bb L}_{R_\infty/\bb Z_p}) := R\varprojlim_n R\Gamma_{\sub{cont}}(\Gamma, \bb L_{R_\infty/\bb Z_p} \dotimes_{\bb Z} \bb Z/p^n) \ .
\]
We want to describe the image of $d\log(T) = \frac{dT}{T} \in H^0(\widehat{\bb L}_{R/\bb Z_p}) = \Omega^{1,\cont}_{R/\roi}$ under this map; note that this is an $R$-module generator of $\Omega^{1,\cont}_{R/\roi}$.

\begin{proposition}\label{prop:identificationmap} Under the identification
\[
R\Gamma_\sub{cont}(\Gamma,\widehat{\bb L}_{R_\infty/\bb Z_p}) = R\Gamma_\sub{cont}(\Gamma,R_\infty)[1]\{1\}\ ,
\]
the image of $d\log(T)\in H^0(\widehat{\bb L}_{R/\bb Z_p})$ in
\[
H^0_\sub{cont}(\Gamma,\widehat{\bb L}_{R_\infty/\bb Z_p}) = H^1_\sub{cont}(\Gamma,R_\infty)\{1\}
\]
is given by the image of
\[
d\log\otimes 1\in H^1_\sub{cont}(\Gamma,\roi\{1\})\otimes_\roi R = H^1_\sub{cont}(\Gamma,\roi\{1\}\otimes_\roi R)\hookrightarrow H^1_\sub{cont}(\Gamma,\roi\{1\}\otimes_\roi R_\infty)\ ,
\]
where $d\log\in H^1_\sub{cont}(\Gamma,\roi\{1\}) = \Hom_\sub{cont}(\Gamma,T_p(\Omega^1_{\roi/\bb Z_p}))$, $\Gamma=\bb Z_p(1) = T_p(\mu_{p^\infty}(\roi))$, is the map on $p$-adic Tate modules induced by the map $d\log: \mu_{p^\infty}(\roi)\to \Omega^1_{\roi/\bb Z_p}$.
\end{proposition}

Note that by Proposition~\ref{prop:compcontcohom}, the map
\[
H^1_\sub{cont}(\Gamma,\roi\{1\})\otimes_\roi R = H^1_\sub{cont}(\Gamma,\roi\{1\}\otimes_\roi R)\hookrightarrow H^1_\sub{cont}(\Gamma,\roi\{1\}\otimes_\roi R_\infty)
\]
induces an equality
\[
(\zeta_p-1)H^1_\sub{cont}(\Gamma,\roi\{1\})\otimes_\roi R = (\zeta_p-1) H^1_\sub{cont}(\Gamma,\roi\{1\}\otimes_\roi R_\infty)\ ,
\]
and the element $d\log\in H^1_\sub{cont}(\Gamma,\roi\{1\}) = \roi\{1\}(-1)$ is a generator of $(\zeta_p-1)\roi\{1\}(-1)$; thus, the proposition gives the remaining step of the proof of Proposition~\ref{prop:OmegavsL}.

\begin{proof} Since we work with $p$-complete objects, it is enough to describe what happens modulo $p^n$ for all $n$. In this case, we can compute $R\Gamma_\sub{cont}(\Gamma,\bb L_{R_\infty/\bb Z_p} \dotimes_{\bb Z} \bb Z/p^n)$ by the total complex of
\[
\left\{ \vcenter{\xymatrix{
\Omega^1_{R_\infty/\bb Z_p} \ar[r]^-{g-1} & \Omega^1_{R_\infty/\bb Z_p} \\
\Omega^1_{R_\infty/\bb Z_p} \ar[r]^-{g-1} \ar[u]^-{p^n} & \Omega^1_{R_\infty/\bb Z_p} \ar[u]^-{p^n}
} } \right\},
\]
where top left term is in bidegree $(0,0)$, and $g\in \Gamma$ is a generator, corresponding to a choice of $p$-power roots of unity $\zeta_{p^r}$, $r\geq 1$. Now $d\log(T)$ defines an element of the top left corner of this bicomplex, and we have
\[
d\log(T) = p^n \cdot d\log(T^{1/p^n})\ .
\]
Thus, in $H^0$ of the totalization of the above bicomplex, the element $d\log(T)$ coming from the top left corner is equivalent to
\[\begin{aligned}
(g-1)d\log(T^{1/p^n})  &= gd\log(T^{1/p^n}) - d\log(T^{1/p^n}) = d\log(\zeta_{p^n} T^{1/p^n}) - d\log(T^{1/p^n}) \\
&= d\log(\zeta_{p^n}) + d\log(T^{1/p^n}) - d\log(T^{1/p^n}) = d\log(\zeta_{p^n})\ ,
\end{aligned}\]
viewed as coming from the bottom right corner. The result follows.
\end{proof}

\newpage

\section{The complex $A\Omega_{\frak X}$}

Let $\frak X/\roi$ be a smooth formal scheme with generic fibre $X$ as in the previous section. In this section, we extend the complex $\widetilde{\Omega}_{\frak X}$ from $\roi$ to $A_\inf=W(\roi^\flat)$ along $\tilde\theta: A_\inf\to \roi$, i.e.~we construct a complex $A\Omega_{\frak X}\in D(\frak X_\sub{Zar})$ of $A_\inf$-modules such that
\[
A\Omega_{\frak X}\dotimes_{A_\inf,\tilde{\theta}} \roi \cong \widetilde{\Omega}_{\frak X}\ .
\]

\subsection{Statement of results}

The definition is very analogous to the definition of $\widetilde{\Omega}_{\frak X}$. Fix a system of primitive $p$-power roots of unity $\zeta_{p^r}$, $r\geq 1$, which give rise to an element $\epsilon=(1,\zeta_p,\ldots)\in \roi^\flat$, and let $\mu=[\epsilon]-1$. Note that the ideal $(\mu)$ is independent of the choice of roots of unity by Lemma~\ref{lem:mapAinfWitt}.

\begin{definition} The complex $A\Omega_{\frak X}\in D(\frak X_\sub{Zar})$ is given by
\[
A\Omega_{\frak X} = L\eta_\mu (R\nu_\ast \bb A_{\inf,X})\ .
\]
\end{definition}

Note that $A\Omega_{\frak X}$ admits a structure of commutative ring in $D(\frak X_\sub{Zar})$ by Proposition~\ref{prop:Letalaxsymmmon}, and is an algebra over (the constant sheaf) $A_\inf$.

\begin{theorem}\label{thm:AOmegavsdRW} The complex $A\Omega_{\frak X}$ has the following properties.
\begin{enumerate}
\item The natural map
\[
A\Omega_{\frak X}\dotimes_{A_\inf,\tilde\theta} \roi\to L\eta_{\zeta_p-1} (R\nu_\ast \hat\roi^+_X) = \widetilde{\Omega}_{\frak X}
\]
is a quasi-isomorphism.
\item More generally, for any $r\geq 1$, the natural map
\[
A\Omega_{\frak X}\dotimes_{A_\inf,\tilde\theta_r} W_r(\roi)\to L\eta_{[\zeta_{p^r}]-1} (R\nu_\ast W_r(\hat\roi^+_X)) =: \widetilde{W_r\Omega}_{\frak X}
\]
is a quasi-isomorphism.
\item For any $r\geq 1$ and $i\in \bb Z$, there is a natural isomorphism
\[
H^i(\widetilde{W_r\Omega}_{\frak X})\cong W_r\Omega^{i,\cont}_{\frak X/\roi}\{-i\}
\]
of sheaves on $\frak X_\sub{Zar}$, where $W_r\Omega^{i,\cont}_{\frak X/\roi} = \varprojlim W_r\Omega^i_{(\frak X/p^n)/(\roi/p^n)}$ is a continuous version of the de~Rham--Witt sheaf of Langer--Zink, \cite{LangerZink}, and $\{-i\}$ denotes a Breuil--Kisin--Fargues twist as in Example~\ref{ex:bkftwist}.
\end{enumerate}
\end{theorem}

Note that part (iii) extends the corresponding result for $\widetilde{\Omega}$ proved in the last section. As in the previous section, it will be important to formulate a stronger local statement.

\begin{definition} Let $R$ be a small formally smooth $\roi$-algebra, and let
\[
\square: \Spf R\to \widehat{\bb G}_m^d = \Spf \roi\langle T_1^{\pm 1},\ldots,T_d^{\pm 1}\rangle \ ,
\]
be a framing, giving rise to
\[
R_\infty = R\hat{\otimes}_{\roi\langle T_1^{\pm 1},\ldots,T_d^{\pm 1}\rangle} \roi\langle T_1^{\pm 1/p^\infty},\ldots,T_d^{\pm 1/p^\infty}\rangle\ ,
\]
on which the Galois group $\Gamma = \bb Z_p(1)^d$ acts. Define the following complexes:
\[\begin{aligned}
\widetilde{W_r\Omega}_R^\square &= L\eta_{[\zeta_{p^r}]-1} R\Gamma_{\mathrm{cont}}(\Gamma,W_r(R_\infty)) \\
\widetilde{W_r\Omega}_R^\sub{prof\'et} &= L\eta_{[\zeta_{p^r}]-1} R\Gamma(X_\sub{prof\'et},W_r(\hat\roi^+_X)) \\
\widetilde{W_r\Omega}_R^\sub{pro\'et} &= L\eta_{[\zeta_{p^r}]-1} R\Gamma(X_\sub{pro\'et},W_r(\hat\roi^+_X))\ ,
\end{aligned}\]
as well as
\[\begin{aligned}
A\Omega_R^\square &= L\eta_\mu R\Gamma_{\mathrm{cont}}(\Gamma,\bb A_\inf(R_\infty)) \\
A\Omega_R^\sub{prof\'et} &= L\eta_\mu R\Gamma(X_\sub{prof\'et},\bb A_{\inf,X}) \\
A\Omega_R^\sub{pro\'et} &= L\eta_\mu R\Gamma(X_\sub{pro\'et},\bb A_{\inf,X})\ .
\end{aligned}\]
\end{definition}

We will prove the following result, which implies Theorem~\ref{thm:AOmegavsdRW}.

\begin{theorem}\label{thm:AOmegavsdRWLocal} Let $R$ be a small formally smooth $\roi$-algebra with a framing $\square$, and let $\frak X=\Spf R$ with generic fibre $X$.
\begin{enumerate}
\item The natural maps
\[
A\Omega_R^\square\dotimes_{A_\inf,\tilde\theta_r} W_r(\roi)\to \widetilde{W_r\Omega}_R^\square
\]
are quasi-isomorphisms.
\item The natural maps
\[
\widetilde{W_r\Omega}_R^\square\to \widetilde{W_r\Omega}_R^\sub{prof\'et}\to \widetilde{W_r\Omega}_R^\sub{pro\'et}\to R\Gamma(\frak X,\widetilde{W_r\Omega}_{\frak X})
\]
are quasi-isomorphisms; we denote the common value by $\widetilde{W_r\Omega}_R$.
\item The natural maps
\[
A\Omega_R^\square\to A\Omega_R^\sub{prof\'et}\to A\Omega_R^\sub{pro\'et}\to R\Gamma(\frak X,A\Omega_{\frak X})
\]
are quasi-isomorphisms; we denote the common value by $A\Omega_R$.
\item For any $r\geq 1$ and $i\in \bb Z$, there is a natural isomorphism
\[
H^i(\widetilde{W_r\Omega}_R)\cong W_r\Omega_{R/\roi}^{i,\cont}\{-i\}\ ,
\]
where $W_r\Omega^{i,\cont}_{R/\roi} = \varprojlim W_r\Omega^i_{(R/p^n)/(\roi/p^n)}$.
\end{enumerate}
\end{theorem}

In this section, we will prove these theorems, except for part (iv) of Theorem~\ref{thm:AOmegavsdRWLocal} (and the corresponding part (iii) of Theorem~\ref{thm:AOmegavsdRW}), which will be proved in the next sections.

\subsection{Proofs}

Let $\roi=\roi_K$ be the ring of integers in a perfectoid field $K$ of characteristic $0$, containing all primitive $p$-power roots of unity $\zeta_{p^r}$, giving rise to the usual elements $\xi,\mu\in A_\inf = W(\roi^\flat)$. Let $R$ be a small formally smooth $\roi$-algebra, with framing
\[
\square: \frak X = \Spf R\to \widehat{\bb G}_m^d = \Spf \roi\langle T_1^{\pm 1},\ldots,T_d^{\pm 1}\rangle\ .
\]
As usual, let
\[
R_\infty = R\hat{\otimes}_{\roi\langle T_1^{\pm 1},\ldots,T_d^{\pm 1}\rangle} \roi\langle T_1^{\pm 1/p^\infty},\ldots,T_d^{\pm 1/p^\infty}\rangle\ ,
\]
on which $\Gamma = \bb Z_p(1)^d$ acts. We get the complexes
\[\begin{aligned}
A\Omega_R^\square &= L\eta_\mu R\Gamma_\sub{cont}(\Gamma,\bb A_\inf(R_\infty))\ ,\\
\widetilde{W_r\Omega}_R^\square &= L\eta_{[\zeta_{p^r}]-1} R\Gamma_\sub{cont}(\Gamma,W_r(R_\infty))\ ;
\end{aligned}\]
note that both of them have canonical representatives as actual differential graded algebras, by computing the continuous group cohomology as the standard Koszul complex (which gives a $\mu$-torsion-free, resp.~$[\zeta_{p^r}]-1$-torsion-free, resolution on which one can apply $\eta_\mu$, resp.~$\eta_{[\zeta_{p^r}]-1}$).

It turns out that $A\Omega_R^\square$ can be described (up to quasi-isomorphism) as a $q$-de~Rham complex, at least after fixing the system of $p$-power roots $\zeta_{p^r}$. Let us first define the relevant version of the $q$-de~Rham complex. Consider the surjection
\[
A_\inf\langle \ul U^{\pm 1}\rangle\to \roi\langle \ul T^{\pm 1}\rangle\ ,\ U_i\mapsto T_i\ ,
\]
which is given by $\theta: A_\inf\to \roi$ on $A_\inf$. As $\square$ is (formally) \'etale, one can lift $R$ uniquely to a $(p,\mu)$-adically complete $A_\inf$-algebra $A(R)^\square$ which is formally \'etale over $A_\inf\langle \ul U^{\pm 1}\rangle$. Moreover, there is an action of $\Gamma = \bb Z_p(1)^d$ on $A_\inf\langle \ul U^{\pm 1}\rangle$: if we fix the $p$-power roots of unity and let $\gamma_i\in \Gamma$ be the corresponding $i$-th basis vector, then it acts by sending $U_i$ to $[\epsilon] U_i$, and $U_j$ to $U_j$ for $j\neq i$. This action respects the quotient $\roi\langle \ul T^{\pm 1}\rangle$ and is trivial there. As $A_\inf\langle \ul U^{\pm 1}\rangle\to A(R)^\square$ is \'etale, this action lifts uniquely to an action of $\Gamma$ on $A(R)^\square$ which is trivial on the quotient $A(R)^\square\to R$. Actually, as the $\Gamma$-action becomes trivial on $A_\inf\langle \ul U^{\pm 1}\rangle/([\epsilon]-1)$, the $\Gamma$-action on $A(R)^\square$ is also trivial on $A(R)^\square/([\epsilon]-1)$.

In particular, for any $i=1,\ldots,d$, we can look at the operation
\[
\frac{\partial_q}{\partial_q \log(U_i)} = \frac{\gamma_i - 1}{[\epsilon]-1}: A(R)^\square\to A(R)^\square\ .
\]
If $R=\roi\langle \ul T^{\pm 1}\rangle$, then $A(R)^\square = A_\inf\langle \ul U^{\pm 1}\rangle$, and
\[
\frac{\partial_q}{\partial_q \log(U_i)}\left(\prod_j U_j^{n_j}\right) = [n_i]_q \prod_j U_j^{n_j}\ ,
\]
where as usual $q=[\epsilon]$. Using this, one verifies that the following definition gives in this case simply the $(p,\mu)$-adic completion of the $q$-de Rham complex $q\op-\Omega^\bullet_{A_\inf[\ul U^{\pm 1}]/A_\inf}$ from Example~\ref{ex:qdR}.

\begin{definition} The $q$-de~Rham complex of the framed small formally smooth $\roi$-algebra $R$ is given by
\[\begin{aligned}
q\op-&\Omega^\bullet_{A(R)^\square/A_\inf} = K_{A(R)^\square}\left(\frac{\partial_q}{\partial_q \log(U_1)},\ldots,\frac{\partial_q}{\partial_q \log(U_d)}\right)\\
&= A(R)^\square\xTo{(\frac{\partial_q}{\partial_q \log(U_i)})_i} (A(R)^\square)^d\to (A(R)^\square)^{\binom d2}\to\ldots \to (A(R)^\square)^{\binom dd}\ .
\end{aligned}\]
\end{definition}

To connect this to $A\Omega_R^\square$, we first observe that there is a canonical isomorphism
\[
A(R)^\square\hat{\otimes}_{A_\inf\langle \ul U^{\pm 1}\rangle} A_\inf\langle \ul U^{\pm 1/p^\infty}\rangle\isoto \bb A_\inf(R_\infty)\ ,\ U_i^{1/p^r}\mapsto [(T_i^{1/p^r},T_i^{1/p^{r+1}},\ldots)]\ ,
\]
equivariant for the $\Gamma$-action. Indeed, this is evident modulo $\xi$, and then follows by rigidity. Reducing along $\tilde\theta_r: A_\inf\to W_r(\roi)$, we get a quasi-isomorphism
\[
A(R)^\square/\tilde\xi_r\hat{\otimes}_{W_r(\roi)\langle \ul U^{\pm 1}\rangle} W_r(\roi)\langle \ul U^{\pm 1/p^\infty}\rangle\isoto W_r(R_\infty)\ ,\ U_i^{1/p^s}\mapsto [T_i^{1/p^{r+s}}]\ ,
\]
cf.~Lemma~\ref{lemma_theta} for the identification of the map. The following lemma proves part (i) of Theorem~\ref{thm:AOmegavsdRWLocal}.

\begin{lemma}\label{lem:qdRvsAOmega} There are injective quasi-isomorphisms
\[
q\op-\Omega^\bullet_{A(R)^\square/A_\inf}=\eta_{q-1} K_{A(R)^\square}(\gamma_1-1,\ldots,\gamma_d-1)\to A\Omega_R^\square = \eta_{q-1} K_{\bb A_\inf(R_\infty)}(\gamma_1-1,\ldots,\gamma_d-1)
\]
and
\[
q\op-\Omega^\bullet_{A(R)^\square/A_\inf}/\tilde\xi_r\to \widetilde{W_r\Omega}_R^\square\ .
\]
Moreover, the left side represents the derived reduction modulo $\tilde\xi_r$, and so the natural map
\[
A\Omega_R^\square\dotimes_{A_\inf,\tilde\theta_r} W_r(\roi)\to \widetilde{W_r\Omega}_R^\square
\]
is a quasi-isomorphism.
\end{lemma}

\begin{proof} We will prove only the first identification of $A\Omega_R^\square$ as a $q$-de~Rham complex; the identification of $\widetilde{W_r\Omega}_R^\square$ works exactly in the same way. For the final statement, note that the $q$-de~Rham complex has $\tilde\xi_r$-torsion-free terms.

We start from the identification
\[
A(R)^\square\hat{\otimes}_{A_\inf\langle \ul U^{\pm 1}\rangle} A_\inf\langle \ul U^{\pm 1/p^\infty}\rangle\isoto \bb A_\inf(R_\infty)\ .
\]
Using this, we get a $\Gamma$-equivariant decomposition
\[
\bb A_\inf(R_\infty) = A(R)^\square\oplus \bb A_\inf(R_\infty)^\sub{nonint}\ ,
\]
where $A(R)^\square$ is the ``integral'' part, and the second summand is the nonintegral part, given by the completed tensor product of $A(R)^\square$ with the $(p,\mu)$-adically complete $A_\inf\langle \ul U^{\pm 1}\rangle$-submodule of $A_\inf\langle \ul U^{\pm 1/p^\infty}\rangle$ generated by non-integral monomials. First, we observe that all cohomology groups
\[
H^i_\sub{cont}(\Gamma,\bb A_\inf(R_\infty)^\sub{nonint})
\]
are killed by $\phi^{-1}(\mu)=[\epsilon]^{1/p}-1$ (and thus by $\mu$), so that in particular
\[
L\eta_\mu R\Gamma_\sub{cont}(\Gamma,\bb A_\inf(R_\infty)^\sub{nonint})
\]
is $0$. In fact, we will check that multiplication by $\phi^{-1}(\mu)$ on $R\Gamma_\sub{cont}(\Gamma,\bb A_\inf(R_\infty)^\sub{nonint})$ is homotopic to $0$. By taking a decomposition according to the first non-integral exponent, we have a decomposition
\[
\bb A_\inf(R_\infty)^\sub{nonint} = \bigoplus_{i=1}^d \bb A_\inf(R_\infty)^\sub{nonint,i}\ .
\]
Now, to prove that multiplication by $\phi^{-1}(\mu)$ on
\[
R\Gamma_\sub{cont}(\Gamma,\bb A_\inf(R_\infty)^\sub{nonint,i}) = K_{\bb A_\inf(R_\infty)^\sub{nonint,i}}(\gamma_1-1,\ldots,\gamma_d-1)
\]
is homotopic to $0$, it suffices to show that multiplication by $\phi^{-1}(\mu)$ on
\[
\bb A_\inf(R_\infty)^\sub{nonint,i}\xTo{\gamma_i-1} \bb A_\inf(R_\infty)^\sub{nonint,i}
\]
is homotopic to $0$. Indeed, the whole Koszul complex is built from this complex by taking successive cones, to which this homotopy will lift. Thus, we have to find the dotted arrow in the diagram
\[\xymatrix{
\bb A_\inf(R_\infty)^\sub{nonint,i}\ar[r]^-{\gamma_i-1}\ar[d]_{\phi^{-1}(\mu)} & \bb A_\inf(R_\infty)^\sub{nonint,i}\ar[d]^{\phi^{-1}(\mu)}\ar@{..>}_h[dl]\\
\bb A_\inf(R_\infty)^\sub{nonint,i}\ar[r]^-{\gamma_i-1} & \bb A_\inf(R_\infty)^\sub{nonint,i}\ .
}\]
This decomposes into a completed direct sum of many pieces of the form
\[
\gamma_i-1: A(R)^\square\cdot T_i^{a(i)} \prod_{j\neq i} T_j^{a(j)}\to A(R)^\square\cdot T_i^{a(i)} \prod_{j\neq i} T_j^{a(j)}\ ,
\]
where $a(i)=m/p^r\in \bb Z[\tfrac 1p]$, $r\geq 1$, $m\in \bb Z\setminus p\bb Z$. This complex is the same as
\[
A(R)^\square\xTo{\gamma_i [\epsilon^{m/p^r}] - 1} A(R)^\square\ .
\]
Up to changing the roots of unity, we may assume that $m=1$. Moreover, the map $\gamma_i [\epsilon^{1/p^r}] - 1$ divides the map $\gamma_i^{p^{r-1}} [\epsilon]^{1/p} - 1$, so it is enough to produce a homotopy $h$ for $\gamma_i^{p^{r-1}}[\epsilon]^{1/p} - 1$. This amounts to finding a map $h: A(R)^\square\to A(R)^\square$ such that
\[
\gamma_i^{p^{r-1}}(h(a))[\epsilon]^{1/p}-h(a) = \phi^{-1}(\mu) a\ .
\]
As $\gamma_i^{p^{r-1}}\equiv \mathop{id}$ modulo $\mu$, we can write $\gamma_i^{p^{r-1}} = \mathop{id}+\mu \delta$ for some map $\delta: A(R)^\square\to A(R)^\square$. The equation becomes
\[
\mu \delta(h(a))[\epsilon]^{1/p} = \phi^{-1}(\mu)(a-h(a))\ ,
\]
or equivalently
\[
h(a) = a - \xi \delta(h(a))[\epsilon]^{1/p}\ .
\]
By successive $\xi$-adic approximation, it is clear that there is a unique solution to this. This handles the non-integral part of $A\Omega_R^\square$.

On the other hand, by the existence of the $q$-derivatives
\[
\frac{\partial_q}{\partial_q \log(U_i)} = \frac{\gamma_i - 1}{[\epsilon]-1}: A(R)^\square\to A(R)^\square\ ,
\]
the differentials in the complex calculating
\[
R\Gamma_\sub{cont}(\Gamma,A(R)^\square) = K_{A(R)^\square}(\gamma_1-1,\ldots,\gamma_d-1)
\]
are divisible by $\mu=[\epsilon]-1$, and one gets (by Lemma \ref{lemma_on_Koszul_1})
\[
\eta_\mu R\Gamma_\sub{cont}(\Gamma,A(R)^\square) = K_{A(R)^\square}\left(\frac{\gamma_1-1}{[\epsilon]-1},\ldots,\frac{\gamma_d-1}{[\epsilon]-1}\right) = q\op-\Omega^\bullet_{A(R)^\square/A_\inf}\ .
\]
\end{proof}

Next, we need some qualitative results on the complex $R\Gamma_\cont(\Gamma,W_r(R_\infty))$.

\begin{lemma}\label{lem:propWrOmega} Consider the Koszul complex
\[
C^\blob = K_{W_r(R_\infty)}(\gamma_1-1,\ldots,\gamma_d-1)
\]
computing $R\Gamma_\cont(\Gamma,W_r(R_\infty))$.
\begin{enumerate}
\item The complex $C^\blob$ can be written as a completed direct sum of Koszul complexes
\[
K_{W_r(\roi)}([\zeta^{a_1}]-1,\ldots,[\zeta^{a_d}]-1)
\]
for varying $a_1,\ldots,a_d\in \bb Z[\tfrac 1p]$. Here $\zeta^k\in \roi$ is short-hand for $\zeta_{p^a}^b$ if $k=\tfrac b{p^a}\in \bb Z[\tfrac 1p]$.
\item The cohomology groups
\[
H^i(\widetilde{W_r\Omega}_R^\square) = H^i(\eta_{[\zeta_{p^r}]-1} C^\blob)
\]
are $p$-torsion-free.
\item For any perfect complex $E\in D(W_r(\roi))$, the $W_r(\roi)$-modules
\[
H^i(C^\blob\dotimes_{W_r(\roi)} E)\ ,\ H^i(C^\blob\dotimes_{W_r(\roi)} E)/([\zeta_{p^r}]-1)
\]
have no almost zero elements, i.e.~no elements killed by $W_r(\frak m)$.
\end{enumerate}
\end{lemma}

\begin{proof} We begin with a rough computation of
\[
R\Gamma_\sub{cont}(\Gamma,W_r(\roi\langle \ul T^{\pm 1/p^\infty}\rangle))
\]
as a complex of $W_r(\roi)\langle \ul U^{\pm 1}\rangle$-modules, where $U_i\mapsto [T_i]$. Here, we normalize the action so that the $i$-th basis vector $\gamma_i\in\Gamma = \bb Z_p(1)^d$ acts by sending $[T_i^{1/p^s}]$ to $[\zeta_{p^s} T_i^{1/p^s}]$.

We can write
\[
R\Gamma_\sub{cont}(\Gamma,W_r(\roi\langle \ul T^{\pm 1/p^\infty}\rangle)) = \widehat{\bigoplus_{a_1,\ldots,a_d\in \bb Z[\frac 1p]\cap [0,1)}} R\Gamma_\sub{cont}(\Gamma,W_r(\roi)\langle \ul U^{\pm 1}\rangle\cdot \prod_{i=1}^d [T_i]^{a_i})\ .
\]
Moreover, each summand can be written as a Koszul complex
\[
R\Gamma_\sub{cont}(\Gamma,W_r(\roi)\langle \ul U^{\pm 1}\rangle\cdot \prod_{i=1}^d [T_i]^{a_i}) = K_{W_r(\roi)\langle \ul U^{\pm 1}\rangle}([\zeta^{a_1}]-1,\ldots,[\zeta^{a_d}]-1)\ .
\]

Next, we want to get a similar description of
\[
R\Gamma_\sub{cont}(\Gamma,W_r(R_\infty))\ .
\]
Recall that $\bb A_\inf(R_\infty) = A_\inf\langle \ul U^{\pm 1/p^\infty}\rangle\hat{\otimes}_{A_\inf\langle \ul U^{\pm 1}\rangle} A(R)^\square$, so that by base change along $\theta_r$, we get
\[
W_r(R_\infty) = W_r(\roi)\langle \ul U^{\pm 1/p^\infty}\rangle\hat{\otimes}_{W_r(\roi)\langle \ul U^{\pm 1}\rangle} A(R)^\square/\xi_r\ ;
\]
also, $W_r(\roi)\langle \ul U^{\pm 1/p^\infty}\rangle = W_r(\roi\langle \ul T^{\pm 1/p^\infty}\rangle)$ by passing to the $p$-adic completion in Lemma~\ref{lem:wittpolynomial} below. This implies
\[
R\Gamma_\sub{cont}(\Gamma,W_r(R_\infty)) = R\Gamma_\sub{cont}(\Gamma,W_r(\roi\langle \ul T^{\pm 1/p^\infty}\rangle))\hat{\otimes}_{W_r(\roi)\langle \ul U^{\pm 1}\rangle} A(R)^\square/\xi_r\ ;
\]
note that the tensor product is underived modulo any power of $p$ by \'etaleness. Therefore, we get a decomposition
\[
R\Gamma_\sub{cont}(\Gamma,W_r(R_\infty)) = \widehat{\bigoplus_{a_1,\ldots,a_d\in \bb Z[\frac 1p]\cap [0,1)}} K_{A(R)^\square/\xi_r}([\zeta^{a_1}]-1,\ldots,[\zeta^{a_d}]-1)\ .
\]
Finally, as in Lemma~\ref{lem:topfree}, $A(R)^\square/\xi_r$ is topologically free over $W_r(\roi)$, finishing the proof of (i).

For (ii), note that by Lemma~\ref{lemma_on_Koszul_1},
\[
\eta_{[\zeta_{p^r}]-1} K_{W_r(\roi)}([\zeta^{a_1}]-1,\ldots,[\zeta^{a_d}]-1)
\]
is acyclic if $p^ra_i\not\in \bb Z$ for some $i$, and otherwise it is given by
\[
K_{W_r(\roi)}\left(\frac{[\zeta^{a_1}]-1}{[\zeta_{p^r}]-1},\ldots,\frac{[\zeta^{a_d}]-1}{[\zeta_{p^r}]-1}\right)\ .
\]
The cohomology groups of this complex are $p$-torsion-free by Lemma~\ref{lemma_on_Koszul_2} and Corollary~\ref{corollary_roots_of_unity}. Also, $\eta_{[\zeta_{p^r}]-1}$ commutes with the completed direct sum by Lemma~\ref{lem:Letacommutecompletion}. Thus, we can apply Lemma~\ref{lem:cohomcompleteddirectsum} to compute the cohomology groups
\[
H^i(\widetilde{W_r\Omega}_R^\square) = H^i(\eta_{[\zeta_{p^r}]-1} C^\blob)
\]
as a classical $p$-adic completion of the direct sum of the $p$-torsion-free cohomology groups of the Koszul complexes above. In particular, they are $p$-torsion-free.

For (iii), assume first that $K=K_{W_r(\roi)}([\zeta^{a_1}]-1,\ldots,[\zeta^{a_d}]-1)$ is a Koszul complex. Then $K\dotimes_{W_r(\roi)} E$ is a perfect complex of $W_r(\roi)$-modules. Thus, as $W_r(\roi)$ is coherent, every cohomology group is finitely presented, and thus contains no almost zero elements by Corollary~\ref{cor:WrFinPresNoAlmostZero}; the same argument works for $H^i/([\zeta_{p^r}]-1)$.

Now we have a decomposition $C^\blob=C^\sub{int}\oplus C^\sub{nonint}$, where $C^\sub{int}$ is a completed direct sum of Koszul complexes
\[
K_{W_r(\roi)}([\zeta^{a_1}]-1,\ldots,[\zeta^{a_d}]-1)\ ,
\]
where the denominator of each $a_i$ is at most $p^r$, and $C^\sub{nonint}$ is a completed direct sum of Koszul complexes
\[
K_{W_r(\roi)}([\zeta^{a_1}]-1,\ldots,[\zeta^{a_d}]-1)
\]
where the denominator of some $a_i$ is at least $p^{r+1}$. Note that $C^\sub{nonint}$ is actually just (quasi-isomorphic to) the direct sum of these Koszul complexes, as multiplication by $[\zeta_{p^r}]-1$ is homotopic to $0$ on each of the Koszul complexes, and thus on their direct sum.

It suffices to prove the similar assertions for $H^i(C^\sub{int}\dotimes_{W_r(\roi)} E)$ and $H^i(C^\sub{nonint}\dotimes_{W_r(\roi)} E)$. Note that only finitely many different Koszul complexes appear in $C^\sub{int}$; by taking a corresponding isotypic decomposition, we can reduce to the case that $C^\sub{int}$ is the $p$-adic completion of a direct sum of copies of one Koszul complex
\[
K = K_{W_r(\roi)}([\zeta^{a_1}]-1,\ldots,[\zeta^{a_d}]-1)\ .
\]
In that case, $H^i(C^\sub{int}\dotimes_{W_r(\roi)} E)$ is the classical $p$-adic completion of a similar direct sum of copies of the finitely presented $W_r(\roi)$-module $H^i(K\dotimes_{W_r(\roi)} E)$ (by Lemma~\ref{lem:cohomcompleteddirectsum}, using that the $p$-torsion submodule of finitely presented $W_r(\roi)$-modules is of bounded exponent), for which we have already checked the assertion. Similarly in the second case, $H^i(C^\sub{nonint}\dotimes_{W_r(\roi)} E)$ decomposes as a (noncompleted) direct sum of the cohomology groups of $H^i(K\dotimes_{W_r(\roi)} E)$ for Koszul complexes $K$.
\end{proof}

We used the following lemma in the proof.

\begin{lemma}\label{lem:wittpolynomial} Let $S$ be any ring. There are natural inclusions
\[
W_r(S[T_1^{p^r},\ldots,T_d^{p^r}])\subset W_r(S)[U_1,\ldots,U_d]\subset W_r(S[T_1,\ldots,T_d])\ ,
\]
and
\[
W_r(S[T_1^{\pm p^r},\ldots,T_d^{\pm p^r}])\subset W_r(S)[U_1^{\pm 1},\ldots,U_d^{\pm 1}]\subset W_r(S[T_1^{\pm 1},\ldots,T_d^{\pm 1}])\ ,
\]
where $U_i=[T_i]$. In particular, by passing to a union over all $p$-power roots, we have equalities
\[\begin{aligned}
W_r(S[T_1^{1/p^\infty},\ldots,T_d^{1/p^\infty}]) &= W_r(S)[U_1^{1/p^\infty},\ldots,U_d^{1/p^\infty}]\ ,\\
W_r(S[T_1^{\pm 1/p^\infty},\ldots,T_d^{\pm 1/p^\infty}]) &= W_r(S)[U_1^{\pm 1/p^\infty},\ldots,U_d^{\pm 1/p^\infty}]\ .
\end{aligned}\]
\end{lemma}

\begin{proof} The Laurent polynomial case follows from the polynomial case by localization. The polynomial case follows for example from~\cite[Corollary 2.4]{LangerZink}.
\end{proof}

Moreover, we need the following base change property.

\begin{lemma}\label{lem:WrOmegaBC} Let $R$ be as above, and let $R\to R^\prime$ be a formally \'etale map, i.e.~$R/p^n\to R^\prime/p^n$ is \'etale for all $n$, and $R^\prime$ is $p$-adically complete. Let $\square^\prime$ be the induced framing of $\Spf R^\prime$. Then the natural map
\[
\widetilde{W_r\Omega}_R^\square\hat{\otimes}_{W_r(R)} W_r(R^\prime)\to \widetilde{W_r\Omega}_{R^\prime}^{\square^\prime}
\]
is a quasi-isomorphism.
\end{lemma}

\begin{remark} We note that modulo $p^n$, the tensor product is underived by Theorem~\ref{thm:Wittetale}. Indeed, by Elkik, \cite{Elkik}, we may always find a smooth $\roi$-algebra $R_0$ and an \'etale $R_0$-algebra $R_0^\prime$ such that $R\to R^\prime$ is the $p$-adic completion of $R_0\to R_0^\prime$. Then $W_r(R_0)\to W_r(R_0^\prime)$ is \'etale and hence so is $W_r(R_0)/p^n\to W_r(R_0^\prime)/p^n$, which agrees with $W_r(R)/p^n\to W_r(R^\prime)/p^n$.
\end{remark}

\begin{proof} Fix a map $R_0\to R_0^\prime$ as in the remark. By Theorem~\ref{thm:Wittetale},
\[
W_r(R_\infty^\prime) = W_r(R_\infty)\hat{\otimes}_{W_r(R)} W_r(R^\prime)\ ,
\]
where the tensor product is underived modulo $p^n$. Taking cohomology, we get
\[
R\Gamma_\cont(\Gamma,W_r(R_\infty^\prime)) = R\Gamma_\cont(\Gamma,W_r(R_\infty))\hat{\otimes}_{W_r(R)} W_r(R^\prime)\ .
\]
Moreover, using Lemma~\ref{lem:Letacommutecompletion} and the observation that $L\eta$ commutes with flat base change $W_r(R_0)\to W_r(R_0^\prime)$, cf.~Lemma~\ref{lem:LEtaChangeTopos}, we get
\[
L\eta_{[\zeta_{p^r}]-1} R\Gamma_\cont(\Gamma,W_r(R_\infty^\prime)) = L\eta_{[\zeta_{p^r}]-1} R\Gamma_\cont(\Gamma,W_r(R_\infty))\hat{\otimes}_{W_r(R)} W_r(R^\prime)\ ,
\]
as desired.
\end{proof}

We can now prove part (ii) of Theorem~\ref{thm:AOmegavsdRWLocal}.

\begin{corollary}\label{cor:AOmegavsdRWLocalPart2} The natural maps
\[
\widetilde{W_r\Omega}_R^\square\to \widetilde{W_r\Omega}_R^\sub{prof\'et}\to \widetilde{W_r\Omega}_R^\sub{pro\'et}\to R\Gamma(\frak X,\widetilde{W_r\Omega}_{\frak X})
\]
are quasi-isomorphisms.
\end{corollary}

\begin{proof} Let $C=R\Gamma_\sub{cont}(\Gamma,W_r(R_\infty))$, and let $D$ be either of
\[
R\Gamma(X_\sub{prof\'et},W_r(\hat{\roi}_X^+))\ ,\ R\Gamma(X_\sub{pro\'et},W_r(\hat{\roi}_X^+))\ ,
\]
where $X$ is the generic fibre of $\frak X = \Spf R$. Then the map $g: C\to D$ is an almost quasi-isomorphism, and hence by Lemma~\ref{lem:LetaActualIsom} applied with $A=W_r(\roi)$, $I=W_r(\frak m)$, $f=[\zeta_{p^r}]-1$, the induced map $L\eta_{[\zeta_{p^r}]-1} g$ is a quasi-isomorphism, as by Lemma~\ref{lem:propWrOmega}, $C$ satisfies the necessary hypothesis.

For the comparison to $R\Gamma(\frak X,\widetilde{W_r\Omega}_{\frak X})$, we look at the map
\[
\widetilde{W_r\Omega}_R\hat{\otimes}_{W_r(R)} W_r(\roi_{\frak X})\to \widetilde{W_r\Omega}_{\frak X}\ .
\]
The same arguments as in the proof of Corollary~\ref{cor:alltildeOmegasame} (iv) show that this is a quasi-isomorphism in $D(\frak X_\sub{Zar})$, using Lemma~\ref{lem:WrOmegaBC}. Passing to global sections gives the result.
\end{proof}

Finally, we prove part (iii) of Theorem~\ref{thm:AOmegavsdRWLocal}. Once more, we need a lemma that $L\eta_\mu$ turns certain almost quasi-isomorphisms into quasi-isomorphisms. Recall that the ideal $W(\frak m^\flat)\subset A_\inf$ does not in general satisfy $W(\frak m^\flat)^2 = W(\frak m^\flat)$, so we have to be careful about the meaning of ``almost'' here.

\begin{lemma}\label{lem:LEtaQuasiIsomAinf} Let $f: C\to D$ be a map of derived $p$-complete complexes in $D(A_\inf)$, and assume that the following conditions are satisfied.
\begin{enumerate}
\item The morphism $f\dotimes_{\bb Z_p} \bb F_p$ in $D(\roi^\flat)$ is an almost quasi-isomorphism.
\item For all $i\in \bb Z$, the map $H^i(L\eta_\mu f): H^i(L\eta_\mu C)\to H^i(L\eta_\mu D)$ is injective.
\item For all $i\in \bb Z$, one has
\[
\bigcap_{m\in W(\frak m^\flat), m|\mu} \frac{\mu}m H^i(C) = \mu H^i(C)\ .
\]
\end{enumerate}
Then $L\eta_\mu f: L\eta_\mu C\to L\eta_\mu D$ is a quasi-isomorphism.
\end{lemma}

\begin{proof} We need to show that for all $i\in \bb Z$, the map
\[
\beta: H^i(L\eta_\mu C) = H^i(C)/H^i(C)[\mu]\to H^i(D)/H^i(D)[\mu] = H^i(L\eta_\mu D)
\]
is an isomorphism; let
\[
\alpha: H^i(C)\to H^i(D)
\]
be the map inducing $\beta$. By assumption (ii), $\beta$ is injective.

To prove surjectivity of $\beta$, we have to see that the map
\[
H^i(D)[\mu]\to \coker \alpha
\]
is surjective. For this, we observe first that for all $r\geq 1$, the map
\[
f\dotimes_{\bb Z_p} \bb Z/p^r \bb Z: C\dotimes_{\bb Z_p} \bb Z/p^r \bb Z\to D\dotimes_{\bb Z_p} \bb Z/p^r \bb Z
\]
is an almost quasi-isomorphism with respect to the ideal $W_r(\frak m^\flat)\subset W_r(\roi^\flat)$. This implies that the induced map
\[
W_r(\frak m^\flat)\otimes_{W_r(\roi^\flat)} (C\dotimes_{\bb Z_p} \bb Z/p^r \bb Z)\to W_r(\frak m^\flat)\otimes_{W_r(\roi^\flat)} (D\dotimes_{\bb Z_p} \bb Z/p^r \bb Z)
\]
is a quasi-isomorphism. In particular, there is a map
\[
W_r(\frak m^\flat)\otimes_{W_r(\roi^\flat)} (D\dotimes_{\bb Z_p} \bb Z/p^r \bb Z)\cong W_r(\frak m^\flat)\otimes_{W_r(\roi^\flat)} (C\dotimes_{\bb Z_p} \bb Z/p^r \bb Z)\to C\dotimes_{\bb Z_p} \bb Z/p^r \bb Z\ .
\]
Thus, for any element $m\in W(\frak m^\flat)$, there is a canonical map
\[
\tilde{m}: D\dotimes_{\bb Z_p} \bb Z/p^r \bb Z\to C\dotimes_{\bb Z_p} \bb Z/p^r \bb Z
\]
whose composite with $f\dotimes_{\bb Z_p} \bb Z/p^r \bb Z$ (on either side) is multiplication by $m$. Passing to the limit over $r$, using that $C$ and $D$ are $p$-complete, we get a canonical map $\tilde{m}: D\to C$ whose composite with $f$ (on either side) is multiplication by $m$.

Now, pick any element $\bar{x}\in \coker \alpha$, and lift it to $x\in H^i(D)$. We claim that
\[
\tilde{\mu}(x)\in \bigcap_{m\in W(\frak m^\flat), m|\mu} \frac{\mu}m H^i(C)\subset \mu H^i(C)\ .
\]
Indeed, for any $m\in W(\frak m^\flat)$, we have $\tilde{m}(x)\in H^i(C)$, and then $\tilde{\mu}(x)= \frac{\mu}m \tilde{m}(x)\in \frac{\mu}m H^i(C)$. By assumption (iii), we get that $\tilde{\mu}(x)\in \mu H^i(C)$, so after subtracting (the image in $H^i(D)$ of) an element of $H^i(C)$ from $x$, we may assume that $\tilde{\mu}(x) = 0$, so that in particular $\mu x=0$, i.e.~$x\in H^i(D)[\mu]$. Thus, $H^i(D)[\mu]\to \coker \alpha$ is surjective, finishing the proof.
\end{proof}

\begin{lemma}\label{lem:niceintersection} Let
\[
C=R\Gamma_\sub{cont}(\Gamma,\bb A_\inf(R_\infty))\in D(A_\inf)\ .
\]
Then for all $i\in \bb Z$, the intersection
\[
\bigcap_{m\in W(\frak m^\flat), m|\mu} \frac{\mu}m H^i(C) = \mu H^i(C)\ .
\]
\end{lemma}

We note that it is actually not so easy to find many elements $m\in W(\frak m^\flat)$ with $m|\mu$. The only elements we know are the $\varphi^{-r}(\mu)$, and we will only use these elements in the proof. In particular, we do not know whether one can write $\mu$ as a product of two elements in $W(\frak m^\flat)$.

\begin{proof} We will freely make use of
\[
\bigcap_{m\in W(\frak m^\flat), m|\mu} \frac{\mu}m A_\inf = \mu A_\inf\ ,
\]
cf.~Lemma~\ref{lem:mapAinfWitt}. We may decompose $C=C^\sub{int}\oplus C^\sub{nonint}$ according to the decomposition
\[
\bb A_\inf(R_\infty) = \bb A_\inf(R_\infty)^\sub{int}\oplus \bb A_\inf(R_\infty)^\sub{nonint}
\]
from the proof of Proposition~\ref{lem:qdRvsAOmega}.

We handle first the non-integral part $C^\sub{nonint}$. This can be written as a completed direct sum of complexes of the form
\[
K_{A(R)^\square}(\gamma_1[\epsilon]^{a(1)}-1,\ldots,\gamma_d[\epsilon]^{a(d)}-1)\ ,
\]
where $a(1),\ldots,a(d)\in \bb Z[\tfrac 1p]\cap [0,1)$, not all $0$. We compute the cohomology groups of each of the summands. Permuting the coordinates, we may assume that $a(1)=m/p^r$ has the largest denominator $p^r$. The argument for existence of $h$ in the proof of Proposition~\ref{lem:qdRvsAOmega} shows that $\gamma_1[\epsilon]^{a(1)}-1$ has image precisely $[\epsilon]^{1/p^r}-1$. Moreover, the image of $\gamma_i[\epsilon]^{a(i)}-1$ is contained in the image of $[\epsilon]^{1/p^r}-1$, as $\gamma_i\equiv 1\mod \mu$.

Applying Lemma~\ref{lemma_on_Koszul_2} (ii) for the commutative algebra of endomorphisms of $A(R)^\square$ generated by $g=[\epsilon]^{1/p^r}-1$, $g_i=\gamma_i[\epsilon]^{a(i)}-1$ and $\tfrac{g_i}g$ shows that
\[
H^i(K_{A(R)^\square}(\gamma_1[\epsilon]^{a(1)}-1,\ldots,\gamma_d[\epsilon]^{a(d)}-1))
\]
can be written as a finite direct sum of copies of $A(R)^\square/([\epsilon]^{1/p^r}-1)$. This is a topologically free $A_\inf/([\epsilon]^{1/p^r}-1)$-module. It follows that the cohomology groups of $C^\sub{nonint}$ are a $p$-adically completed direct sum of copies of $A_\inf/([\epsilon]^{1/p^r}-1)$ for varying $r\geq 1$. Thus, by Lemma \ref{lem:cohomcompleteddirectsum}, it suffices to prove the similar assertion for $A_\inf/([\epsilon]^{1/p^r}-1)$, which is easy.

It remains to handle the integral part
\[
C^\sub{int} = K_{A(R)^\square}(\gamma_1-1,\ldots,\gamma_d-1)\ .
\]
Here, we note that all $\gamma_i-1$ are divisible by $\mu$. This implies that $H^i(C^\sub{int})/\mu$ is isomorphic to $Z^i(C^\sub{int})/\mu$. Thus, it remains to prove that
\[
\bigcap_{m\in W(\frak m^\flat), m|\mu} \frac{\mu} m Z^i(C^\sub{int}) = \mu Z^i(C^\sub{int})\ .
\]
But as the cocycles form a submodule of the corresponding term of $K_{A(R)^\square}(\gamma_1-1,\ldots,\gamma_d-1)$, which is a complex of $\mu$-torsion-free modules, it suffices to prove the similar result for the terms of the Koszul complex. Now any term is a topologically free $A_\inf$-module, for which the claim is known.
\end{proof}

\begin{proposition}\label{prop:AOmegaallthesame} The canonical maps
\[
A\Omega^\square_R\to A\Omega^\sub{prof\'et}_R\to A\Omega^\sub{pro\'et}_R\to R\Gamma(\frak X,A\Omega_{\frak X})
\]
are quasi-isomorphisms.
\end{proposition}

\begin{proof} Let $C=R\Gamma_\sub{cont}(\Gamma,\bb A_\inf(R_\infty))$, and let $D$ be either of
\[
R\Gamma(X_\sub{prof\'et},\bb A_{\inf,X})\ ,\ R\Gamma(X_\sub{pro\'et},\bb A_{\inf,X})\ ,
\]
so there is a natural map $f: C\to D$. We want to verify the conditions of Lemma~\ref{lem:LEtaQuasiIsomAinf}. Condition (i) is immediate from the almost purity theorem. Condition (iii) is the content of Lemma~\ref{lem:niceintersection}. It remains to prove that
\[
H^i(L\eta_\mu C)\to H^i(L\eta_\mu D)
\]
is injective. For this, we note that for each $r\geq 1$, there is a commutative diagram
\[\xymatrix{
L\eta_\mu C\ar[r]\ar[d] & L\eta_\mu D\ar[d]\\
\widetilde{W_r\Omega}_R^\square\ar[r] & \widetilde{W_r\Omega}_R^\sub{pro\'et}.
}\]
(More precisely, one has such a commutative diagram in the derived category of $\bb N$-indexed projective systems, where the upper row is regarded as a constant system.) Passing to the limit over $r$, we get a commutative diagram
\[\xymatrix{
L\eta_\mu C\ar[r]\ar[d] & L\eta_\mu D\ar[d]\\
R\projlim_r \widetilde{W_r\Omega}_R^\square\ar[r] & R\projlim_r \widetilde{W_r\Omega}_R^\sub{pro\'et}.
}\]
Now we note that by Lemma~\ref{lem:qdRvsAOmega}, the left vertical map is a quasi-isomorphism. By Corollary~\ref{cor:AOmegavsdRWLocalPart2}, the lower horizontal map is a quasi-isomorphism. Thus, looking at cohomology groups, we get the desired injectivity.

This shows that
\[
A\Omega_R^\square\simeq A\Omega_R^\sub{prof\'et}\simeq A\Omega_R^\sub{pro\'et}\ ;
\]
we denote them simply $A\Omega_R$ in the following. It remains to show that $A\Omega_R\simeq R\Gamma(\frak X,A\Omega_{\frak X})$. Previously, we argued by extending some variant of $A\Omega_R$ to (some kind of) a quasicoherent sheaf, and did the comparison on the sheaf level. However, $A\Omega_R$ is not a module over $R$, or any variant of $R$ (like $W_r(R)$), so this does not work here. Instead, we argue by reducing to the known case of $\widetilde{W_r\Omega}$ by an inverse limit argument.

Let $\frak X_\sub{Zar}^\sub{psh}$ be the presheaf topos on the set of affine opens $\Spf R^\prime\subset \frak X$. There is a map of topoi $j: \frak X_\sub{Zar}\to \frak X_\sub{Zar}^\sub{psh}$, where $j_\ast$ is the forgetful functor, and $j^\ast$ is the sheafification functor. We can form
\[
A\Omega_{\frak X}^\sub{psh} = L\eta_\mu R\nu^\sub{psh}_\ast \bb A_{\inf,X}\ ,
\]
where $\nu^\sub{psh} = j\circ \nu: X_\sub{pro\'et}\to \frak X_\sub{Zar}^\sub{psh}$. By Lemma~\ref{lem:LEtaChangeTopos}, the value of $A\Omega_{\frak X}^\sub{psh}$ on an affine open $\Spf R^\prime\subset \frak X$ is given by $A\Omega_{R^\prime}^\sub{pro\'et} = A\Omega_{R^\prime}$. Moreover, using Lemma~\ref{lem:LEtaChangeTopos} again, we have
\[
j^\ast A\Omega_{\frak X}^\sub{psh} = j^\ast L\eta_\mu Rj_\ast R\nu_\ast \bb A_{\inf,X} = L\eta_\mu j^\ast Rj_\ast R\nu_\ast \bb A_{\inf,X} = L\eta_\mu R\nu_\ast \bb A_{\inf,X} = A\Omega_{\frak X}\ ,
\]
i.e.~$A\Omega_{\frak X}$ is the sheafification of $A\Omega_{\frak X}^\sub{psh}$. By adjunction, we get a map
\[
A\Omega_{\frak X}^\sub{psh}\to Rj_\ast A\Omega_{\frak X} = Rj_\ast j^\ast A\Omega_{\frak X}^\sub{psh}\ ,
\]
which we want to prove is a quasi-isomorphism (as then on global sections, it gives the desired quasi-isomorphism $A\Omega_R\simeq R\Gamma(\frak X,A\Omega_{\frak X})$). In other words, we want to prove that $A\Omega_{\frak X}^\sub{psh}$ is already a sheaf. But as for any $\Spf R^\prime\subset \frak X$, we have
\[
A\Omega_{R^\prime} = R\projlim_r A\Omega_{R^\prime}/\tilde\xi_r = R\projlim_r \widetilde{W_r\Omega}_{R^\prime}\ ,
\]
we have an equality
\[
A\Omega_{\frak X}^\sub{psh} = R\projlim_r \widetilde{W_r\Omega}_{\frak X}^\sub{psh}\ ,
\]
for the evident definition of $\widetilde{W_r\Omega}_{\frak X}^\sub{psh}$. By Theorem~\ref{thm:AOmegavsdRWLocal} (iii), we know that $\widetilde{W_r\Omega}_{\frak X}^\sub{psh}$ is a sheaf, i.e.
\[
\widetilde{W_r\Omega}_{\frak X}^\sub{psh}\to Rj_\ast j^\ast \widetilde{W_r\Omega}_{\frak X}^\sub{psh}
\]
is a quasi-isomorphism. We conclude by using the following lemma, saying that an inverse limit of sheaves is a sheaf (which holds true in vast generality).

\begin{lemma} Let $C_r\in D(\frak X_\sub{Zar}^\sub{psh})$, $r\geq 1$, be a projective system, with homotopy limit $C=R\projlim C_r$. Assume that for each $r\geq 1$, $C_r$ is a sheaf, i.e. $C_r\to Rj_\ast j^\ast C_r$ is a quasi-isomorphism. Then $C$ is a sheaf, i.e.~$C\to Rj_\ast j^\ast C$ is a quasi-isomorphism.
\end{lemma}

\begin{proof} Let $\tilde{C}_r = j^\ast C_r$, and let $\tilde{C} = R \projlim \tilde{C}_r\in D(\frak X_\sub{Zar})$; we note that this is not a priori given by $j^\ast C$. There is a quasi-isomorphism $C\isoto Rj_\ast \tilde{C}$, given as a limit of the quasi-isomorphisms $C_r\isoto Rj_\ast \tilde{C}_r$. Applying $j^\ast$ shows that $j^\ast C\cong j^\ast Rj_\ast \tilde{C} = \tilde{C}$, and thus $C\cong Rj_\ast j^\ast C$ as desired.
\end{proof}
\end{proof}

\subsection{Further properties of $A\Omega$}

Let us end this section by noting several further properties of $A\Omega_{\frak X}$. First, the complex $A\Omega_R$ satisfies a K\"unneth formula.

\begin{lemma}\label{lem:AOmegaKunneth} Let $R_1$ and $R_2$ be small formally smooth $\roi$-algebras with completed tensor product $R=R_1\hat{\otimes}_\roi R_2$. Then the natural map
\[
A\Omega_{R_1}\hat{\dotimes}_{A_\inf} A\Omega_{R_2}\to A\Omega_R
\]
is a quasi-isomorphism.
\end{lemma}

\begin{proof} As both sides are derived $\tilde\xi$-complete, it suffices to check modulo $\tilde\xi$, where it follows from Proposition~\ref{prop:kuenneth}.
\end{proof}

Also, by construction $A\Omega_{\frak X}$ comes equipped with a Frobenius.

\begin{proposition}\label{prop:PhionAOmega} Let $R$ be a small formally smooth $\roi$-algebra. Then there is a natural $\phi$-linear map $\phi: A\Omega_R\to A\Omega_R$ which factors as the composite of a $\phi$-linear quasi-isomorphism $A\Omega_R\simeq L\eta_{\tilde\xi} A\Omega_R$ and the natural map $L\eta_{\tilde\xi} A\Omega_R\to A\Omega_R$.

In particular, if $\frak X$ is a smooth formal scheme over $\roi$, then there is a $\phi$-linear map $\phi: A\Omega_{\frak X}\to A\Omega_{\frak X}$ factoring over a $\phi$-linear quasi-isomorphism $A\Omega_{\frak X}\simeq L\eta_{\tilde\xi} A\Omega_{\frak X}$.
\end{proposition}

\begin{proof} Let $X$ be the generic fibre of $\frak X=\Spf R$. The Frobenius $\phi_X$ is an automorphism of $R\Gamma_\sub{pro\'et}(X,\bb A_{\inf,X})$, and thus induces a quasi-isomorphism
\[\begin{aligned}
\phi^* A\Omega_R &= \phi^* L\eta_\mu R\Gamma_\sub{pro\'et}(X,\bb A_{\inf,X})\simeq L\eta_{\phi(\mu)} \phi^* R\Gamma_\sub{pro\'et}(X,\bb A_{\inf,X})\\
&= L\eta_{\tilde\xi} L\eta_\mu \phi^* R\Gamma_\sub{pro\'et}(X,\bb A_{\inf,X}) \stackrel{\phi_X}{\simeq} L\eta_{\tilde\xi} L\eta_\mu R\Gamma_{\sub{pro\'et}}(X, \bb A_{\inf,X}) \simeq L\eta_{\tilde{\xi}} A\Omega_R\ .
\end{aligned}\]
\end{proof}

Moreover, let us note that $L\eta$ behaves in a symmetric monoidal way in a relevant case.

\begin{lemma}\label{lem:LetasymmmonWrOmega} Let $R$ be a small formally smooth $\roi$-algebra, and let $D=R\Gamma_\sub{pro\'et}(X,W_r(\hat\roi_X^+))$, so that $\widetilde{W_r\Omega}_R = L\eta_{[\zeta_{p^r}]-1} D$. Let $E\in D(W_r(\roi))$ be any complex. The natural map
\[
L\eta_{[\zeta_{p^r}]-1} D\dotimes_{W_r(\roi)} L\eta_{[\zeta_{p^r}]-1} E\to L\eta_{[\zeta_{p^r}]-1}(D\dotimes_{W_r(\roi)} E)
\]
is a quasi-isomorphism.

In fact, the same result holds if $D$ is replaced by any complex which admits an almost quasi-isomorphism $R\Gamma_\cont(\Gamma,W_r(R_\infty))\to D$, where $R\Gamma_\cont(\Gamma,W_r(R_\infty))$ is defined using a framing as usual.
\end{lemma}

\begin{proof} We may assume that $E$ is perfect, as the general result follows by passage to a filtered colimit. Choose a framing $\square$, and let $C=R\Gamma_\cont(\Gamma,W_r(R_\infty))$ using standard notation. In that case, the argument of Corollary~\ref{cor:AOmegavsdRWLocalPart2} works to prove that
\[
L\eta_{[\zeta_{p^r}]-1}(C\dotimes_{W_r(\roi)} E)\isoto L\eta_{[\zeta_{p^r}]-1}(D\dotimes_{W_r(\roi)} E)\ ,
\]
using the general form of Lemma~\ref{lem:propWrOmega} (iii). Thus, it is enough to show that
\[
L\eta_{[\zeta_{p^r}]-1} C\dotimes_{W_r(\roi)} L\eta_{[\zeta_{p^r}]-1} E\to L\eta_{[\zeta_{p^r}]-1}(C\dotimes_{W_r(\roi)} E)
\]
is a quasi-isomorphism. But $C$ decomposes into a completed direct sum of Koszul complexes. Thus, the result follows from the case of Koszul complexes, Lemma~\ref{lemma_on_Koszul_1}, and the commutation of $L\eta$ with $p$-adic completion, Lemma~\ref{lem:Letacommutecompletion}.
\end{proof}

\newpage

\section{The relative de~Rham--Witt complex}\label{section_Witt_complexes}
In this section we review the theory of de~Rham--Witt complexes.

\subsection{Witt groups}\label{subsection_Witt_vectors} Let $A$ be a ring. As before, we use $W_r(A)$ to denote the finite length $p$-typical Witt vectors (normalized so that $W_1(A) = A$) and $W(A):=\projlim_r W_r(A)$. In this section we recall some results about how ideals of $A$ induce ideals of $W_r(A)$.

If $I\subset A$ is an ideal then $W_r(I):=\ker(W_r(A)\to W_r(A/I))$, which may be alternatively defined as the Witt vectors of the non-unital ring $I$. We also let $[I]\subset W_r(A)$ denote the ideal generated by $\{[a]:a\in I\}$, which is contained in $W_r(I)$.

\begin{lemma}\label{lemma_powers_of_ideals_in_W_r}
Suppose that $I$ is a finitely generated ideal of a ring $A$, and let $\Sigma\subset I$ be a finite set of generators. Then the following five chains of ideals of $W_r(A)$ are all intertwined:
\[
\pid{[a^s]:a\in \Sigma}\qquad [I]^s\qquad [I^s]\qquad W_r(I)^s\qquad W_r(I^s)\ ,\qquad s\ge 1\ .
\]
(The first denotes the ideal generated by the elements $[a^s]$, for $a\in \Sigma$.) More precisely, we have containments
\[
W_r(I^{|\Sigma|p^rs}) \subset  \pid{[a^s]:a\in \Sigma}\subset [I]^s\subset [I^s]\subset W_r(I^s), [I]^s\subset W_r(I)^s\subset W_r(I^s)
\]
\end{lemma}

\begin{proof} Firstly, any element of $W_r(A)$ may be written as $\sum_{i=0}^{r-1} V^i[a_i]$ for some unique $a_0,\dots,a_{r-1}$; applying the same observation to $A/I$ we see that $W_r(I)$ is precisely the set of elements of $W_r(A)$ such that each element $a_i$ occurring in this expansion belongs to $I$. Moreover, for any two ideals $J_1,J_2\subset A$, we have $W_r(J_1+J_2) = W_r(J_1)+W_r(J_2)$ (induct on $r$ and use the formula for $[a]+[b]$).

The inclusions $\pid{[a^s]:a\in \Sigma}\subset[I]^s \subset [I^s]\subset W_r(I^s)$ and $[I]^s\subset W_r(I)^s$ are then clear, and $W_r(I)^s\subset W_r(I^s)$ is a consequence of the identity $V^i[a]V^j[b]=p^j V^i([ab^{p^{i-j}}])$ (cf.~proof of Lemma~\ref{lemma_witt_alg_1}) for all $a,b\in A$ and $i\ge j$, and do not require finite generation of $I$. Conversely, $I^{|\Sigma|p^r s}\subset \pid{a^{p^r s}:a\in \Sigma}$, and
\[
W_r(\pid{a^{p^r s}:a\in \Sigma}) = \sum_{a\in \Sigma} W_r(a^{p^r s} A)\ ,
\]
by the additivity of $W_r$ of ideals. Finally, $W_r(a^{p^r s} A)\subset [a]^s W_r(A)$. Combining these observations shows that
\[
W_r(I^{|\Sigma|p^rs})\subset \pid{[a^s]:a\in \Sigma}\ .
\]
\end{proof}

\begin{corollary}\label{corollary_almost_Witt}
If $I\subset A$ is an ideal satisfying $I=I^2$ such that $I$ can be written as an increasing union of principal ideals generated by non-zero-divisors, then $W_r(I)=[I]$, and $W_r(I)\subset W_r(A)$ is again an ideal satisfying $W_r(I)^2 = W_r(I)$ such that $W_r(I)$ can be written as an increasing union of principal ideals generated by non-zero-divisors.
\end{corollary}

\begin{proof} Write $I=\bigcup_j f_j A$, where $f_j\in A$ is a non-zero-divisor. Applying the previous lemma to all $f_j A$ and passing to a direct limit over $j$ (noting that the constants are independent of $j$) shows that the sequences of ideals
\[
\bigcup_j [f_j]^s W_r(A)\qquad [I^s]\qquad [I]^s\qquad W_r(I)^s\qquad W_r(I^s)\ ,\qquad s\ge 1
\]
are intertwined, and are all contained in the last sequence $W_r(I^s)$. However, this last sequence is constant as $I=I^2=I^3=\ldots$. Thus, all systems are constant and equal, and in particular $[I]=W_r(I)=\bigcup_j [f_j] W_r(A)$. Since the Teichm\"uller lift of a non-zero-divisor is still a non-zero-divisor, this completes the proof.
\end{proof}

The next lemma shows that $[p]$-adic and $p$-adic completion are the same:

\begin{lemma}\label{lemma_W_r_of_p_completion}
Let $A$ be a ring. The following chains of ideals are intertwined:
\[
[p]^sW_r(A)\qquad W_r(pA)^s\qquad p^sW_r(A)\ ,\qquad s\ge1\ .
\]
More precisely,
\[
[p]^{2s} W_r(A)\subset p^s W_r(A)\ ,\ p^{rs} W_r(A)\subset W_r(pA)^s\ ,\ W_r(pA)^{p^r s}\subset [p]^s W_r(A)\ .
\]
\end{lemma}

\begin{proof} Recall from Lemma~\ref{lemma_witt_alg_1} that $[p]^2\in pW_r(A)$; this implies $[p]^{2s} W_r(A)\subset p^s W_r(A)$. As $p^r=0$ in the $W_r(\bb F_p)=\bb Z/p^r \bb Z$-algebra $W_r(A/pA)$, we have $p^r\in W_r(pA)$ and thus $p^{rs} W_r(A)\subset W_r(pA)^s$. Finally, the last inclusion was proved in Lemma~\ref{lemma_powers_of_ideals_in_W_r}.
\end{proof}

Let us also recall that Witt rings behave well with respect to the \'etale topology. The first part of the following theorem appeared first in work of van der Kallen, \cite[Theorem~2.4]{vanderKallen1986}. Under the assumption that the rings are $F$-finite, the result is proved by Langer--Zink in \cite[Corollary A.18]{LangerZink}. The general result appears (in even greater generality) in work of Borger, \cite[Theorem 9.2, Corollary 9.4]{Borger}.

\begin{theorem}\label{thm:Wittetale}
Let $A\to B$ be an \'etale morphism. Then $W_r(A)\to W_r(B)$ is also \'etale. Moreover, if $A\to A'$ is any map with base extension $B' = B\otimes_A A'$, then the natural map
\[
W_r(A')\otimes_{W_r(A)} W_r(B)\to W_r(B')
\]
is an isomorphism.
\end{theorem}

\begin{proof} If $R$ is a $\bb Z[\frac{1}{p}]$-algebra, then $W_r(R) \simeq \prod_{i=1}^r R$ as rings functorially in $R$ via the ghost maps. Thus, if $A$ (and thus every ring involved) is a $\bb Z[\frac{1}{p}]$-algebra, the claim is clear. As the functor $W_r(-)$ commutes with localization, we may then assume that $A$ (and thus every ring involved) is a $\bb Z_{(p)}$-algebra. Now if $A$ and $B$ are $F$-finite (e.g., finitely generated over $\bb Z_{(p)}$) and $A'$ is arbitrary, this is \cite[Corollary A.18]{LangerZink}. Let us observe that this formally implies the general case: Indeed, we may find a finitely generated $\bb Z_{(p)}$-algebra $A_0$ and an \'etale $A_0$-algebra $B_0$ such that $B=B_0\otimes_{A_0} A$ along some morphism $A_0\to A$. Then $W_r(B_0)$ is \'etale over $W_r(A_0)$, and
\[
W_r(B_0)\otimes_{W_r(A_0)} W_r(A)\to W_r(B)
\]
is an isomorphism. Thus, $W_r(B)$ is \'etale over $W_r(A)$, as the base extension of an \'etale map. Similarly,
\[
W_r(B_0)\otimes_{W_r(A_0)} W_r(A')\to W_r(B')
\]
is an isomorphism, so that
\[
W_r(A')\otimes_{W_r(A)} W_r(B) = W_r(A')\otimes_{W_r(A_0)} W_r(B_0) = W_r(B')\ ,
\]
as desired.
\end{proof}

\subsection{Relative de~Rham--Witt complex}
We recall the notion of an $F$-$V$-procomplex from the work of Langer--Zink, \cite{LangerZink}. From now on, we assume that $A$ is a $\bb Z_{(p)}$-algebra.

\begin{definition} Let $B$ be an $A$-algebra. An $F$-$V$-procomplex for $B/A$ consists of the following data $(\cal W_r^\blob,R,F,V,\lambda_r)$:
\begin{enumerate}
\item a commutative differential graded $W_r(A)$-algebra $\cal W_r^\blob=\bigoplus_{n\ge0} \cal W_r^n$ for each integer $r\ge 1$;
\item morphisms $R:\cal W_{r+1}^\blob\to R_*\cal W_r^\blob$ of differential graded $W_{r+1}(A)$-algebras for $r\ge 1$;
\item morphisms $F:\cal W_{r+1}^\blob\to F_*\cal W_r^\blob$ of graded $W_{r+1}(A)$-algebras for $r\ge 1$;
\item morphisms $V:F_*\cal W_r^\blob\to\cal W_{r+1}^\blob$ of graded $W_{r+1}(A)$-modules for $r\ge 1$;
\item morphisms $\lambda_r: W_r(B)\to \cal W_r^0$ for each $r\geq 1$, commuting with the $F$, $R$ and $V$ maps;
\end{enumerate}
such that the following identities hold: $R$ commutes with both $F$ and $V$, $FV=p$,  $FdV=d$, $V(F(x)y)=xV(y)$, and the Teichm\"uller identity
\[
Fd\lambda_{r+1}([b])=\lambda_r([b])^{p-1}d\lambda_r([b])
\]
for $b\in B$, $r\ge 1$.
\end{definition}

In the classical work on the de~Rham--Witt complex, the restriction operator $R$ is regarded as the ``simplest'' part of the data; however, in our work, it will actually be the most subtle of the operators (in close analogy to what happens in topological cyclic homology). In particular, we will be explicit about the use of the operator $R$, and it would probably be more appropriate to use the term $F$-$R$-$V$-procomplex, but we stick to Langer--Zink's notation.

\begin{remark}\label{remark_p-torsion-free_Teichmuller}
The Teichm\"uller rule of the previous definition is automatic in the case that $\cal W_r^1$ is $p$-torsion-free, since one deduces from the other rules that $dF(x) = FdVF(x) = Fd(V(1)x) = F(V(1)dx) = p Fdx$, and thus
\[
p\lambda_r([b])^{p-1}d\lambda_r([b])=d\lambda_r([b]^p)=dF\lambda_{r+1}([b])=pFd\lambda_{r+1}([b])\ .
\]
\end{remark}

There is an obvious definition of morphism between $F$-$V$-procomplexes. In particular, it makes sense to ask for an initial object in the category of all $F$-$V$-procomplexes for $B/A$.

\begin{theorem}[{\cite{LangerZink}}]\label{thm:dRWExists} There is an initial object $\{W_r\Omega_{B/A}^\blob\}_r$ in the category of $F$-$V$-procomplexes, called the {\em relative de~Rham--Witt complex}.

In other words, if $(\cal W_r^\blob,R,F,V,\lambda_r)$ is any $F$-$V$-procomplex for $B/A$, then there are unique morphisms of differential graded $W_r(A)$-algebras
\[
\lambda_r^\blob:W_r\Omega_{B/A}^\blob\to \cal W_r^\blob
\]
which are compatible with $R,F,V$ in the obvious sense and such that $\lambda_r^0:W_r(B)\to \cal W_r^0$ is the structure map $\lambda_r$ of the Witt complex for each $r\ge 1$.
\end{theorem}

\subsection{Elementary properties of relative de~Rham--Witt complexes}\label{subsection_further_properties}
In this section we summarise various properties of relative de~Rham--Witt complexes.

\begin{lemma}[\'Etale base change]\label{lem:dRWEtale}
Let $A\to R$ be a morphism of $\bb Z_{(p)}$-algebras, and let $R'$ be an \'etale $R$-algebra. The natural map
\[
W_r\Omega_{R/A}^n\otimes_{W_r(R)}W_r(R')\isoto W_r\Omega_{R'/A}^n
\]
is an isomorphism.
\end{lemma}

\begin{proof} If $p$ is nilpotent in $S$ or $S$ is $F$-finite, this is \cite[Proposition~1.7]{LangerZink}; this assumption is used in \cite{LangerZink} only to guarantee that $W_r(R)\to W_r(R^\prime)$ is \'etale, which is however always true by Theorem~\ref{thm:Wittetale}. Thus, one can either reduce the general case to the $F$-finite case by Noetherian approximation, or observe that by Theorem~\ref{thm:Wittetale}, the argument of \cite{LangerZink} works in general.
\end{proof}

The next lemma complements Lemma \ref{lemma_powers_of_ideals_in_W_r}; if $I\subset R$ is an ideal, then we write
\[
W_r\Omega^n_{(R,I)/A}:=\ker(W_r\Omega_{R/A}^n\To W_r\Omega_{(R/I)/A}^n)\ .
\]

\begin{lemma}[Quotients]
Let $A\to R$ be a morphism of $\bb Z_{(p)}$-algebras, and $I\subset R$ an ideal. Then:
\begin{enumerate}
\item $\bigoplus_{n\ge0}W_r\Omega^n_{(R,I)/A}$ is the differential graded ideal of $W_r\Omega_{R/A}^\blob$ generated by $W_r(I)$.
\item If $I$ is finitely generated and $\Sigma\subset I$ is a finite set of generators, then the following two chains of ideals of $W_r\Omega_{R/A}^\blob$ are intertwined:
\[
\pid{[a^s]:a\in \Sigma}W_r\Omega_{R/A}^\blob\qquad W_r\Omega^\blob_{(R,I^s)/A}\ ,\qquad s\ge1\ .
\]
\end{enumerate}
\end{lemma}

\begin{proof}
(i): Write $\bigoplus_{n\ge0}W_r'\Omega^n_{(R,I)/A}$ for the differential graded ideal of $W_r\Omega_{R/A}^\blob$ generated by $W_r(I)$; certainly $W_r'\Omega^n_{(R,I)/A}\subset W_r\Omega^n_{(R,I)/A}$ and so there is a canonical surjection
\[
\pi:W_r\Omega_{R/A}^\blob/W_r'\Omega^\blob_{(R,I)/A}\onto W_r\Omega_{(R/I)/A}^\blob\ .
\]
Elements of $W_r'\Omega^n_{(R,I)/A}$ are by definition finite sums of terms of the form $a_0da_1\cdots da_n$ where at least one of $a_0,\dots,a_n\in W_r(A)$ belongs to $W_r(I)$. From this it is relatively straightforward to prove that $R(W_r'\Omega^n_{(R,I)/A})\subset W_{r-1}'\Omega^n_{(R,I)/A}$, $F(W_r'\Omega^n_{(R,I)/A})\subset W_{r-1}'\Omega^n_{(R,I)/A}$, and $V(W_{r-1}'\Omega^n_{(R,I)/A})\subset W_r'\Omega^n_{(R,I)/A}$: we refer the reader to \cite[Lemma~2.4]{GeisserHesselholt2006} for the detailed manipulations, where the same result is proved for the absolute de~Rham--Witt complex. Since $W_r(R)/W_r(I)=W_r(R/I)$ by definition, it follows that the quotients $W_r\Omega_{R/A}^\blob/W_r'\Omega_{(R,I)/A}^\blob$, $r\ge1$, inherit the structure of an $F$-$V$-procomplex for $R/I$ over $A$. The universal property of the relative de~Rham--Witt complex therefore implies that $\pi$ has a section; since $\pi$ is surjective, it is therefore actually an isomorphism and so $W_r'\Omega^\blob_{(R,I)/A}=W_r\Omega^\blob_{(R,I)/A}$, as required.

(ii): The inclusion $\pid{[a^s]:a\in \Sigma}W_r\Omega_{R/A}^\blob\subset W_r\Omega^\blob_{(R,I^s)/A}$ is clear. Conversely, in Lemma \ref{lemma_powers_of_ideals_in_W_r} we proved that for each $s\ge 1$ there exists $t\ge 1$ such that $W_r(I^t)\subset\pid{[a^s]:a\in \Sigma}$. It follows that any element of $W_r'\Omega^n_{(R,I^t)/A}$ is a finite sum of terms of the form $\omega=a_0da_1\cdots da_n$ where at least one of the elements $a_0,\dots,a_n\in W_r(R)$ may be written as $[a^s]b$, with $a\in \Sigma$ and $b\in W_r(R)$; the Leibniz rule now easily shows that $\omega\in \pid{[a^{s-1}]:a\in \Sigma}W_r\Omega^n_{R/A}$, as required.
\end{proof}

\begin{corollary}
Let $A\to R$ be a morphism of $\bb Z_{(p)}$-algebras, and $I\subset A$ a finitely generated ideal. Then the canonical map of pro-differential graded-$W_r(A)$-algebras
\[
\{W_r\Omega^\blob_{R/A}\otimes_{W_r(A)}W_r(A)/[I^s]\}_s\To \{W_r\Omega^\blob_{(R/I^sR)/(A/I^s)}\}_s
\]
is an isomorphism. In particular,
\[
\projlim_sW_r\Omega^\blob_{R/A}\otimes_{W_r(A)}W_r(A)/[I^s]\Isoto \projlim_sW_r\Omega^\blob_{(R/I^s R)/(A/I^s)}\ ,
\]
\end{corollary}

\begin{proof}
This follows directly from part (ii) of the previous lemma, noting that $W_r\Omega^\blob_{(R/I^sR)/A}=W_r\Omega^\blob_{(R/I^sR)/(A/I^s)}$.
\end{proof}

In particular, we make the following definition, where the stated equality follows from the previous corollary applied to $I=pA$, and we use Lemma~\ref{lemma_W_r_of_p_completion}.

\begin{definition} The continuous de~Rham--Witt complex of a morphism $A\to R$ of $\bb Z_{(p)}$-algebras is given by
\[
W_r\Omega^{i,\cont}_{R/A} = \projlim_s W_r\Omega^i_{(R/p^s)/(A/p^s)} = \projlim_s W_r\Omega^i_{R/A}/p^s\ .
\]
\end{definition}

It would perhaps be more appropriate to let this notion depend on a choice of ideal of definition of $A$, but we will only need this version in the paper. We note that $W_r\Omega^{i,\cont}_{R/A}$ still has the structure of an $F$-$V$-procomplex for $R/A$.

\subsection{Relative de~Rham--Witt complex of a (Laurent) polynomial algebra}\label{subsec:dRWLaurent}
We now recall Langer--Zink's results concerning the relative de~Rham--Witt complex of a polynomial algebra $A[\ul T]:=A[T_1,\dots,T_d]$. We will be more interested in the Laurent polynomial algebra $A[\ul T^{\pm 1}]:=A[T_1^{\pm1},\dots,T_d^{\pm 1}]$, and trivially extend their results to this case by noting that $W_r\Omega_{A[\ul T^{\pm1}]/A}^n$ is the localisation of the $W_r(A[\ul T])$-module $W_r\Omega_{A[\ul T]/A}^n$ at the non-zero-divisors $[T_1],\dots,[T_d]$ by \cite[Remark~1.10]{LangerZink}.

Fix a function $a:\{1,\dots,d\}\to p^{-r}\bb Z$ (this notation is slightly more convenient than thinking of $a$ as an element of $p^{-r}\bb Z^d$), which is usually called a ``weight''. Then we set $v(a):=\min_{i}v(a(i))$, where $v(a(i))=v_p(a(i))\in \bb Z\cup \{\infty\}$ is the $p$-adic valuation of $a(i)$; more generally, given a subset $I\subset \{1,\ldots,d\}$, we define $v(a|_I):=\min_{i\in I}v(a(i))$. Let $P_a$ denote the collection of disjoint partitions $I_0,\dots,I_n$ of $\{1,\ldots,d\}$ satisfying the following conditions:
\begin{enumerate}
\item $I_1,\dots,I_n$ are non-empty, but $I_0$ is possibly empty;
\item all elements of $a(I_{j-1})$ have $p$-adic valuation $\le$ those elements of $a(I_j)$, for $j=1,\dots,n$;
\item an additional ordering condition, strengthening (ii) and only necessary in the case that $v:\{1,\ldots,d\}\to\bb Z$ is not injective, to eliminate the possibility that two different such partitions might be equal after reordering the indices; to be precise, we fix a total ordering $\preceq_a$ on $\{1,\ldots,d\}$ with the property that $v:\{1,\ldots,d\}\to\bb Z$ is weakly increasing, and then insist that all elements of $I_{j-1}$ are strictly $\preceq_a$-less than all elements of $I_j$, for $j=1,\dots,n$.
\end{enumerate}
Fix such a partition $(I_0,\dots,I_d)\in P_a$, and let $\rho_1$ be the greatest integer between $0$ and $n$ such that $v(a|_{I_{\rho_1}})<0$ (take $\rho_1=0$ if there is no such integer); similarly, let $\rho_2$ be the greatest integer between $0$ and $n$ such that $v(a|_{I_{\rho_2}})<\infty$.

It is convenient to set $u(a):=\max\{-v(a),0\}$. Then, given $x\in W_{r-u(a)}(A)$, we define an element $e(x,a,I_0,\dots,I_n)\in W_r\Omega_{A[\ul T^{\pm 1}]/A}^n$ as follows:
\begin{description}
\item[Case 1] ($I_0\neq\emptyset$) the product of the elements
\begin{align*}
V^{-v(a|_{I_0})}(x \prod_{i\in I_0}[T_i]^{a(i)/p^{v(a|_{I_0})})})&\\
dV^{-v(a|_{I_j})}\prod_{i\in I_j}[T_i]^{a(i)/p^{v(a|_{I_j})}}&\qquad j=1,\dots,\rho_1,\\
F^{v(a|_{I_j})}d\prod_{i\in I_j}[T_i]^{a(i)/p^{v(a|_{I_j})}}&\qquad j=\rho_1+1,\dots,\rho_2,\\
\dlog \prod_{i\in I_j} [T_i]\qquad\qquad&\qquad j=\rho_2+1,\ldots,n.
\end{align*}
\item[Case 2] ($I_0=\emptyset$ and $v(a)<0$) the product of the elements
\begin{align*}
dV^{-v(a|_{I_1})}(x\prod_{i\in I_1}[T_i]^{a(i)/p^{v(a|_{I_1})}})&\\
dV^{-v(a|_{I_j})}\prod_{i\in I_j}[T_i]^{a(i)/p^{v(a|_{I_j})}}&\qquad j=2,\dots,\rho_1,\\
F^{v(a|_{I_j})}d\prod_{i\in I_j}[T_i]^{a(i)/p^{v(a|_{I_j})}}&\qquad j=\rho+1,\dots,\rho_2,\\
\dlog \prod_{i\in I_j} [T_i]\qquad\qquad&\qquad j=\rho_2+1,\ldots,n.
\end{align*}
\item[Case 3] ($I_0=\emptyset$ and $v(a)\ge 0$) the product of $x\in W_r(A)$ with the elements
\begin{align*}
F^{v(a|_{I_j})}d\prod_{i\in I_j}[T_i]^{a(i)/p^{v(a|_{I_j})}}&\qquad j=1,\dots,\rho_2,\\
\dlog \prod_{i\in I_j} [T_i]\qquad\qquad&\qquad j=\rho_2+1,\ldots,n.
\end{align*}
\end{description}

\begin{theorem}[{\cite[Proposition~2.17]{LangerZink}}]\label{theorem_LZ1}
The map of $W_r(A)$-modules
\[
e:\bigoplus_{a:\{1,\dots,d\}\to p^{-r}\bb Z}\bigoplus_{(I_0,\dots,I_n)\in P_a} V^{u(a)}W_{r-u(a)}(A)\To W_r\Omega_{A[\ul T^{\pm1}]/A}^n
\]
given by the sum of the maps
\[
V^{u(a)}W_{r-u(a)}(A)\to W_r\Omega_{A[\ul T^{\pm1}]/A}^n,\quad V^{u(a)}(x)\mapsto e(x,a,I_0,\dots,I_n)
\]
is an isomorphism.
\end{theorem}

\begin{proof}
Langer--Zink prove this for $A[\ul T]$ in place of $A[\ul T^{\pm 1}]$, in which case $p^{-r}\bb Z$ should be replaced by $p^{-r}\bb Z_{\ge 0}$. To deduce the desired result for Laurent polynomials, recall that $W_r\Omega_{A[\ul T^{\pm1}]/A}^n$ is the localisation of the $W_r(A[\ul T])$-module $W_r\Omega_{A[\ul T]/A}^n$ at the non-zero-divisors $[T_1],\dots,[T_d]$, and hence $W_r\Omega_{A[\ul T^{\pm1}]/A}^n=\bigcup_{j\ge 0} [T_1]^{-j}\cdots[T_d]^{-j} W_r\Omega_{A[\ul T]/A}^n$ is an increasing union of copies of $W_r\Omega_{A[\ul T]/A}^n$.

We also remark that Langer--Zink work with weights whose valuations are bounded below by $1-r$ rather than $-r$; since $W_{r-\max(r,0)}(A)=0$ this means that we are only adding redundant zero summands to the description.
\end{proof}

The {\em integral part} $W_r^\sub{int}\Omega_{A[\ul T^{\pm 1}]/A}^\blob$ of $W_r\Omega_{A[\ul T^{\pm 1}]/A}^\blob$ is its differential graded $W_r(A)$-subalgebra generated by the elements $[T_1]^{\pm 1},\dots,[T_d]^{\pm 1}\in W_r(A[\ul T^{\pm1}])$. In other words, the integral part is the image of the canonical map of differential graded $W_r(A)$-algebras
\[
\tau:\Omega_{W_r(A)[\ul U^{\pm 1}]/W_r(A)}^\blob\To W_r\Omega_{A[\ul T^{\pm 1}]/A}^\blob
\]
induced by $U_i\mapsto [T_i]$. We note that the integral part depends on the choice of coordinates.

\begin{theorem}[{\cite[Proof of Theorem~3.5]{LangerZink}}]\label{thm:dRvsdRW}
The map of complexes $\tau$ is an injective quasi-isomorphism.
\end{theorem}

\begin{proof}
In terms of the previous theorem, the image of $\tau$ is easily seen to be the $W_r(A)$-submodule spanned by the weights $a:\{1,\dots,d\}\to p^{-r}\bb Z$ with $v(a)\ge0$, i.e., with value in $\bb Z$. One then checks directly firstly that the complement, i.e., the part of $W_r\Omega_{A[\ul T^{\pm1}]/A}^\blob$ corresponding to weights with $v(a)<0$, is acyclic, and secondly, by writing a similar explicit description of $\Omega_{W_r(A)[\ul U^{\pm 1}]/W_r(A)}^\blob$, that $\tau$ is an isomorphism onto its image.
\end{proof}

\subsection{The case of smooth algebras over a perfectoid base}
Finally, we want to explain some nice features in the case where the base ring $A$ is perfectoid, and $R$ is a smooth $A$-algebra. The next proposition will be applied in particular to the homomorphism $\roi \to k$ of perfectoid rings.

\begin{proposition}\label{prop:dRWperfectoid}
Let $A\to A'$ be a homomorphism of perfectoid rings, and $R$ a smooth $A$-algebra, with base change $R' = R\otimes_A A'$.
\begin{enumerate}
\item The $W_r(A)$-modules $W_r\Omega^n_{R/A}$ and $W_r(A')$ are Tor-independent.
\item The canonical map of differential graded $W_r(A')$-algebras
\[
W_r\Omega^\blob_{R/A}\otimes_{W_r(A)} W_r(A')\to W_r\Omega^\blob_{R'/A'}
\]
is an isomorphism.
\end{enumerate}
\end{proposition}

\begin{proof} Both statements can be checked locally on $\Spec R$, so we may assume that there is an \'etale map $A[\ul T^{\pm 1}] = A[T_1^{\pm 1},\ldots,T_d^{\pm 1}]\to R$. In that case, Lemma~\ref{lem:dRWEtale} shows
\[
W_r\Omega^n_{R/A} = W_r(R)\otimes_{W_r(A[\ul T^{\pm 1}])} W_r\Omega^n_{A[\ul T^{\pm 1}]/A}\ ,
\]
and similarly
\[
W_r\Omega^n_{R'/A'} = W_r(R')\otimes_{W_r(A'[\ul T^{\pm 1}])} W_r\Omega^n_{A'[\ul T^{\pm 1}]/A'}\ .
\]
From Theorem~\ref{theorem_LZ1}, Lemma~\ref{lemma_tor_independence_1} and Remark~\ref{rem:idealsgeneralcase}, we see that $W_r\Omega^n_{A[\ul T^{\pm 1}]/A}$ is Tor-independent from $W_r(A')$ over $W_r(A)$, and
\[
W_r\Omega^n_{A[\ul T^{\pm 1}]/A}\otimes_{W_r(A)} W_r(A')\cong W_r\Omega^n_{A'[\ul T^{\pm 1}]/A'}\ .
\]
As $W_r(R)$ is flat over $W_r(A[\ul T^{\pm 1}])$, we see that $W_r\Omega^n_{R/A}$ is Tor-independent from $W_r(A')$ over $W_r(A)$, and
\[\begin{aligned}
W_r\Omega^n_{R'/A'} &= W_r(R')\otimes_{W_r(A'[\ul T^{\pm 1}])} W_r\Omega^n_{A'[\ul T^{\pm 1}]/A'}\\
&= W_r(R)\otimes_{W_r(A[\ul T^{\pm 1}])} W_r(A'[\ul T^{\pm 1}])\otimes_{W_r(A'[\ul T^{\pm 1}])} W_r\Omega^n_{A'[\ul T^{\pm 1}]/A'}\\
&= W_r(R)\otimes_{W_r(A[\ul T^{\pm 1}])} W_r\Omega^n_{A'[\ul T^{\pm 1}]/A'} \\
&= W_r(R)\otimes_{W_r(A[\ul T^{\pm 1}])} W_r\Omega^n_{A[\ul T^{\pm 1}]/A}\otimes_{W_r(A)} W_r(A') \\
&= W_r\Omega^n_{R/A}\otimes_{W_r(A)} W_r(A') \ ,
\end{aligned}\]
using Theorem~\ref{thm:Wittetale} in the second step.
\end{proof}

\newpage

\section{The comparison with de~Rham--Witt complexes}

In this section, we will give the proof of part (iv) of Theorem~\ref{thm:AOmegavsdRWLocal}:

\begin{theorem}\label{thm:AOmegavsdRWLocalPart4} Let $R$ be a small formally smooth $\roi$-algebra. Then for $r\geq 1$, $i\geq 0$, there is a natural isomorphism
\[
H^i(\widetilde{W_r\Omega}_R)\cong W_r\Omega^{i,\cont}_{R/\roi}\{-i\}\ .
\]
\end{theorem}

We start with a general construction that starts from a complex like $R\Gamma_\sub{pro\'et}(X,\bb A_{\inf,X})$ and produces the structure of an $F$-$V$-procomplex. In this way, we define first the elaborate structure of an $F$-$V$-procomplex on the left side, and then show that the resulting universal map is an isomorphism (that is then automatically compatible with the extra structure).

\subsection{Constructing $F$-$V$-procomplexes}\label{section_constructing_Witt}
Let $S$ be a perfectoid ring and $\xi$ a generator of $\ker\theta: \bb A_\inf(S)\to S$ which satisfies $\theta_r(\xi)=V(1)$ for all $r\ge 1$; in particular, $\theta_r(\tilde\xi) = p$ for all $r\ge 1$. Let $D\in D(\bb A_\inf(S))$ be a commutative algebra object, with $H^i(D)=H^i(D/\xi)=0$ for $i<0$, equipped with a $\phi$-linear automorphism $\phi_D: D\isoto D$. Note that by assumption $H^0(D)$ is $\xi$-torsion-free, and thus also $\phi^r(\xi)$-torsion-free for all $r\in \bb Z$; in particular, it is $\tilde\xi_r$-torsion-free for all $r\geq 1$, and so $H^i(D/\tilde\xi_r)=0$ for $i<0$.

\subsubsection{First construction}\label{subsection_first_construction}
We now present a construction of (essentially) an $F$-$V$-procomplex from $D$. It is interesting to see the rather elaborate structure of an $F$-$V$-procomplex emerge from the rather simple input that is $D$. It will turn out that this preliminary construction must be refined, which will be done in the next subsection.

For each $r\ge 1$ we may form the algebra $D\dotimes_{\bb A_\inf(S),\tilde\theta_r} W_r(S)$ over $W_r(S)=\bb A_\inf(S)/\tilde\xi_r$ and take its cohomology
\[
\cal W_r^*(D)_\sub{pre}:=H^*(D\dotimes_{\bb A_\inf(S)}\bb A_\inf(S)/\tilde\xi_r)
\]
to form a graded $W_r(S)$-algebra. Equipping these cohomology groups with the Bockstein differential $d:\cal W_r^n(D)_\sub{pre}\to\cal W_r^{n+1}(D)_\sub{pre}$ associated to
\[
0\To D\dotimes_{\bb A_\inf(S)} \bb A_\inf(S)/\tilde\xi_r\xTo{\tilde\xi_r}D\dotimes_{\bb A_\inf(S)}\bb A_\inf(S)/\tilde\xi_r^2\To D\dotimes_{\bb A_\inf(S)}\bb A_\inf(S)/\tilde\xi_r\To 0
\]
makes $\cal W_r^*(D)_\sub{pre}$ into a differential graded $W_r(S)$-algebra.

Now let
\[\begin{aligned}
R'&: \cal W_{r+1}^*(D)_\sub{pre}\to\cal W_r^*(D)_\sub{pre}\\
F&: \cal W_{r+1}^*(D)_\sub{pre}\to\cal W_r^*(D)_\sub{pre}\\
V&: \cal W_r^*(D)_\sub{pre}\to\cal W_{r+1}^*(D)_\sub{pre}
\end{aligned}\]
be the maps of graded $W_r(S)$-modules induced respectively by
\[\begin{aligned}
D\dotimes_{\bb A_\inf(S)}\bb A_\inf(S)/\tilde\xi_{r+1}&\xTo{\phi_D^{-1}} D\dotimes_{\bb A_\inf(S)}\bb A_\inf(S)/\tilde\xi_r\\
D\dotimes_{\bb A_\inf(S)}\bb A_\inf(S)/\tilde\xi_{r+1}&\xTo{\mathrm{can.\ proj.}} D\dotimes_{\bb A_\inf(S)}\bb A_\inf(S)/\tilde\xi_r\\
D\dotimes_{\bb A_\inf(S)}\bb A_\inf(S)/\tilde\xi_r&\xTo{\varphi^{r+1}(\xi)} D\dotimes_{\bb A_\inf(S)}\bb A_\inf(S)/\tilde\xi_{r+1},
\end{aligned}\]
c.f. Lemma~\ref{lemma_theta_r_diagrams}. Instead of $R'$, we will be primarily interested in
\[
R:=\tilde\theta_r(\xi)^n R': \cal W_{r+1}^n(D)_\sub{pre}\to\cal W_r^n(D)_\sub{pre}\ .
\]

\begin{proposition}\label{proposition_first_construction}
The groups $\cal W_r^n(D)_\sub{pre}$, together with the $F$, $R$, $V$, $d$ and multiplication maps, satisfy the following properties.
\begin{enumerate}\itemsep0pt
\item $\cal W_r^\blob(D)_\sub{pre}$ is a differential graded $W_r(S)$-algebra, and satisfies the (anti)commutativity $xy = (-1)^{|x||y|} yx$ for homogeneous elements $x,y$ of degree $|x|, |y|$;
\item $R'$ is a homomorphism of graded $W_r(S)$-algebras, and $R$ is a homomorphism of differential graded $W_r(S)$-algebras;
\item $V$ is additive, commutes with both $R'$ and $R$, and satisfies $V(F(x)y)=xV(y)$;
\item $F$ is a homomorphism of graded rings and commutes with both $R'$ and $R$;
\item $FdV=d$;
\item $FV$ is multiplication by $p$.
\end{enumerate}
\end{proposition}

We note that in general, $\cal W_r^\blob(D)_\sub{pre}$ may fail to be a commutative differential graded algebra, as the equation $x^2=0$ for $|x|$ odd may fail (if $A$ is $2$-adic).

\begin{proof} Part (i) is formal.

(ii): $R'$ is a homomorphism of graded rings by functoriality; the same is true of $R$ since it is twisted by increasing powers of an element. Moreover, the commutativity of
\[\xymatrix{
0\ar[r] & \bb A_\inf(S)/\tilde\xi_{r+1}\ar[r]^-{\tilde\xi_{r+1}}\ar[d]_{\xi \varphi^{-1}} & \bb A_\inf(S)/\tilde\xi_{r+1}^2\ar[r]\ar[d]_{\phi^{-1}} & \bb A_\inf(S)/\tilde\xi_{r+1}\ar[r]\ar[d]_{\phi^{-1}} & 0\\
0\ar[r] & \bb A_\inf(S)/\tilde\xi_r\ar[r]^-{\tilde\xi_r} & \bb A_\inf(S)/\tilde\xi_r^2\ar[r] & \bb A_\inf(S)/\tilde\xi_r\ar[r] & 0
}\]
and functoriality of the resulting Bocksteins implies that
\[\xymatrix{
\cal W_{r+1}^n(D)_\sub{pre}\ar[r]^d\ar[d]_{R'}&\cal W_{r+1}^{n+1}(D)_\sub{pre}\ar[d]^{\tilde\theta_r(\xi)R'} \\
\cal W_r^n(D)_\sub{pre}\ar[r]_d&\cal W_r^{n+1}(D)_\sub{pre} \\
}\]
commutes; hence $d$ commutes with $R$.

(iii): $V$ is clearly additive, and it commutes with $R'$ since it already did so before taking cohomology; it therefore also commutes with $R$. Secondly, $V(F(x)y)=xV(y)$ follows by tensoring the commutative diagram below with $D$ over $\bb A_\inf(S)$ (resp. with $D\otimes D$ over $\bb A_\inf(S)\otimes \bb A_\inf(S)$ on the left), and passing to cohomology:
\[\xymatrix{
\bb A_\inf(S)/\tilde\xi_{r+1}\otimes \bb A_\inf(S)/\tilde\xi_{r+1}\ar[rrr]^{\sub{mult}} &&& \bb A_\inf(S)/\tilde\xi_{r+1}\\
\bb A_\inf(S)/\tilde\xi_{r+1}\otimes \bb A_\inf(S)/\tilde\xi_{r}\ar[u]^{\mathrm{id}\otimes \varphi^{r+1}(\xi)}\ar[d]_{\mathrm{can.\ proj.}\otimes\op{id}}&\\
\bb A_\inf(S)/\tilde\xi_{r}\otimes \bb A_\inf(S)/\tilde\xi_{r}\ar[rrr]^{\sub{mult}} &&& \bb A_\inf(S)/\tilde\xi_r\ar[uu]_{\varphi^{r+1}(\xi)}
}\]

(iv): $F$ is a graded ring homomorphism, and it commutes with $R^\prime$ by definition, and then also with $R$.

(v): This follows by tensoring the commutative diagram below with $D$ over $\bb A_\inf(S)$, and looking at the associated boundary maps on cohomology:

\[\xymatrix{
0\ar[r] & \bb A_\inf(S)/\tilde\xi_r\ar[r]^-{\tilde\xi_r}\ar@{=}[d] & \bb A_\inf(S)/\tilde\xi_r^2\ar[r]\ar^{\varphi^{r+1}(\xi)}[d] & \bb A_\inf(S)/\tilde\xi_r\ar^{\varphi^{r+1}(\xi)}[d]\ar[r] & 0\\
0\ar[r] & \bb A_\inf(S)/\tilde\xi_r\ar[r]^-{\tilde\xi_{r+1}} & \bb A_\inf(S)/\tilde\xi_r \tilde\xi_{r+1} \ar[r] & \bb A_\inf(S)/\tilde\xi_{r+1}\ar[r] & 0 \\
0\ar[r] & \bb A_\inf(S)/\tilde\xi_{r+1}\ar[r]^-{\tilde\xi_{r+1}}\ar^{\mathrm{can.\ proj.}}[u] & \bb A_\inf(S)/\tilde\xi_{r+1}^2\ar[r]\ar[u] & \bb A_\inf(S)/\tilde\xi_{r+1}\ar[r]\ar@{=}[u] & 0
}\]

(vi): This is a consequence of the assumption that $\tilde\theta_r(\varphi^{r+1}(\xi))=p$ for all $r\geq 1$ (which is equivalent to $\theta_r(\tilde\xi)=p$ for $r\geq 1$).
\end{proof}

Now suppose further that there exists an $S$-algebra $B$ and $W_r(S)$-algebra homomorphisms $\lambda_r:W_r(B)\to \cal W_r^0(D)$ which are compatible with $R,F,V$, i.e., such that the diagrams
\[
\xymatrix{
W_{r+1}(B) \ar[d]_R \ar[r]^-{\lambda_{r+1}} & H^0(D/\tilde\xi_{r+1})\ar[d]^{\varphi_D^{-1}}\\
W_r(B) \ar[r]^-{\lambda_r} & H^0(D/\tilde\xi_r)
}
\
\xymatrix{
W_{r+1}(B) \ar[d]_F \ar[r]^-{\lambda_{r+1}} & H^0(D/\tilde\xi_{r+1})\ar[d]^{\mathrm{can.\ proj.}}\\
W_r(B) \ar[r]^-{\lambda_r} & H^0(D/\tilde\xi_r)
}
\
\xymatrix{
W_{r+1}(B) \ar[r]^-{\lambda_{r+1}} & H^0(D/\tilde\xi_{r+1})\\
W_r(B) \ar[u]^{V} \ar[r]^-{\lambda_r} & H^0(D/\tilde\xi_r)\ar[u]_{\varphi^{r+1}(\xi)}
}
\]
commute, and which satisfy the Teichm\"uller rule $Fd\lambda_{r+1}([b])=\lambda_r([b])^{p-1}d\lambda_r([b])$ for $b\in B$, $r\ge 1$. Moreover, assume that $\cal W_r^\blob(D)_\sub{pre}$ is a commutative differential graded algebra; the only remaining issue here being the equation $x^2=0$ for $|x|$ odd.

Then the data $(\cal W_r^\blob(D)_\sub{pre},R,V,F,\lambda_r)$ form an $F$-$V$-procomplex for $B$ over $S$, and so there exist unique maps of differential graded $W_r(S)$-algebras $\lambda_r^\blob:W_r\Omega^\blob_{B/A}\to\cal W_r^\blob(D)_\sub{pre}$ which are compatible with $R,F,V$ and satisfy $\lambda_r^0=\lambda_r$.

\begin{remark}[The need to improve the construction]
Unfortunately, from the surjectivity of the restriction maps for $W_r\Omega^\blob_{B/S}$ and the definition of the restriction map for $\cal W_r^\blob(D)_\sub{pre}$, we see that
\[
\Im\lambda_r^n\subset \bigcap_{s\ge 1}\Im(\cal W_{r+s}^n(D)_\sub{pre}\xto{R^s}\cal W_r^n(D)_\sub{pre})\subset \bigcap_{s\ge 1}\tilde\theta_r(\xi_s)^n\cal W_r^n(D)_\sub{pre}\ ,
\]
where the right side is in practice much smaller than $\cal W_r^n(D)_\sub{pre}$. Hence $\cal W_r^\blob(D)_\sub{pre}$ is too large in applications: in the next section we will modify its construction to cut it down by a carefully controlled amount of torsion.
\end{remark}

\subsubsection{Improvement}\label{subsection_improvement}
Let $D\in D(\bb A_\inf(S))$ be an algebra as above, equipped with a Frobenius isomorphism $\varphi_D: D\isoto D$. Moreover, we assume that there is a system of primitive $p$-power roots of unity $\zeta_{p^r}\in S$, and $S$ is $p$-torsion-free, so we are in the situation of Proposition~\ref{proposition_roots_of_unity} above. This gives rise to the element $\epsilon=(1,\zeta_p,\zeta_{p^2},\ldots)\in S^\flat$, and $\mu = [\epsilon]-1\in \bb A_\inf(S)$, which is a non-zero-divisor. We let $\xi = \mu/\varphi^{-1}(\mu)$, which satisfies the assumption $\theta_r(\xi) = V(1)$ for all $r\ge 1$. Finally, we assume that $H^0(D)$ is $\mu$-torsion-free.

We can now refine the construction of $\cal W_r^\blob(D)_\sub{pre}$ in the previous section by replacing $D$ by the algebra $L\eta_\mu D$ over $\bb A_\inf(S)$, on which $\phi_D$ induces a $\phi$-linear map $\phi_D: L\eta_\mu D\quis L\eta_{\tilde\xi} (L\eta_\mu D)\to L\eta_\mu D$ (as $L\eta_{\tilde\xi} L\eta_\mu = L\eta_{\tilde\xi \mu} = L\eta_{\phi(\mu)}$). Moreover, there is a natural map $L\eta_\mu D\to D$ by Lemma~\ref{lem:LetaCoconnectiveMap}, and the diagram
\[\xymatrix{
L\eta_\mu D\ar[d]\ar[r]^-{\phi_D}& L\eta_\mu D\ar[d]\\
D\ar[r]^-{\phi_D}&D
}\]
commutes.

More precisely, we consider the cohomology groups
\[
\cal W_r^n(D):=H^n(L\eta_\mu D\dotimes_{\bb A_\inf(S)}\bb A_\inf(S)/\tilde\xi_r)\ .
\]
Equipped with the Bockstein differential, they form a differential graded $W_r(S)$-algebra as before (satisfying the Leibniz rule, and the anticommutativity $xy = (-1)^{|x||y|} yx$, but not necessarily $x^2=0$ for $|x|$ odd), and the map $L\eta_\mu D\to D$ induces a morphism of differential graded $W_r(S)$-algebras
\[
i:\cal W_r^\blob(D)\To\cal W_r^\blob(D)_\sub{pre}\ .
\]

Moreover, letting $F:\cal W_{r+1}^n(D)\to\cal W_r^n(D)$ and $V:\cal W_r^n(D)\to\cal W_{r+1}^n(D)$ be the maps induced respectively by
\[
L\eta_\mu D\dotimes_{\bb A_\inf(S)}\bb A_\inf(S)/\tilde\xi_{r+1}\xTo{\mathrm{can.\ proj.}} L\eta_\mu D\dotimes_{\bb A_\inf(S)}\bb A_\inf(S)/\tilde\xi_r
\]
and
\[
L\eta_\mu D\dotimes_{\bb A_\inf(S)}\bb A_\inf(S)/\tilde\xi_r\xTo{\varphi^{r+1}(\xi)} L\eta_\mu D\dotimes_{\bb A_\inf(S)}\bb A_\inf(S)/\tilde\xi_{r+1},
\]
it is clear that $i$ commutes with $F$ and $V$. It is more subtle to define $R:\cal W_{r+1}^n(D)\to\cal W_r^n(D)$; in the proof below, we give a ``point-set level'' construction based on picking an actual model of $D$ as a complex. It is not clear to us whether the construction is independent of the choice of this model, so we impose the following assumption which helps us prove independence; it is verified in our applications.

\begin{assumption}\label{ass:torsionfree} For all $r\geq 1$, $n\geq 0$, the group $\cal W_r^n(D)$ is $p$-torsion-free.
\end{assumption}

\begin{proposition}\label{prop:improvedconstruction} Assume that Assumption~\ref{ass:torsionfree} is verified. Then the following statements hold.
\begin{enumerate}
\item The differential graded $W_r(S)$-algebra $\cal W_r^\bullet(D)$ is commutative; in particular, it satisfies $x^2=0$ for $|x|$ odd.
\item For all $r\geq 1$, $n\geq 0$, the map $\cal W_r^n(D)\to \cal W_r^n(D)_\sub{pre}$ is injective.
\item The maps $F, R: \cal W_{r+1}^n(D)_\sub{pre}\to \cal W_r^n(D)_\sub{pre}$, $V: \cal W_r^n(D)_\sub{pre}\to W_{r+1}^n(D)_\sub{pre}$ and $d: \cal W_r^n(D)_\sub{pre}\to \cal W_r^{n+1}(D)_\sub{pre}$ induce (necessarily unique) maps $F, R: \cal W_{r+1}^n(D)\to \cal W_r^n(D)$, $V: \cal W_r^n(D)\to W_{r+1}^n(D)$ and $d: \cal W_r^n(D)\to \cal W_r^{n+1}(D)$. In the case of $F$, $V$ and $d$, these agree with the maps described above.
\item The map $R: \cal W_{r+1}^\bullet(D)\to R_\ast \cal W_r^\bullet(D)$ is a map of differential graded $W_{r+1}(A)$-algebras, the map $F: \cal W_{r+1}^\bullet(D)\to F_\ast \cal W_r^\bullet(D)$ is a map of graded $W_{r+1}(A)$-algebras, the map $V: F_\ast \cal W_r^\bullet(D)\to \cal W_{r+1}^\bullet(D)$ is a map of graded $W_{r+1}(A)$-modules, and the identities $RF=FR$, $RV=VR$, $V(F(x)y)=xV(y)$, $FV=p$ and $FdV=d$ hold.
\item Assume that $B$ is an $S$-algebra equipped with $W_r(S)$-algebra maps $\lambda_r: W_r(B)\to \cal W_r^0(D)$ for $r\geq 1$, compatible with $F$, $R$ and $V$. Then the Teichm\"uller identity
\[
F d\lambda_{r+1}([b]) = \lambda_r([b])^{p-1} d\lambda_r([b])
\]
holds true for all $x\in B$, $r\geq 1$. In particular, $\cal W_r^\bullet(D)$ forms an $F$-$V$-procomplex for $B/S$, and there is an induced map
\[
\lambda_r^\bullet: W_r\Omega^\bullet_{B/S}\to \cal W_r^\bullet(D)
\]
of differential graded algebras for $r\geq 1$, compatible with the $F$, $R$ and $V$ maps.
\end{enumerate}
\end{proposition}

\begin{proof} For (i), we only need to verify that $x^2=0$ for $|x|$ odd, which under the standing assumption follows from $2x^2=0$, which is a consequence of the anticommutativity.

For part (ii), the statement does not depend on the algebra structure of $D$, so we may assume that $D\in D^{[0,n+1]}(\bb A_\inf(S))$ by passing to a truncation; note that this does not change $\cal W_r^n(D)_\sub{pre} = H^n(D/\tilde\xi_r)$ or $\cal W_r^n(D) = H^n((L\eta_\mu D)/\tilde\xi_r)$ for any $r$. Then there are maps $D\to L\eta_\mu D$, $L\eta_\mu D\to D$ whose composite in either direction is multiplication by $\mu^{n+1}$ by Lemma~\ref{lem:mapLetabackforth}. Since $\mu$ divides $p^r$ modulo $\tilde\xi_r$ by Proposition~\ref{proposition_roots_of_unity} (iv), the kernel of the map
\[
H^n((L\eta_\mu D)/\tilde\xi_r)\to H^n(D/\tilde\xi_r)
\]
is $p$-torsion. By our assumption, $\cal W_r^n(D) = H^n((L\eta_\mu D)/\tilde\xi_r)$ is $p$-torsion-free, so we get the desired injectivity.

In part (iii), it is clear that the $d$, $F$ and $V$ maps defined above commute with the corresponding maps on $\cal W_r^n(D)_\sub{pre}$. It remains to handle the case of $R$, so fix $n\geq 0$. Note that the definition of $R$ depends only on $D\in D(\bb A_\inf(S))$ with the automorphism $\phi_D: D\isoto D$, but not on the algebra structure of $D$. We may assume that $D\in D^{[0,n+1]}(\bb A_\inf(S))$, and then pick a bounded above representative $D^\blob$ of $D$ by projective $\bb A_\inf(S)$-modules. Then $\phi_D: D\to D$ can be represented by a map $\phi_{D^\blob}: D^\blob\to D^\blob$. Replacing $D^\blob$ by the homotopy colimit of $D^\blob$ under $\phi_{D^\blob}$, we can assume that $D^\blob$ is a bounded above complex of flat $\bb A_\inf(S)$-modules, on which there is a $\phi$-linear automorphism $\phi_{D^\blob}: D^\blob\isoto D^\blob$.

Now pick an element $\bar{\alpha}\in \cal W_{r+1}^n(D) = H^n((\eta_\mu D^\blob)/\tilde\xi_{r+1})$. This can be represented by an element $\alpha\in \mu^n D^n$ with $d\alpha = \tilde\xi_{r+1} \beta$ for some $\beta\in \mu^{n+1} D^{n+1}$. The element $\alpha^\prime = \tilde\xi^n \alpha\in \phi(\mu)^n D^n$ satisfies
\[
d\alpha^\prime = \tilde\xi^n\tilde\xi_{r+1} \beta\in \varphi(\tilde\xi_r) \tilde\xi^{n+1} \mu^{n+1}  D^{n+1}= \varphi(\tilde\xi_r) \varphi(\mu)^{n+1} D^{n+1}\ ,
\]
so that $\alpha^\prime\in (\eta_{\phi(\mu)} D)^n$. Thus, $R(\alpha) := \varphi_{D^n}^{-1}(\alpha^\prime)\in (\eta_\mu D)^n$, and it satisfies
\[
d(R(\alpha)) = \varphi_{D^{n+1}}^{-1}(d\alpha^\prime)\in \varphi_{D^{n+1}}^{-1}(\varphi(\tilde\xi_r) \varphi(\mu)^{n+1} D^{n+1}) = \tilde\xi_r \mu^{n+1} D^{n+1}\ ,
\]
so that in fact $d(R(\alpha))=0\in (\eta_{\phi(\mu)} D)^{n+1}/\tilde\xi_r$. This shows that $R(\alpha)\mod \tilde\xi_r$ induces an element of $H^n((\eta_\mu D^\blob)/\tilde\xi_r)$. One checks that under the inclusion $\cal W_r^n(D)\hookrightarrow \cal W_r^n(D)_\sub{pre}$, this is the image of $\bar{\alpha}$ under $R$.

In part (iv), all statements follow formally from the results for $\cal W_r^n(D)_\sub{pre}$, and (ii).

Finally, in part (v), the Teichm\"uller identity always holds after multiplication by $p$, cf.~Remark~\ref{remark_p-torsion-free_Teichmuller}, so that by our assumption, it holds on the nose.
\end{proof}

Note also that the map $\cal W_r^n(D)\to \cal W_r^n(D)_\sub{pre}$ has image in $\tilde\theta_r(\mu)^n \cal W_r^n(D)_\sub{pre}$, and is an isomorphism if $n=0$.

\begin{remark} Assume in addition that for all $r\geq 1$, the natural map
\[
(L\eta_\mu D)/\tilde\xi_r\to L\eta_\mu(D/\tilde\xi_r)
\]
is a quasi-isomorphism, as is the case for $D=R\Gamma_\sub{pro\'et}(X,\bb A_{\inf,X})$ by Theorem~\ref{thm:AOmegavsdRW} (i), where $X$ is the generic fibre of $\frak X=\Spf R$ for a small formally smooth $\roi$-algebra $R$. In that case, the image of
\[
\cal W_r^n(D) = H^n((L\eta_\mu D)/\tilde\xi_r)\isoto H^n(L\eta_\mu(D/\tilde\xi_r))\to H^n(D/\tilde\xi_r) = \cal W_r^n(D)_\sub{pre}
\]
is exactly $\tilde\theta_r(\mu)^n \cal W_r^n(D)_\sub{pre}$. Indeed, in general the image of $H^n(L\eta_f C)\to H^n(C)$, for $C\in D^{\geq 0}$ with $H^0(C)$ being $f$-torsion-free, is given by $f^n H^n(C)$. This makes it easy to see that $R$ preserves $\cal W_r^n(D)$. Moreover, one can give a different description of the restriction map, as follows. Indeed, composing the map
\[
\cal W_r^n(D) = H^n(L\eta_\mu(D/\tilde\xi_r))\to H^n(L\eta_\mu(D/\tilde\xi_r))/H^n(L\eta_\mu(D/\tilde\xi_r))[\tilde\xi] = H^n(L\eta_{\tilde\xi}L\eta_\mu(D/\tilde\xi_r))
\]
with
\[\begin{aligned}
H^n(L\eta_{\tilde\xi}L\eta_\mu(D/\tilde\xi_r))&= H^n(L\eta_{\phi(\mu)}(D/\tilde\xi_r))\to H^n(L\eta_{\phi(\mu)} (D/\phi(\tilde\xi_{r-1})))\\
&\cong^{\phi^{-1}} H^n(L\eta_\mu(D/\tilde\xi_{r-1})) = \cal W_{r-1}^n(D)
\end{aligned}\]
defines the restriction map.
\end{remark}

\subsection{A realization of the de~Rham--Witt complex of the torus}

Let $\roi = \roi_K\subset K$ be the ring of integers in a perfectoid field $K$ of characteristic $0$ containing all $p$-power roots of unity; we fix a choice of $\zeta_{p^r}\in \roi$, giving rise to the elements $\epsilon=(1,\zeta_p,\ldots)\in \roi^\flat$, $\mu=[\epsilon]-1\in A_\inf=W(\roi^\flat)$ and $\xi = \mu/\phi^{-1}(\mu)$ as usual.

Consider the Laurent polynomial algebra
\[
A_\inf[\ul U^{\pm 1/p^\infty}] := A_\inf[U_1^{\pm 1/p^\infty},\ldots,U_d^{\pm 1/p^\infty}]\ .
\]
It admits an action of $\bb Z^d = \bigoplus_{i=1}^d \gamma_i^{\bb Z}$, where the element $\gamma_i$ acts by sending $U_i^{1/p^r}$ to $[\epsilon]^{1/p^r} U_i^{1/p^r}$, and $U_j^{1/p^r}$ to $U_j^{1/p^r}$ for $j\neq i$. We consider
\[
D=R\Gamma(\bb Z^d,A_\inf[\ul U^{\pm 1/p^\infty}])\in D(A_\inf)\ ,
\]
which is a commutative algebra in $D(A_\inf)$. Note that $H^i(D)=0$ for $i<0$, and $H^0(D)\subset A_\inf[\ul U^{\pm 1/p^\infty}]$ is torsion-free. We will see below in Theorem~\ref{thm:dRWtorusprecise} that $D$ satisfies Assumption~\ref{ass:torsionfree}; thus, we may apply the constructions of Section~\ref{section_constructing_Witt}. Our goal is to prove the following theorem.

\begin{theorem}\label{thm:dRWtorus} There are natural isomorphisms
\[
\cal W_r^n(D) = H^n(L\eta_\mu D\dotimes_{A_\inf} A_\inf/\tilde\xi_r) \cong W_r\Omega^n_{\roi[T_1^{\pm 1},\ldots,T_d^{\pm 1}]/\roi}\ ,
\]
compatible with the $d$, $F$, $R$, $V$ and multiplication maps.
\end{theorem}

We begin by computing $L\eta_\mu D$. The result will turn out to be the $q$-de~Rham complex
\[
q\op-\Omega_{A_\inf[\ul U^{\pm1}]/A_\inf}^\blob = \bigotimes_{i=1}^d \left(A_\inf[U_i^{\pm 1}]\To A_\inf[U_i^{\pm 1}]\dlog U_i\right),\quad U_i^k\mapsto [k]_q U_i^k\dlog U_i
\]
from Example~\ref{ex:qdR}, where $q=[\epsilon]$, the tensor product is taken over $A_\inf$, and $[k]_q=\tfrac{q^k-1}{q-1}$ is the $q$-analogue of the integer $k\in \bb Z$.

Note that there is a standard Koszul complex computing $D$, namely the complex
\[
D^\blob = K_{A_\inf[\ul U^{\pm 1/p^\infty}]}(\gamma_1-1,\ldots,\gamma_d-1)\ .
\]
Recall also that there is a Frobenius automorphism $\phi_D$ of $D$, coming from the automorphism of $A_\inf[\ul U^{\pm 1/p^\infty}]$ which is the Frobenius of $A_\inf$, and sends $U_i$ to $U_i^p$ for all $i=1,\ldots,d$. This automorphism $\phi_D$ of $D$ lifts to an automorphism $\phi_{D^\blob}$ of $D^\blob$, given by acting on each occurence of $A_\inf[\ul U^{\pm 1/p^\infty}]$. Note that $D^\blob$ is a complex of free $A_\inf$-modules, so that one can use it to compute $L\eta_\mu D$.

\begin{proposition}\label{prop:IdentqdeRham} There is a natural injective quasi-isomorphism
\[\begin{aligned}
{[\epsilon]\op-\Omega^\blob_{A_\inf[\ul U^{\pm 1}]/A_\inf}}&=\eta_{q-1} K_{A_\inf[\ul U^{\pm 1}]}(\gamma_1-1,\ldots,\gamma_d-1)\\
\to \eta_\mu D^\blob&=\eta_{q-1} K_{A_\inf[\ul U^{\pm 1/p^\infty}]}(\gamma_1-1,\ldots,\gamma_d-1)\ .
\end{aligned}\]
Moreover, the natural map
\[
(L\eta_\mu D)/\tilde\xi_r\to L\eta_{[\zeta_{p^r}]-1} (D/\tilde\xi_r)
\]
is a quasi-isomorphism.
\end{proposition}

\begin{proof} This is an easier version of Lemma~\ref{lem:qdRvsAOmega}. Note that $A_\inf[\ul U^{\pm 1/p^\infty}]$ is naturally $\bb Z[\tfrac 1p]^d$-graded, and this grading extends to the complex $D^\blob$, giving a decomposition
\[
D^\blob = \bigoplus_{a:\{1,\ldots,d\}\to \bb Z[\frac 1p]} \bigotimes_{i=1}^d \left(A_\inf\cdot U_i^{a(i)}\xTo{\gamma_i - 1} A_\inf\cdot U_i^{a(i)}\right)\ .
\]
Here, the complex
\[
\left(A_\inf\cdot U_i^{a(i)}\xTo{\gamma_i - 1} A_\inf\cdot U_i^{a(i)}\right)\cong (A_\inf\xTo{[\epsilon]^{a(i)} - 1} A_\inf) = K_{A_\inf}([\epsilon]^{a(i)}-1)\ ,
\]
so that
\[
D^\blob = \bigoplus_{a:\{1,\ldots,d\}\to \bb Z[\frac 1p]} K_{A_\inf}([\epsilon]^{a(1)}-1,\ldots,[\epsilon]^{a(d)}-1)\ .
\]
Observe that if $k\not\in \bb Z$, then $[\epsilon]^k-1$ divides $\mu = [\epsilon]-1$; indeed, this is clear for $[\epsilon]^{1/p^r}-1$, and in general if $k=j/p^r$ with $j\in \bb Z\setminus p\bb Z$, then $[\epsilon]^k-1$ differs from $[\epsilon]^{1/p^r}-1$ by a unit. On the other hand, if $k\in \bb Z$, then $\mu=[\epsilon]-1$ divides $[\epsilon]^k-1$, with quotient $[k]_q$, where $q=[\epsilon]$.

Now, we distinguish two cases. If $a(i)\not\in \bb Z$ for some $i$, then
\[
\eta_\mu K_{A_\inf}([\epsilon]^{a(1)}-1,\ldots,[\epsilon]^{a(d)}-1)
\]
is acyclic by Lemma~\ref{lemma_on_Koszul_1}. On the other hand, if $a(i)\in \bb Z$ for all $i$, then by the same lemma,
\[
\eta_\mu K_{A_\inf}([\epsilon]^{a(1)}-1,\ldots,[\epsilon]^{a(d)}-1) = K_{A_\inf}([a(1)]_q,\ldots,[a(d)]_q)
\]
where $q=[\epsilon]$. Assembling the summands for $a:\{1,\ldots,d\}\to \bb Z$ gives precisely $[\epsilon]\op-\Omega^\blob_{A_\inf[\ul U^{\pm 1}]/A_\inf}$.

The final statement follows by repeating the calculation after base extension along $\tilde\theta_r: A_\inf\to W_r(\roi)$.
\end{proof}

It will be useful to have an a priori description of the groups
\[
\cal W_r^n(D) = H^n((L\eta_\mu D)/\tilde\xi_r)\ .
\]

\begin{lemma}\label{lemma_explicit_decomposition_of_Wtor}
For each $n\ge 0$ there is an isomorphism of $W_r(\roi)$-modules
\[
\cal W_r^n(D)\cong \bigoplus_{a:\{1,\dots,d\}\to p^{-r}\bb Z} W_{r-u(a)}(\roi)^{\binom dn}
\]
where $u(a)$ is as in Section~\ref{subsec:dRWLaurent}. In particular, $\cal W_r^n(D)$ is $p$-torsion-free.
\end{lemma}

\begin{proof} Using the interpretation of $L\eta_\mu D$ as a $q$-de~Rham complex from Proposition~\ref{prop:IdentqdeRham}, we have
\[
L\eta_\mu D\simeq \bigoplus_{a:\{1,\ldots,d\}\to \bb Z} \bigotimes_{i=1}^d (A_\inf\cdot U_i^{a(i)}\xTo{[a(i)]_q} A_\inf\cdot U_i^{a(i)}) = \bigoplus_{a:\{1,\ldots,d\}\to \bb Z} K_{A_\inf}([a(1)]_q,\ldots,[a(d)]_q)\ ,
\]
where as usual $q=[\epsilon]$. Taking the base change along $\tilde\theta_r: A_\inf\to W_r(\roi)$, we get
\[
L\eta_\mu D/\tilde\xi_r\simeq \bigoplus_{a:\{1,\dots,d\}\to \bb Z}K_{W_r(\roi)}\left(\tfrac{[\zeta_{p^r}^{a(1)}]-1}{[\zeta_{p^r}]-1},\ldots,\tfrac{[\zeta_{p^r}^{a(d)}]-1}{[\zeta_{p^r}]-1}\right).
\]
Since each element on the right side is divisible by $\tfrac{[\zeta_{p^r}^{u(p^{-r} a)}]-1}{[\zeta_{p^r}]-1}$, and at least one element agrees with it up to a unit, it follows from Lemma~\ref{lemma_on_Koszul_2} (ii) that the Koszul complex in the summand on the right side has cohomology
\[
\Ann_{W_r(\roi)}\left(\tfrac{[\zeta_{p^r}^{u(p^{-r} a)}]-1}{[\zeta_{p^r}]-1}\right)^{\binom{d-1}n}\oplus \left(W_r(\roi)/\tfrac{[\zeta_{p^r}^{u(p^{-r} a)}]-1}{[\zeta_{p^r}]-1} W_r(\roi)\right)^{\binom{d-1}{n-1}}.
\]
This is isomorphic to $W_{r-u(p^{-r} a)}(\roi)^{\binom dn}$ by Corollary~\ref{corollary_roots_of_unity} (iii). Renaming $p^{-r} a$ by $a$ finishes the proof.
\end{proof}

\begin{remark}\label{remark_torus_no_p-torsion}
It may be useful to contrast the $p$-torsion-freeness of $\cal W_r^n(D)$ with the cohomology groups $\cal W_r^n(D)_\sub{pre}=H^n(D/\tilde\xi_r)$ obtained without applying $L\eta_\mu$, which are well-known to contain a lot of torsion, coming from the summands parametrized by nonintegral $a$. This is one important motivation for introducing the improved construction of Section~\ref{subsection_improvement}.
\end{remark}

In order to equip $\cal W_r^\bullet(D)$ with the structure of an $F$-$V$-procomplex for $\roi[\ul T^{\pm 1}]/\roi$, it remains to construct the maps $\lambda_r: W_r(\roi)[\ul T^{\pm 1}]\to \cal W_r^0(D)$. This is the content of the next lemma.

\begin{lemma}\label{lemma_existence_of_structure_maps_for_torus}
There is a unique collection of $W_r(\roi)$-algebra morphisms $\lambda_r:W_r(\roi[\ul T^{\pm1}])\to \cal W_r^0(D)$ for $r\ge 1$, which satisfy $\lambda_r([T_i])=U_i$ for $i=1,\dots,d$ and which commute with the $F$, $R$ and $V$ maps. Moreover, each morphism $\lambda_r$ is an isomorphism.
\end{lemma}

\begin{proof} We have
\[
\cal W_r^0(D) = H^0((L\eta_\mu D)/\tilde\xi_r) = H^0(L\eta_{[\zeta_{p^r}]-1} (D/\tilde\xi_r)) = H^0(D/\tilde\xi_r)\ ,
\]
as $H^0(D/\tilde\xi_r)$ is $p$-torsion-free (and thus $[\zeta_{p^r}]-1$-torsion-free). Note that by definition of
\[
D=R\Gamma(\bb Z^d,A_\inf[\ul U^{\pm 1/p^\infty}])\ ,
\]
so
\[
H^0(D/\tilde\xi_r) = W_r(\roi)[\ul U^{\pm 1/p^\infty}]^{\bb Z^d}\ .
\]
where $\gamma_i$ acts by sending $U_i^{1/p^s}$ to $[\zeta_{p^{r+s}}]U_i^{1/p^s}$, and $U_j^{1/p^s}$ to $U_j^{1/p^s}$ for $j\neq i$; let us recall that $[\epsilon]^{1/p^s}\mapsto [\zeta_{p^{r+s}}]$ by Lemma~\ref{lemma_theta}.

Now note that by (a renormalization of) Lemma~\ref{lem:wittpolynomial}, there is an identification
\[
W_r(\roi)[\ul U^{\pm 1/p^\infty}] = W_r(\roi[\ul T^{\pm 1/p^\infty}])\ ,\ U_i^{1/p^s}\mapsto [T_i^{1/p^{r+s}}]\ .
\]
Under this identification, $\gamma_i$ acts by sending $T_i^{1/p^s}$ to $\zeta_{p^s} T_i^{1/p^s}$, and $T_j^{1/p^s}$ to $T_j^{1/p^s}$ for $j\neq i$; in particular, the $\bb Z^d$-action on $W_r(\roi[\ul T^{\pm 1/p^\infty}])$ is induced by an action on $\roi[\ul T^{\pm 1/p^\infty}]$, with invariants $\roi[\ul T^{\pm 1}]$. It follows that
\[
H^0(D/\tilde\xi_r) =W_r(\roi[\ul T^{\pm 1/p^\infty}])^{\bb Z^d} = W_r(\roi[\ul T^{\pm 1}])\ ,
\]
and one verifies compatibility with $F$, $R$ and $V$.
\end{proof}

\begin{corollary} There are unique maps
\[
\lambda_r^\bullet: W_r\Omega^\bullet_{\roi[\ul T^{\pm 1}]/\roi}\to \cal W_r^\bullet(D)
\]
compatible with the $d$, $F$, $R$, $V$ and multiplication maps.
\end{corollary}

\begin{proof} This follows from Proposition~\ref{prop:improvedconstruction} (v), Lemma~\ref{lemma_existence_of_structure_maps_for_torus} and Lemma~\ref{lemma_explicit_decomposition_of_Wtor}.
\end{proof}

We can now state the following more precise form of Theorem~\ref{thm:dRWtorus}.

\begin{theorem}\label{thm:dRWtorusprecise} For each $r\geq 1$, $n\geq 0$, the map
\[
\lambda_r^n: W_r\Omega^n_{\roi[\ul T^{\pm 1}]/\roi}\to \cal W_r^n(D)
\]
is an isomorphism.
\end{theorem}

\begin{proof} We first observe that the source and target of $\lambda_r^n$ look alike. More precisely, both admit natural direct sum decompositions according to functions $a:\{1,\ldots,d\}\to p^{-r} \bb Z$, by Theorem~\ref{theorem_LZ1} and Lemma~\ref{lemma_explicit_decomposition_of_Wtor} respectively, with similar terms. We need to make this observation more explicit.

Define an action of $\bb Z[\tfrac1p]^d = \bigoplus_{i=1}^d \bb Z[\tfrac 1p] \gamma_i$ on $\roi[\ul T^{\pm1}]$ and $A_\inf[\ul U^{\pm 1/p^\infty}]$, via $\roi$- resp.~$A_\inf$-algebra automorphisms, by specifying that $\tfrac 1{p^r} \gamma_i$ acts via $T_i\mapsto \zeta_{p^r} T_i$ and $T_j\mapsto T_j$ for $j\neq i$, resp.~$U_i\mapsto [\epsilon]^{1/p^r} U_i$ and $U_j\mapsto U_j$ for $j\neq i$. In the latter case this action is of course extending the action of $\bb Z^d$ on $\bb A_\inf(S)[\ul U^{\pm1/p^\infty}]$ which has been considered since the start of the section; in the former case, the action of $\bb Z^d\subset \bb Z[\tfrac 1p]^d$ is trivial.

There are induced actions of $\bb Z[\tfrac1p]^d$ on $\cal W_r^\blob(D)$ and $W_r\Omega_{\roi[\ul T^{\pm 1}]/\roi}^\blob$ which are compatible with all extra structure and (thus) commute with $\lambda_r^\blob$.

\begin{lemma}\label{lemma_Witt_decomposition_into_isotypical}
Fix $n\ge 0$. Then the $W_r(\roi)$-modules $\cal W_r^n(D)$ and $W_r\Omega_{\roi[\ul T^{\pm1}]/\roi}^n$ admit unique direct sum decompositions of the form $\bigoplus_{a:\{1,\dots,d\}\to p^{-r}\bb Z}M_a$, where
\begin{enumerate}
\item the decomposition is compatible with the action of $\bb Z[\tfrac1p]^d$, in such a way that $\tfrac 1{p^s}\gamma_i\in\bb Z[\tfrac1p]^d$ acts on $M_a$ as multiplication by $[\zeta_{p^s}^{a(i)}]\in W_r(\roi)$, where $\zeta_{p^s}^{a(i)} := \zeta_{p^{r+s}}^{p^ra(i)}$.
\item each $M_a$ is isomorphic to a finite direct sum of copies of $W_{r-u(a)}(\roi)$;
\item the decompositions are compatible with $\lambda_r^n$.
\end{enumerate}

Moreover, $\lambda_r^n$ is an isomorphism if and only if $\lambda_r^n\otimes_{W_r(\roi)} W_r(k)$ is an isomorphism.
\end{lemma}

\begin{remark} The reader may worry that the description of the action in (i) does not seem to be trivial on $\bb Z^d\subset \bb Z[\tfrac 1p]^d$; however, $[\zeta_{p^r}]^{p^ra(i)}$ does act trivially on $W_{r-u(a)}(\roi)$, and thus on $M_a$ by (ii).
\end{remark}

\begin{proof} In the case of $W_r\Omega^n_{\roi[\ul T^{\pm1}]/\roi}$ we use Theorem~\ref{theorem_LZ1}: by directly analysing Cases 1--3 of the definition of the element $e(x,a,I_0,\dots,I_n)$ one sees that the weight $a$ part of $W_r\Omega_{A[\ul T^{\pm1}]/A}^n$ has property (i); it has property (ii) by Corollary~\ref{corollary_roots_of_unity} (iii). In the case of $\cal W_r^n(D)$, the result follows from Lemma~\ref{lemma_explicit_decomposition_of_Wtor}.

Now, knowing that both sides of the map $\lambda_r^n:W_r\Omega_{\roi[\ul T^{\pm 1}]/\roi}^\blob\to \cal W_r^n(D)$ admit decompositions satisfing (i) and (ii), we claim that the map is automatically compatible with the decompositions. This follows by a standard ``isotypical component argument'' from the observation that if $a:\{1,\dots,d\}\to p^{-r}\bb Z$ is non-zero and $x\in W_j(\roi)$ is an element fixed by $[\zeta_{p^s}^{a(i)}]$ for all $\tfrac 1{p^s}\gamma_i$, then $x=0$; this observation is proved by noting that the hypotheses imply that $x$ is killed in particular by $[\zeta_{p^j}]-1$, which is a non-zero-divisor of $W_j(\roi)$ by Proposition~\ref{proposition_roots_of_unity} (i).

For the final statement, it suffices to prove that if $f: M\to N$ is map between two $W_r(\roi)$-modules $M$, $N$ which are finite direct sums of copies of $W_j(\roi)$ for some fixed $0\leq j\leq r$ (regarded as $W_r(\roi)$-module via $F^{r-j}$), and $f\otimes_{W_r(\roi)} W_r(k)$ is an isomorphism, then so is $f$. To check this, we may assume that $j=r$. Now $W_r(\roi)$ is a local ring, over which a map of finite free modules is an isomorphism if and only if it is an isomorphism over the residue field.
\end{proof}

By the lemma, it is enough to prove that
\[
\overline{\lambda}_r^\blob := \lambda_r^\blob\otimes_{W_r(\roi)} W_r(k): W_r\Omega_{\roi[\ul T^{\pm1}]/\roi}^\blob\otimes_{W_r(\roi)} W_r(k)\to \cal W_r^n(D)\otimes_{W_r(\roi)} W_r(k)=:\cal W_r^n(D)_k
\]
is an isomorphism. By Proposition~\ref{prop:dRWperfectoid}, the source
\[
W_r\Omega_{\roi[\ul T^{\pm1}]/\roi}^\blob\otimes_{W_r(\roi)} W_r(k) = W_r\Omega_{k[\ul T^{\pm 1}]/k}^\blob\ .
\]

\begin{lemma} There is an isomorphism of differential graded algebras
\[
\cal W_r^\blob(D)_k\cong W_r\Omega_{k[\ul T^{\pm 1}]/k}^\blob\ .
\]
In degree $0$, it is compatible with the identification $\lambda_r^0\otimes_{W_r(\roi)} W_r(k): \cal W_r^n(D)_k = W_r(k[\ul T^{\pm 1}])$.
\end{lemma}

Note that we do not a priori claim that this isomorphism is related to $\overline{\lambda}_r^\blob$.

\begin{proof} First, we note that $\cal W_r^n(D)$ is Tor-independent from $W_r(\roi)$ over $W_r(k)$ by part (ii) of the previous lemma and Lemma~\ref{lemma_tor_independence_1}; this implies that
\[
\cal W_r^n(D)_k = H^n(L\eta_\mu D\dotimes_{A_\inf} W_r(k))\ .
\]
This identification is multiplicative; the differential on the left is induced from the Bockstein differential in the triangle
\[
L\eta_\mu D\dotimes_{A_\inf} W_r(k)\xTo{p^r} L\eta_\mu D\dotimes_{A_\inf} W_{2r}(k)\to L\eta_\mu D\dotimes_{A_\inf} W_r(k)\ .
\]
But one has an identification between $L\eta_\mu D\dotimes_{A_\inf} W(k)$ and $\Omega_{W(k)[\ul T^{\pm 1}]/W(k)}^\blob$ by Proposition~\ref{prop:IdentqdeRham}, noting that the $q$-de~Rham complex becomes a usual de~Rham complex over $W(k)$ as $q=[\epsilon]\mapsto 1\in W(k)$. This identification is compatible with the multiplicative structure. We get an isomorphism of graded algebras
\[
\cal W_r^n(D)_k = H^n(L\eta_\mu D\dotimes_{A_\inf} W_r(k)) = H^n(\Omega_{W_r(k)[\ul T^{\pm 1}]/W_r(k)}^\blob) = W_r\Omega^n_{k[\ul T^{\pm 1}]/k}\ ,
\]
using the Cartier isomorphism~\cite[\S III.1.5]{IllusieRaynaud} in the last step. This identification is compatible with the differential, as both are given by the same Bockstein. One checks that in degree $0$, this is the previous identification.
\end{proof}

Thus,
\[
\overline{\lambda}_r^\blob: W_r\Omega_{k[\ul T^{\pm 1}]/k}^\blob\to \cal W_r^\blob(D)_k\cong W_r\Omega_{k[\ul T^{\pm 1}]/k}^\blob
\]
can be regarded as a differential graded endomorphism of $W_r\Omega_{k[\ul T^{\pm 1}]/k}^\blob$, which is the identity in degree $0$. But $W_r\Omega_{k[\ul T^{\pm 1}]/k}^\blob$ is generated in degree $0$, so it follows that the displayed map is the identity, and so $\overline{\lambda}_r^\blob$ is an isomorphism.
\end{proof}

\subsection{Proof of Theorem~\ref{thm:AOmegavsdRWLocalPart4}}
Finally, we can prove part (iv) of Theorem~\ref{thm:AOmegavsdRWLocal}. Recall that this states for a small formally smooth $\roi$-algebra $R$, there is a natural isomorphism
\[
H^i(\widetilde{W_r\Omega}_R)\cong W_r\Omega_{R/\roi}^{i,\cont}\{-i\}\ .
\]

\begin{proof} Note that we have already proved in Lemma~\ref{lem:propWrOmega} and Corollary~\ref{cor:AOmegavsdRWLocalPart2} that all
\[
H^i(\widetilde{W_r\Omega}_R) = H^i(L\eta_\mu D/\tilde\xi_r)
\]
are $p$-torsion-free, where $D=R\Gamma_\sub{pro\'et}(X,\bb A_{\inf,X})$. Moreover, we have
\[
H^0(\widetilde{W_r\Omega}_R) = H^0_\sub{pro\'et}(X,W_r(\hat\roi_X^+)) = W_r(R)\ .
\]
Thus, we can apply the machinery from Section~\ref{section_constructing_Witt} to get canonical maps of $F$-$V$-procomplexes
\[
\lambda_r^\bullet: W_r\Omega^\bullet_{R/\roi}\to H^\bullet(\widetilde{W_r\Omega}_R)\ .
\]
To verify that these are isomorphisms after $p$-completion, we use Elkik's theorem, \cite{Elkik}, to choose a smooth $\roi$-algebra $R_0$ with an \'etale map $\roi[\ul T^{\pm 1}]\to R_0$ which after $p$-completion gives $\roi\langle \ul T^{\pm 1}\rangle\to R$, and consider the diagram
\[\xymatrix{
W_r\Omega^n_{\roi[\ul T^{\pm 1}]/\roi}\otimes_{W_r(\roi[\ul T^{\pm 1}])} W_r(R_0)\ar[r]\ar[d] & H^n(\widetilde{W_r\Omega}_{\roi\langle \ul T^{\pm 1}\rangle})\otimes_{W_r(\roi[\ul T^{\pm 1}])} W_r(R_0)\ar[d]\\
W_r\Omega^{n,\cont}_{R/\roi}\ar[r] &H^n(\widetilde{W_r\Omega}_R)\ .
}\]
Here, the left vertical map is an isomorphism after $p$-completion by Lemma~\ref{lem:dRWEtale} (and the equation $W_r\Omega^{n,\cont}_{R_0/\roi} = W_r\Omega^{n,\cont}_{R/\roi}$), the upper horizontal arrow is an isomorphism after $p$-completion by Theorem~\ref{thm:dRWtorusprecise}, and the right vertical arrow is an isomorphism after $p$-completion by Lemma~\ref{lem:WrOmegaBC}.

Note that in this section, we have regarded roots of unity as fixed; undoing the choice introduces the Breuil--Kisin--Fargues twist, as can easily be checked from the definition of the differential in Section~\ref{section_constructing_Witt} as a Bockstein for
\[
0\to \tilde\xi_r A_\inf/\tilde\xi_r^2 A_\inf\to A_\inf/\tilde\xi_r^2 A_\inf\to A_\inf/\tilde\xi_r A_\inf\to 0\ .
\]

Finally, to see that the isomorphism for $r=1$ agrees with the one from Theorem~\ref{thm:IntegralCartierLocal}, it suffices to check in degree $i=1$ by multiplicativity. It suffices to check on basis elements of $\Omega^{1,\cont}_{R/\roi}$, so one reduces to the case $R=\roi\langle T^{\pm 1}\rangle$, where it is a direct verification.
\end{proof}

\subsection{A variant}

Let us end this section by observing that as a consequence, one gets the following variant. Let $R$ be a small formally smooth $\roi$-algebra as above, with $X$ the generic fibre of $\frak X=\Spf R$.

\begin{proposition}\label{prop:interpretjunk} For any integers $r\geq 1$, $s\geq 0$, there is a natural isomorphism
\[\begin{aligned}
H^i(L\eta_{[\zeta_{p^{r+s}}]-1} R\Gamma_\sub{pro\'et}(X,W_r(\hat\roi_X^+)))&\cong^{\phi^s} H^i(\widetilde{W_{r+s}\Omega}_R\dotimes_{W_{r+s}(\roi)} W_r(\roi))\\
&\cong (W_{r+s}\Omega^{i,\cont}_{R/\roi} / V^r W_s\Omega^{i,\cont}_{R/\roi})\{-i\}\ ,
\end{aligned}\]
where the map $W_{r+s}(\roi)\to W_r(\roi)$ is the restriction map.
\end{proposition}

Note that as $s\to\infty$, the left side becomes almost isomorphic to $H^i_\sub{pro\'et}(X,W_r(\hat\roi_X^+))$, so this gives an interpretation of the ``junk torsion'' (i.e., the cohomology of the terms coming from non-integral exponents $a$ in the computation) in terms of the de~Rham--Witt complex.

\begin{proof} The first isomorphism follows from Lemma~\ref{lem:LetasymmmonWrOmega} applied to $\widetilde{W_{r+s}\Omega}_R$ and $E=W_r(\roi)$ considered as $W_{r+s}(\roi)$-module via restriction, as
\[
W_r(\hat\roi_X^+)\cong W_{r+s}(\hat\roi_X^+)\dotimes_{W_{r+s}(\roi)} W_r(\roi)
\]
by Lemma~\ref{lemma_tor_independence_1}. For the identification with de~Rham--Witt groups, note that there is an exact triangle
\[
\widetilde{W_s\Omega}_R\xTo{\phi^{s+1}(\xi)\cdots \phi^{s+r}(\xi)} \widetilde{W_{r+s}\Omega}_R\to \widetilde{W_{r+s}\Omega}_R\dotimes_{W_{r+s}(\roi)} W_r(\roi)\ ,
\]
as one has a short exact sequence
\[
0\to W_s(\roi)\xTo{\phi^{s+1}(\xi)\cdots \phi^{s+r}(\xi)} W_{r+s}(\roi)\to W_r(\roi)\to 0\ ,
\]
and $\widetilde{W_s\Omega}_R = \widetilde{W_{r+s}\Omega}_R\dotimes_{W_{r+s}(\roi),F^r} W_s(\roi)$. Passing to cohomology, we get a long exact sequence
\[
\ldots\to W_s\Omega^{n,\cont}_{R/\roi}\{-i\}\xTo{V^r} W_{r+s}\Omega^{n,\cont}_{R/\roi}\{-i\}\to H^n(\widetilde{W_{r+s}\Omega}_R\dotimes_{W_{r+s}(\roi)} W_r(\roi))\to \ldots\ .
\]
As $V^r$ is injective (since $F^rV^r=p^r$ and the groups are $p$-torsion-free), this splits into short exact sequences, giving the result.
\end{proof}

\newpage

\section{The comparison with crystalline cohomology over $A_\crys$}\label{sec:cryscomp}

Let $\frak X/\roi$ be a smooth $p$-adic formal scheme, and let $Y=\frak X\times_{\Spf \roi} \Spec \roi/p$ be the fiber modulo $p$ of $\frak X$. Note that this is a large nilpotent thickening of the special fiber $\frak X\times_{\Spf \roi} \Spec k$.

Let $u:(Y/\bb Z_p)_{\crys} \to Y_\sub{Zar} = \frak{X}_\sub{Zar}$ be the canonical map from the (absolute) crystalline site of $Y$ down to the Zariski site. Recall that $A_\crys$ is the universal $p$-adically complete PD thickening of $\roi/p$ (compatible with the PD structure on $\bb Z_p$), so we have $(Y/\bb Z_p)_\crys = (Y/A_\crys)_\crys$, and for psychological reasons we prefer the second interpretation. Recall that we have defined $A\Omega_{\frak X} = L\eta_\mu(R\nu_\ast \bb A_{\inf,X})$, where $X$ is the generic fibre of $\frak X$. Our goal is to prove the following comparison result:

\begin{theorem}\label{thm:cryscomp}
There is a canonical isomorphism
\[
A\Omega_{\frak X} \widehat{\otimes}_{A_\inf} A_{\crys} \simeq Ru_* \roi_{Y/A_\crys}^\crys
\]
in $D(\frak X_\sub{Zar})$. In particular, if $\frak X$ is qcqs, this gives an isomorphism
\[
R\Gamma(\frak X,A\Omega_{\frak X}) \widehat{\otimes}_{A_\inf} A_\crys \simeq R\Gamma_{\crys}(Y/A_\crys)\ .
\]
\end{theorem}

The first step of the proof is to construct the identification locally using a framing of $\mathfrak{X}$ in \S \ref{AcrysCompCoord}. To globalize, in \S \ref{subsec:cancryscomp}, we  reinterpret the previous identification in a choicefree fashion: instead of working with \'etale maps (i.e., the framing) attached to a finite set of units, we work with closed immersions provided by working with  ``all possible units''; this gives a strictly functorial isomorphism, and thus globalizes.

\subsection{The local isomorphism}
\label{AcrysCompCoord}

We start by verifying the assertion in the case $\frak X=\Spf R$ for a small formally smooth $\roi$-algebra $R$, with a fixed framing
\[
\square: \Spf R\to \Spf \roi\langle T_1^{\pm 1},\ldots,T_d^{\pm 1}\rangle = \Spf \roi\langle \ul T^{\pm 1}\rangle\ .
\]
The isomorphism will a priori be noncanonical.

Recall that in this situation we have a formally smooth $A_\inf$-algebra $A(R)^\square$, with $A(R)^\square/\xi = R$; more precisely, it is formally \'etale over $A_\inf\langle \ul U^{\pm 1}\rangle$. The action of $\Gamma = \bb Z_p(1)^d$ which lets the basis vector $\gamma_i\in \Gamma$ act by sending $U_i$ to $[\epsilon]U_i$ and $U_j$ to $U_j$ for $j\neq i$ lifts uniquely to an action on $A(R)^\square$, and we have the $q$-derivatives
\[
\frac{\partial_q}{\partial_q \log(U_i)} = \frac{\gamma_i-1}{[\epsilon]-1}: A(R)^\square\to A(R)^\square\ .
\]
This gives rise to the $q$-de~Rham complex
\[
q\op-\Omega^\blob_{A(R)^\square/A_\inf} = K_{A(R)^\square}\left(\frac{\partial_q}{\partial_q \log(U_1)},\ldots,\frac{\partial_q}{\partial_q \log(U_d)}\right)\ .
\]
On the other hand, we have the usual de~Rham complex
\[
\Omega^\blob_{A(R)^\square/A_\inf} = K_{A(R)^\square}\left(\frac{\partial}{\partial \log(U_1)},\ldots,\frac{\partial}{\partial \log(U_d)}\right)
\]
written using the basis $\dlog(U_1),\ldots,\dlog(U_d)$ of $\Omega^{1,\cont}_{A(R)^\square/A_\inf}$. Also, define the $A_\crys$-algebra $A_\crys(R)^\square = A(R)^\square\hat{\otimes}_{A_\inf} A_\crys$; then
\[
\Omega^\blob_{A(R)^\square/A_\inf}\hat{\otimes}_{A_\inf} A_\crys\cong \Omega^\blob_{A_\crys(R)^\square/A_\crys}\ .
\]

Before go on, we observe a few facts about elements of $A_\crys$.

\begin{lemma}\label{lem:elementsofAcrys} Let $q=[\epsilon]\in A_\crys$ as usual.
\begin{enumerate}
\item The element $\frac{(q-1)^{p-1}}p$ lies in $A_\crys$, and it is topologically nilpotent in the $p$-adic topology.
\item For any $n\geq 0$, the element $\frac{(q-1)^n}{(n+1)!}$ lies in $A_\crys$, and converges to $0$ in the $p$-adic topology as $n\to\infty$.
\item The element $\log(q)\in A_\crys$ can be written as $\log(q) = (q-1)u$ for some unit $u\in A_\crys$.
\end{enumerate}
In particular, the elements
\[
\frac{\log(q)^n}{n!(q-1)} = u^n \frac{(q-1)^{n-1}}{n!}
\]
lie in $A_\crys$, and converge to $0$ in the $p$-adic topology.
\end{lemma}

\begin{proof} For (i), note that $\xi = \frac{q-1}{q^{1/p}-1}$ lies in $\ker\theta: A_\crys\to\roi$. Thus, $\tfrac{\xi^p}p\in A_\crys$. On the other hand, $\xi^p\equiv (q-1)^{p-1}\mod p$, already in $A_\inf$. Therefore, $\tfrac{(q-1)^{p-1}}p\in A_\crys$. As it lies in the kernel of $\theta: A_\crys\to\roi$, it has divided powers, and in particular is topologically nilpotent.

For part (ii), let $m=\lfloor \tfrac n{p-1}\rfloor$. Then by part (i)
\[
\frac{(q-1)^{m(p-1)}}{p^m}\in A_\crys
\]
converges to $0$ as $m\to\infty$. But note that the $p$-adic valuation of $(n+1)!$ is bounded by $m$. Thus,
\[
\frac{(q-1)^n}{(n+1)!} = (q-1)^{n-m(p-1)} \frac{p^m}{(n+1)!} \frac{(q-1)^{m(p-1)}}{p^m}\in A_\crys\ ,
\]
where each factor lies in $A_\crys$, and the last factor converges to $0$.

Finally, for part (iii), write
\[
\log(q) = (q-1)(1 + \sum_{n\geq 1} \frac{(-1)^n}{n+1} (q-1)^n)\ .
\]
We claim that the sum is topologically nilpotent. As all the terms with $n\geq p$ are divisible by $p$ by (ii), it suffices to check that the terms with $n<p$ are topologically nilpotent. This is clear if $n<p-1$, as $q-1$ is topologically nilpotent, and for $n=p-1$, it follows from (i).
\end{proof}

In the sequel, the following torsionfreeness property of $A_\crys(R)^\square$ shall be used freely. 

\begin{lemma}
The ring $A_{\crys}(R)^\square$ is $\mu$-torsionfree.
\end{lemma}
\begin{proof}
Let us first note that $A_\crys(R)^\square$ is the derived $p$-completion of a free $A_\crys$-module: by base change (first along $A_{\inf} \to A_\crys$ and then along $A_\inf \to \mathcal{O}$), this reduces to showing that $R$ is the derived $p$-completion of a free $\mathcal{O}$-module. As both $R$ and $\mathcal{O}$ are $p$-torsionfree, it suffices to show any smooth $\mathcal{O}/p$-algebra $S$ is free as a $\mathcal{O}/p$-module. But this is clear: any such $S$ is the base change of a smooth algebra defined over an artinian subring of $\mathcal{O}/p$ (as $\mathcal{O}/p$ is $0$-dimensional and thus a direct limit of its artinian subrings), and any flat module over an artinian ring is free.

We now show $A_\crys(R)^\square[\mu] = 0$ by showing that the Koszul complex $\mathrm{Kos}(A_\crys(R)^\square; \mu)$ has no $H_1$. By the previous paragraph, the complex $\mathrm{Kos}(A_\crys(R)^\square; \mu)$ is the derived $p$-completion of a complex of the form $\mathrm{Kos}(\oplus_I A_\crys; \mu)$ for some set $I$. Since $A_\crys$ is itself $\mu$-torsionfree, we have $\mathrm{Kos}(\oplus_I A_\crys; \mu) \simeq \oplus_I A_\crys/\mu$ via projection to $H^0$. We are thus reduced to showing that $H_1(\widehat{\oplus}_I A_\crys/\mu) = 0$, where $\widehat{\oplus}$ denotes the derived $p$-completion of the direct sum. By general properties of derived completions of abelian groups, this $H_1$ is identified with the $p$-adic Tate module $T_p(\oplus_I A_\crys/\mu)$. But the obvious map $T_p(\oplus_I A_\crys/\mu) \to \prod_I T_p(A_\crys/\mu)$ is trivially injective, and $T_p(A_\crys/\mu) = 0$ as $A_\crys/\mu$ is already derived $p$-complete, so we get the desired vanishing. 
\end{proof}

The following formula expresses the $q$-derivative in terms of the derivative, via a Taylor expansion.

\begin{lemma}\label{lem:qDerTaylor} One has an equality of endomorphisms of $A_\crys(R)^\square$,
\[\begin{aligned}
\frac{\partial_q}{\partial_q \log(U_i)} &= \frac{\log(q)}{q-1} \frac{\partial}{\partial \log(U_i)} + \frac{\log(q)^2}{2(q-1)} \left(\frac{\partial}{\partial \log(U_i)}\right)^2+\ldots\\
&= \sum_{n\geq 1} \frac{\log(q)^n}{n!(q-1)} \left(\frac{\partial}{\partial \log(U_i)}\right)^n\ .
\end{aligned}\]
\end{lemma}

\begin{proof} As $A_{\crys}(R)^\square$ is $\mu$-torsionfree, the formula is equivalent to the formula
\[
\gamma_i = \sum_{n\geq 0} \frac{\log(q)^n}{n!} \left(\frac{\partial}{\partial \log(U_i)}\right)^n = \exp(\log(q) \frac{\partial}{\partial \log(U_i)})\ .
\]
To check this formula, we must show that the right side is a well-defined continuous $A_\crys$-algebra endomorphism of $A_\crys(R)^\square$, reducing to the identity on $R=A_\crys(R)^\square\otimes_{A_\crys,\theta} \roi$, and that the identity holds in the case $R=\roi\langle \ul T^{\pm 1}\rangle$ (as these properties determine $\gamma_i$).

The formula is well-defined by Lemma~\ref{lem:elementsofAcrys}. Moreover, it defines a continuous $A_\crys$-linear map. Multiplicativity follows from standard manipulations. Also, after base extension along $\theta: A_\crys\to\roi$, $\log(q)$ vanishes, and the formula reduces to $1=1$. Finally, we need to check the action on the $U_j$ is correct. Certainly, the right side leaves $U_j$ for $j\neq i$ fix. It sends $U_i$ to
\[
\sum_{n\geq 0} \frac{\log(q)^n}{n!} U_i = q U_i\ ,
\]
as desired.
\end{proof}

\begin{corollary}\label{cor:compqdRdR} There is an isomorphism of complexes
\[
q\op-\Omega^\blob_{A(R)^\square/A_\inf}\hat{\otimes}_{A_\inf} A_\crys\cong \Omega^\blob_{A_\crys(R)^\square/A_\crys}
\]
inducing the identity $A(R)^\square\hat{\otimes}_{A_\inf} A_\crys = A_\crys(R)^\square$ in degree $0$.
\end{corollary}

Note that as $A_\crys(R)^\square$ is a (formally) smooth lift of $R/p$ from $\roi/p$ to $A_\crys$, the right side computes $R\Gamma_\crys((\Spec R/p)/A_\crys)$. Also recall from Lemma~\ref{lem:qdRvsAOmega} that $q\op-\Omega^\blob_{A(R)^\square/A_\inf}$ computes $A\Omega_R$. Thus, the proposition verifies the existence of some isomorphism as in Theorem~\ref{thm:cryscomp} in this case. We note that the isomorphism of complexes will not be an isomorphism of differential graded algebras (as the left side is non-commutative, but the right side is commutative).

The isomorphism constructed in the proof will agree with the canonical isomorphism from Theorem~\ref{thm:cryscomp} in the derived category.

\begin{proof} For each $i$, one can write
\[
\frac{\partial_q}{\partial_q \log(U_i)} = \frac{\partial}{\partial \log(U_i)}\left(\frac{\log(q)}{q-1}+\sum_{n\geq 2} \frac{\log(q)^n}{n!(q-1)} \left(\frac{\partial}{\partial \log(U_i)}\right)^{n-1}\right)\ ,
\]
where the second factor is invertible. Indeed, $\frac{\log(q)}{q-1}$ is invertible, and $\frac{\log(q)^n}{n!(q-1)}\in A_\crys$ is topologically nilpotent and converges to $0$ by Lemma~\ref{lem:elementsofAcrys}.

In general, if $g_i$, $i=1,\ldots,d$, are commuting endomorphisms of $M$, and $h_i$, $i=1,\ldots,d$, are automorphisms of $M$ commuting with each other and with the $g_i$, then
\[
K_M(g_1h_1,\ldots,g_dh_d)\cong K_M(g_1,\ldots,g_d)\ .
\]
Applying this in our case with $M=A_\crys(R)^\square$, $g_i = \frac{\partial}{\partial \log(U_i)}$ and $g_ih_i = \frac{\partial_q}{\partial_q \log(U_i)}$, where $h_i$ itself is given by the formula above, we get the result.
\end{proof}

\subsection{The canonical isomorphism}\label{subsec:cancryscomp}

In this subsection, we modify the construction of the previous subsection to construct specific complexes computing $R\Gamma_\crys(\Spec(R/p)/A_\crys,\roi)$ and $A\Omega_R\hat{\otimes}_{A_\inf} A_\crys$, and a map of complexes between them, which is a quasi-isomorphism. These explicit complexes, and the map between them, will be functorial in $R$, and thus globalize.

Let $R/\roi$ be a formally smooth $\roi$-algebra. Assume that $R$ is small, i.e., there is an \'etale map $\Spf R\to \widehat{\bb G}_m^d$. Here, we assume additionally that there is a closed immersion $\Spf R\subset \widehat{\bb G}_m^n$ for some $n\geq d$; let us call such $R$ \emph{very small}. Of course, any formally smooth $\roi$-algebra $R$ is locally on $(\Spf R)_\sub{Zar}$ very small.

The (simple) idea is to extra roots not just of some system of coordinates, but instead of any sufficiently large set of invertible functions on $R$. Thus, fix any finite set $\Sigma\subset R^\times$ of units of $R$ such that the induced map $\Spf R\to \widehat{\bb G}_m^n$, $n=|\Sigma|$, is a closed embedding, and there is some subset of $d$ elements of $\Sigma$ for which the induced map $\Spf R\to \widehat{\bb G}_m^d$ is \'etale. Let $S_\Sigma$ be the group algebra over $A_\crys$ of the free abelian group $\bigoplus_{u\in \Sigma} \bb Z$ generated by the set $\Sigma$; for $u \in \Sigma$, we write $x_u\in S_\Sigma$ for the corresponding variable. This gives a torus $\Spec(S_\Sigma)$ over $A_\crys$. There is an obvious map $S_\Sigma\otimes_{A_\crys} \roi\to R$ sending $x_u$ to $u$, and we get a natural closed immersion $\Spec(R/p)\subset \Spec(S_\Sigma)$ by assumption on $R$. Let $D_\Sigma$ be the $p$-adically completed PD envelope (compatible with the PD structure on $A_\crys$) of $S_\Sigma\to R/p$; as $R/p$ is smooth over $\roi/p$, $D_\Sigma$ is flat over $\bb Z_p$.

Let $S_\Sigma\to S_{\infty,\Sigma}$ be the map on group algebras corresponding to the map
\[
\bigoplus_{u\in \Sigma} \bb Z\to \bigoplus_{u\in \Sigma} \bb Z[\tfrac 1p]
\]
of abelian groups, so there is a well-defined element $x_u^k \in S_{\infty,\Sigma}$ for each $u\in \Sigma$ and $k\in \bb Z[\tfrac 1p]$, extending the obvious meaning in $S_\Sigma$ if $k\in \bb Z$. Using this, let $R_{\infty,\Sigma}$ be the $p$-adic completion of the normalization of $R$ in $(R\otimes_{S_\Sigma} S_{\infty,\Sigma})[\tfrac 1p]$. Note that $R_{\infty,\Sigma}$ is perfectoid.

There is a natural map $S_\Sigma\to \bb A_\inf(R_{\infty,\Sigma})$ sending $x_u$ to $[u^\flat]$, where $u^\flat=(u,u^{1/p},u^{1/p^2},\ldots)\in R_{\infty,\Sigma}^\flat$ is a well-defined element, as we have freely adjoined $p$-power roots of $u$. This extends to a map $D_\Sigma\to A_\crys(R_{\infty,\Sigma}/p)$ by passing to $p$-adically complete PD envelopes. Here, $A_\crys(R_{\infty,\Sigma}/p)$ denotes the universal $p$-adically complete PD thickening of $R_{\infty,\Sigma}/p$ compatible with the PD structure on $\bb Z_p$; equivalently,
\[
A_\crys(R_{\infty,\Sigma}/p) = \bb A_\inf(R_{\infty,\Sigma})\hat{\otimes}_{A_\inf} A_\crys\ .
\]

Let $\Gamma = \prod_{u\in \Sigma} \bb Z_p(1)$ be the corresponding profinite group, so there is a natural $\Gamma$-action on $S_\Sigma$, $D_\Sigma$, $S_{\infty,\Sigma}$ and $R_{\infty,\Sigma}$. Explicitly, if one fixes primitive $p$-power roots $\zeta_{p^r}\in \roi$, giving rise to $[\epsilon]\in A_\inf$, then the generator $\gamma_u\in \Gamma$ corresponding to $u\in \Sigma$ acts on $S_{\infty,\Sigma}$ by fixing $x_v^{1/p^r}\in S$ for $u\neq v\in \Sigma$, and sends $x_u^{1/p^r}$ to $[\epsilon]^{1/p^r} \cdot x_u^{1/p^r}$.

Let $\Lie \Gamma\cong \prod_{u\in \Sigma} \bb Z_p(1)$ denote the Lie algebra of $\Gamma$. In this simple situation of an additive group, this is just the same as $\Gamma$, and there is a natural ``exponential'' isomorphism $e: \Lie \Gamma\cong \Gamma$ (which is just the identity $\prod_{u\in \Sigma} \bb Z_p(1) = \prod_{u\in \Sigma} \bb Z_p(1)$).

\begin{lemma}\label{lem:Liealgacts} There is a natural action of $\Lie \Gamma$ on $D_\Sigma$, via letting $g\in \Lie \Gamma$ (with $\gamma=e(g)\in \Gamma$) act via the derivation
\[
g=\log(\gamma) = \sum_{n\geq 1} \frac{(-1)^{n-1}}n (\gamma-1)^n\ .
\]
One can recover the action of $\Gamma$ on $D_\Sigma$ from the action of $\Lie \Gamma$ by the formula
\[
\gamma = \exp(g) = \sum_{n\geq 0} \frac{g^n}{n!}\ .
\]

Moreover, the action of the basis vector $g_u\in \Lie\Gamma$ corresponding to $u\in \Sigma$ (and a choice of primitive $p$-power roots of unity) is given by $\log([\epsilon]) \frac{\partial}{\partial\log(x_u)}$; recall here that the derivations $\frac{\partial}{\partial\log(x_u)}$ of $S_\Sigma$ extend uniquely to continuous derivations of the $p$-adically completed PD envelope $D_\Sigma$.
\end{lemma}

\begin{proof} Note first that $\gamma-1$ takes values in $([\epsilon]-1)D_\Sigma$. Indeed, acting on $S_\Sigma$, it is clear that $\gamma-1$ takes values in $([\epsilon]-1)S_\Sigma$. Now, if $x\in S_\Sigma$ lies in the kernel of $S_\Sigma\to R/p$ with divided power $\frac{x^n}{n!}\in D_\Sigma$, then $\gamma x = x + ([\epsilon]-1)y$ for some $y\in S_\Sigma$, and thus
\[\begin{aligned}
\gamma\left(\frac{x^n}{n!}\right) = \frac{(x+([\epsilon]-1)y)^n}{n!}& = \sum_{m=0}^n \frac{x^{n-m}}{(n-m)!} \frac{([\epsilon]-1)^m y^m}{m!}\\
&=\frac{x^n}{n!}+([\epsilon]-1) \sum_{m=1}^n \frac{x^{n-m}}{(n-m)!} \frac{([\epsilon]-1)^{m-1}}{m!}y^m\in \frac{x^n}{n!}+([\epsilon]-1)D_\Sigma\ ,
\end{aligned}\]
where we use that $\frac{([\epsilon]-1)^{m-1}}{m!}\in D_\Sigma$ by Lemma~\ref{lem:elementsofAcrys}.

Therefore, the $n$-fold composition $(\gamma-1)^n$ takes values in $([\epsilon]-1)^n D_\Sigma$. The element $\frac{([\epsilon]-1)^n}n$ lies in $D_\Sigma$ and converges to $0$ as $n\to\infty$; this shows that the formula for $\log(\gamma)$ converges to an endomorphism of $D_\Sigma$, which in fact takes values in $([\epsilon]-1)D_\Sigma$. For this last observation, use that in fact $\frac{([\epsilon]-1)^{n-1}}n$ lies in $D_\Sigma$, by Lemma~\ref{lem:elementsofAcrys}. Similarly, using the same lemma, one checks that $\exp(g)$ converges. To verify the identity $\gamma=\exp(g)$, note that $\exp(g)$ defines a continuous $A_\crys$-algebra endomorphism; it is then enough to check the behaviour on the elements $x_u$, which is done as in the proof of Lemma~\ref{lem:qDerTaylor} above.

By uniqueness, the formula for the action of $\Lie \Gamma$ can be checked on $S_\Sigma$. This decomposes into a tensor product of Laurent polynomial algebras in one variable, so it suffices to check the similar assertion for the action of $\bb Z_p(1)$ on $A_\crys[X^{\pm 1}]$. Here,
\[
g(X^i) = \sum_{n\geq 1} \frac{(-1)^{n-1}}n ([\epsilon]^i-1)^n X^i = \log([\epsilon]^i) X^i = i \log([\epsilon]) X^i\ .
\]
\end{proof}

\begin{corollary}\label{cor:LieAlgCohom} Consider the Koszul complex
\[
K_{D_\Sigma}((g_u)_{u\in \Sigma})
\]
corresponding to $D_\Sigma$ and the endomorphisms $g_u$ for all $u\in \Sigma$; it computes the Lie algebra cohomology $R\Gamma(\Lie\Gamma,D_\Sigma)$.
\begin{enumerate}
\item There is a natural isomorphism of complexes
\[
K_{D_\Sigma}((\frac{\partial}{\partial \log(x_u)})_{u\in \Sigma})\cong \eta_\mu K_{D_\Sigma}((g_u)_{u\in \Sigma})\ .
\]
Here, the left side computes $R\Gamma_\crys(\Spec(R/p)/A_\crys,\roi)$.
\item There is a natural isomorphism of complexes
\[
K_{D_\Sigma}((g_u)_{u\in \Sigma})\cong K_{D_\Sigma}((\gamma_u-1)_{u\in \Sigma})\ ,
\]
where the right side computes $R\Gamma_\cont(\Gamma,D_\Sigma)$.
\end{enumerate}

In particular, there is a natural map
\[
\alpha_R^0: K_{D_\Sigma}((\frac{\partial}{\partial \log(x_u)})_{u\in \Sigma})\to \eta_\mu K_{D_\Sigma}((\gamma_u-1)_{u\in \Sigma})\to \eta_\mu K_{A_\crys(R_{\infty,\Sigma}/p)}((\gamma_u-1)_{u\in \Sigma})\ ,
\]
where the source computes $R\Gamma_\crys(\Spec(R/p)/A_\crys,\roi)$.
\end{corollary}

We note that a similar passage between group cohomology and Lie algebra cohomology also appears in the work of Colmez--Niziol, \cite{ColmezNiziolBst}.

Again, the isomorphism in (ii) is not compatible with the structure of differential graded algebras. However, the left side is naturally a commutative differential graded algebra, and one can check that it models the $E_\infty$-algebra $R\Gamma_\cont(\Gamma,D_\Sigma)$.

\begin{proof} By the formula $g_u = \log(\mu) \tfrac{\partial}{\partial\log(x_u)}$ and the observation that $\log(\mu) = \mu v$ for some unit $v\in A_\crys$, cf.~Lemma~\ref{lem:elementsofAcrys}, part (i) follows from Lemma~\ref{lemma_on_Koszul_1}.

For part (ii), one uses that $g_u = (\gamma_u-1) h_u$ for some automorphism $h_u$ of $D_\Sigma$ commuting with everything else, as in the proof of Corollary~\ref{cor:compqdRdR} above.
\end{proof}

The map $\alpha_R^0$ is essentially the map we wanted to construct, but unfortunately we do not know whether the target actually computes $A\Omega_R\hat{\dotimes}_{A_\inf} A_\crys$. The problem is that $A_\crys$ is a rather ill-behaved ring, and notably $A_\crys/\mu$ is not $p$-adically separated. However, we have the following lemma.

\begin{lemma}\label{lem:LetaAcrys} Let $A_\crys^{(m)}\subset A_\crys$ be the $p$-adic completion of the $A_\inf$-subalgebra generated by $\tfrac{\xi^j}{j!}$ for $j\leq m$, so that $A_\crys$ is the $p$-adic completion of $\varinjlim_m A_\crys^{(m)}$.

\begin{enumerate}
\item If $m\geq p^2$, then $\tilde\xi_r = p^r v$ for some unit $v\in A_\crys^{(m)}$, and Lemma~\ref{lem:elementsofAcrys} holds true with $A_\crys^{(m)}$ in place of $A_\crys$.

\item The systems of ideals $(\{x\mid \mu x\in p^r A_\crys^{(m)}\})_r$ and $(p^r A_\crys^{(m)})_r$ are intertwined.

\item The intersection
\[
\bigcap_r \frac{\mu}{\varphi^{-r}(\mu)} A_\crys^{(m)} = \mu A_\crys^{(m)}\ .
\]

\item For any $m\geq p^2$, the natural map
\[
(\eta_\mu K_{\bb A_\inf(R_{\infty,\Sigma})}((\gamma_u-1)_{u\in \Sigma}))\hat{\dotimes}_{A_\inf} A_\crys^{(m)}\to \eta_\mu K_{\bb A_\inf(R_{\infty,\Sigma})\hat{\otimes}_{A_\inf} A_\crys^{(m)}}((\gamma_u-1)_{u\in \Sigma})
\]
is a quasi-isomorphism. Here, the left side computes $A\Omega_R\hat{\dotimes}_{A_\inf} A_\crys^{(m)}$.

\item Under the identification $A_\crys(R_{\infty,\Sigma}/p) = \bb A_\inf(R_{\infty,\Sigma})\hat{\otimes}_{A_\inf} A_\crys$, the map
\[
\alpha_R^0: K_{D_\Sigma}((\frac{\partial}{\partial \log(x_u)})_{u\in \Sigma})\to \eta_\mu K_{A_\crys(R_{\infty,\Sigma}/p)}((\gamma_u-1)_{u\in \Sigma})
\]
of complexes factors canonically over a map of complexes
\[
\alpha_R: K_{D_\Sigma}((\frac{\partial}{\partial \log(x_u)})_{u\in \Sigma})\to \left(\varinjlim_m \eta_\mu K_{\bb A_\inf(R_{\infty,\Sigma})\hat{\otimes}_{A_\inf} A_\crys^{(m)}}((\gamma_u-1)_{u\in \Sigma})\right)^\wedge_p\ ,
\]
where the left side computes $R\Gamma_\crys(\Spec(R/p)/A_\crys)$ as before, and the right side computes $A\Omega_R\hat{\dotimes}_{A_\inf} A_\crys$.
\end{enumerate}
\end{lemma}

\begin{proof} For part (i), the arguments given in the case of $A_\crys$ work as well for $A_\crys^{(m)}$. For parts (ii) and (iii), we approximate the situation by noetherian subrings. More precisely, consider $A_0=\bb Z_p[[T]]\subset A_\inf$, where $T$ is sent to $[\epsilon]^{1/p}$. Then the element $\mu\in A_\inf$ is the image of $T^p-1$, and $\xi$ is the image of $\xi_0=T^{p-1}+\ldots+T+1\in A_0$. One can then define analogues $A_{0,\crys}$, $A_{0,\crys}^{(m)}$ of $A_\crys$ and $A_\crys^{(m)}$; for example, $A_\crys$ is the $p$-adic completion of the PD envelope of $A_0\to A_0/\xi_0$. Then $A_\crys = A_{0,\crys}\hat{\otimes}_{A_0} A_\inf$ and $A_\crys^{(m)} = A_{0,\crys}^{(m)}\hat{\otimes}_{A_0} A_\inf$. As $A_\inf$ is topologically free over $A_0$, it suffices to prove the analogue of (ii) for $A_{0,\crys}^{(m)}$. But $A_{0,\crys}^{(m)}$ is a noetherian ring. Thus, the Artin--Rees lemma for the inclusion $(T^p-1) A_{0,\crys}^{(m)}\subset A_{0,\crys}^{(m)}$ and the $p$-adic topology gives (ii).

Part (iii) is equivalent to the statement
\[
\bigcap_r \frac{\mu}{\varphi^{-r}(\mu)} A_\crys^{(m)}/\mu = 0\ .
\]
But by part (ii), $A_\crys^{(m)}/\mu = \varprojlim_s A_\crys^{(m)}/(\mu,p^s)$, so it suffices to prove the similar statement for $A_\crys^{(m)}/(\mu,p^s)$. Now note that
\[
A_\crys^{(m)}/(\mu,p^s) = A_{0,\crys}^{(m)}/(T^p-1,p^s)\otimes_{A_0/(T^p-1,p^s)} A_\inf/(\mu,p^s)\ .
\]
We claim that more generally, for any $A_0/(T^p-1,p^s)$-module $M$, there are no elements in
\[
M\otimes_{A_0/(T^p-1,p^s)} A_\inf/(\mu,p^s)
\]
that are killed by $\phi^{-r}(\mu)$ for all $r\geq 1$. Assume that $x$ was such an element. In particular, $x$ is killed by $q-1$, so as $A_\inf/(\mu,p^s)$ is flat over $A_0/(T^p-1,p^s)$, $x$ lies in $M^\prime\otimes_{A_0/(T^p-1,p^s)} A_\inf/(\mu,p^s)$, where $M^\prime\subset M$ is the $T-1$-torsion submodule. We can then assume that $M=M^\prime$ is $T-1$-torsion, i.e.~an $A_0/(T-1,p^s) = \bb Z/p^s\bb Z$-module. We can also assume that $px=0$; if not, replace $x$ by $p^ix$ with $i$ maximal such that $p^ix\neq 0$. In that case, we can assume that $M$ is $p$-torsion, and thus an $\bb F_p$-vector space. Finally, it remains to see that
\[
\bb F_p\otimes_{A_0/(T^p-1,p^s)} A_\inf/(\mu,p^s) = A_\inf/(\phi^{-1}(\mu),p) = \roi^\flat/(\epsilon^{1/p}-1)
\]
has no elements killed by all $\epsilon^{1/p^r}-1$, which is clear.

For part (iv), pick an \'etale map $\square: \Spf R\to \Spf \roi\langle T_1^{\pm 1},\ldots,T_d^{\pm 1}\rangle$, corresponding to fixed units $u_1,\ldots,u_d\in \Sigma$; this exists by choice of $\Sigma$. This gives rise to $R_\infty\subset R_{\infty,\Sigma}$, on which the quotient $\prod_{i=1}^d \bb Z_p(1)$ of $\Gamma$ acts.

The proof of Proposition~\ref{prop:AOmegaallthesame} shows that
\[
\eta_\mu K_{\bb A_\inf(R_\infty)}((\gamma_{u_i}-1)_{i=1,\ldots,d})\to \eta_\mu K_{\bb A_\inf(R_{\infty,\Sigma})}((\gamma_u-1)_{u\in \Sigma})
\]
is a quasi-isomorphism (in particular, the right side computes $A\Omega_R$), and the proof of Lemma~\ref{lem:qdRvsAOmega} shows that
\[
(\eta_\mu K_{\bb A_\inf(R_\infty)}((\gamma_{u_i}-1)_{i=1,\ldots,d}))\hat{\dotimes}_{A_\inf} A_\crys^{(m)}\to \eta_\mu K_{\bb A_\inf(R_\infty)\hat{\otimes}_{A_\inf} A_\crys^{(m)}}((\gamma_{u_i}-1)_{i=1,\ldots,d})
\]
is a quasi-isomorphism.

It remains to see that
\[
\eta_\mu K_{\bb A_\inf(R_\infty)\hat{\otimes}_{A_\inf} A_\crys^{(m)}}((\gamma_{u_i}-1)_{i=1,\ldots,d})\to \eta_\mu K_{\bb A_\inf(R_{\infty,\Sigma})\hat{\otimes}_{A_\inf} A_\crys^{(m)}}((\gamma_u-1)_{u\in \Sigma})
\]
is a quasi-isomorphism. This can be proved using Lemma~\ref{lem:LEtaQuasiIsomAinf} (one does not need a variant for $A_\crys^{(m)}$). Let $C^\bullet = K_{\bb A_\inf(R_\infty)\hat{\otimes}_{A_\inf} A_\crys^{(m)}}((\gamma_{u_i}-1)_{i=1,\ldots,d})$ and $D^\bullet = K_{\bb A_\inf(R_{\infty,\Sigma})\hat{\otimes}_{A_\inf} A_\crys^{(m)}}((\gamma_u-1)_{u\in \Sigma})$. Condition (i) is immediate from Faltings' almost purity, and condition (iii) is proved like Lemma~\ref{lem:niceintersection}, using part (iii) of the current lemma. Finally, in order to verify the injectivity condition (ii) of Lemma~\ref{lem:LEtaQuasiIsomAinf}, we will momentarily prove that the map
\[
L\eta_\mu C^\bullet\to R\varprojlim_r (L\eta_\mu(C^\bullet/\tilde\xi_r))
\]
is a quasi-isomorphism, and for each $r$, $L\eta_\mu(C^\bullet/\xi_r)\to L\eta_\mu(D^\bullet/\xi_r)$ is a quasi-isomorphism; the commutative diagram
\[\xymatrix{
L\eta_\mu C^\bullet\ar[r]\ar[d] & L\eta_\mu D^\bullet\ar[d]\\
R\projlim_r (L\eta_\mu(C^\bullet/\tilde\xi_r))\ar[r] & R\projlim_r (L\eta_\mu(D^\bullet/\tilde\xi_r))
}\]
then proves the desired injectivity. Note that Lemma~\ref{lem:LetasymmmonWrOmega} shows that
\[
L\eta_\mu(D^\bullet/\xi_r)\simeq L\eta_\mu(C^\bullet/\xi_r)\simeq \widetilde{W_r\Omega}_R\hat{\dotimes}_{W_r(\roi)} L\eta_\mu(A_\crys^{(m)}/\tilde\xi_r)\simeq A\Omega_R\hat{\dotimes}_{A_\inf} L\eta_\mu(A_\crys^{(m)}/\tilde\xi_r)\ ,
\]
and, as $L\eta_\mu(A_\crys^{(m)}/\tilde\xi_r) = A_\crys^{(m)}/\{x\mid \mu x\in \tilde\xi_r A_\crys^{(m)}\}$, parts (i) and (ii) show that (as $A\Omega_R$ is derived $p$-complete)
\[
R\projlim_r(A\Omega_R\hat{\dotimes}_{A_\inf} L\eta_\mu(A_\crys^{(m)}/\tilde\xi_r)) = A\Omega_R\hat{\dotimes}_{A_\inf} A_\crys^{(m)}\ .
\]

For part (v), one can write $D_\Sigma$ similarly as the $p$-adic completion of the union of $p$-adically complete subrings $D^{(m)}_\Sigma\subset D_\Sigma$, where $D^{(m)}_\Sigma\subset D_\Sigma$ only allows divided powers of order at most $m$. Following the construction of $\alpha_R^0$ through with $D^{(m)}_\Sigma$ in place of $D_\Sigma$ gives, for $m$ large enough, maps from $K_{D^{(m)}_\Sigma}((\frac{\partial}{\partial \log(x_u)})_{u\in \Sigma})$ to $\eta_\mu K_{\bb A_\inf(R_{\infty,\Sigma})\hat{\otimes}_{A_\inf} A_\crys^{(m)}}((\gamma_u-1)_{u\in \Sigma})$. Passing to the direct limit over $m$ and $p$-completing gives the desired map $\alpha_R$.
\end{proof}

To finish the proof of Theorem~\ref{thm:cryscomp}, it remains to prove that $\alpha_R$ is a quasi-isomorphism: Passing to the filtered colimit over all sufficiently large $\Sigma$, all our constructions become strictly functorial in $R$, and thus immediately globalize.

\begin{proposition}\label{prop:alphaRQuasiIsom} The map
\[
\alpha_R: K_{D_\Sigma}((\frac{\partial}{\partial \log(x_u)})_{u\in \Sigma})\to \left(\varinjlim_m \eta_\mu K_{\bb A_\inf(R_{\infty,\Sigma})\hat{\otimes}_{A_\inf} A_\crys^{(m)}}((\gamma_u-1)_{u\in \Sigma})\right)^\wedge_p
\]
is a quasi-isomorphism.
\end{proposition}

\begin{proof} Pick an \'etale map $\square: \Spf R\to \Spf \roi\langle T_1^{\pm 1},\ldots,T_d^{\pm 1}\rangle$ as in the previous proof. We get a diagram
\[\xymatrix{
\Spec R/p\ar@{^(->}[d]\ar[r] & \Spf A_\crys(R)^\square\ar[d]^\square\\
\Spf D_\Sigma\ar[r] & \Spf A_\crys\langle T_1^{\pm 1},\ldots,T_d^{\pm 1}\rangle.
}\]
As $\Spf D_\Sigma$ is a (pro-)thickening of $\Spec R/p$, the infinitesimal lifting criterion for (formally) \'etale maps shows that there is a unique lift $\Spf D_\Sigma\to \Spf A_\crys(R)^\square$ making the diagram commute. One can then redo the construction of $\alpha_R$ using only the coordinates $T_1,\ldots,T_d$, and (using notation from the previous proof) one gets a commutative diagram
\[\xymatrix{
K_{A_\crys(R)^\square}((\frac{\partial}{\partial\log(T_i)})_{i=1,\ldots,d})\ar[rr]\ar[d] &&  \left(\varinjlim_m \eta_\mu K_{\bb A_\inf(R_\infty)\hat{\otimes}_{A_\inf} A_\crys^{(m)}}((\gamma_{u_i}-1)_{i=1,\ldots,d})\right)^\wedge_p\ar[d]\\
K_{D_\Sigma}((\frac{\partial}{\partial \log(x_u)})_{u\in \Sigma})\ar[rr]^{\!\!\!\!\!\!\!\!\!\!\!\!\!\!\!\!\!\!\!\!\!\!\!\!\!\!\!\!\!\alpha_R} &&  \left(\varinjlim_m \eta_\mu K_{\bb A_\inf(R_{\infty,\Sigma})\hat{\otimes}_{A_\inf} A_\crys^{(m)}}((\gamma_u-1)_{u\in \Sigma})\right)^\wedge_p.
}\]
Here, the right vertical map is a quasi-isomorphism, as was proved in the previous proof, and the left vertical map is a quasi-isomorphism, as both compute $R\Gamma_\crys(\Spec(R/p)/A_\crys,\roi)$. Finally, the upper horizontal map is a quasi-isomorphism by Corollary~\ref{cor:compqdRdR} (noting that in this situation, the map
\[
\left(\varinjlim_m \eta_\mu K_{\bb A_\inf(R_\infty)\hat{\otimes}_{A_\inf} A_\crys^{(m)}}((\gamma_{u_i}-1)_{i=1,\ldots,d})\right)^\wedge_p\to \eta_\mu K_{A_\crys(R_\infty/p)}((\gamma_{u_i}-1)_{i=1,\ldots,d})
\]
is a quasi-isomorphism, as both sides compute $A\Omega_R^\square\hat{\dotimes}_{A_\inf} A_\crys$). 
\end{proof}

\subsection{Multiplicative structures}

The previous discussion had the defect that it was not compatible with the structure of differential graded algebras. Let us note that this is a defect of the explicit models we have chosen. More precisely, we claim that the isomorphism of Theorem~\ref{thm:cryscomp} can be made into an isomorphism of (sheaves of) $E_\infty$-$A_\crys$-algebras. For this discussion, we admit that $L\eta$ can be lifted to a lax symmetric monoidal functor on the level of symmetric monoidal $\infty$-categories. Then $A\Omega_R = L\eta_\mu R\Gamma_\sub{pro\'et}(X,\bb A_{\inf,X})$ is an $E_\infty$-$A_\inf$-algebra, and we want to show that
\[
R\Gamma_\crys(\Spec(R/p)/A_\crys,\roi)\cong A\Omega_R\hat{\dotimes}_{A_\inf} A_\crys
\]
as $E_\infty$-$A_\crys$-algebras, functorially in $R$. This implies formally the global case (as the $E_\infty$-structure encodes all the information necessary to globalize).

We want to redo the construction of the previous section by replacing all Koszul complexes computing group cohomology by the $E_\infty$-algebra $R\Gamma_\cont(\Gamma,-)$. This has the advantage of keeping more structure, but the disadvantage that we have no explicit complexes anymore. However, the construction of the map $\alpha_R^0$ in Corollary~\ref{cor:LieAlgCohom} is done in two steps: Part (i) is an isomorphism of commutative differential graded algebras, which gives an isomorphism of $E_\infty$-algebras. On the other hand, part (ii) can be checked without reference to explicit models, and indeed one can check directly that the commutative differential graded algebra $K_{D_\Sigma}((g_u)_{u\in \Sigma})$ models the $E_\infty$-algebra $R\Gamma_\cont(\Gamma,D_\Sigma)$. These steps work exactly the same with $D^{(m)}_\Sigma$ in place of $D_\Sigma$. As the final map
\[
L\eta_\mu R\Gamma_\cont(\Gamma,D^{(m)}_\Sigma)\to L\eta_\mu R\Gamma_\cont(\Gamma,\bb A_\inf(R_{\infty,\Sigma})\hat{\otimes}_{A_\inf} A_\crys^{(m)})
\]
is a map of $E_\infty$-algebras, this gives (by passing to the filtered colimit over all sufficiently large $\Sigma$) the desired functorial map of $E_\infty$-$A_\crys$-algebras
\[
\alpha_R: R\Gamma_\crys(\Spec(R/p)/A_\crys,\roi)\to A\Omega_R\hat{\dotimes}_{A_\inf} A_\crys\ ,
\]
which we have already proved to be an equivalence.

\newpage

\section{Rational $p$-adic Hodge theory, revisited}\label{sec:ratPAdicHodgeNew}

Let $C$ be an algebraically closed complete extension of $\bb Q_p$, with ring of integers $\roi$ and residue field $k$ as usual. The goal of this section is to prove a de Rham comparison theorem for rigid spaces over $C$. As the continuous projection $B_\dR^+ \to C$ does not admit a continuous section, the usual formulation of the de Rham comparison theorem does not make sense in this case. On the other hand, as the map $B_\dR^+ \to C$ can be regarded as pro-infinitesimal thickening of $C$, one has a well-behaved crystalline cohomology for proper smooth $C$-schemes taking values in $B_\dR^+$-modules, and deforming usual de Rham cohomology along $B_\dR^+ \to C$. It is then natural to wonder if this deformation of de Rham cohomology to $B_\dR^+$ can be compared with \'etale cohomology. The primary goal of this section is to explain how to construct this deformation more generally for proper smooth rigid spaces,  and to prove the de Rham comparison theorem:

\begin{theorem}\label{thm:ratpadicHodgeC} Let $X$ be a proper smooth adic space over $C$. Then there are cohomology groups $H^i_\crys(X/B_\dR^+)$ which come with a canonical isomorphism
\[
H^i_\crys(X/B_\dR^+)\otimes_{B_\dR^+} B_\dR\cong H^i_\sub{\'et}(X,\bb Z_p)\otimes_{\bb Z_p} B_\dR\ .
\]
In case $X=X_0\hat{\otimes}_K C$ arises via base change from some complete discretely valued extension $K$ of $\bb Q_p$ with perfect residue field, this isomorphism agrees with the comparison from Theorem~\ref{thm:ratPAdicHodge} above, under the identification
\[
H^i_\crys(X/B_\dR^+) = H^i_\dR(X_0)\otimes_K B_\dR^+
\]
of Remark~\ref{rmk:BdRCohDiscValued} below.
\end{theorem}

Our strategy is to define a cohomology theory $R\Gamma_\crys(X/B_\dR^+)$ for any smooth adic space $X$ by imitating one possible definition of crystalline cohomology, namely, in terms of de Rham complexes of formal completions of embeddings of $X$ into smooth spaces over $B_\dR^+$; in order to get a strictly functorial theory, we simply take the colimit over all possible choices of embeddings.

More precisely, for any smooth affinoid $C$-algebra $R$ equipped with a sufficiently large finite subset $\Sigma$ of units in $R^\circ$, we consider the canonical surjective map $B_\dR^+ \langle (X_u^{\pm 1})_{u \in \Sigma} \rangle \to R$, viewed roughly as (dual to) an embedding of $\Spa(R, R^\circ)$ into a smooth rigid space over $B_\dR^+$; the precise language to set this up involves taking a limit over $n$ of ``rigid geometry over $B_\dR^+/\xi^n$'', and is set up in Lemma \ref{lem:BdRmodxin}. The completion $D_{\Sigma}(R)$ of $B_\dR^+ \langle (X_u^{\pm 1})_{u \in \Sigma} \rangle$ along the kernel of this map is then shown to be a well-behaved object, roughly analogous to the formal completion of the afore-mentioned embedding; the precise statement is recorded in Lemma~\ref{lem:DRnice}, and the proof entails approximating our smooth $C$-algebra $R$ in terms of smooth algebras defined over a much smaller base $A$. The de Rham complex $\Omega^\bullet_{D_{\Sigma}(R)/B_\dR^+}$ is then shown to be independent of $\Sigma$ up to quasi-isomorphism in Lemma~\ref{lem:crysdRcomp}; the key point here is that $\Omega^\bullet_{D_{\Sigma}(R)/B_\dR^+}/\xi$ is canonically identified up to quasi-isomorphism with $\Omega^\bullet_{R/C}$, which is obviously independent of $\Sigma$. Taking a filtered colimit over all possible choices of $\Sigma$ then gives a functorial (in $R$) complex, independent of all choices. For a general smooth adic space $X$ over $C$, this construction gives a presheaf of complexes on a basis of $X$ whose hypercohomology is (by definition) $R\Gamma_\crys(X/B_{\dR}^+)$; when $X$ is proper, this theory is then shown to satisfy Theorem~\ref{thm:ratpadicHodgeC}.

\begin{remark}
It is probably possible to develop a full-fledged analogue of the crystalline site in this context (which actually reduces to the infinitesimal site), replacing the usual topologically nilpotent thickening $W(k)\to k$ by $B_\dR^+\to C$. Our somewhat pedestrian approach, via building strictly functorial complexes on affinoid pieces, is engineered to be compatible with the $A_\crys$-comparison of the previous section.
\end{remark}

As an application of the construction of the $B_\dR^+$-cohomology theory, we can prove degeneration of the Hodge--Tate spectral sequence, \cite{ScholzeSurvey}, in general. 

\begin{theorem}\label{thm:hodgetate} 
Let $X$ be a proper smooth adic space over $C$. 
\begin{enumerate}
\item (Conrad-Gabber) The Hodge--de~Rham spectral sequence
\[
E_1^{ij} = H^j(X,\Omega_{X/C}^i)\Rightarrow H^{i+j}_\dR(X)
\]
degenerates at $E_1$.
\item The Hodge--Tate spectral sequence
\[
E_2^{ij} = H^i(X,\Omega_{X/C}^j)(-j)\Rightarrow H^{i+j}_\sub{\'et}(X,\bb Z_p)\otimes_{\bb Z_p} C
\]
degenerates at $E_2$.
\end{enumerate}
\end{theorem}

Both parts of Theorem~\ref{thm:hodgetate} rely on the work \cite{ConradGabber} of Conrad-Gabber yielding a ``Lefschetz principle'' for proper rigid spaces. In fact, the degeneration of the Hodge--de~Rham spectral sequence follows directly from (and was one of the motivations for) \cite{ConradGabber}; the degeneration of the Hodge--Tate spectral sequence also uses the $B_\dR^+$-cohomology theory.   The work \cite{ConradGabber} relies on establishing a relative version of classical results in the deformation theory of proper varieties. Since the classical version actually suffices for our application, we give a self-contained exposition of the relevant statements in \S \ref{ss:ConradGabber}.

\subsection{The $B_\dR^+$-cohomology of affinoids} 
In this section, we explain how to construct the $B_\dR^+$-cohomology for certain smooth affinoids. To do so, we need some basic lemmas on ``rigid geometry over $B_\dR^+/\xi^n$''. Note that $B_\dR^+/\xi^n=A_\inf/\xi^n[\tfrac 1p]$ is a complete Tate-$\bb Q_p$-algebra.

\begin{lemma}\label{lem:BdRmodxin} Let $R$ be a complete Tate-$B_\dR^+/\xi^n$-algebra.
\begin{enumerate}
\item The following conditions on $R$ are equivalent.
\begin{enumerate}
\item[{\rm (a)}] There is a surjective map $B_\dR^+/\xi^n\langle X_1,\ldots,X_m\rangle\to R$ for some $m$.
\item[{\rm (b)}] The algebra $R/\xi$ is topologically of finite type over $C$.
\end{enumerate}
 In case they are satisfied, we say that $R$ is topologically of finite type over $B_\dR^+/\xi^n$.

\item If $R$ is topologically of finite type over $B_\dR^+/\xi^n$, the following further properties are satisfied.
\begin{enumerate}
\item[{\rm (a)}] The ring $R$ is noetherian.
\item[{\rm (b)}] Any ideal $I\subset R$ is closed.
\end{enumerate}

\item A $p$-adically complete $p$-torsion free $A_\inf/\xi^n$-algebra $R_0$ is by definition topologically of finite type if there is a surjective map $A_\inf/\xi^n\langle X_1,\ldots,X_m\rangle\to R_0$ for some $m$. In this case, the following properties are satisfied.
\begin{enumerate}
\item[{\rm (a)}] The ring $R_0$ is coherent.
\item[{\rm (b)}] Any ideal $I\subset R_0$ such that $R_0/I$ is $p$-torsion free is finitely generated.
\item[{\rm (c)}] The Tate-$B_\dR^+/\xi^n$-algebra $R=R_0[\tfrac 1p]$ is topologically of finite type.
\end{enumerate}

\item If $R$ is topologically of finite type over $B_\dR^+/\xi^n$, then there exists a ring of definition $R_0\subset R$ such that $R_0$ is topologically of finite type over $A_\inf/\xi^n$.
\end{enumerate}
\end{lemma}

We note that all assertions are well-known for $n=1$, i.e.~over $B_\dR^+/\xi = C$. We will use this freely in the proof.

\begin{proof} For (i), clearly condition (a) implies (b). On the other hand, given a surjection
\[
C\langle X_1,\ldots,X_m\rangle\to R/\xi\ ,
\]
one can lift the $X_i$ arbitrarily to $R$; they will still be powerbounded as $R\to R/\xi$ has nilpotent kernel. Thus, one gets a map $B_\dR^+/\xi^n\langle X_1,\ldots,X_m\rangle\to R$, which is automatically surjective.

In part (ii), it is enough to prove these assertions in the case $R=B_\dR^+/\xi^n\langle X_1,\ldots,X_m\rangle$. This is a successive square-zero extension of the noetherian ring $C\langle X_1,\ldots,X_m\rangle$ by finitely generated ideals, and thus noetherian itself. We will prove part (b) at the end.

For part (iii), part (c) is clear, and the other assertions reduce to $R_0=A_\inf/\xi^n\langle X_1,\ldots,X_m\rangle$. This is a successive square-zero extension of the coherent ring $\roi\langle X_1,\ldots,X_m\rangle$ by finitely presented ideals, and thus coherent itself, cf.~Lemma~\ref{LiftCoherence}. For part (b), let more generally $M$ be a finitely generated $p$-torsion free $R_0$-module; we want to prove that $M$ is finitely presented. Applying this to $M=R_0/I$ gives (b). Let $\bar{M}=\mathrm{im}(M\to M/\xi[\tfrac 1p])$. Then $\bar{M}$ is a $p$-torsion free finitely generated $R_0/\xi$-module, and thus finitely presented as $R_0/\xi$-module, and thus also as $R_0$-module, cf.~Lemma~\ref{QuotCoherent} (i). Therefore, $M^\prime = \ker(M\to \bar{M})$ is also a finitely generated $p$-torsion free $R_0$-module. But $M^\prime$ is killed by $\xi^{n-1}$: If $m\in M^\prime$, then $p^k m\in \xi M$ for some $k$, and then $\xi^{n-1} p^k m\in \xi^n M=0$. As $M$ is $p$-torsion free, this implies that $\xi^{n-1} m=0$. We see that $M^\prime$ is a finitely generated $p$-torsion free $R_0/\xi^{n-1}$-module, so induction on $n$ finishes the proof.

In part (iv), if $B_\dR^+/\xi^n\langle X_1,\ldots,X_m\rangle\to R$ is surjective, then the image of
\[
A_\inf/\xi^n\langle X_1,\ldots,X_m\rangle\to R
\]
defines a ring of definition $R_0\subset R$ which is topologically of finite type.

Finally, for (ii) (b), let $I\subset R$ be any (necessarily finitely generated) ideal, and $R_0\subset R$ a ring of definition which is topologically of finite type. Let $I_0 = I\cap R_0$. Then $R_0/I_0\subset R/I$ is $p$-torsion free, and thus $I_0$ is finitely generated over $R_0$. This implies that $I_0$ is $p$-adically complete, and thus $I_0\subset R_0$ is closed, and so is $I\subset R$.
\end{proof}

\begin{definition}
\label{DefVSBdR}
Let $R$ be a smooth Tate $C$-algebra of dimension $d$. We say that $R$ is {\em very small} if there exist finite subsets $\{T_1,....,T_d\} \subset \Sigma \subset (R^\circ)^\ast$ with the following properties:
\begin{enumerate}
\item The map 
\[C (\langle X_u^{\pm 1})_{u \in \Sigma} \rangle \to R\]
defined by $X_u \mapsto u$ is surjective.
\item On adic spectra, the map 
\[ \mathrm{Spa}(R,R^\circ) \to \mathbb{T}^d := \mathrm{Spa}(C \langle T_1^{\pm 1},...,T_d^{\pm 1} \rangle, \mathcal{O} \langle T_1^{\pm 1},....,T_d^{\pm 1} \rangle)\] 
by the $T_i$'s is \'etale and factors as a composition of rational embeddings and finite \'etale maps.
\end{enumerate}
\end{definition}

Note that the subset $\Sigma \subset (R^\circ)^\ast$ appearing in the definition of very smallness is not fixed; in particular, we are allowed to enlarge $\Sigma$ without affecting either of $(i)$ or $(ii)$ above.  Let us explain how to construct pro-infinitesimal thickenings of very small rings relative to $B_\dR^+$.

\begin{construction}
\label{ConsThick1}
Fix very small $R$ and subset $\{T_1,...,T_d\} \subset \Sigma \subset (R^\circ)^\ast$ as in Definition~\ref{DefVSBdR}. We have a surjective map
\[
B_\dR^+\langle (X_u^{\pm 1})_{u\in \Sigma}\rangle\to R\ ,
\]
sending $X_u^{\pm 1}$ to $u^{\pm 1}$. Here, for any finite set $I$,
\[
B_\dR^+\langle (X_i^{\pm 1})_{i\in I}\rangle := \varprojlim_r B_\dR^+/\xi^r\langle (X_i^{\pm 1})_{i\in I}\rangle\ .
\]
For $v\in \Sigma$, there are natural commuting continuous derivations
\[
\frac{\partial}{\partial \log(X_v)} = X_v \frac{\partial}{\partial X_v}: B_\dR^+\langle (X_u^{\pm 1})_{u\in\Sigma}\rangle\to B_\dR^+\langle (X_u^{\pm 1})_{u\in\Sigma}\rangle\ .
\]
Now let $D_\Sigma(R)$ be the completion of $B_\dR^+\langle (X_u^{\pm 1})_{u\in \Sigma}\rangle$ with respect to the ideal
\[
I(R) = \ker(B_\dR^+\langle (X_u^{\pm 1})_{u\in \Sigma}\rangle\to R)\ .
\]
By Lemma~\ref{lem:BdRmodxin}, all powers $I(R)^n\subset B_\dR^+\langle (X_u^{\pm 1})_{u\in \Sigma}\rangle$ are closed, so that with its natural topology, $D_\Sigma(R)$ is a complete and separated $B_\dR^+$-algebra. The derivations $\frac{\partial}{\partial \log(X_u)}$ for $u\in \Sigma$ extend continuously to $D_\Sigma(R)$.
\end{construction}

To proceed further, we shall need the following noetherian approximation lemma that roughly says that a very small smooth Tate $C$-algebra can be defined over a smooth algebra over a discretely valued field in such a way that the set of units witnessing ``very smallness'' are $p$-adically close to units that also descend. 

\begin{lemma}
\label{ApproxVerySmall}
Let $R$ be a very small smooth Tate $C$-algebra $R$; fix finite subsets $\{T_1,...,T_d\} \subset \Sigma \subset (R^\circ)^\ast$ as in Definition~\ref{DefVSBdR}. Then, at the expense of enlarging $\Sigma$, we can find the following:
\begin{enumerate}
\item A smooth adic space $S=\mathrm{Spa}(A,A^\circ)$ of finite type over $W(k^\prime)[\tfrac 1p]$ for some perfect field $k^\prime\subset k$ and a $W(k')$-algebra map $A \to C$.
\item A smooth morphism $\mathrm{Spa}(R_A,R_A^\circ)\to \mathrm{Spa}(A,A^\circ)$ and an identification $R \simeq R_A\hat{\otimes}_A C$.
\item Finite subsets $\{T_1,...,T_d\} \subset \Sigma_A \subset (R_A^\circ)^\ast$ such that
\begin{enumerate}
\item The identification $R_A \hat{\otimes}_A C \simeq R$ carries $\Sigma_A$ into $\Sigma$ while preserving the $T_i$'s.
\item The map 
\[A (\langle X_u^{\pm 1})_{u \in \Sigma_A} \rangle \to R_A\]
 defined by $X_u \mapsto u$ is surjective.
\item On adic spectra, the map 
\[ \mathrm{Spa}(R_A,R_A^\circ) \to \mathbb{T}_S^d := S \times_{\mathrm{Spa}(\mathbb{Q}_p,\mathbb{Z}_p)}  \mathrm{Spa}(\mathbb{Q}_p \langle T_1^{\pm 1},...,T_d^{\pm 1} \rangle, \mathbb{Z}_p \langle T_1^{\pm 1},....,T_d^{\pm 1} \rangle)\] 
by the $T_i$'s is \'etale and factors as a composition of rational embeddings and finite \'etale maps.
\end{enumerate}
\end{enumerate}
\end{lemma}

The proof below shows that we can take $k' = \mathbb{F}_p$ in (i), i.e., we can take $A$ to be a smooth Tate $\mathbb{Q}_p$-algebra. 

\begin{proof}
By de Jong's theorem \cite{dJAlterations}, we can write $\mathcal{O}$ as a filtered colimit of finite type regular $\mathbb{Z}_p$-algebras $B_i$ (and thus $B_i[\frac{1}{p}]$ is smooth over $\mathbb{Q}_p$). We shall show that taking $A = \widehat{B}_j[\frac{1}{p}]$ for $j$ sufficiently large does the job; note that $\mathrm{Spa}(C,\mathcal{O}) \cong \lim_i \mathrm{Spa}(\widehat{B}_i[\frac{1}{p}],\widehat{B}_i)$.

Consider the \'etale map 
\[\Spa(R,R^\circ)\to \mathbb T^d = \Spa(C\langle T_1^{\pm 1},\ldots,T_d^{\pm 1}\rangle,\roi\langle T_1^{\pm 1},\ldots,T_d^{\pm 1}\rangle).\]
This map  factors as a composition of rational embeddings and finite \'etale maps by hypothesis. As both rational embeddings and finite \'etale maps admit suitable ``noetherian approximation'' results, setting $A = \widehat{B}_j[\frac{1}{p}]$ for sufficiently large $j$, we immediately get (i), (ii), and and a map as in part (c) of (iii) that descends the previous map. It remains to show that, after possibly enlarging $\Sigma$ and replacing $A$ with a finer approximation, we can also find the subset $\Sigma_A \subset (R^\circ_A)^\ast$ satisfying parts (a) and (b) in (iii). For this, we first enlarge $\Sigma$ by adding in small perturbations, and then replace $A$ with rational localizations. More precisely, we first note that there exists some $N \geq 0$ such that any map 
\[C (\langle X_u^{\pm 1})_{u \in \Sigma} \rangle \to R\]
defined by $X_u \mapsto u + p^N a_u$ (for some $a_u \in R^\circ$) is surjective:  this follows from Lemma~\ref{PerturbSurj} below applied to the map on power bounded elements (and our hypothesis that this map is surjective when all the $a_u$ equal $0$). Now the map
\[ \varinjlim_{i \geq j} R_A^\circ \widehat{\otimes}_{\widehat{B}_j} \widehat{B}_i \to R^\circ \]
has dense image in a ring of definition. It follows that at the expense of enlarging $j$, we can choose a subset $\Sigma_A \subset (R_A^\circ)^\ast$ containing $\{T_1,...,T_d\}$ such that the correponding map
\[\alpha:B := A (\langle X_u^{\pm 1})_{u \in \Sigma_A} \rangle \to R_A\]
is surjective after base change along $A \to C$; this immediately gives (iii) (a). We also obtain the surjectivity required in (iii) (b) by replacing $A$ with a rational localization around the point $x \in \mathrm{Spa}(A,A^\circ)$ determined by the map $A \to C$ using Lemma~\ref{ApproxAffd1} below.
\end{proof}

The next three lemmas were used above.

\begin{lemma}
\label{PerturbSurj}
Let $f:M \to N$ be a map of $p$-torsionfree and $p$-adically complete abelian groups with $f[1/p]$ is surjective. There exists some $m \geq 0$ such that any map $g:M \to N$ with $g \equiv f \mod p^m$, the map $g[1/p]$ is also surjective.
\end{lemma}
\begin{proof}
By the open mapping theorem, there exists some $n \geq 0$ such that $N' := p^n N \subset f(M)$. Write $M' := f^{-1}(N')$, so $f$ restricts to a map $f':M' \to N'$ that is surjective. We shall show that taking $m = n+1$ does the job. Fix a map $h:M \to N$, and let $g = f + p^{n+1}h$. We must show that $g[1/p]$ is surjective. Now if $x \in M'$, then $f(x) \in N$ and $p^{n+1} h(x) = p \cdot p^n h(x) \in pN$. It immediately follows that $g$ carries $M'$ into $N'$ and that the induced map $g':M' \to N'$ agrees with $f'$ modulo $p$. In particular, $g'$ surjective modulo $p$. But then $g'$ must be surjective: any map between derived $p$-complete modules that is surjective modulo $p$ is surjective: apply \cite[Tag 09B9]{StacksProject} to the cokernel. It is also clear that $g'[1/p] = g[1/p]$, so the claim follows.
\end{proof}

\begin{lemma}
\label{ApproxAffd1}
Let $A \to B \xrightarrow{\alpha} C$ be maps of affinoid algebras that are topologically of finite type over a nonarchimedean field $K$. Assume that there exists a rank $1$ point $x \in \mathrm{Spa}(A,A^\circ)$ such that $B \widehat{\otimes}_A k(x) \xrightarrow{\alpha_x} C \widehat{\otimes}_A k(x) $ is surjective. Then there exists some rational subset $U \subset \mathrm{Spa}(A,A^\circ)$ containing $x$ such that $\alpha_U:B_U \to C_U$ is surjective; here $A_U = \mathcal{O}_{\mathrm{Spa}(A,A^\circ)}(U)$, $B_U := B  \widehat{\otimes}_A A_U$ and similarly for $C_U$. 
\end{lemma}

The assumption that $x$ be a rank $1$ point is critical to the conclusion above. In fact, taking $x$ be {\em any} point of rank $> 1$ on any affinoid $A$ gives a counterxample as follows. Take $A=B$ and let $C = A_V$ be the rational localization corresponding to a rational subset $V \subset \mathrm{Spa}(A,A^\circ)$ that contains the unique rank $1$ generalization $x_{gen}$ of $x$ but does not contain $x$; these exist as rational subsets give a basis for the topology. As $k(x) = k(x_{gen})$, the map $\alpha_x$ is bijective. Now if $U$ is any rational open that contains $x$, then the map $\alpha_U:A_U \to A_U \widehat{\otimes}_A A_V \simeq A_{U \cap V}$ is a rational localization corresponding to the inclusion $U \cap V \subset U$. If $\alpha_U$ were surjective for some $U$ containing $x$, then $U \cap V \subset U$ would be a closed subset by \cite[\S 1.4.1]{Huber}. But this is impossible as $U \cap V$ is not closed under specialization in $U$: the point $x_{gen}$ lies in $U \cap V$ and its specialization $x$ lies in $U$ (by assumption on $U$) but not in $V$ (by choice of $V$).

\begin{proof}
In this proof, the symbol $U$ will be reserved to denote an element of the collection $\mathfrak{U}$ of all rational subsets of $\mathrm{Spa}(A,A^\circ)$ that contains $x$. Let us begin by fixing some compatible rings of definition. Without loss of generality, we may assume $A = K \langle T_1,...,T_n \rangle$ is a Tate algebra. In particular, for each $U \in \mathfrak{U}$, we simply use $A_{0,U} := A_U^\circ$ as the ring of definition for $A_U$; write $A_0 = A_{0,\mathrm{Spa}(A,A^\circ)}$ for the ring of definition of $A$ itself.  Write $C = A \langle Y_1,...,Y_r \rangle/I$ for some ideal $I$, write $f_i \in C$ for the image of $Y_i$, and choose a ring of definition $C_0 \subset C$ that contains the $f_i$'s as well as the image of $A_0$. By writing $B$ as a quotient of a Tate algebra, we may assume without loss of generality that $B = A \langle X_1,...,X_n \rangle$, so $B_0 = A_0 \langle X_1,...,X_n \rangle \subset B$ is a ring of definition. We may enlarge the ring of definition $C_0$ if necessary to ensure that $\alpha(B_0) \subset C_0$. Then $B_{0,x} := B_0 \widehat{\otimes}_{A_0} k(x)^+ \simeq k(x)^+ \langle X_1,..,X_n \rangle$ is a ring of definition of $B_x := B \widehat{\otimes}_A k(x) \simeq k(x) \langle X_1,...,X_n \rangle$; similarly, $B_{0,U} := B_0 \widehat{\otimes}_{A_0} A_{0,U} \subset B_U$  is a ring of definition.  The $p^\infty$-torsion in $C_0 \widehat{\otimes}_{A_0} A_{0,U}$ and $C_0 \widehat{\otimes}_{A_0} k(x)^+$ is bounded by \cite[Lemma 1.2 (c)]{BoschLut1}, and the quotients $C_{0,U} := C_0 \widehat{\otimes}_{A_0} A_{0,U}/(p^\infty\text{-torsion})$ and $C_{0,x} := C_0 \widehat{\otimes}_{A_0} k(x)^+/(p^\infty\text{-torsion})$ give rings of definition of $C_U$ and $C_x$ respectively. For future reference, note that the natural maps $\varinjlim_{U \in \mathfrak{U}} B_{0,U} \to B_{0,x}$ and $\varinjlim_{U \in \mathfrak{U}} C_{0,U} \to C_{0,x}$ are isomorphisms after $p$-adic completions, and that $C_U$ and $C_x$ are topologically generated by the $f_i$'s over $A_U$ and $k(x)$ respectively. 

Next, let us fix some constants that we are allowed to perturb the $f_i$'s by without affecting the fact that they topologically generate $C$ or its localizations. Recall that we have chosen a presentation $C = A \langle Y_1,...,Y_r \rangle/I$. In particular, the natural map $A_0 \langle Y_1,...,Y_r \rangle \to C_0$ is surjective after inverting $p$. By the open mapping theorem, this map has a cokernel annihilated by $p^N$ for some fixed $N \geq 0$. By base change, the cokernel of the map $A_{0,U} \langle Y_1,...,Y_r \rangle \to C_{0,U}$ is also annihilated by $p^N$ for $U \in \mathfrak{U}$. It follows from Lemma~\ref{PerturbSurj} (and its proof)  that any $f_i' \in C_{0,U}$ such that $f_i' \equiv f_i \mod p^{N+1} C_{0,U}$ also provides a topological generating set for $C_U$ over $A_U$ for all $U \in \mathfrak{U}$ as well as for $C_x$ over $k(x)$. 

Now consider the map $\alpha_x:B_x \to C_x$. By assumption, this map is surjective. By the open mapping theorem,  the induced map $\alpha_x:B_{0,x} \to C_{0,x}$ has cokernel killed by $p^m$ for some $m \geq 0$. So we can choose $g_1,...,g_r \in B_{0,x}$ such that $\alpha_x(g_i) = p^m f_i$ for all $i$. Moreover, note that both $\alpha_x(B_{0,x})$ and $C_{0,x}$ are topologically finitely generated rings of definition of the tft $k(x)$-algebra $C_x$. Their integral closures must coincide with the subring $C_x^\circ$ of power bounded elements by \cite[\S 6.3.4, Proposition 1]{BGR} (see also \cite[Lemma 4.4]{HuberDefAdic}). In particular, $C_{0,x}$ is integral over $\alpha_x(B_{0,x})$. So for each $i \in \{1,...,r\}$, there exists a monic polynomial $Q_i(T) \in B_{0,x}[T]$ such that $\alpha_x(Q_i)(f_i) = 0$. 

As the natural map $\varinjlim_{U \in \mathfrak{U}} B_{0,U} \to B_{0,x}$ is surjective modulo any power of $p$, we can find $h_1,...,h_r \in B_{0,U}$ for a sufficiently small $U \in \mathfrak{U}$ such that the image of $h_i$ in $B_{0,x}$ differs from $g_i$ by $p^{N+m+1} B_{0,x}$. Then $\alpha_U(h_i) \in C_{0,U}$ are elements whose image in $C_{0,x}$ differs from $p^m f_i$ by $p^{N+m+1} C_{0,x}$. As the map $\varinjlim_{U \in \mathfrak{U}} C_{0,U} \to C_{0,x}$ is an isomorphism after $p$-adic completion, it follows that after possibly shrinking $U \in \mathfrak{U}$, we can ensure that $\alpha_U(h_i) - p^m f_i \in p^{N+m+1} C_{0,U}$. Dividing by $p^m$ shows that $f_i' := \frac{\alpha_U(h_i)}{p^m} \in C_{0,U}$ and that $f_i' - f_i \in p^{N+1} C_{0,U}$. By our choice of $N$ in the second paragraph of this proof, it follows that $f_1',...,f_r' \in C_{0,U}$ give a topological generating set that lies in the image of $B_{0,U}[\frac{1}{p}] = B_U \to C_U$. 

By shrinking $U \in \mathfrak{U}$ further, we can find monic polynomials $P_i(T) \in B_{0,U}[T]$ such that the image of $P_i(T)$ in $B_{0,x}[T]$ differs from $Q_i(T)$ by $p B_{0,x}$. Since $f_i' - f_i \in p^{N+1}C_{0,U}$, it follows that the image of $\alpha_U(P_i)(f_i')$ in $C_{0,x}$ lies in $pC_{0,x}$. By shrinking $U \in \mathfrak{U}$ further, we can ensure that $\alpha_U(P_i)(f_i') \in p C_{0,U}$. Applying Lemma~\ref{ApproxAffd2} below to the inclusion $\mathrm{im}(B_{0,U} \to C_{0,U}) \subset C_{0,U}$ and the elements $f_1',...,f_r' \in C_{0,U}$ then gives the result. 
\end{proof}

\begin{lemma}
\label{ApproxAffd2}
Let $B_0 \subset C_0$ be an inclusion of $p$-adically complete and $p$-torsionfree rings such that $C_0$ is generated as a $p$-adically complete $B_0$-algebra  by $f_1,...,f_r \in C_0$. Assume this data satisfies the following:
\begin{enumerate}
\item We have $f_i \in B_0[\frac{1}{p}]$ for all $i$. 
\item There exist monic polynomials $P_1,..,P_r \in B_0[x]$ such that $P_i(f_i) \in p C_0$.
\end{enumerate}
Then $p^k C_0  \subset B_0$ for some $k \gg 0$. 
\end{lemma}
\begin{proof}
It is enough to show that $C_0$ is a finite $B_0$-module. Indeed, then each $f_i$ would satisfy a monic polynomial over $B_0$, so (i) would show $B_0[\frac{1}{p}] = C_0[\frac{1}{p}]$, whence $p^k C_0 \subset B_0$ for $k \gg 0$ by the open mapping theorem. To show the finiteness, by Nakayama's lemma and completeness, it is enough to show the same modulo $p$. But then it is clear: $B_0/p \to C_0/p$ is a finite ring map simply because the generators $f_1,...,f_r \in C_0/p$ are integral over $B_0/p$ by (ii). 
\end{proof}

For the rest of this section, fix notation as in Lemma~\ref{ApproxVerySmall}, though we shall use the flexibility of enlarging $\Sigma$ as necessary.  Note that by smoothness of $A$ over a discretely valued field, there are continuous maps $A\to B_\dR^+$ lifting the map to $C$; we fix one such map. Let $D_{\Sigma_A}(R_A)$ be the completion of
\[
A\langle (X_u^{\pm 1})_{u\in \Sigma_A}\rangle\to R_A\ .
\]
Again, all powers of the ideal $\ker(A\langle (X_u^{\pm 1})_{u\in \Sigma_A}\rangle\to R_A)$ are closed, and thus this defines a complete and separated algebra. Our next goal is to compare this with Construction~\ref{ConsThick1}. 

For this, we shall need a structural property that we prove first. Let $R_A\hat{\otimes}_A B_\dR^+$ be defined as the inverse limit of $R_A\hat{\otimes}_A B_\dR^+/\xi^n$, where we note that $R_A$, $A$ and $B_\dR^+/\xi^n$ are all complete Tate $\bb Q_p$-algebras, and hence there is a well-defined completed tensor product: if $S_2\leftarrow S_1\to S_3$ is a diagram of complete Tate-$\bb Q_p$-algebras with rings of definition $S_{2,0}\leftarrow S_{1,0}\to S_{3,0}$, then
\[S_2\hat{\otimes}_{S_1} S_3 = (\mathrm{im}(S_{2,0}\otimes_{S_{1,0}} S_{3,0}\to S_2\otimes_{S_1} S_3))^\wedge_p[\tfrac 1p]\ .\]

The structural property that we need is the following:

\begin{lemma}\label{lem:etalelift} The algebra $R_A\hat{\otimes}_A B_\dR^+$ is a $\xi$-adically complete flat $B_\dR^+$-algebra, with
\[
(R_A\hat\otimes_A B_\dR^+)/\xi = R
\]
and more generally
\[
(R_A\hat{\otimes}_A B_\dR^+)/\xi^n = R_A\hat{\otimes}_A B_\dR^+/\xi^n\ ,
\]
which is topologically free over $B_\dR^+/\xi^n$.
\end{lemma}

\begin{proof} It is enough to see that $R_A\hat{\otimes}_A B_\dR^+/\xi^n$ is topologically free (in particular, flat) over $B_\dR^+/\xi^n$ for all $n\geq 1$, with $(R_A\hat{\otimes}_A B_\dR^+/\xi^n)/\xi = R$.

There is a finitely generated $A^\circ[T_1^{\pm 1},\ldots,T_d^{\pm 1}]$-algebra $R_{A,\mathrm{alg}}$, \'etale after inverting $p$, such that $R_A=(R_{A,\mathrm{alg}})_p^\wedge[\tfrac 1p]$ by~\cite[Corollary 1.7.3 (iii)]{Huber}. Fix any topologically finitely generated ring of definition  $(B_\dR^+/\xi^n)_0\subset B_\dR^+/\xi^n$ containing $\xi$ and the image of $A^\circ$. Then
\[
R_A\hat{\otimes}_A B_\dR^+/\xi^n = ((R_{A,\mathrm{alg}}\otimes_{A^\circ} (B_\dR^+/\xi^n)_0)/(p\mathrm{-torsion}))^\wedge_p[\tfrac 1p]\ .
\]
Now $S_n = R_{A,\mathrm{alg}}\otimes_{A^\circ} (B_\dR^+/\xi^n)_0$ is a finitely presented $(B_\dR^+/\xi^n)_0$-algebra which is smooth, and in particular flat, after inverting $p$. Then $S_n/\xi$ is a finitely presented $\roi$-algebra which is smooth after inverting $p$. As it is finitely presented over $\roi$, the $p$-power torsion $T\subset S_n/\xi$ is finitely generated; thus, there is some power of $p$ killing $T$. Now, if $S_n$ has no connected components living entirely over the generic fibre $\Spec B_\dR^+/\xi^n$, then also $(S_n/\xi)/T$ has no connected components living entirely over $\Spec C$, and thus $(S_n/\xi)/T$ is free over $\roi$ by a result of Raynaud--Gruson, \cite[Th\'eor\`eme 3.3.5]{RaynaudGruson}. We assume that this is the case; in general one simply passes to the biggest direct factor of $S_n$ with this property. Pick a basis $(\bar{s}_i)_{i\in I}$ of $(S_n/\xi)/T$ as $\roi$-module, and lift the elements $\bar{s}_i$ to $s_i\in S_n$. This gives a map
\[
\alpha: \bigoplus_{i\in I} (B_\dR^+/\xi^n)_0\to S_n\ .
\]
We claim that $\alpha$ is injective, and that the cokernel of $\alpha$ is killed by a power of $p$. For injectivity, it is enough to check that
\[
\bigoplus_{i\in I} B_\dR^+/\xi^n\to S_n[\tfrac 1p]
\]
is an isomorphism. But both modules are flat over $B_\dR^+/\xi^n$, so it is enough to check that
\[
\bigoplus_{i\in I} C\to S_n/\xi[\tfrac 1p]
\]
is an isomorphism, which follows from the choice of the $s_i$. Now, to check that the cokernel of $\alpha$ is killed by a power of $p$, it suffices to check modulo $\xi$; there, again the result follows from the choice of the $s_i$, and the fact that $T$ is killed by a power of $p$.

It follows that in the formula
\[
R_A\hat\otimes_A B_\dR^+/\xi^n = (S_n/(p\mathrm{-torsion}))^\wedge_p[\tfrac 1p]\ ,
\]
one can replace $S_n$ by $\bigoplus_{i\in I} (B_\dR^+/\xi^n)_0$, which shows that $R_A\hat\otimes_A B_\dR^+/\xi^n$ is topologically free over $B_\dR^+/\xi^n$. Moreover, the proof shows that $(R_A\hat{\otimes}_A B_\dR^+/\xi^n)/\xi = R$.
\end{proof}

We can now give the promised comparison between $D_{\Sigma_A}(R_A)$ and $D_\Sigma(R)$.

\begin{lemma}\label{lem:DRnice} One has the following description of $D_{\Sigma_A}(R_A)$ and $D_\Sigma(R)$.
\begin{enumerate}
\item There is a unique isomorphism of topological algebras
\[
D_{\Sigma_A}(R_A)\cong R_A[[(X_u-u)_{u\in \Sigma_A, u\neq T_1,\ldots,T_d}]]
\]
compatible with the projections to $R_A$, and the structure of $A\langle (X_u^{\pm 1})_{u\in \Sigma_A}\rangle$-algebras, where $X_{T_i}\mapsto T_i$ on the right.
\item If $\Sigma$ is sufficiently large, there is an isomorphism of topological algebras
\[
D_\Sigma(R)\cong (R_A\hat{\otimes}_A B_\dR^+)[[(X_u-\tilde{u})_{u\in \Sigma,u\neq T_1,\ldots,T_d}]]\ ,
\]
compatibly with the projection to $R$, and the structure of $B_\dR^+\langle (X_u^{\pm 1})_{u\in \Sigma}\rangle$-algebras (via $X_{T_i}\mapsto T_i$). Here, $\tilde{u}\in R_A\hat{\otimes}_A B_\dR^+$ is a lift of $u\in R_A\hat{\otimes}_A C$. In particular, $D_\Sigma(R)$ is $\xi$-adically complete and $\xi$-torsion-free.
\end{enumerate}
\end{lemma}

\begin{proof} For (i), we first want to find a lift $R_A\to D_{\Sigma_A}(R_A)$ of the projection $D_{\Sigma_A}(R_A)\to R_A$. The strategy is to pick the obvious lifting on $A\langle T_1^{\pm 1},\ldots,T_d^{\pm 1}\rangle$ sending $T_i$ to $X_{T_i}$, and then extend to $R_A$ by \'etaleness; however, the second step needs some care because of topological issues.

As above, there is a finitely generated $A^\circ[T_1^{\pm 1},\ldots,T_d^{\pm 1}]$-algebra $R_{A,\mathrm{alg}}$, \'etale after inverting $p$, such that $R_A=(R_{A,\mathrm{alg}})_p^\wedge[\tfrac 1p]$ by~\cite[Corollary 1.7.3 (iii)]{Huber}.

The map $A^\circ[T_1^{\pm 1},\ldots,T_d^{\pm 1}]\to D_{\Sigma_A}(R_A)$ given by $T_i\mapsto X_{T_i}$ lifts uniquely to $R_{A,\mathrm{alg}}$. We claim that it also extends to the $p$-adic completion. For this, note that the completion of $A\langle (X_u^{\pm 1})_{u\in \Sigma}\rangle\to R_A$ is an inverse limit of complete Tate $A$-algebras $D_j$ which are topologically of finite type, with reduced quotient $R_A$. In particular, the subring of powerbounded elements $D_j^\circ\subset D_j$ is the preimage of $R_A^\circ\subset R_A$. Thus, $R_{A,\mathrm{alg}}$ is a finitely generated $A^\circ$-algebra mapping into $D_j^\circ$; as such, it maps into some ring of definition of $D_j$, and therefore the map extends to the $p$-adic completion. This gives the desired map $R_A\to D_{\Sigma_A}(R_A)$.

In particular, we get a canonical continuous map
\[
R_A[[(X_u-u)_{u\in \Sigma_A,u\neq T_1,\ldots,T_d}]]\to D_{\Sigma_A}(R_A)\ .
\]
We claim that this is a topological isomorphism. For this, we use the commutative diagram
\[\xymatrix{
A\langle (X_u^{\pm 1})_{u\in \Sigma_A}\rangle\ar[r]\ar[dr]& R_A[[(X_u-u)_{u\in \Sigma_A,u\neq T_1,\ldots,T_d}]]\ar[d]\ar[dr] \\
&D_{\Sigma_A}(R_A)\ar[r]&R_A,
}\]
where we use the identity
\[
X_u^{-1} = u^{-1}\left(1+\frac{X_u-u}{u}\right)^{-1}
\]
in $R_A[[(X_u-u)_{u\in \Sigma_A,u\neq T_1,\ldots,T_d}]]$ to define the upper map. The upper part of the diagram implies that there is a continuous map
\[
D_{\Sigma_A}(R_A)\to R_A[[(X_u-u)_{u\in \Sigma_A,u\neq T_1,\ldots,T_d}]]\ .
\]
The maps are inverse: In the direction from $D_{\Sigma_A}(R_A)$ back to $D_{\Sigma_A}(R_A)$, this follows by construction. In the other direction, the resulting endomorphism of the separated ring $R_A[[(X_u-u)_{u\in \Sigma_A,u\neq T_1,\ldots,T_d}]]$ must be the identity on $R_{A,\mathrm{alg}}$ and all $X_u$, and thus by continuity everywhere, finishing the proof of (i).

For part (ii), we repeat the same arguments, using Lemma~\ref{lem:BdRmodxin} and Lemma~\ref{lem:etalelift}.
\end{proof}

As observed in Construction~\ref{ConsThick1}, the derivations $\frac{\partial}{\partial \log(X_u)}$ extend continouously to $D_\Sigma(R)$. Thus, we can build a de~Rham complex
\[
K_{D_\Sigma(R)}\left((\frac{\partial}{\partial \log(X_u)})_{u\in \Sigma}\right)\ ,
\]
which starts with
\[
0\to D_\Sigma(R)\xTo{(\frac{\partial}{\partial \log(X_u)})_u} \bigoplus_{u\in \Sigma} D_\Sigma(R)\to \ldots\ .
\]
By abuse of notation, we will denote it by $\Omega^\bullet_{D_\Sigma(R)/B_\dR^+}$. This complex, or rather the filtered colimit over all sufficiently large $\Sigma$, is our explicit model for (the so far undefined)
\[
R\Gamma_\crys(\Spa(R,R^\circ)/B_\dR^+)\ .
\]
We note that in Lemma~\ref{lem:crysdRcomp}, we will check that the transition maps $\Omega^\bullet_{D_\Sigma(R)/B_\dR^+}\to \Omega^\bullet_{D_{\Sigma^\prime}(R)/B_\dR^+}$ are quasi-isomorphisms, for any inclusion $\Sigma\subset \Sigma^\prime$ of sufficiently large subsets of $R^{\circ\times}$.

We want to compare crystalline and de~Rham cohomology. For this, it is convenient to introduce an intermediate object: Namely, let $\tilde{D}_\Sigma(R)$ be the completion of
\[
(D_\Sigma(R)/\xi)\hat{\otimes}_C R\to R\ .
\]
This comes with derivations $\frac{\partial}{\partial \log(X_u)}$ for $u\in \Sigma$, and $\frac{\partial}{\partial \log (T_i)}$ for $i=1,\ldots,d$, and one can build a corresponding de~Rham complex $\Omega^\bullet_{\tilde{D}_{\Sigma}(R)/C}$ (taking into account both derivations). Note that this complex does not actually depend on the choice of coordinates $T_1,\ldots,T_d$, as one can parametrize the second set of derivations canonically by (the dual of) $\Omega^{1,\cont}_{R/C}$. Then there are natural maps of complexes
\[
\Omega^\bullet_{D_\Sigma(R)/B_\dR^+}/\xi\to \Omega^\bullet_{\tilde{D}_{\Sigma}(R)/C}\leftarrow \Omega^\bullet_{R/C}\ .
\]
Again, there is also a version taking into account the algebra $R_A$. Namely, let $\tilde{D}_\Sigma(R_A)$ be the completion of
\[
D_\Sigma(R)\hat{\otimes}_{B_\dR^+} (R_A\hat{\otimes}_A B_\dR^+)\to R\ .
\]
In this case, there are natural maps of complexes as follows:
\[
\Omega^\bullet_{D_\Sigma(R)/B_\dR^+}\to \Omega^\bullet_{\tilde{D}_\Sigma(R_A)/B_\dR^+}\leftarrow \Omega^\bullet_{R_A/A}\hat{\otimes}_A B_\dR^+\ .
\]

\begin{lemma}\label{lem:crysdRcomp} The maps
\[
\Omega^\bullet_{D_\Sigma(R)/B_\dR^+}/\xi\to \Omega^\bullet_{\tilde{D}_\Sigma(R)/C}\leftarrow \Omega^\bullet_{R/C}
\]
and
\[
\Omega^\bullet_{D_\Sigma(R)/B_\dR^+}\to \Omega^\bullet_{\tilde{D}_\Sigma(R_A)/B_\dR^+}\leftarrow \Omega^\bullet_{R_A/A}\hat{\otimes}_A B_\dR^+
\]
are quasi-isomorphisms.

In particular, for any inclusion $\Sigma\subset \Sigma^\prime$ of sufficiently large subsets of $R^{\circ\times}$, the map
\[
\Omega^\bullet_{D_\Sigma(R)/B_\dR^+}\to \Omega^\bullet_{D_{\Sigma^\prime}(R)/B_\dR^+}
\]
is a quasi-isomorphism.
\end{lemma}

\begin{proof} Explicitly,
\[
\tilde{D}_\Sigma(R) = R[[(X_u-u)_{u\in \Sigma}]]\ ,
\]
which easily shows that the second map is a quasi-isomorphism. On the other hand, we claim that
\[
\tilde{D}_\Sigma(R) = (D_\Sigma(R)/\xi)[[(T_i-X_{T_i})_{i=1,\ldots,d}]]\ .
\]
This presentation implies that the first map is a quasi-isomorphism. To check the claim, we use Lemma~\ref{lem:DRnice} to see that $\tilde{D}_{\Sigma}(R)$ is the completion of
\[
R[[(X_u-u)_{u\in \Sigma, u\neq T_1,\ldots,T_d}]]\hat{\otimes}_C R\to R\ .
\]
But the completion of $R\hat{\otimes}_C R\to R$ is given by $R[[(T_i\otimes 1 - 1\otimes T_i)_{i=1,\ldots,d}]]$. Combining these observations, we see that $\tilde{D}_\Sigma(R) = (D_\Sigma(R)/\xi)[[(T_i-X_{T_i})_{i=1,\ldots,d}]]$, as desired.

The second part follows, as everything is derived $\xi$-complete (as the terms of the complexes are $\xi$-adically complete and $\xi$-torsion free), so it suffices to check that one gets a quasi-isomorphism modulo $\xi$, which reduces to the first part.
\end{proof}

Using the last statement of the preceding lemma, one can define a fully functorial $B_\dR^+$-valued cohomology theory on the category of very small smooth affinoids over $C$ as follows.

\begin{definition}
For a very small smooth Tate $C$-algebra $R$ with $X := \mathrm{Spa}(R,R^\circ)$, define the $B_{\dR}^+$-complex $C^\bullet_{\crys}(X/B_{\dR}^+)$ as the filtered colimit $\varinjlim_{\Sigma}\Omega^\bullet_{D_\Sigma(R)/B_\dR^+}$ where $\Sigma$ ranges over all sufficiently large finite subsets of units in $R^\circ$. Write $R\Gamma_{\crys}(X/B_\dR^+) \in D(B_\dR^+)$ for the image of $C^\bullet_\crys(X/B_\dR^+)$ in the derived category. 
\end{definition}

It is easy to see that the $C^\bullet_{\crys}(-/B_\dR^+)$ gives a presheaf of $B_\dR^+$-complexes on the category of very small smooth affinoids over $C$.  Moreover, by Lemma~\ref{lem:crysdRcomp}, we  have a natural quasi-isomorphism 
\[ R\Gamma_{\crys}(X/B_{\dR}^+) \dotimes_{B_\dR^+} C \simeq \Omega^\bullet_{R/C}.\]
We shall later extend these constructions to proper smooth rigid spaces over $C$.

\subsection{Interlude: spreading out proper rigid spaces, following Conrad-Gabber}
\label{ss:ConradGabber}

In this section, we prove that any proper rigid space can be realized as the fibre of a family defined over a discretely valued field (Corollary~\ref{ConradGabberAlgebraization}). Our strategy is to reduce to a similar statement about formal models. The latter is a special case of the following result.

\begin{proposition}
\label{AlgebraizeFormal}
Let $(W,\mathfrak{m})$ be a complete noetherian local ring with residue field $k$. Let $\mathcal{O}$ be an $\mathfrak{m}$-adically complete local $W$-algebra such that the local ring $\mathcal{O}/\mathfrak{m}\mathcal{O}$ is $0$-dimensional with residue field $k$.
Let $\mathfrak{X}_{\mathcal{O}}/\mathcal{O}$ be a proper flat adic formal scheme, where $\mathcal{O}$ is topologized $\mathfrak{m}$-adically.  Then there exist the following:
\begin{enumerate}
\item A complete noetherian local $W$-algebra $R$ with residue field $k$, and a proper flat adic formal scheme $\mathfrak{X}_R/R$, where $R$ is topologized by powers of its maximal ideal.
\item A $W$-algebra map $\eta:R \to \mathcal{O}$ and an isomorphism $\psi:\eta^* \mathfrak{X}_R \simeq \mathfrak{X}_{\mathcal{O}}$ of formal $\mathcal{O}$-schemes. 
\end{enumerate}
\end{proposition}

Note that any ring $R$ as in $(1)$ above is a quotient of a formal power series ring over $W$: if $a_1,...,a_n \in R$ are generators of the maximal ideal, then the map $W [[ x_1,...,x_n ]] \xrightarrow{x_i \mapsto a_i} R$ is a surjection of local rings.

\begin{proof}
For any discrete $\mathcal{O}$-algebra $B$, write $\mathfrak{X}_B/B$ for base change of $\mathfrak{X}_{\mathcal{O}}/\mathcal{O}$; as $B$ is discrete, $\mathfrak{X}_B/B$ is a proper flat $B$-scheme. In particular, the special fibre $\mathfrak{X}_k/k$ is a proper $k$-scheme; we shall construct the required pair $\mathfrak{X}_R/R$ as a versal deformation of $\mathfrak{X}_k/k$ relative to $W$.

Let $\mathrm{Art}$ be the category of artinian $W$-algebras with residue field $k$. Consider the functor $\mathrm{Def}_{\mathfrak{X}_k}:\mathrm{Art} \to \mathrm{Set}$ of deformations of $\mathfrak{X}_k$, i.e., $\mathrm{Def}_{\mathfrak{X}_k}(A)$ is the set of isomorphism classes of lifts of $\mathfrak{X}_k$ to proper flat $A$-schemes. By Schlessinger \cite[Proposition 3.10]{Schlessinger} (see also \cite[Tag 0ET6]{StacksProject}), this functor admits a versal deformation, i.e., there exists a complete noetherian local $W$-algebra $R$ with residue field $k$ and a proper flat adic formal scheme $\mathfrak{X}_R/R$ (where $R$ is topologized by powers of its maximal ideal) deforming $\mathfrak{X}_k/k$ such that the induced classifying map $h_R := \mathrm{Hom}_W(R,-) \to \mathrm{Def}_{\mathfrak{X}_k}$ is formally smooth, i.e., for any surjection $B \to A$ in $\mathrm{Art}$, the map $h_R(B) \to h_R(A) \times_{\mathrm{Def}_{\mathfrak{X}_k}(A)} \mathrm{Def}_{\mathfrak{X}_k}(B)$ is surjective. We shall check that this construction does the job. Note that (i) and (ii) are clear from the construction. 

Let us first explain how to extend functors defined on $\mathrm{Art}$ to a slightly larger class of $W$-algebras that includes rings of the form $\mathcal{O}/\mathfrak{m}\mathcal{O}$. Let $\mathrm{IndArt}$ be the category of local $0$-dimensional $W$-algebras $A$ with residue field $k$. Note that $\mathrm{Art} \subset \mathrm{IndArt}$, and each $\mathcal{O}/\mathfrak{m}^n\mathcal{O}$ also lies in $\mathrm{IndArt}$. Further, the maximal ideal  $\mathfrak{m}_A$ of any $A \in \mathrm{IndArt}$ is locally nilpotent as $\mathrm{Spec}(A)$ has a single point by $0$-dimensionality. We can thus write such an $A$ as a filtered colimit of its artinian $W$-subalgebras: any finite subset $S := \{a_1,...,a_n\} \subset \mathfrak{m}_A$ lies in the image of the map $W[[ x_1,...,x_n ]] \xrightarrow{x_i \mapsto a_i} A$, and this image is artinian as some power of the maximal ideal of $W[[ x_1,..,x_n ]]$ maps to $0$ by local nilpotence of $\mathfrak{m}_A$. By length considerations, it follows that $\mathrm{Art} \subset \mathrm{IndArt}$ is exactly the category of compact objects, and that the map $\mathrm{Art} \to \mathrm{IndArt}$ realizes the target as the $\mathrm{Ind}$-completion of the source.  In particular, any functor $F:\mathrm{Art} \to \mathrm{Set}$ has a unique extension $\underrightarrow{F}:\mathrm{IndArt} \to \mathrm{Set}$ that preserves filtered colimits: explicitly, if $A \in \mathrm{IndArt}$, then we simply set $\underrightarrow{F}(A) := \varinjlim_{A_i \subset A} F(A_i)$, where the colimit runs over all artinian subalgebras of $A$. Crucial to our purposes will be the following stability property of this construction: if $F \to G$ is a formally smooth map of functors on $\mathrm{Art}$, then $\underrightarrow{F} \to \underrightarrow{G}$ is also formally smooth, i.e., for any surjection $B \to A$ in $\mathrm{IndArt}$, the map $\underrightarrow{F}(B) \to \underrightarrow{F}(A) \times_{\underrightarrow{G}(A)} \underrightarrow{G}(B)$ is surjective. To see this, one first observes that the surjection $B \to A$ can be written as a filtered colimit of surjections $B_i \to A_i$ in $\mathrm{Art}$: write $B$ as a union of its artinian subalgebras $B_i \subset B$, and set $A_i \subset A$ to be the image of $B_i$. The desired surjectivity now follows as the formation of filtered colimits in the category of sets commutes with fibre products and preserves surjections.

We now specialize the considerations in the previous paragraph to the functors of interest. First, note that the extension $\underrightarrow{h_R}:\mathrm{IndArt} \to \mathrm{Set}$ as defined above coincides with $\mathrm{Hom}_W(R,-)$ as $R$ is a quotient of a formal power series ring over $W$ in finitely many variables. Similarly, as the functor specifying finitely presented schemes or their isomorphisms commutes with filtered colimits of rings, the set $\underrightarrow{\mathrm{Def}_{\mathfrak{X}_k}}(A)$ is simply the set of isomorphism classes of deformations of $\mathfrak{X}_k$ to $A$ for any $A \in \mathrm{IndArt}$.  Also, by the previous paragraph, the induced map $\underrightarrow{h_R}  \to \underrightarrow{\mathrm{Def}_{\mathfrak{X}_k}}$ of functors on $\mathrm{IndArt}$ is formally smooth.

Let us now give the proof of (iii). We have a canonical map $\eta_0:R \to k$ and an isomorphism $\psi_0:\eta_0^* \mathfrak{X}_R \simeq \mathfrak{X}_k$ of $k$-schemes. Applying the formal smoothness of $\underrightarrow{h_R} \to \underrightarrow{\mathrm{Def}_{\mathfrak{X}_k}}$ to the surjection $\mathcal{O}/\mathfrak{m}\mathcal{O} \to k$ in $\mathrm{IndArt}$, we can choose a map $\eta_1:R \to \mathcal{O}/\mathfrak{m}\mathcal{O}$ lifting $\eta_0$ and an isomorphism $\psi_1:\eta_1^* \mathfrak{X}_R \simeq \mathfrak{X}_{\mathcal{O}/\mathfrak{m}\mathcal{O}}$ of $\mathcal{O}/\mathfrak{m}\mathcal{O}$-schemes lifting $\psi_0$. Similarly, we can inductively choose a compatible system of maps $\eta_n:R \to \mathcal{O}/\mathfrak{m}^n\mathcal{O}$ and isomorphisms $\psi_n:\eta_n^* \mathfrak{X}_R \simeq \mathfrak{X}_{\mathcal{O}/\mathfrak{m}^n\mathcal{O}}$ of $\mathcal{O}/\mathfrak{m}^n\mathcal{O}$-schemes for each $n \geq 1$. The proposition now follows by taking an inverse limit in $n$.
\end{proof}

\begin{corollary}[Conrad-Gabber \cite{ConradGabber}]
\label{ConradGabberAlgebraization}
Let $C/K$ be an extension of complete nonarchimedean fields with the same residue field. If $X/C$ is a proper rigid space, then there exists a proper flat morphism $f:\mathcal{X} \to S$ of rigid spaces over $K$ such that $X/C$ arises as the fibre of $f$ over a point $\eta \in S(C)$. If $X/C$ is smooth, then we may choose $S$, $\mathcal{X}$ and $f$ to be smooth. 

In particular, any proper smooth rigid space over $C$ can be realized as the fibre of a proper smooth morphism of smooth rigid spaces defined over a discretely valued subfield of $C$.
\end{corollary}

\begin{proof}
We are free to replace $K$ with smaller complete nonarchimedean subfields of $C$ in proving the corollary. Taking $K$ to be the fraction field of a Cohen ring of the residue field of $C$, we may thus assume that $K$ is discretely valued. Let $W \subset K$ and $\mathcal{O} \subset C$ be the valuation rings, so $W$ is discrete. By the theory of formal models, the proper rigid space $X/C$ arises as the generic fibre of a proper flat adic formal scheme $\mathfrak{X}/\mathcal{O}$ (see \cite[Lemma 2.6]{LProper} for an explanation of the properness of the formal model). Choose $\mathfrak{X}_R/R$ and the map $\eta:R \to \mathcal{O}_C$ as in Proposition~\ref{AlgebraizeFormal}. Setting $X/S$ to be the generic fibre (i.e., the base change along $\mathrm{Spa}(K,W) \to \mathrm{Spf}(W)$ in the language of adic spaces) of $\mathfrak{X}_R/R$ then gives the desired family. The smoothness assertions in the last part follow immediately: given a proper flat morphism $f:\mathcal{X} \to S$ of rigid spaces over $K$ and a point $\eta \in S(C)$ where $f$ is smooth, we may replace $S$ by a suitable locally closed subset containing $\eta$ to conclude that both $S$ and $f$ (and hence $\mathcal{X}$) may be taken to be smooth. 
\end{proof}

\begin{remark}
With a little extra effort, the method described in this subsection can be used to prove a refinement of Corollary~\ref{ConradGabberAlgebraization} where the assumption that $C$ and $K$ have the same residue field is relaxed to the assumption that the residue field of $C$ is purely inseparable over that of $k$. We do not spell this out as the preprint \cite{ConradGabber} proves a strong form of Corollary~\ref{ConradGabberAlgebraization} by dropping the assumption on the residue field completely. In particular, for a $p$-adic field $C$, they show that any proper rigid space $X/C$ arises as the fibre of a family that is defined over $\mathbb{Q}_p$. This stronger statement is not necessary for our purposes. It is also considerably more complicated to prove as the necessary analog of Proposition~\ref{AlgebraizeFormal} entails developing an analog of Schlessinger's work \cite{Schlessinger} for the versal deformation rings of proper schemes over {\em positive dimensional} base rings (arising by approximating the residue field $k$ with smooth algebras over the prime field). 
\end{remark}

\subsection{$B_\dR^+$-cohomology of proper smooth rigid spaces}

We may extend the construction of the $B_\dR^+$-valued cohomology theory from small affinoids to the proper case by taking hypercohomology.

\begin{definition}
For a proper smooth adic space $X/C$, write $R\Gamma_{\crys}(X/B_\dR^+) \in D(B_\dR^+)$ for the hypercohomology of the presheaf $U \mapsto C^\bullet_{\crys}(U/B_\dR^+)$ defined on the category of all smooth open affinoids $U \subset X$ that are very small. 
\end{definition}

Glueing analogous isomorphisms for affinoids shows that
\[ R\Gamma_\crys(X/B_\dR^+)\dotimes_{B_\dR^+} C\cong R\Gamma_\dR(X)\ . \]
As $R\Gamma_\crys(X/B_\dR^+)$ is derived $\xi$-complete and de Rham cohomology is finite-dimensional, this implies, in particular, that each $H^i_\crys(X/B_\dR^+)$ is a finitely generated $B_\dR^+$-module which vanishes for $|i| \gg 0$. In particular, $R\Gamma_\crys(X/B_\dR^+)$ is a perfect $B_\dR^+$-complex. In fact, we can do better:

\begin{theorem}\label{thm:cryscohomfree} Let $X/C$ be a proper smooth adic space. 
Then $H^i_\crys(X/B_\dR^+)$ is finite free over $B_\dR^+$ for all $i\in \bb Z$.
\end{theorem}

\begin{proof} 
If $k$ denotes the perfect residue field of $C$, then we can split the projection $\mathcal{O}_C/p \to k$ by the ind-smoothness of $k/\mathbf{F}_p$. By deformation theory, this lifts uniquely to a map $W(k) \to \mathcal{O}_C$, and thus gives an inclusion $W(k)[\frac{1}{p}] =: K \subset C$ of complete nonarchimedean fields with the same residue field.  By Corollary~\ref{ConradGabberAlgebraization}, we can find a proper smooth map $f:\mathcal{X} \to S$ of smooth adic spaces over $K$ such that $X/C$ arises as the fibre of $f$ at a point $\eta \in S(C)$. By shrinking $S$, we may assume $S := \mathrm{Spa}(A,A^\circ)$ is a smooth affinoid, so the map $\eta$ corresponds to a continuous map $A \to C$ of Tate $K$-algebras. By the smoothness of $A/K$, we can lift this to a continuous map $A \to B_\dR^+$. Write $ Rf_{\dR\ast} \mathcal{O}_{\mathcal{X}}$ for the relative de Rham cohomology of $f$, viewed as a complex of $A$-modules. Applying Lemma~\ref{lem:crysdRcomp} to a hypercover of $\mathcal{X}$ by small smooth affinoids gives a map
\[ Rf_{\dR\ast} \mathcal{O}_{\mathcal{X}} \dotimes_A B_{\dR}^+ \to R\Gamma_{\crys}(X/B_\dR^+)\] 
in $D(B_\dR^+)$ that is an isomorphism after applying $- \dotimes_{B_\dR^+} B_\dR^+/\xi$ by base change for de Rham cohomology along the map $A \to C$. Now each $R^i f_{\dR\ast} \roi_{\cal X}$ is a coherent $\roi_S$-module equipped with an integrable connection, and therefore locally free. In particular, both the source and target of the above map are derived $\xi$-complete; as the map was an isomorphism modulo $\xi$, it must thus be an isomorphism. Since each $R^i f_{\dR\ast} \roi_{\cal X}$ is a finite projective $A$-module, it now follows that each $H^i_\crys(X/B_\dR^+)$ is a finite projective, and hence finite free, $B_\dR^+$-module.
\end{proof}

\begin{remark}
\label{rmk:BdRCohDiscValued}
In the special case of Theorem~\ref{thm:cryscohomfree} where the proper smooth adic space $X/C$ arises as the base change of a proper smooth adic space $X_0/K$ defined over a discretely valued subfield $K \subset C$, the proof above shows that there is a canonical identification
\[ H^i_\dR(X_0/K) \otimes_{K} B_\dR^+ \simeq H^i_\crys(X/B_\dR^+)\]
of $B_\dR^+$-modules, where the implicit map $K \to B_\dR^+$ is the unique continuous lift of $K \to C$ that exists since $K$ is discretely valued.
\end{remark}

Finally, we can prove Theorem~\ref{thm:ratpadicHodgeC} and Theorem~\ref{thm:hodgetate}.

\begin{proof}[Proof of Theorem~\ref{thm:ratpadicHodgeC}] We start by constructing a natural map
\[
R\Gamma_\crys(X/B_\dR^+)\to R\Gamma(X_\sub{pro\'et},\bb B_{\dR,X}^+)\cong R\Gamma_\sub{\'et}(X,\bb Z_p)\otimes_{\bb Z_p} B_\dR^+\ .
\]
Afterwards, we will check that after inverting $\xi$, this gives a quasi-isomorphism. Our strategy is to construct a strictly functorial map of complexes locally, so that this map is already locally a quasi-isomorphism after inverting $\xi$; this reduces us to the local case.

In the local situation, assume that $X=\Spa(R,R^\circ)$ admits an \'etale map to the torus $\bb T^d$ that factors as a composite of rational embeddings and finite \'etale maps. In this case, for any sufficiently large $\Sigma\subset R^{\circ\times}$, we have the $B_\dR^+$-algebra $D_\Sigma(R)$ which is defined as the completion of
\[
B_\dR^+\langle (X_u^{\pm 1})_{u\in \Sigma}\rangle\to R\ .
\]
Moreover, we have a canonical pro-finite-\'etale tower $X_{\infty,\Sigma} = \projlimf_i X_i\to X$ which extracts $p$-power roots of all elements $u\in \Sigma$. In particular, this tower contains the tower of Lemma~\ref{lem:perfectoidtower}, so that $X_{\infty,\Sigma} = \projlimf_i X_i$ is affinoid perfectoid. Let $\Gamma = \prod_{u\in \Sigma} \bb Z_p(1)$ be the Galois group of the tower $X_{\infty,\Sigma}/X$. Then, by Lemma~\ref{lem:descrperiodsheaves} and \cite[Corollary 6.6]{ScholzePAdicHodge}, we have
\[
R\Gamma(X_\sub{pro\'et},\bb B_{\dR,X}^+) = R\Gamma_\sub{cont}(\Gamma,\bb B_\dR^+(R_{\infty,\Sigma}))\ ,
\]
where $(R_{\infty,\Sigma},R_{\infty,\Sigma}^+)$ is the completed direct limit of $(R_i,R_i^+)$, where $X_i=\Spa(R_i,R_i^+)$.

Let us fix primitive $p$-power roots of unity $\zeta_{p^r}\in \roi$; one checks easily that the following constructions are independent of this choice up to canonical isomorphisms. We get basis elements $\gamma_u\in \Gamma$ for each $u\in \Sigma$, and one can compute $R\Gamma_\sub{cont}(\Gamma,\bb B_\dR^+(R_{\infty,\Sigma}))$ by a Koszul complex
\[
K_{\bb B_\dR^+(R_{\infty,\Sigma})}((\gamma_u-1)_{u\in \Sigma}): \bb B_\dR^+(R_{\infty,\Sigma})\xTo{(\gamma_u-1)_u} \bigoplus_u \bb B_\dR^+(R_{\infty,\Sigma})\to \ldots\ .
\]

Now, by repeating the arguments of Section~\ref{subsec:cancryscomp}, there is a natural map of complexes
\[\xymatrix{
D_\Sigma(R)\ar[rrr]^{(\frac{\partial}{\partial \log(X_u)})_u}\ar[d] &&& \bigoplus_u D_\Sigma(R)\ar[r]\ar[d] & \ldots\\
\bb B_\dR^+(R_{\infty,\Sigma})\ar[rrr]^{(\gamma_u-1)_u} &&& \bigoplus_u \bb B_\dR^+(R_{\infty,\Sigma})\ar[r] & \ldots .
}\]
Here, the map $D_\Sigma(R)\to \bb B_\dR^+(R_{\infty,\Sigma})$ in degree $0$ comes via completion from the map
\[
B_\dR^+\langle (X_u^{\pm 1})_u\rangle\to \bb B_\dR^+(R_{\infty,\Sigma})
\]
sending $X_u$ to $[(X_u,X_u^{1/p},\ldots)]\in \bb B_\dR^+(R_{\infty,\Sigma})$, which is a well-defined element as we have freely adjoined $p$-power roots of all $X_u$.

We claim that this induces a quasi-isomorphism between $\Omega^\bullet_{D_\Sigma(R)/B_\dR^+}$ and $\eta_\xi K_{\bb B_\dR^+(R_{\infty,\Sigma})}((\gamma_u-1)_{u\in \Sigma})$, which finishes the proof of the comparison. This is completely analogous to the proof of Proposition~\ref{prop:alphaRQuasiIsom}.

To check that this construction is compatible with the isomorphism from Theorem~\ref{thm:ratPAdicHodge}, use that in that case $R=R_K\hat{\otimes}_K C$ comes as a base change, and there is a commutative diagram
\[\xymatrix{
D_\Sigma(R)\ar[r]\ar[d] & \tilde{D}_\Sigma(R_K)\ar[d] & R_K\hat{\otimes}_K B_\dR^+\ar[l]\ar@{=}[d]\\
\bb B_\dR^+(R_{\infty,\Sigma})\ar[r] & \roi \bb B_\dR^+(R_{\infty,\Sigma}) & R_K\hat{\otimes}_K B_\dR^+\ar[l].
}\]
Here, the left vertical arrow gives rise to the comparison isomorphism just constructed (after passing to Koszul complexes), the lower row encodes the comparison isomorphism from Theorem~\ref{thm:ratPAdicHodge} (after simultaneously passing to Koszul and de~Rham complexes), and the upper row encodes the comparison between crystalline and de~Rham cohomology in Lemma~\ref{lem:crysdRcomp}. The commutativity of the diagram (together with the relevant extra structures) proves the desired compatibility.
\end{proof}

\begin{proof}[Proof of Theorem~\ref{thm:hodgetate}] 
Using Corollary~\ref{ConradGabberAlgebraization}, we may realize $X/C$ as the fibre of a proper smooth morphism $f: \cal X\to S$ of smooth adic spaces over a discretely valued subfield $K \subset C$.  By passage to a suitable locally closed subset, we can assume that $R^i f_\ast \Omega^j_{\cal X/S}$ is a locally free $\cal O_S$-module for all $i$ and $j$, as is $R^i f_{\dR\ast} \cal O_{\cal X}$, and everything commutes with arbitrary base change. To check (i), we need to check that the ranks of Hodge cohomology add up to the rank of de~Rham cohomology. This can now be checked on classical points, where it is~\cite[Corollary 1.8]{ScholzePAdicHodge}.

Thus, we see that the dimension of de~Rham cohomology is the sum of the dimensions of Hodge cohomology. On the other hand, the dimension of de~Rham cohomology is the same as the rank of the free $B_\dR^+$-module $H^i_\crys(X/B_\dR^+)$, which is the same as the rank of the free $B_\dR^+$-module $H^i_\sub{\'et}(X,\bb Z_p)\otimes_{\bb Z_p} B_\dR^+$ by Theorem~\ref{thm:ratpadicHodgeC}. This, in turn, is the same as the dimension of \'etale cohomology; it follows that the Hodge--Tate spectral sequence degenerates.
\end{proof}

\subsection{The $B_\dR^+$-cohomology in the good reduction case}

Let us give an alternate description of $R\Gamma_\crys(X/B_\dR^+)$ in the good reduction case in terms of the $A_\crys$-cohomology theory. Let $\frak X$ be a proper smooth formal scheme over $\roi$, with generic fibre $X$. In this situation, we can consider the scheme $Y = \frak X\times_{\Spf \roi} \Spec \roi/p$. The universal $p$-adically complete PD thickening (compatible with the natural PD structure on $\bb Z_p$) of $\roi/p$ is Fontaine's ring $A_\crys$. Thus, we can consider the crystalline cohomology groups
\[
H^i_\crys(Y/A_\crys)\ .
\]
On the other hand, we can consider the special fibre $\bar{Y} = \frak X\times_{\Spf \roi} \Spec k$, and its crystalline cohomology groups
\[
H^i_\crys(\bar{Y}/W(k))\ ,
\]
which are finitely generated $W(k)$-modules.

\begin{proposition}\label{prop:BerthelotOgusAcrys} Fix a section $k\to \roi/p$. Then there is a canonical $\phi$-equivariant isomorphism
\[
H^i_\crys(Y/A_\crys)[\tfrac 1p]\cong H^i_\crys(\bar{Y}/W(k))\otimes_{W(k)} A_\crys[\tfrac 1p]\ .
\]
In particular, $H^i_\crys(Y/A_\crys)[\tfrac 1p]$ is a finite free $A_\crys[\tfrac 1p]$-module.
\end{proposition}

This is a variant on a result of Berthelot--Ogus, \cite{BerthelotOgus2}.

\begin{proof} First, we check that for any qcqs smooth $\roi/p$-scheme $Z$, the Frobenius
\[
\phi: H^i_\crys(Z/A_\crys)\otimes_{A_\crys,\phi} A_\crys\to H^i_\crys(Z/A_\crys)
\]
is an isomorphism after inverting $p$. Indeed, this reduces to the affine case. In that case, there is an isomorphism $Z=\bar{Z}\times_{\Spec k} \Spec \roi/p$, where $\bar{Z} = Z\times_{\Spec \roi/p} \Spec k$ (as by finite presentation, there is such an isomorphism modulo $p^{1/p^n}$ for some $n$, and one can lift this isomorphism by smoothness), and the result follows by base change from the case of $\bar{Z}/k$.

Note that
\[
H^i_\crys(Y/A_\crys)\otimes_{A_\crys,\phi} A_\crys = H^i(Y_{\roi/p^{1/p}}/\phi^{-1}(A_\crys))\otimes_{\phi^{-1}(A_\crys),\phi} A_\crys
\]
by base change. Repeating, we see that
\[
H^i_\crys(Y/A_\crys)\otimes_{A_\crys,\phi^n} A_\crys = H^i(Y_{\roi/p^{1/p^n}}/\phi^{-n}(A_\crys))\otimes_{\phi^{-n}(A_\crys),\phi^n} A_\crys\ ,
\]
where the left side agrees with $H^i_\crys(Y/A_\crys)$ after inverting $p$. On the other hand, if $n$ is large enough, then there is an isomorphism
\[
Y\times_{\roi/p^{1/p^n}}\cong \bar{Y}\times_{\Spec k} \Spec \roi/p^{1/p^n}
\]
reducing to the identity over $\Spec k$, by finite presentation. Moreover, any two such isomorphisms agree after increasing $n$. Base change for crystalline cohomology implies the result.
\end{proof}

\begin{remark}
The choice of section $k \to \roi/p$ in Proposition~\ref{prop:BerthelotOgusAcrys} is unique in the important special case when $k = \overline{\mathbb{F}}_p$. Indeed, to see this, it is enough to observe the following: if $R \to \mathbb{F}_q$ is a surjection of $\mathbb{F}_p$-algebras with a locally nilpotent kernel, then there is a {\em unique} section $\mathbb{F}_q \to R$. To prove this, we can write $R = \varinjlim R_i$ as a filtered colimit of its finitely generated $\mathbb{F}_p$-algebras $R_i \subset R$. Passing to a cofinal subsystem, we may assume that the composite $R_i \to R \to \mathbb{F}_q$ is surjective for each $i$. But then $R_i$ is an artinian local $\mathbb{F}_p$-algebra with residue field $\mathbb{F}_q$, so there is a unique section $\mathbb{F}_q \to R$ since $\mathbb{F}_p \to \mathbb{F}_q$ is  \'etale.
\end{remark}

In particular, we get a finite free $B_\dR^+$-module
\[
H^i_\crys(Y/A_\crys)\otimes_{A_\crys} B_\dR^+\ .
\]

\begin{proposition} There is a natural quasi-isomorphism
\[
R\Gamma_\crys(Y/A_\crys)\otimes_{A_\crys} B_\dR^+\cong R\Gamma_\crys(X/B_\dR^+)\ .
\]
In particular, $H^i_\crys(X/B_\dR^+)$ is free over $B_\dR^+$.
\end{proposition}

\begin{proof} The crystalline cohomology of $Y$ over $A_\crys$ can be computed via explicit complexes as in the definition of $R\Gamma_\crys(X/B_\dR^+)$, as in Section~\ref{subsec:cancryscomp}. Using these explicit models, one can write down an explicit map, which is locally, and thus globally, a quasi-isomorphism (as locally, both complexes are quasi-isomorphic to de~Rham complexes for a smooth lift to $A_\crys$, resp. $B_\dR^+$).
\end{proof}

\newpage

\section{Proof of main theorems}

Finally, we can assemble everything to prove our main results. Let $C$ be a complete algebraically closed extension of $\bb Q_p$ with ring of integers $\roi$ and residue field $k$. Let $\frak X$ be a smooth formal scheme over $\roi$, with generic fibre $X$. Recall Theorem~\ref{ThmC}:

\begin{theorem}\label{ThmCRestated} There are canonical quasi-isomorphisms of complexes of sheaves on $\frak X_\sub{Zar}$ (compatible with multiplicative structures).
\begin{enumerate}
\item With crystalline cohomology of $\frak X_k$:
\[
A\Omega_{\frak X}\hat{\dotimes}_{A_\inf} W(k)\simeq W\Omega^\bullet_{\frak X_k/W(k)}\ .
\]
Here, the tensor product is $p$-adically completed, and the right side denotes the de~Rham--Witt complex of $\frak X_k$, which computes crystalline cohomology of $\frak X_k$.
\item With de~Rham cohomology of $\frak X$:
\[
A\Omega_{\frak X}\dotimes_{A_\inf} \roi\simeq \Omega^{\bullet,\cont}_{\frak X/\roi}\ ,
\]
where $\Omega^{i,\cont}_{\frak X/\roi} = \projlim_n \Omega^i_{(\frak X/p^n)/(\roi/p^n)}$.
\item With crystalline cohomology of $\frak X_{\roi/p}$: If $u: (\frak X_{\roi/p}/A_\crys)_\crys\to \frak X_\sub{Zar}$ denotes the projection, then
\[
A\Omega_{\frak X}\hat{\dotimes}_{A_\inf} A_\crys\simeq Ru_\ast \roi_{\frak X_{\roi/p}/A_\crys}^\crys\ .
\]
\item With (a variant of) \'etale cohomology of the generic fibre $X$ of $\frak X$: If $\nu: X_\sub{pro\'et}\to \frak X_\sub{Zar}$ denotes the projection, then
\[
A\Omega_{\frak X}\otimes_{A_\inf} A_\inf[\tfrac 1\mu]\simeq (R\nu_\ast \bb A_{\inf,X})\otimes_{A_\inf} A_\inf[\tfrac 1\mu]\ .
\]
\end{enumerate}
\end{theorem}

\begin{remark} In fact, this result needs only that $C$ is perfectoid, with all $p$-power roots of unity.
\end{remark}

\begin{proof} Part (iii) is Theorem~\ref{thm:cryscomp}, and part (iv) follows directly from the definition of $A\Omega_{\frak X}$. Moreover, part (iii) implies parts (i) and (ii).

Alternatively, one can use the relation to the de~Rham--Witt complex to prove (i) and (ii). For simplicity, let us fix roots of unity for this discussion. For example, one can prove (ii) via
\[\begin{aligned}
A\Omega_{\frak X}\dotimes_{A_\inf,\theta} \roi &= (L\eta_\mu R\nu_\ast \bb A_{\inf,X})\dotimes_{A_\inf,\theta} \roi\\
&\isoto^\phi (L\eta_{\phi(\mu)} R\nu_\ast \bb A_{\inf,X})\dotimes_{A_\inf,\tilde\theta} \roi\\
&=L\eta_{\tilde\xi}(L\eta_\mu R\nu_\ast \bb A_{\inf,X})\dotimes_{A_\inf,\tilde\theta} \roi\\
&=(L\eta_{\tilde\xi} A\Omega_{\frak X})\dotimes_{A_\inf,\tilde\theta}\roi\\
&\cong H^\bullet(A\Omega_{\frak X}/\tilde\xi)\\
&\cong \Omega^{\bullet,\cont}_{\frak X/\roi}\ ,
\end{aligned}\]
using Proposition~\ref{prop:LetaBock} in the second-to-last step, and Theorem~\ref{thm:IntegralCartier} in the last step. More generally, for any $r\geq 1$,
\[\begin{aligned}
A\Omega_{\frak X}\dotimes_{A_\inf,\theta_r} W_r(\roi) &= (L\eta_\mu R\nu_\ast \bb A_{\inf,X})\dotimes_{A_\inf,\theta_r} W_r(\roi)\\
&\isoto^{\phi^r} (L\eta_{\phi^r(\mu)} R\nu_\ast \bb A_{\inf,X})\dotimes_{A_\inf,\tilde\theta_r} W_r(\roi)\\
&=L\eta_{\tilde\xi_r}(L\eta_\mu R\nu_\ast \bb A_{\inf,X})\dotimes_{A_\inf,\tilde\theta_r} W_r(\roi)\\
&=(L\eta_{\tilde\xi_r} A\Omega_{\frak X})\dotimes_{A_\inf,\tilde\theta_r}W_r(\roi)\\
&\cong H^\bullet(A\Omega_{\frak X}/\tilde\xi_r)\\
&\cong W_r\Omega^{\bullet,\cont}_{\frak X/\roi}\ ,
\end{aligned}\]
using Theorem~\ref{thm:AOmegavsdRWLocalPart4} in the last step. Extending this quasi-isomorphism from $W_r(\roi)$ to $W_r(k)$ and taking the limit over $r$ proves (i).

Note that we now have two quasi-isomorphisms
\[
A\Omega_{\frak X}\dotimes_{A_\inf,\theta_r} W_r(\roi)\simeq W_r\Omega^{\bullet,\cont}_{\frak X/\roi}\ :
\]
The one just constructed, coming from Theorem~\ref{thm:AOmegavsdRWLocalPart4}, and the one resulting from Theorem~\ref{thm:cryscomp} by extending along $A_\crys\to W_r(\roi)$ and using Langer--Zink's comparison, \cite[Theorem 3.5]{LangerZink}, between de~Rham--Witt cohomology and crystalline cohomology. Let us give a sketch that these quasi-isomorphisms are the same; for this, we use freely notation from Section~\ref{sec:cryscomp}. We look at the functorial complex
\[
\eta_\mu K_{\bb A_\inf(R_{\infty,\Sigma})}((\gamma_u-1)_{u\in \Sigma})
\]
computing $A\Omega_R$ for very small affine open $\Spf R\subset \frak X$ (where we are suppressing the filtered colimit over all sufficiently large $\Sigma\subset R^{\circ\times}$ from the notation). By Proposition~\ref{prop:LetaBock}, this admits a map of complexes
\[
\eta_\mu K_{\bb A_\inf(R_{\infty,\Sigma})}((\gamma_u-1)_{u\in \Sigma})\to H^\bullet((\eta_{\phi^{-r}(\mu)} K_{\bb A_\inf(R_{\infty,\Sigma})}((\gamma_u-1)_{u\in \Sigma}))/\xi_r)\cong^{\phi^r} H^\bullet(A\Omega_R/\tilde\xi_r)\ ,
\]
as above. Now we observe that this map factors through a map
\[
\left(\eta_\mu K_{\bb A_\inf(R_{\infty,\Sigma})\hat{\otimes}_{A_\inf} A_\crys^{(m)}}((\gamma_u-1)_{u\in \Sigma})\right)^\wedge_p\to H^\bullet(A\Omega_R/\tilde\xi_r)\ .
\]
Indeed, for any $m$, there is a natural map
\[
\eta_\mu K_{\bb A_\inf(R_{\infty,\Sigma})\hat{\otimes}_{A_\inf} A_\crys^{(m)}}((\gamma_u-1)_{u\in \Sigma})\to H^\bullet((\eta_{\phi^{-r}(\mu)} K_{\bb A_\inf(R_{\infty,\Sigma})\hat{\otimes}_{A_\inf} A_\crys^{(m)}}((\gamma_u-1)_{u\in \Sigma}))/\xi_r)\ ,
\]
and there is a quasi-isomorphism
\[
\eta_{\phi^{-r}(\mu)} K_{\bb A_\inf(R_{\infty,\Sigma})\hat{\otimes}_{A_\inf} A_\crys^{(m)}}((\gamma_u-1)_{u\in \Sigma})\simeq^{\phi^r} A\Omega_R\hat{\otimes}_{A_\inf} \phi^r(A_\crys^{(m)})
\]
by the usual arguments. Therefore,
\[
H^\bullet((\eta_{\phi^{-r}(\mu)} K_{\bb A_\inf(R_{\infty,\Sigma})\hat{\otimes}_{A_\inf} A_\crys^{(m)}}((\gamma_u-1)_{u\in \Sigma}))/\xi_r)\cong^{\phi^r} H^\bullet(A\Omega_r\hat\otimes_{A_\inf} \phi^r(A_\crys^{(m)})/\tilde\xi_r)\ ,
\]
and there is a natural map $\phi^r(A_\crys^{(m)})/\tilde\xi_r\to W_r(\roi)$, leading to a canonical map
\[
\left(\eta_\mu K_{\bb A_\inf(R_{\infty,\Sigma})\hat{\otimes}_{A_\inf} A_\crys^{(m)}}((\gamma_u-1)_{u\in \Sigma})\right)^\wedge_p\to H^\bullet(A\Omega_R/\tilde\xi_r)\cong W_r\Omega^{\bullet,\cont}_{R/\roi}\ ,
\]
as desired. This map is compatible with multiplication by construction; to check compatibility with the Bockstein differential, use that the target is $p$-torsionfree, and that there is a map $\phi^r(A_\crys{(m)})/\tilde\xi_r^2[\tfrac 1p]\to A_\inf/\tilde\xi_r^2[\tfrac 1p]$.

In particular, one can compose the map $\alpha_R$ from Lemma~\ref{lem:LetaAcrys} (v) with this map to get a functorial map of complexes
\[
K_{D_\Sigma}((\frac{\partial}{\partial \log(x_u)})_{u\in \Sigma})\to W_r\Omega^{\bullet,\cont}_{R/\roi}\ .
\]
In fact, this is a map of commutative differential graded algebras: To check compatibility with multiplication, use that $\alpha_R$ becomes compatible with multiplication after base extension to $W_r(\roi)$.

Here, the left side is the complex computing crystalline cohomology in terms of the embedding into the torus given by all units in $\Sigma$. One can then check that this map agrees with the similar map constructed by Langer--Zink in \cite[\S 3.2]{LangerZink}: As it is a continuous map of commutative differential graded algebras generated in degree $0$, one has to check only that it behaves correctly in degree $0$.
\end{proof}

Now assume that $\frak X$ is also proper. Recall Theorem~\ref{ThmB}:

\begin{theorem}\label{ThmBRestated} Let $\frak X$ be a proper smooth formal scheme over $\roi$ with generic fibre $X$. Then
\[
R\Gamma_{A_\inf}(\frak X) = R\Gamma(\frak X,A\Omega_{\frak X})
\]
is a perfect complex of $A_\sub{inf}$-modules, equipped with a $\phi$-linear map $\phi: R\Gamma_{A_\inf}(\frak X)\to R\Gamma_{A_\inf}(\frak X)$ inducing an isomorphism $R\Gamma_{A_\inf}(\frak X)[\tfrac 1\xi]\simeq R\Gamma_{A_\inf}(\frak X)[\tfrac 1{\phi(\xi)}]$, such that all cohomology groups are Breuil--Kisin--Fargues modules. Moreover, one has the following comparison results.
\begin{enumerate}
\item With crystalline cohomology of $\frak X_k$:
\[
R\Gamma_{A_\inf}(\frak X)\dotimes_{A_\inf} W(k)\simeq R\Gamma_\crys(\frak X_k/W(k))\ .
\]
\item With de~Rham cohomology of $\frak X$:
\[
R\Gamma_{A_\inf}(\frak X)\dotimes_{A_\inf} \roi\simeq R\Gamma_\dR(\frak X)\ .
\]
\item With crystalline cohomology of $\frak X_{\roi/p}$:
\[
R\Gamma_{A_\inf}(\frak X)\dotimes_{A_\inf} A_\crys\simeq R\Gamma_\crys(\frak X_{\roi/p}/A_\crys)\ .
\]
\item With \'etale cohomology of $X$:
\[
R\Gamma_{A_\inf}(\frak X)\otimes_{A_\inf} A_\inf[\tfrac 1\mu]\simeq R\Gamma_\sub{\'et}(X,\bb Z_p)\otimes_{\bb Z_p} A_\inf[\tfrac 1\mu]\ .
\]
\end{enumerate}
\end{theorem}

\begin{proof} The comparison results (i), (ii), (iii) and (iv) follow immediately from Theorem~\ref{ThmCRestated}, using Theorem~\ref{thm:etalevsAinfcohom} for part (iv).

Note that $A\Omega_R$ is derived $\xi$-complete for any small affine open $\Spf R\subset \frak X$ by Lemma~\ref{lem:Letapreservecompleteness}; thus, $A\Omega_{\frak X}$ is derived $\xi$-complete, and so is $R\Gamma_{A_\inf}(\frak X)$. Then, to prove that $R\Gamma_{A_\inf}(\frak X)$ is perfect, it is enough to prove that $R\Gamma_{A_\inf}(\frak X)\dotimes_{A_\inf,\theta}\roi$ is perfect, which follows from (ii).

By Proposition \ref{prop:PhionAOmega}, there is a $\phi$-linear quasi-isomorphism
\[
\phi: A\Omega_{\frak X}\cong L\eta_{\tilde\xi} A\Omega_{\frak X}\ ,
\]
inducing in particular a $\phi$-linear map $\phi: A\Omega_{\frak X}\to A\Omega_{\frak X}$. This induces a similar map on $R\Gamma_{A_\inf}(\frak X)$, which becomes an isomorphism after inverting $\tilde\xi=\phi(\xi)$.

It follows that all cohomology groups are finite free after inverting $p$ by Corollary~\ref{cor:crysspecbetter} and comparisons (iii) and (iv), using also Proposition~\ref{prop:BerthelotOgusAcrys}. Thus, all cohomology groups are Breuil--Kisin--Fargues modules.
\end{proof}

\begin{remark}
\label{rmk:TorsionFreedRCrysEquiv}
In the situation of Theorem~\ref{ThmBRestated}, if we fix a cohomological index $i$, then the following conditions are equivalent:
\begin{enumerate}
\item $H^i_\crys(\mathfrak{X}_k/W(k))$ is torsion-free.
\item $H^i_\dR(\mathfrak{X})$ is torsion-free.
\end{enumerate}
Indeed, this follows by combining parts (i) and (ii) of Theorem~\ref{ThmBRestated} with Remark~\ref{rmk:TorsionFreedRCrysEquivAbstract}. The weaker equivalence  where both $H^i$ and $H^{i+1}$ are simultaneously required to be torsion-free can be proven easily using the universal coefficients theorem relating $H^*_\dR(\mathfrak{X})$ and $H^*_\crys(\mathfrak{X}_k/W(k))$ with $H^*_\dR(\mathfrak{X}_k)$. However, for a fixed index $i$ as above, we do not know a direct ``crystalline'' proof of this equivalence.
\end{remark}

Let us now state a version of Theorem~\ref{ThmA} over $C$.

\begin{theorem}\label{ThmAoverC} Let $\frak X$ be a proper smooth formal scheme over the ring of integers $\roi$ in a complete algebraically closed extension of $\bb Q_p$, with residue field $k$; let $X$ be the generic fibre of $\frak X$. Let $i\geq 0$.
\begin{enumerate}
\item There is a canonical isomorphism
\[
H^i_\crys(\frak X_{\roi/p}/A_\crys)\otimes_{A_\crys} B_\crys\cong H^i_\sub{\'et}(X,\bb Z_p)\otimes_{\bb Z_p} B_\crys\ .
\]
It is compatible with the isomorphism
\[
H^i_\crys(X/B_\dR^+)\otimes_{B_\dR^+} B_\dR\cong H^i_\sub{\'et}(X,\bb Z_p)\otimes_{\bb Z_p} B_\dR
\]
via the identification
\[
H^i_\crys(\frak X_{\roi/p}/A_\crys)\otimes_{A_\crys} B_\dR^+\cong H^i_\crys(X/B_\dR^+)\ .
\]
\item For all $n\geq 0$, we have the inequality
\[
\mathrm{length}_{W(k)} (H^i_\crys(\frak X_k/W(k))_\sub{tor}/p^n)\geq \mathrm{length}_{\bb Z_p} (H^i_\sub{\'et}(X,\bb Z_p)_\sub{tor}/p^n)\ .
\]
In particular, if $H^i_\crys(\frak X_k/W(k))$ is $p$-torsion-free, then so is $H^i_\sub{\'et}(X,\bb Z_p)$.
\item Assume that $H^i_\crys(\frak X_k/W(k))$ and $H^{i+1}_\crys(\frak X_k/W(k))$ are $p$-torsion-free. Then one can recover $H^i_\crys(\frak X_k/W(k))$ with its $\varphi$-action from $H^i_\sub{\'et}(X,\bb Z_p)$ with the natural $B_\dR^+$-lattice
\[
H^i_\crys(X/B_\dR^+)\subset H^i_\sub{\'et}(X,\bb Z_p)\otimes_{\bb Z_p} B_\dR\ .
\]
More precisely, the pair of $H^i_\sub{\'et}(X,\bb Z_p)$ and this $B_\dR^+$-lattice give rise to a finite free Breuil--Kisin--Fargues module $\mathrm{BKF}(H^i_\sub{\'et}(X,\bb Z_p))$ by Theorem~\ref{ThmFargues}. Then, assuming only that $H^i_\crys(\frak X_k/W(k))$ is $p$-torsion-free, we have a canonical isomorphism
\[
H^i_{A_\inf}(\frak X)\cong \mathrm{BKF}(H^i_\sub{\'et}(X,\bb Z_p))\ ,
\]
and
\[
H^i_\crys(\frak X_k/W(k))\supset \mathrm{BKF}(H^i_\sub{\'et}(X,\bb Z_p))\otimes_{A_\inf} W(k)\ ,
\]
compatibly with the $\phi$-action. If $H^{i+1}_\crys(\frak X_k/W(k))$ is also $p$-torsion-free, then the last inclusion is an equality.
\end{enumerate}
\end{theorem}

\begin{proof} The isomorphism in part (i) follows from Theorem~\ref{ThmCRestated}. The compatibility with the $B_\dR^+$-lattice $H^i_\crys(X/B_\dR^+)$ amounts to the compatibility between the isomorphisms of Theorem~\ref{thm:cryscomp} and Theorem~\ref{thm:ratpadicHodgeC}, which one checks on the level of the explicit complexes.

For part (ii), we use Theorem~\ref{ThmCRestated}, Lemma~\ref{lem:CrysSpecializationInjects} and Corollary~\ref{cor:ElementaryDivisorsSpecialization} together with the observation that for any injective map $M\hookrightarrow N$ of finitely generated $W(k)$-modules with torsion cokernel $Q$,
\[
\mathrm{length}_{W(k)} (N/p^n)\geq \mathrm{length}_{W(k)} (M/p^n)\ ,
\]
as follows from the exact sequence
\[
\mathrm{Tor}_1^{W(k)}(Q,W(k)/p^n)\to M/p^n\to N/p^n\to Q/p^n\to 0\ ,
\]
and the equality
\[
\mathrm{length}_{W(k)} \mathrm{Tor}_1^{W(k)}(Q,W(k)/p^n) = \mathrm{length}_{W(k)} (Q/p^n)\ ,
\]
which holds for any torsion $W(k)$-module.

For part (iii), we use the equivalence of Theorem~\ref{ThmFargues} together with Corollary~\ref{cor:crysspecbetter}, and the identification of the $B_\dR^+$-lattice in part (i).
\end{proof}

Finally, we can prove Theorem~\ref{ThmA}:

\begin{theorem}\label{ThmARestated} Let $\frak X$ be a proper smooth formal scheme over $\roi$, where $\roi$ is the ring of integers in a complete discretely valued nonarchimedean extension $K$ of $\bb Q_p$ with perfect residue field $k$, and let $i\geq 0$. Let $C$ be a completed algebraic closure of $K$, with corresponding absolute Galois group $G_K$, and let $X/K$ be the rigid-analytic generic fibre of $\frak X$.
\begin{enumerate}
\item There is a comparison isomorphism
\[
H^i_\sub{\'et}(X_C,\bb Z_p)\otimes_{\bb Z_p} B_\crys\cong H^i_\crys(\frak X_k/W(k))\otimes_{W(k)} B_\crys\ ,
\]
compatible with the Galois and Frobenius actions, and the filtration. In particular, $H^i_\sub{\'et}(X_C,\bb Q_p)$ is a crystalline Galois representation.
\item For all $n\geq 0$, we have the inequality
\[
\mathrm{length}_{W(k)} (H^i_\crys(\frak X_k/W(k))_\sub{tor}/p^n)\geq \mathrm{length}_{\bb Z_p} (H^i_\sub{\'et}(X_C,\bb Z_p)_\sub{tor}/p^n)\ .
\]
In particular, if $H^i_\crys(\frak X_k/W(k))$ is $p$-torsion-free, then so is $H^i_\sub{\'et}(X_C,\bb Z_p)$.
\item Assume that $H^i_\crys(\frak X_k/W(k))$ and $H^{i+1}_\crys(\frak X_k/W(k))$ are $p$-torsion-free. Then one can recover $H^i_\crys(\frak X_k/W(k))$ with its $\varphi$-action from $H^i_\sub{\'et}(X_C,\bb Z_p)$ with its $G_K$-action.

More precisely, Theorem~\ref{ThmKisin} associates a finite free Breuil--Kisin module
\[
\mathrm{BK}(H^i_\sub{\'et}(X_C,\bb Z_p))
\]
over $\gS = W(k)[[T]]$ to the lattice $H^i_\sub{\'et}(X_C,\bb Z_p)$ in a crystalline $G_K$-representation. This comes with an identification
\[
\mathrm{BK}(H^i_\sub{\'et}(X_C,\bb Z_p))\otimes_{\gS} B_\crys^+\cong H^i_\crys(\frak X_k/W(k))\otimes_{W(k)} B_\crys^+
\]
by Proposition~\ref{prop:compbkbkf} and part (i). In particular, by extending scalars along $B_\crys^+\to W(\bar{k})[\tfrac 1p]$, we get an identification $\mathrm{BK}(H^i_\sub{\'et}(X_C,\bb Z_p))\otimes_{\gS} W(k)[\tfrac 1p]\cong H^i_\crys(\frak X_k/W(k))[\tfrac 1p]$.

Then
\[
\mathrm{BK}(H^i_\sub{\'et}(X_C,\bb Z_p))\otimes_{\gS} W(k) = H^i_\crys(\frak X_k/W(k))
\]
as submodules of the common base extension to $W(k)[\tfrac 1p]$.
\end{enumerate}
\end{theorem}

\begin{proof} Note that in this situation, there is a canonical section $k\to \roi/p\to \roi_C/p$, so part (i) follows from Theorem~\ref{ThmAoverC} (i) and Proposition~\ref{prop:BerthelotOgusAcrys}. For the compatibility with the filtration, we also use that the isomorphism
\[
H^i_\crys(X_C/B_\dR^+)\otimes_{B_\dR^+} B_\dR\cong H^i_\sub{\'et}(X_C,\bb Z_p)\otimes_{\bb Z_p} B_\dR
\]
from Theorem~\ref{thm:ratpadicHodgeC} is compatible with the isomorphism
\[
H^i_\dR(X)\otimes_K B_\dR\cong H^i_\sub{\'et}(X_C,\bb Z_p)\otimes_{\bb Z_p} B_\dR
\]
from Theorem~\ref{thm:ratPAdicHodge}.

Part (ii) is immediate from Theorem~\ref{ThmAoverC} (ii). Finally, part (iii) follows from Theorem~\ref{ThmAoverC} (iii) and Proposition~\ref{prop:compbkbkf}.
\end{proof}

\begin{remark}
Using Remark~\ref{rmk:TorsionFreedRCrysEquiv}, each torsion-freeness hypothesis on $H^i_\crys(\mathfrak{X}_k/W(k))$ in parts (ii) and (iii) of Theorem~\ref{ThmAoverC} and Theorem~\ref{ThmARestated} can be replaced by the hypothesis that the $\roi$-module $H^i_\dR(\mathfrak{X})$ is torsion-free.
\end{remark}

\newpage

\bibliographystyle{acm}
\bibliography{integralpadicHodge}

\end{document}